\documentclass{amsart}

\usepackage[all]{xy}
\usepackage{mathrsfs}
\usepackage{amssymb}

\newtheorem*{theorem*}{Theorem}
\newtheorem{theorem}[equation]{Theorem}
\newtheorem{corollary}[equation]{Corollary}
\newtheorem{lemma}[equation]{Lemma}
\newtheorem{proposition}[equation]{Proposition}
\newtheorem{conjecture}{Conjecture}
\theoremstyle{definition}
\newtheorem{definition}[equation]{Definition}
\newtheorem{construction}[equation]{Definition}
\theoremstyle{remark}
\newtheorem{remark}[equation]{Remark}
\newtheorem{notation}[equation]{Notation}
\newtheorem{example}[equation]{Example}

\DeclareMathOperator*\tprod{{\textstyle\prod}}
\DeclareMathOperator*\tcoprod{{\textstyle\coprod}}
\DeclareMathOperator\Map{Map}
\DeclareMathOperator\End{End}
\DeclareMathOperator*\colim{colim}
\DeclareMathOperator*\hcolim{hcolim}
\DeclareMathOperator\supp{supp}
\DeclareMathOperator\BAR{B}
\DeclareMathOperator\hatBAR{\hat B}

\DeclareMathOperator\ob{ob}

\newcommand{\out}[1]{}
\newcommand{\cat}{\mathfrak C}

\newcommand{\Oo}{\mathcal P}
\newcommand{\Ms}{\mathscr M}
\newcommand{\Ns}{\mathscr N}

\newcommand{\Bb}{\mathfrak B}
\newcommand{\FF}{\mathfrak F}
\newcommand{\Gg}{\mathcal G}
\newcommand{\Ee}{\mathcal E}
\newcommand{\Ff}{\mathcal F}
\newcommand{\Hh}{\mathbb H}
\newcommand{\Mm}{\mathfrak M}
\renewcommand{\Mc}{\mathcal M}
\newcommand{\MI}{\mathfrak M\mathcal I}

\newcommand{\Rr}{\mathbb R}
\newcommand{\Rc}{\mathcal R}
\newcommand{\Cc}{\mathbb C}

\newcommand{\Zz}{\mathbb Z}

\newcommand{\Ll}{\mathscr L}

\newcommand{\Id}{\mathbf 1}
\newcommand{\ident}{\mathrm{id}}
\newcommand{\Ii}{\mathscr I}
\newcommand{\CP}{\mathbb{P}}
\newcommand{\cp}[1]{\tilde{\mathbb{P}}^2_{\!#1}}
\newcommand{\cpp}[1]{\tilde{\mathbb{P}}^2_{\!\{\!#1\!\}}}
\newcommand{\set}[1]{\{\hspace{-0.1em}#1\hspace{-0.1em}\}}
\newcommand{\Aa}{\mathcal A}

\newcommand{\ds}{\displaystyle}

\newcommand\Top{\underline{\mathrm{Top}}}
\newcommand\hTop{\underline{\mathrm{hTop}}}
\newcommand\xx{\mathbf x}
\newcommand\ff{\mathbf f}
\newcommand\gG{\mathbf f}
\newcommand\reg{\mathrm{reg}}
\newcommand\op{\mathrm{op}}

\newcommand{\tboxplus}{\mathbin{\tilde\boxplus}}
\newcommand{\Cs}{\mathscr C}
\newcommand\Vv{\mathbb V}
\newcommand\Ww{\mathbb W}
\newcommand\Fu{\mathbb U}
\newcommand\bg{\mathrm{big}}

\numberwithin{equation}{section}

\begin{document}

\title[Holomorphic bundles and the bar construction]{Holomorphic bundles on the blown-up plane and the bar construction}

\author{Jo\~ao Paulo Santos}


\address{Centro de An\'alise Matem\'atica, Geometria e Sistemas Din\^amicos
Departamento de Matem\'{a}tica, Instituto Superior T\'{e}cnico\\
Universidade de Lisboa\\
Av.\ Rovisco Pais 1, 1049-001 Lisboa, Portugal}

\email{jsantos@math.tecnico.ulisboa.pt}


\maketitle

\begin{abstract}
  We study the moduli space $\Mm_k^r(\tilde\CP^2_{\!q})$ of rank $r$ holomorphic bundles with trivial determinant and second Chern class $c_2=k$, over the  blowup $\tilde\CP^2_{\!q}$
  of the projective plane at $q$ points, trivialized on a rational curve. We show that, for $k=1,2$, we have a homotopy equivalence between
  $\Mm_k^r(\tilde\CP^2_{\!q})$ and the degree $k$ component of the bar construction $\BAR\bigl(\Mm^r\CP^2,(\Mm^r\CP^2)^{q},(\Mm^r\tilde\CP_{\!1}^2)^{q}\bigr)$.
  The space $\Mm_k^r(\tilde\CP^2_{\!q})$ is isomorphic
  to the moduli space $\MI_k^r(X_q)$ of charge $k$ based $SU(r)$ instantons on a connected sum $X_q$ of $q$ copies of $\overline{\CP^2}$ and we show that, 
  for $k=1,2$, we have a homotopy equivalence between $\MI_k^r(X_q\# X_s)$ and
  the degree $k$ component of $\BAR\bigl(\MI^r(X_q),\MI^r(S^4),\MI^r(X_s)\bigr)$. Analogous results hold in the limit when $k\to\infty$.
  As an application we obtain upper bounds for the cokernel of the Atiyah-Jones map in homology, in the rank-stable limit.
\end{abstract}

\bibliographystyle{amsplain}

\section{Introduction}

Let 
$\CP^2=\Cc^2\cup\CP^1$ be 
the projective plane seen as a
compactification of $\Cc^2$ and denote by $L_\infty\subset\CP^2$ the rational curve at infinity.
Given a finite set $I\subset\Cc^2$, 
let 
$\cp I$ 
be the blowup of $\CP^2$ along $I$.
In this paper we study the  moduli space  $\Mm_I^r$ of rank $r$ holomorphic bundles $\Ee$ on $\cp{I}$
with first Chern class $c_1(\Ee)=0$, trivialized at $L_\infty$. These bundles are topologically classified by their second Chern class $c_2=k$ and
by Bogolomov's inequality \cite[Theorem 4.1]{Kob87} we must have $k\geq0$. We write  $\Mm_I^r=\coprod_{k\geq0}\Mm_{I,k}^r$ where $\Mm_{I,k}^r$
denotes the subspace of bundles with $c_2=k$. 
By collapsing $L_\infty\subset\cp{I}$ to a point we obtain a smooth 4-manifold diffeomorphic to a connected sum of $\#I$ copies of the projective plane
with reversed orientation, which we represent by $\#_I\overline{\CP^2}$.
It was shown in \cite{Buc93} and \cite{Mat00} that
$\Mm_{I,k}^r$ is isomorphic as a real analytic space to the moduli space 
$\MI^{\,r}_{k}(\#_I\overline{\CP^2})$ of instantons with charge $k$ on a $SU(r)$ bundle over $\#_I\overline{\CP^2}$ (a standard reference for instantons is \cite{DoKr90}).

It has long been known \cite{Tau82,Tau84,Don86} that, given 4-manifolds $X$ and $Y$, under rather general hypothesis one can glue together instantons on $X$ and $Y$ to obtain an instanton
on the connected sum $X\#Y$.
When $X=Y=S^4$, this construction gives the moduli space $\MI^{\,r}(S^4)$ the structure of a 
graded homotopy algebra over the $C_4$ operad
\cite[section 6]{BM88}, where $C_4$ denotes the little 4-cubes operad \cite[section 4]{May72}. In particular, $\MI^{\,r}(S^4)$ behaves homologically like an algebra over the $C_4$ operad.
For general $X$ and $Y$ we have a map $\MI_{k_1}^{\,r}(X)\times\MI_{k_2}^{\,r}(Y)\to\MI_{k_1+k_2}^{\,r}(X\#Y)$.
Conversely, by ``stretching the neck'' an instanton $\nabla_{\!A}$ on $X\#Y$
decomposes as an instanton on $X$, an instanton on $Y$ and several instantons (the number depends on $\nabla_{\!A}$) on infinite cylinders \cite[Theorem 6.3.3]{MMR94}, which we may regard as an element in the space
$\MI^{\,r}(X)\times\MI^{\,r}(S^4)^n\times\MI^{\,r}(Y)$ for some $n$. These spaces are the building blocks of the bar construction
$\BAR\bigl(\MI^{\,r}(X),\MI^{\,r}(S^4),\MI^{\,r}(Y)\bigr)$
(see section~\ref{sec:prelim-bar-monoid} below), which suggests that this bar construction should be related to 
the moduli space $\MI^{\,r}(X\#Y)$.
In this paper we investigate this relationship. 

We will state our results in terms of the moduli space $\Mm_I^r$ of holomorphic bundles on $\cp{I}$. Holomorphic bundles on a blowup $\tilde X$ of a complex surface $X$ can be studied using gluing techniques
(see \cite{Buc00}, \cite{San02}, \cite{Gas08}). For $k=2$,
a neighbourhood of the subspace of bundles over $\tilde X$ which are non-trivial on one of the exceptional divisors can be obtained by gluing in ``concentrated'' framed holomorphic bundles on $\CP^2$
\cite[Propositions 4.8, 4.9]{San05},
mimicking the description of a neighbourhood of points at infinity in the compactified moduli space of instantons (see \cite[section 8.2]{DoKr90}).
The bar construction is homeomorphic to the nerve of the open cover thus obtained.

The procedure described above 
for $k=2$ can also be implemented, for any $k$, 
in the limit when the rank $r$ goes to infinity: the gluing maps correspond to Whitney sum. Whitney sum is not strictly associative so some care has to be taken in defining the bar construction.
Our approach is similar to the one in \cite{Ang09}: we replace the simplicial category with 
a topological category whose spaces of morphisms are homotopy equivalent to the discrete spaces of morphisms in the simplicial category.

In finite rank
we define the gluing maps for $k=1,2$ using the monad description of holomorphic bundles introduced in \cite{Don84}, \cite{Kin89}.
Although the same techniques could in principle be used
for $k>2$, there are several complicating factors. For $k=2$ the only non-trivial products we need to consider are of the form
$\Mm_{J,1}^r\times\Mm_{K,1}^r\to\Mm^r_{J\cup K,2}$ and we can further reduce to the case where $\#J,\#K\leq 1$ where we have simple monad descriptions of the moduli spaces.
Also, the gluing maps depend on a choice of disjoint open sets in $\Cc^2$ (of the form $U\times\Cc$). Since
there is never a need to multiply together more than 2 non-trivial factors,
associativity is not an issue so we don't need the little cubes operad. All these simplifying factors fail for $k>2$.

\subsection{Results}

For $k=1,2$ we obtain a description of the moduli space $\Mm_I^r$ in terms of the moduli spaces $\Mm_\emptyset^r$ and $\Mm_{x}^r$, with $x\in I$:
using results in \cite{San05}, we construct, for any
disjoint finite sets $I_1,\dots,I_n\subset\Cc^2$, maps
$\boxplus\colon \Mm_{I_1,k_1}^r\times\dots\times \Mm_{I_n,k_n}^r\to \Mm_{I,k}^r$, where $I=\bigcup_iI_i$ and $k=\sum_ik_i\leq 2$.
Using these maps we build the degree 0, 1 and 2 components of the bar constructions
$\BAR(\Mm_I^r,\Mm_\emptyset^r,\Mm_J^r)$ and $\BAR\bigl(\Mm_\emptyset^r,\smash[b]{\prod\limits_{x\in I}}\Mm_\emptyset^r,\smash[b]{\prod\limits_{x\in I}}\Mm_{x}^r\bigr)$,
which we represent by $\BAR(\cdots)_{\leq2}$.

\begin{theorem}\label{theor0}
Let $I=\{x_1,\dots,x_q\}\subset\Cc^2$. Then:
\begin{enumerate}
\item The map $\boxplus\colon \Mm^r_\emptyset\times\bigl(\prod_i\Mm^r_{x_i}\bigr)\to \Mm^r_I$ induces a map
\[
\textstyle h_\boxplus\colon\BAR\bigl(\Mm^r_\emptyset,\prod_i\Mm^r_\emptyset,\prod_i\Mm^r_{x_i}\bigr)_{\leq2}\to \Mm^r_{I,\leq2}
\]
which is a homotopy equivalence. 
\item If $I=J\cup K$, with $J\cap K=\emptyset$, then
the map $\boxplus\colon \Mm^r_J\times \Mm^r_K\to \Mm^r_{I}$ induces a map
$\BAR(\Mm^r_J,\Mm^r_\emptyset,\Mm^r_K)_{\leq2}\to \Mm^r_{I,\leq2}$ which is a homotopy equivalence. 
\end{enumerate}
\end{theorem}

Direct sum with a trivial rank $r'-r$ bundle induces a map
$\Mm_{I}^r\to\Mm_{I}^{r'}$ and we let $\Mm^\infty_{I}=\colim\Mm_I^r$.
In \cite[Theorem 1.2]{Kir94}, \cite[Theorem 3]{San95}, \cite[Theorem 1.1]{BrSa97} it was shown that we have homotopy equivalences
$\Mm_\emptyset^\infty\simeq\coprod_{k\geq0} BU(k)$ and $\Mm_x^\infty\simeq\coprod_{k\geq0} BU(k)\times BU(k)$ (for $x\in\Cc^2$).
For $J\subset I$, pullback of bundles induces a map $\pi_{J,I}^*:\Mm_J^r\to\Mm_I^r$.
Combining the maps $\pi_{J,I}^*$ with Whitney sum allows us to define, for each $I\subset\Cc^2$, 
a bar construction (see sections~\ref{sec:3} and~\ref{sec:4}), which we denote by:
\[
\|\Bb_I\|=\BAR\Bigl(\Mm_\emptyset^\infty,\textstyle\prod\limits_{x\in I}\Mm_\emptyset^\infty,\prod\limits_{x\in I}\Mm_x^\infty\Bigr)\,,
\]
and a map: 
\[
h_I\colon\|\Bb_I\|\to\Mm_I^\infty\,.
\]
The second Chern class of the bundles gives a grading of the spaces $\|\Bb_I\|$.
In the limit when $r\to\infty$ the maps $\boxplus$ agree with the Whitney sum maps so
Theorem~\ref{theor0} implies that $h_{I}$ is a homotopy equivalence in degrees 1 and 2 (see Theorem~\ref{thm:hI2simeq}).

\begin{conjecture}\label{conjecture}
The map $h_I\colon\|\Bb_I\|\to\Mm_I^\infty$ is a homotopy equivalence.
\end{conjecture}

For disjoint finite sets $I$ and $J$ we have (see Proposition~\ref{prop:IcupJ}):
\begin{equation}\label{eq:barbarbar}
  \BAR\bigl(\|\Bb_I\|,\Mm_\emptyset^\infty,\|\Bb_J\|\bigr)\simeq\|\Bb_{I\cup J}\|\,.
\end{equation}
Assuming Conjecture~\ref{conjecture} holds, it follows that $\Mm^\infty_{I\cup J}\simeq\BAR(\Mm^\infty_I,\Mm^\infty_\emptyset,\Mm^\infty_J)$.
Also, from a finite rank version of equation \eqref{eq:barbarbar} 
(see Proposition~\ref{prop:barbarbardiagramcommutesinfiniterank})
it can be seen that parts (1) and (2) of Theorem~\ref{theor0} are equivalent.

We also show that, in order to
check the conjecture in a given degree $k$, it is enough to show that the map $h_I$  is a homotopy equivalence when $\#I\leq k$:

\begin{theorem}\label{theoI}
If, for every finite set $J\subset I$ with
$\#J\leq k$, the map $h_{J}$ is a homotopy equivalence in degree $k$, then $h_{I}$ is a homotopy equivalence in degree $k$.
\end{theorem}

Note that finite rank versions of Theorem~\ref{theoI}
and equation~\eqref{eq:barbarbar}
(see equation~\eqref{eq:nleqkisenoughinfiniter}
and Proposition~\ref{prop:barbarbardiagramcommutesinfiniterank}) 
imply Theorem~\ref{theor0} for $k=1$.

There is an analog to Theorem~\ref{theor0} 
in the limit when $k\to +\infty$.  
Write $\MI_I^r=\MI^r(\#_I\overline{\CP^2})$ and let
$\mathscr C_{I,k}^r$ denote the space of all connections on an $SU(r)$ bundle over $\#_I\overline{\CP^2}$ with second Chern class equal to $k$ mod
gauge equivalence. This space
is homotopically equivalent to $\Map_*\bigl(\#_I\overline{\CP^2},BSU(r)\bigr)$ \cite[Theorem 1.3]{CoMi94}.
For each degree $k\geq0$ 
we have $\MI_{I,k}^{\,r}\subset \mathscr C_{I,k}^r$
and in \cite{Tau89}, Taubes described, for $k'>k$, 
homotopy equivalences $\mathscr C_{I,k}^r\to\mathscr C_{I,k'}^r$
which restrict to maps
$\MI_{I,k}^{\,r}\to\MI_{I,k'}^{\,r}$,
and showed that ``in the direct limit''
the natural inclusion maps 
\begin{equation}\label{eq:AtiyahJonesmapimath_k}
\imath_k\colon\MI_{I,k}^{\,r}\to\mathscr C_{I,k}^r
\end{equation}
induce a homotopy equivalence 
$\MI_{I,\infty}^{\,r}\xrightarrow{\simeq}\mathscr C_{I,\infty}^r$. 
The $r=+\infty$ case of the following theorem shows that Conjecture~\ref{conjecture}
holds in the limit when $k\to+\infty$:

\begin{theorem}\label{theoIII}
Let $r$ be either a non-negative integer or $r=+\infty$. Then there is a homotopy equivalence:
\[
\BAR\Bigl(\mathscr C_{\emptyset,\infty}^{\,r},\tprod_{x\in I}\mathscr C_{\emptyset,\infty}^{\,r},\tprod\limits_{x\in I}\mathscr C_{x,\infty}^{\,r}\Bigr)\to \mathscr C_{I,\infty}^{\,r}\,.
\]
\end{theorem}

As an application, we study the image in homology of the map
$\imath_k$ in equation~\eqref{eq:AtiyahJonesmapimath_k}
in the limit when $r\to+\infty$.
This map was first studied 
by Atiyah and Jones in \cite{AtJo78} in the finite rank case.
They conjectured that $\imath_k$ induces an isomorphism 
in homology and homotopy through a range $q(k,r)$ that grows with $k$. The original conjecture (over $S^4$) was proven in \cite{BHMM93,
Kir94, Ytian94} and generalizations were proved in \cite{HuMi95, Gas08}. We compute the homology of $\|\Bb_I\|$ and show that the map $\|\Bb_{I,k}\|\to\mathscr C_{I,k}^\infty$
is an isomorphism in homology up to homological dimension $2k+1$. 
As a corollary we show that:

\begin{theorem}\label{theoIV}
Let $h_{I,k}\colon\|\Bb_{I,k}\|\to\Mm_{I,k}^\infty$ 
be the $c_2=k$ component of the map $h_I$. Then
$h_{I,k}$ is injective in homology and,
in homological degree up to $2k+1$, we have:
\begin{enumerate}
\item The inclusion map $\imath_k\colon\MI_{I,k}^{\,\infty}\to\mathscr C_{I,k}^\infty$ 
induces surjective homomorphisms in homology.
\item The map $h_{I,k}$
is an isomorphism in homology up to the range of validity of the Atiyah-Jones conjecture.
\end{enumerate}
\end{theorem}

We also give upper bounds for the cokernel of the map $\imath_k$ in homology up to dimension $4k-2\#I+7$; 
furthermore, since the space $\|\Bb_{I,k}\|$ is simply connected 
(Proposition~\ref{prop:simplyconnected}),
statements analogous to (1) and (2) above also hold in homotopy
(see Theorem~\ref{theoIV}* on page~\pageref{theor:upperboundstocokernel}).

The paper is organized as follows:
In section~\ref{sec:preliminaries} we briefly recall the definitions and basic facts about the bar construction, homotopy colimits and operads.
In section~\ref{sec:2}, we describe the moduli spaces of holomorphic bundles $\Mm_I^r$ and, in the limit when $r\to \infty$, give them
the structure of $E_\infty$-spaces.
In sections~\ref{sec:3} and~\ref{sec:4}, we describe the bar constructions in the limit when $r\to\infty$.
In section~\ref{sec:5}, we prove Theorem~\ref{theoI}.
In section~\ref{sec:glue}, we define the products $\boxplus$ used in the statement of Theorem~\ref{theor0} and prove their properties.
In section~\ref{sec:6}, we prove Theorem~\ref{theor0} and use it to show that $h_{I,k}$ is a homotopy equivalence for $k=1,2$.
In section~\ref{sec:7}, we prove Theorem~\ref{theoIII}.
In section~\ref{sec:8}, we compute the homology of $\|\Bb_I\|$ and prove Theorem~\ref{theoIV}.
In appendix~\ref{app:bar}, we gather some technical proofs concerning the bar construction.

\section{Preliminaries on operads and the bar construction}\label{sec:preliminaries}

The two-sided bar construction for monoids and its generalization for algebras over operads play a central role in our work.
We gather here some terminology and facts which we'll be using throughout the paper.
Let $\Top$ be the category of compactly generated spaces in the sense of \cite{Vog71} 
and let $\hTop$ be the homotopy category. Unless otherwise stated, mapping spaces always have the compact open topology.


\subsection{The bar construction}

In this section we provide some background on the bar construction and homotopy colimits. More details can be found in \cite{HoVo92}.

\subsubsection{Diagrams of spaces}\label{sec:diagramsofspaces}

A \emph{topological category} $\cat$ is a small category with topologized morphism sets such that composition is continuous and
$\ob\cat\subset\text{mor}\,\cat$ is a closed cofibration.
A continuous functor $\cat_1\to\cat_2$ is said to be
an \emph{equivalence of categories} if it is the identity on objects and a homotopy equivalence on the morphism spaces.

A $\cat$-\emph{diagram} of spaces is a continuous functor $X\colon\cat\to\Top$.
A \emph{weak equivalence} between $\cat$-diagrams $X_1$, $X_2$ is a natural transformation $\alpha\colon X_1\to X_2$
such that $\alpha(A)\colon X_1(A)\to X_2(A)$ is a homotopy equivalence for all $A\in\ob\cat$. We say $\alpha$ is a homeomorphism if $\alpha(A)$
is a homeomorphism for all $A\in\ob\cat$.
A functor between topological categories
$i\colon\cat_1\to\cat_2$ induces a pullback of diagrams: given a $\cat_2$-diagram $X\colon\cat_2\to\Top$
we let $i^*X=X\circ i\colon\cat_1\to\Top$.

\begin{example}\label{ex:nerveofopencover}
  Let $\mathcal U=\{U_\alpha\}_{\alpha\in A}$ be an open covering of a space $X$ and for each
  finite subset $\sigma\subset A$ let $U_\sigma=\bigcap_{\alpha\in\sigma}U_\alpha$.
  Let $\cat_{\mathcal U}$ be the category
  whose objects are the non-empty open sets $U_\sigma$ and whose morphisms are the inclusions.
  We refer to the tautological diagram of spaces $N\colon\cat_{\mathcal U}\to\Top$  as the \emph{nerve} of the open cover $\mathcal U$.
\end{example}

\subsubsection{Simplicial spaces}

Let $\Delta$ be the simplicial category, whose objects are the finite ordered sets $[n]=\{0,\dots,n\}$ and whose morphisms
$\Delta(m,n)$ are the order preserving maps $\mu\colon[m]\to[n]$. Let $\Delta^\op$ be the opposite category.
The $\Delta^\op$-diagrams of spaces are called \emph{simplicial spaces}.
Given a simplicial space $X_\bullet\colon\Delta^\op\to\Top$, its \emph{geometric realization} \cite[Definition 11.1]{May72}
is the topological space $|X_\bullet|=\bigl(\coprod_{n\geq0}\Delta\!{}^n\times X_n\bigr)/\sim$,
where $\Delta\!{}^n$ is the standard $n$-simplex and
the equivalence relation is generated by $(\xi\cdot\mu,a)\sim(\xi,\mu\cdot a)$ for every $\xi\in\Delta\!{}^m$, $a\in X_n$ and $\mu\in\Delta(m,n)$.
A natural transformation $X_\bullet\to Y_\bullet$ induces a map $|X_\bullet|\to|Y_\bullet|$.

\begin{example}\label{ex:constantsimplicialspace}
  A space $Y$ can be identified with the simplicial space $Y_\bullet$ with $Y_n=Y$ for all $[n]\in\mathrm{ob}\,\Delta$:
  we have a canonical homeomorphism $|Y_\bullet|\cong Y$ \cite[Lemma 11.8]{May72}.
  Given a simplicial space $X_\bullet$, a sequence of maps
  \[
  f_n\colon X_n\to Y\quad\text{with}\quad f_m(\mu x)=f_n(x)\quad\bigl(x\in X_n,\ \mu\in\Delta(m,n)\bigr)
  \]
  induces a map $|X_\bullet|\to Y$.
\end{example}

\out{
We call a morphism $\mu\in\Delta(m,n)$,                      
a \emph{face map} if $\mu\colon[n]\to[m]$ is injective, 
and a \emph{degeneracy map} if it is surjective. Given a simplicial space 
$X_\bullet$,
we denote by $\|X_\bullet\|$ the quotient of $\bigl(\coprod_{n\geq0}\Delta\!{}^n\times X_n\bigr)$ by the equivalence relation $(\xi\cdot\mu,a)\sim(\xi,\mu\cdot a)$
where $\mu$ is a face map.

\begin{proposition}
	Given a simplicial space $X_\bullet$, 
	there is a simplicial space $\tau X_\bullet$
	weakly equivalent to $X_\bullet$ such that 
	$|\tau X_\bullet|$ is homotopically equivalent to $\|X_\bullet\|$. 
	Furthermore, if $X_\bullet$ is good 
	(in the sense of \cite[Definition A.4]{Seg74}) then 
	we have homotopy equivalences
	$|X_\bullet|\simeq|\tau X_\bullet|\simeq\|X_\bullet\|$.
\end{proposition}
\begin{proof}
	\cite[Propositions A.1(iv) and A.2(v)]{Seg74}
\end{proof}

We call $|\tau X_\bullet|$ the \emph{thick geometric realization} of $X_\bullet$.

}

\subsubsection{The bar construction for monoids}\label{sec:prelim-bar-monoid}

Let $G$ be a topological monoid (that is, $G$ has an associative product with unit) acting on the left and on the right 
on topological spaces $Y$ and $X$ respectively.
The bar construction $\BAR(X,G,Y)$ is the homotopy invariant version of the orbit space
$X\times_G Y$. It is defined as the geometric realization of the simplicial space $\BAR_\bullet(X,G,Y)$ defined on objects by
\[
\BAR_n(X,G,Y)=X\times G^n\times Y
\]
and defined on morphisms as follows: Given $\mu\in\Delta(m,n)$ let $\mu_\alpha=\mu(\alpha)$, with $\alpha=1,\dots,m$; then we can write
\begin{equation}\label{eq:barconstructionformonoids}
X\times G^n\times Y=(X\times G^{\mu_0})\times\left(\prod_{\alpha=1}^mG^{\mu_\alpha-\mu_{\alpha-1}}\right)\times(G^{n-\mu_m}\times Y)
\end{equation}
and the map $X\times G^n\times Y\to X\times G^m\times Y$ is induced by the maps
\[
X\times G^{\mu_0}\to X\,,\quad G^{n-\mu_m}\times Y\to Y\quad\text{and}\quad G^{\mu_\alpha-\mu_{\alpha-1}}\to G\quad(\alpha=1,\dots,m)\,.
\]
defined by the product in $G$ and the actions of $G$ on $X$ and $Y$.

If we let $G^\op$ denote the space $G$ with the opposite monoid structure, then $G^\op$ acts on the left on $X$ and on the right on $Y$ and 
$\BAR(Y,G^\op,X)$ is naturally homeomorphic to $\BAR(X,G,Y)$.

\begin{proposition}\label{prop:B(G,G,Y)=Y}
	The actions $G\times Y\to Y$ and $X\times G\to X$
	induce maps $\BAR(G,G,Y)\to Y$ and $\BAR(X,G,G)\to X$
	respectively, which are homotopy equivalences.
\end{proposition}
\begin{proof}
	
	For the map $\BAR(G,G,Y)\to Y$ see \cite[Proposition 3.1(5)]{HoVo92} 
	(taking into account Remark~\ref{rmk:barformonoidsspecialcase} below). 
	The result for the map
	$\BAR(X,G,G)\to X$ then follows from the homeomorphism $\BAR(X,G,G)\cong \BAR(G^\op,G^\op,X)$.
\end{proof}

One important example is the case where
$X=Y=*$ and $G$ is a group. Then $BG=\BAR(*,G,*)$ is the classifying space of $G$ and $EG=\BAR(G,G,*)\simeq *$ is the universal principal bundle over $BG$.

\subsubsection{The categorical bar construction}

Given a topological category $\cat$, a $\cat$-diagram $Y$ and a $\cat^\op$-diagram $X$,
the bar construction $\BAR(X,\cat,Y)$ is the geometrical realization
of the simplicial space defined on objects by
\begin{multline*}
\BAR_n(X,\cat,Y)\\
=\tcoprod_{A,B\in\cat}
\bigl\{(x;f_n,\dots,f_1;y):\text{$f_n\circ\dots\circ f_1\in\cat(A,B)$ is defined, $y\in Y(A)$, $x\in X(B)$}\bigr\}
\end{multline*}
and defined on morphisms by composition of arrows in $\cat$ and by the evaluation maps
$\cat(A,B)\times Y(A)\to Y(B)$ and
$X(B)\times\cat(A,B)\to X(A)$.

\begin{remark}\label{rmk:barformonoidsspecialcase}
  The bar construction for monoids in section~\ref{sec:prelim-bar-monoid} is the special case when
  the category $\cat$ has only one object; then its space of morphisms is a monoid $G$ and a $\cat$-diagram
  is a space $Y$ with a left action of $G$.
\end{remark}

\out{%
We state two well known properties of the bar construction:

\begin{proposition}\label{prop:propofbarfromHoVo}
\begin{enumerate}
\item Let $\cat\colon\cat^\op\times\cat\to\Top$ be the Hom-functor. Then the evaluation map induces a weak equivalence
\[
B(\cat,\cat,X)\to X\,.
\]
\item Given categories $\cat_1$, $\cat_2$, $\cat$-diagrams $Y_i$ and $\cat^\op$ diagrams $Y_i$, a functor
$F\colon\cat_1\to\cat_2$ and natural transformations $\alpha\colon X_1\to F^*X_2$ and $\beta\colon Y_1\to F^*Y_2$, we have an induced map
\[
\BAR(\alpha,F,\beta)\colon\BAR(X_1,\cat_1,Y_1)\to\BAR(X_2,\cat_2,Y_2)
\]
which is a homotopy equivalence if $F$ is an equivalence and $\alpha$, $\beta$ are weak equivalences.
\end{enumerate}
\end{proposition}
\begin{proof}
  \cite[Proposition 3.1]{HoVo92}
\end{proof}
}%

\subsubsection{Homotopy colimit}\label{sec:hcolim}

Let $X$ be a $\cat$-diagram. The colimit of $X$ is $\colim_\cat X=\Bigl(\coprod_{a\in\ob\cat}X_a\Bigr)/\sim$ where the equivalence relation identifies
$x\sim \mu\cdot x$  for all $x\in X_a$ and $\mu\in\cat(a,b)$. The homotopy colimit is the homotopy invariant version of the colimit.

\begin{definition}\label{def:hcolim}
  Let $\cat$ be a topological category, let $X$ be a $\cat$-diagram
  and let $*$ be the constant $\cat^\op$-diagram
  on a one-point space. The \emph{homotopy colimit} of $X$ along $\cat$ is
  \[
  \hcolim_\cat X=\BAR(*,\cat,X)\,.
  \]
\end{definition}

For convenience, we state here some basic properties of the homotopy colimit which we will be using later:

\begin{proposition}\label{prop:hcolim}\hfill
  \begin{enumerate}
  	\item Given a simplicial space $X$, 
  	there is a simplicial space $\tau X$
  	weakly equivalent to $X$ whose geometric realization
  	$|\tau X|$ is homotopically equivalent to 
  	$\hcolim_{\Delta^{\mathrm{op}}}X$. Furthermore,
    if $X$ is a good simplicial space (in the sense of \cite[Definition A.4]{Seg74}) then $\hcolim_{\Delta^{\mathrm{op}}}X\simeq|X|\simeq|\tau X|$.
  \item Given a $\cat$-diagram $X$ and a topological space $Y$, we have $\hcolim(Y\times X)\cong Y\times\hcolim X$.
  \item If $X_1,\dots,X_n$ are $\cat$-diagrams then $\hcolim_\cat\bigl(\coprod_i X_i\bigr)=\coprod_i\bigl(\hcolim_\cat X_i\bigr)$.
  \end{enumerate}
\end{proposition}
\begin{proof}
  Statement (1) is \cite[Propositions A.1(iv) and A.3]{Seg74}, and
  statements (2) and (3) follow easily from the definition and
  properties of the bar construction 
  \cite[Proposition 3.1]{HoVo92}.
\out{
  \begin{enumerate}                     
  \item Let $\nabla$ be the $\Delta$-diagram which assigns to each $[n]\in\mathrm{ob}\,\Delta$ the $n$-simplex $\Delta^{\!n}$.
  Since $X$ is good, we have $|X|\simeq|\BAR(\Delta^\op,\Delta^\op,X)|$ (see \cite[Proposition A.1(ii)]{Seg74}) and:
  \[
  |\BAR(\Delta^\op,\Delta^\op,X)|\cong\BAR(\nabla,\Delta^\op,X)\simeq\BAR(*,\Delta^\op,X)=\hcolim X\,.
  \]
  \item Identifying $Y$ with $\BAR(*,*,Y)$ we have
  \[
  Y\times\hcolim X=\BAR(*,*,Y)\times\BAR(*,\cat,X)\cong(*\times *,*\times\cat,Y\times X)=\hcolim(Y\times X)\,.
  \]
  \item Let $\mathfrak D$ be the category with $\mathrm{ob}\,\mathfrak D=\{1,\dots,n\}$ and with only the identities as morphisms.
  Then the collection $\{X_i\}$ defines a $\cat\times\mathfrak D$-diagram and the statement follows from the commutation of homotopy colimits
  (see \cite[section 6]{HoVo92}).\qedhere
  \end{enumerate}                                
}
\end{proof}

Given a functor $F\colon\cat\to\mathfrak D$ and an object 
$D\in\mathfrak D$, the undercategory
$D\downarrow F$ is the category whose objects are the pairs $(f,C)$
with $C\in\cat$ and $f\colon D\to F(C)$ and whose morphisms
$p\colon (f,C)\to(f',C')$ are the morphisms $p\colon C\to C'$
in $\cat$ such that $F(p)\circ f=f'$. The overcategory
$F\downarrow D$
is the category whose objects are the pairs $(C,f)$
with $f\colon F(C)\to D$ and whose morphisms
$p\colon (C,f)\to(C',f')$ are the morphisms $p\colon C\to C'$
in $\cat$ such that $f'\circ F(p)=f$. We say
the functor $F$ is \emph{right cofinal} if:
\begin{equation}\label{eq:cofinal}
B(*,D\downarrow F,*)\cong *\quad \text{for all $D\in\mathfrak D$.}
\end{equation}
If $F$ is right cofinal then the induced map $\hcolim_\cat F^*X\to\hcolim_{\mathfrak D}X$
is a homotopy equivalence \cite[Proposition 4.4]{HoVo92}.

     
\subsection{Operads}

We wish to generalize the bar construction of section~\ref{sec:prelim-bar-monoid}
to spaces $A$ which, instead of a strictly associative product, have a family of products with each product associative only up to homotopy.
The technical tool to handle this situation is the theory of operads, of which we give some background in this section.
General treatments of this topic can be found in \cite{MSS02,Fre09}.

Before we define operad we need to introduce some notation about symmetric groups:

\begin{definition}\label{def:blockpermutation}
  Let $\Sigma_n$ denote the symmetric group of permutations of $n$ elements.
  \begin{itemize}
\item Given an $n$-tuple $(a_1,\dots,a_n)$ and $\sigma\in\Sigma_n$ we write \[\sigma(a_1,\dots,a_n)=(a_{\sigma^{-1}(1)},\dots,a_{\sigma^{-1}(n)})\,.\]
\item Let $j_1,\dots,j_k$ be non-negative integers, let $j=j_1+\dots+j_k$
and consider the partition $P$ of $\{1,\dots,j\}$
into $k$ intervals of sizes $j_1,\dots,j_k$.
\begin{itemize}
\item Given $\sigma\in\Sigma_k$, let
$\sigma_{j_1,\dots,j_k}\in\Sigma_{j}$ be the block permutation which permutes the $k$ blocks of the partition $P$ 
in the same way that $\sigma$ permutes the elements of $\{1,\dots,k\}$.
\item Given $\tau_i\in\Sigma_{j_i}$ let $\tau_1\oplus\dots\oplus\tau_k\in\Sigma_j$ be the image 
  of $(\tau_1,\dots,\tau_k)$ under the inclusion $\Sigma_{j_1}\times\dots\times\Sigma_{j_k}\to\Sigma_j$ induced by the partition $P$.
\end{itemize}
\end{itemize}
\end{definition}

An operad $\Oo$ is a sequence of spaces $\Oo(n)$ parametrizing $n$-adic operations $A^n\to A$ on a space $A$.
Note that, given maps $\theta\colon A^k\to A$ and $\theta_i\colon A^{j_i}\to A$, composition defines a map
$\theta\circ(\theta_1,\dots,\theta_k)\colon A^{j_1+\dots+j_k}\to A$.

\begin{definition}\label{def:operad}
  An operad $\Oo$ (in the category $\Top$) is a sequence of spaces $\{\Oo(n)\}_{n\geq0}$ with $\Oo(0)=\{*\}$, together with the following data:
  a right action of the symmetric group $\Sigma_n$ on $\Oo(n)$; a unit $\Id\in\Oo(1)$; and continuous maps
\[
\gamma\colon\Oo(k)\times\Oo(j_1)\times\dots\times\Oo(j_k)\to\Oo(j)\quad(j=j_1+\dots+j_k)
\]
which we also represent by $\gamma(\theta;\theta_1,\dots,\theta_k)=\theta\circ(\theta_1,\dots,\theta_k)$,
satisfying:
\begin{description}
\item[Equivariance] Given $\theta\in\Oo(k)$, $\theta_i\in\Oo(j_i)$, $\sigma\in\Sigma_k$ and $\tau_i\in\Sigma_{j_i}$ we have
\begin{align*}
\theta\circ(\theta_1\tau_1,\dots,\theta_k\tau_k)&=\bigl(\theta\circ(\theta_1,\dots,\theta_k)\bigr)(\tau_1\oplus\dots\oplus\tau_k)\,;\\
(\theta\sigma)\circ(\theta_1,\dots,\theta_k)&=\bigl(\theta\circ\sigma(\theta_1,\dots,\theta_k)\bigr)\sigma_{j_1,\dots,j_k}\,.
\end{align*}
\item[Associativity]
\begin{multline*}
  \theta\circ\bigl(
  \theta_1\circ(\theta_{11},\dots,\theta_{1j_1}),\dots,
  \theta_k\circ(\theta_{k1},\dots,\theta_{kj_k})\bigr)\\
  =\bigl(\theta\circ(\theta_1,\dots,\theta_k)\bigr)\circ
  (\theta_{11},\dots,\theta_{1j_1},
  \dots,\theta_{k1},\dots,\theta_{kj_k})\,.
\end{multline*}
\item[Unit] $\Id\circ\theta=\theta\circ(\Id,\dots,\Id)=\theta$.
\end{description}
A morphism of operads $\Oo_1\to\Oo_2$ is a sequence of equivariant maps $\Oo_1(n)\to\Oo_2(n)$ such that the induced map
$\psi\colon\coprod_n\Oo_1(n)\to\coprod_n\Oo_2(n)$ satisfies $\psi(\Id)=\Id$ and
\[
\psi\bigl(\theta\circ(\theta_1,\dots,\theta_k)\bigr)=(\psi\theta)\circ(\psi\theta_1,\dots,\psi\theta_k)\,.
\]
\end{definition}


\begin{example}\label{ex:endoperad}
Given a based topological space $A$, the sequence $\End_A(n)=\Map(A^n,A)$ defines an operad called the \emph{endomorphism operad} of $A$.
The group $\Sigma_n$ acts by permuting the factors in $A^n$.
\end{example}


\begin{definition}
  An operad $\Oo$ is said to be an $E_\infty$-operad if for each $n$ the space $\Oo(n)$ is contractible and the action of $\Sigma_n$ on $\Oo(n)$ is free.
\end{definition}

An $E_\infty$-operad models multiplications which are associative and commutative up to all higher homotopies.
The following example plays a central role in our work:

\begin{definition}\label{ex:Ll^H}
  A \emph{universe} is a countably infinite dimensional complex Hermitian vector space. Given a universe $\Hh$,
  the \emph{complex linear isometries operad} $\Ll^\Hh$ over $\Hh$ has
  $\Ll^\Hh(n)$ the space of linear isometries $\Hh^n\to\Hh$. The group $\Sigma_n$ acts on $\Ll^\Hh(n)$ by
  $(\phi\sigma)(\mathbf v)=\phi(\sigma\mathbf v)$ where $\phi\in\Ll^\Hh(n)$ and $\mathbf v=(v_1,\dots,v_n)\in\Hh^n$.
\end{definition}

The action of $\Sigma_n$ on $\Ll^\Hh(n)$ is clearly free and the spaces $\Ll^\Hh(n)$ are contractible \cite[Lemma 1.3]{May77} so $\Ll^\Hh$ is an $E_\infty$-operad.



\subsubsection{Algebras over operads}

An algebra $A$ over an operad $\Oo$, also called a $\Oo$-algebra, is a based topological space $A$ together with 
a sequence of equivariant maps
$q\colon\Oo(n)\times A^n\to A$ which we represent by $(\theta,a_1,\dots,a_n)\mapsto\theta\cdot(a_1,\dots,a_n)$,
such that $\Id(a)=a$ and
the following diagram commutes:
\[
\xymatrix@C=0pt{
  &A^{j_1}\times\dots\times A^{j_k}\ar[ld]_(.6){\theta_1\times\dots\times\theta_k}\ar[rd]^(.6){\theta\circ(\theta_1,\dots,\theta_k)}\\
  A^k\ar[rr]^-{\theta}&&A
}
\]
A morphism of $\Oo$-algebras is a map $f\colon A\to B$ such that $f\bigl(\theta\cdot(a_1,\dots,a_k)\bigr)=\theta\cdot\bigl(f(a_1),\dots,f(a_k)\bigr)$.

\subsubsection{Modules over algebras over operads}\label{sec:modoverPalgebras}

Let $\Oo$ be an operad and let $A$ be a $\Oo$-algebra. A left module over $A$ \cite[section 1.6]{GiKa94}
is a topological space $M$
together with maps $q_L\colon\Oo(k+1)\times A^k\times M\to M$ (with $k\geq0$),
equivariant with respect to the action of $\Sigma_k\subset\Sigma_{k+1}$ (the subgroup of permutations preserving $k+1$),
such that $q_L(\Id,m)=m$ and
the following diagram is commutative:
\[
\xymatrix@C=0pt{
  & A^{j_1}\times\dots\times A^{j_k}\times (A^{j_{k+1}}\times M)\ar[ld]_(.6){\theta_1\times\dots\times\theta_k\times\theta_{k+1}\quad}
    \ar[rd]^(.6){\quad\theta\circ(\theta_1,\dots,\theta_k,\theta_{k+1})}\\
  A^k\times M\ar[rr]^-{\theta}&&M
}
\]
The above definition corresponds to the notion of a left module, but, using equivariance, a module $M$ over a $\Oo$-algebra $A$
%
can be equivalently defined 
\cite[Definition 1.1]{BeMo09} as a family of maps
$\Oo(n)\times A^{k-1}\times M\times A^{n-k}\to M$
satisfying suitable unit and associativity relations and inducing a total action
\begin{equation}\label{eq:totalaction}
q_T\colon\Oo(n)\times_{\Sigma_n}\Bigl(\coprod_{k=1}^nA^{k-1}\times M\times A^{n-k}\Bigr)\to M\,.
\end{equation}
In particular, for $k=1$ we have a map 
$\Oo(n)\times M\times A^{n-1}\to M$ giving $M$ the structure of
a right module over the $\Oo$-algebra $A$.


\begin{example}\label{ex:moduleinducedbymap}
Let $\Oo$ be an operad. A morphism of $\Oo$-algebras $f\colon A\to B$ gives $B$ the structure of an $A$-module by defining
\[
q_L(\theta;a_1,\dots,a_k,b_{k+1})=q(\theta;f(a_1),\dots,f(a_k),b_{k+1})\,.
\]
\end{example}

\section{Moduli spaces of holomorphic bundles}\label{sec:2}

In this section we give the moduli space of holomorphic bundles the structure of an
algebra over the linear isometries operad (Definition~\ref{ex:Ll^H}) using the machinery of \cite{BoVo68,May77}. 

\begin{definition}\label{def:Ii}
	Let $\Ii$ denote the \emph{category of linear isometries} whose objects are the finite or countably infinite dimensional
	complex Hermitian vector spaces topologized as the limits of their finite dimensional subspaces and whose morphisms $\Ii(\Vv,\Ww)$
	are the linear isometries $\alpha\colon \Vv\to \Ww$.
	For each $r\geq0$ let $\Ii_r$ be the full subcategory whose objects are the $r$-dimensional complex Hermitian vector spaces and
	let $\Ii_*$ be the union of the categories $\Ii_r$.
\end{definition}

Recall that 
$L_\infty\subset\CP^2$ denotes the rational curve at infinity and we write $\cp{I}$ for the blowup of $\CP^2$ along a finite set
$I\subset\CP^2\setminus L_\infty$. 
We work with the following construction of the moduli space:

\begin{definition}\label{def:ofMm}
	Let $\Vv\in\ob\Ii_r$ and
	let $E\to\cp{I}$ be a rank $r$ smooth complex vector bundle with first Chern class
	$c_1(E)=0$. A holomorphic structure on $E$ is a
	semi-connection $\bar\partial_E:\Omega^0(E)\to\Omega^{0,1}(E)$
	satisfying the integrability condition $\bar\partial_E^{\,2}=0$.
	Let $\mathcal C(I,E,\Vv)$ be the space of pairs $(\bar\partial_E,\phi)$ where
	$\bar\partial_E$ is a holomorphic structure on $E$ and
	$\phi:E|_{L_\infty}\to \Vv\times L_\infty$
	is a holomorphic trivialization. Let $\mathrm{Aut}(E)$ denote the group of smooth bundle automorphisms of $E$.
	Then we let $\Mm(I,E,\Vv)=\mathcal C(I,E,\Vv)/\mathrm{Aut}(E)$.
\end{definition}

In \cite[Theorem 1.1 and Lemma 2.6]{Lub93}, it was shown that
the group $\mathrm{Aut}(E)$ acts freely on
$\mathcal C(I,E,\Vv)$ and the quotient has the
structure of a finite dimensional Hausdorff complex analytic space.

\begin{proposition}\label{prop:M(I,E1,V)congM(I,E2,V)}
	Let $E_1,E_2\to\cp{I}$ be two isomorphic smooth complex vector bundles. Then there is
	a canonical isomorphism $\Mm(I,E_1,\Vv)\cong\Mm(I,E_2,\Vv)$.
\end{proposition}
\begin{proof}
	Given an isomorphism $\psi:E_1\to E_2$ define a map
	$\psi_*:\mathcal C(I,E_1,\Vv)\to\mathcal C(I,E_2,\Vv)$ by
	$\psi_*(\bar\partial,\phi)=(\psi\circ\bar\partial\circ\psi^{-1},\phi\circ\psi^{-1})$.
	This map descends to the quotient to give a homeomorphism
	$\Mm(I,E_1,\Vv)\to\Mm(I,E_2,\Vv)$ which is independent of the choice of isomorphism $\psi$.
\end{proof}

Since the isomorphism class of $E$ is completely determined by $c_2(E)=k$
and $\mathrm{rk}\,E=\dim \Vv$, we will use the notation 
$\Mm_{I,k}^\Vv=\Mm(I,E,\Vv)$.

\begin{definition}\label{def:omega,pi*,Malpha}
	Let  $I\subset\Cc^2$ be a finite set.
	\begin{enumerate}
		\item For each $\Vv\in\ob\Ii_*$ we represent the moduli space of holomorphic bundles on $\cp{I}$ with rank $r=\dim \Vv$ by
		\[
		\Mm_I^\Vv=\coprod\limits_{k=0}^\infty\Mm_{I,k}^\Vv\,,
		\]
		with base point the trivial bundle $\Vv\times\cp{I}\in\Mm_{I,0}^\Vv$.
		\item Let $\alpha\in\Ii_r(\Vv_1,\Vv_2)$. For each rank $r$ smooth bundle $E$
		we define the map $\mathcal C\alpha\colon\mathcal C(I,E,\Vv_1)\to\mathcal C(I,E,\Vv_2)$
		by sending $(\bar\partial_E,\phi)$ to $\bigl(\bar\partial_E,(\alpha\times\ident)\circ\phi\bigr)$.
		These maps descend to the quotient to give a map 
		\[
		\Mm_I^\alpha\colon\Mm_I^{\Vv_1}\to\Mm_I^{\Vv_2}\,.
		\]
		\item Given $\Vv_1,\Vv_2\in\ob\Ii_*$ and smooth bundles $E_1$, $E_2$
		let $\omega\colon\mathcal C(I,E_1,\Vv_1)\times\mathcal C(I,E_2,\Vv_2)\to\mathcal C(I,E_1\oplus E_2,\Vv_1\oplus \Vv_2)$ be the map
		defined by $\omega\bigl((\bar\partial_{1},\phi_1),(\bar\partial_2,\phi_2)\bigr)=
		(\bar\partial_1\oplus\bar\partial_2,\phi_1\oplus\phi_2)$. These maps descends to the quotient to give a map
		\[
		\omega\colon\Mm_I^{\Vv_1}\times\Mm_I^{\Vv_2}\to\Mm_I^{\Vv_1\oplus \Vv_2}\,.
		\]
		\item Given $J\subset I$ let
		$\pi_{J,I}\colon\cp{I}\to\cp{J}$ be
		the blowup of $\cp{J}$ along $I\setminus J$. Then, for each smooth bundle $E$, pullback 
		of holomorphic bundles induces a map
		$\pi_{J,I}^*\colon\mathcal C(J,E,\Vv)\to\mathcal C(I,\pi_{J,I}^*E,\Vv)$.
		These maps descend to the quotient to give a map 
		\[
		\pi_{J,I}^*\colon\Mm_J^\Vv\to\Mm_I^\Vv\,.
		\]
	\end{enumerate}
\end{definition}

\begin{lemma}\label{lemma:closedembedding}
	Let $\Vv,\Ww\in\ob\Ii_*$ and let $*\in\Mm_I^\Ww$ be the basepoint. Then the map
	\[
	\omega\circ(\ident\times *)\colon\Mm_I^{\Vv}\to\Mm_I^{\Vv\oplus \Ww}
	\] 
	is a closed embedding.
\end{lemma}
\begin{proof}
The map $\omega\circ(\ident\times *)$ is given by Whitney sum with the trivial bundle $\Ww\times\cp{I}$.
  For each $k$, the moduli space $\Mm_{I,k}^\Vv$ can be described \cite[page 532]{Buc04} as the space of some matrices $a$, $a_{00}^0$, $a_{00}^1$, $c$, $d$ 
obeying the integrability condition in \cite[equation (3.4)]{Buc04}, modulo the free action
of a certain group $G$ \cite[equation (3.7)]{Buc04}. Here $c\in\text{Hom}(\Cc^k,\Vv)$, $d\in\text{Hom}(\Vv,\Cc^k)$ and the group action on $c$ and $d$ is given
by $c\mapsto ch_{00}$ and $d\mapsto g_{00}d$ with $h_{00},g_{00}\in GL(k)$. 
The holomorphic bundle is given as the cohomology of a monad of the form (see \cite[equation (3.1)]{Buc04}):
\[
0\to \Ee_0\oplus \Ee_1\xrightarrow{A} W_0\oplus\Vv\xrightarrow{B} \mathcal F_0\oplus\mathcal F_1 \to 0
\]
for some bundles $\Ee_i$, $\mathcal F_i$ over $\cp{I}$ and a vector space $W_0$. If $[z_0,z_1,z_2]\in\CP^2$ are homogeneous coordinates then
$A=\bigl(\begin{smallmatrix}\cdot&\cdot\\ cz_2 &0\end{smallmatrix}\bigr)$
  and $B=\bigl(\begin{smallmatrix}\cdot& dz_2\\ \cdot &0\end{smallmatrix}\bigr)$ \cite[equations (27) and (28)]{Hen14}.
It follows that the map $\omega\circ(\ident\times *)$ is induced by
the assignment $(a,a_{00}^0,a_{00}^1,c,d)\mapsto(a,a_{00}^0,a_{00}^1,\imath\circ c,d\circ p)$ where $\imath\colon \Vv\to \Vv\oplus \Ww$ is the inclusion
and $p\colon \Vv\oplus \Ww\to \Vv$ is the projection. So $\omega\circ(\ident\times *)$ embeds $\Mm_{I,k}^{\Vv}$ as a closed submanifold
of $\Mm_{I,k}^{\Vv\oplus \Ww}$.
\end{proof}

\out{%
	\begin{lemma}\label{lemm:L->Mapiscontinuous}
		The assignment $\alpha\mapsto\Mm_I^\alpha$ defines a continuous map between
		$\Ii_*(\Vv,\Ww)$ and the space of maps
		$\Map(\Mm_I^\Vv,\Mm_I^\Ww)$ with the compact open topology.
	\end{lemma}
	\begin{proof}
		Let $\alpha\in\Ii_*(\Vv,\Ww)$. Then $\mathcal C\alpha$ is the map 
		$(\bar\partial_E,\phi)\mapsto(\bar\partial_E,(\alpha\times\ident)\circ\phi)$
		and the result is clear. 
	\end{proof}
}%

Using the terminology of \cite{May77}, the pair $(\Mm_I,\omega)$ 
is an $\Ii_*$-functor. That is:

\begin{proposition}
	Let $\oplus:\Ii_*\times\Ii_*\to\Ii_*$ be the direct sum functor.
	Then the assignments $\Vv\mapsto\Mm_I^\Vv$ and $\alpha\mapsto\Mm_I^\alpha$
	define  a continuous functor $\Mm_I$
	from $\Ii_*$ to $\Top$ and $\omega:\Mm_I\times\Mm_I\to\Mm_I\circ\oplus$
	is a commutative, associative and continuous natural transformation satisfying $\omega(x,*)=x$, 
	(where $*\in\Mm_I^0$ is the basepoint) and such that $\omega\circ(\ident\times *)$ is a closed embedding.
\end{proposition}
\begin{proof}
	The proof is straightforward.
\end{proof}

\begin{proposition}\label{prop:MisI-functor}
	The functor $\Mm_I$ extends to a functor $\Mm_I:\Ii\to \Top$
	and $\omega$ extends to a natural transformation $\Mm_I\times\Mm_I\to\Mm_I\circ\oplus$
\end{proposition}
\begin{proof}
	We can extend $\Mm_I$ to infinite dimensional vector spaces $\Hh$ by letting
	$\Mm_I^\Hh=\colim\Mm_I^\Vv$ where the colimit is taken over the finite dimensional subspaces $\Vv\subset \Hh$.
	See \cite[Proposition 1.9]{May77} for details.
\end{proof}

\begin{proposition}\label{prop:MisaL-algebra}
	Let $\Ll^\Hh$ be the complex linear isometries operad over a universe $\Hh$ (Definition~\ref{ex:Ll^H}). Then
	$\omega$ induces an $\Ll^\Hh$-algebra structure on $\Mm_I^\Hh$.
	Furthermore, a linear isometric isomorphism $\alpha:\Hh_1\to\Hh_2$ induces a map of $\Ll^{\Hh_1}$-algebras
	$\Mm_I^\alpha\colon\Mm_I^{\Hh_1}\to\Mm_I^{\Hh_2}$ which is a homeomorphism. 
\end{proposition}
\begin{proof}
	See \cite[Definition 1.6, Remark 1.7]{May77}.
\end{proof}

Given finite subsets $J\subset I\subset\Cc^2$, the map $\pi_{J,I}^*$ from Definition~\ref{def:omega,pi*,Malpha}(4) passes to the colimit to give a map 
$\pi_{J,I}^*\colon\Mm_J^\Hh\to\Mm_I^\Hh$.

\begin{definition}\label{def:cat}
	Let $\cat$ be the category of finite subsets of $\Cc^2$
	with morphisms the inclusions.
\end{definition}

\begin{proposition}\label{prop:Mm:cat->Top}
	The assignments $I\mapsto\Mm_I^\Hh$ and
	$(J\subset I)\mapsto\pi_{J,I}^*$ define a functor $\Mm$ between $\cat$ 
	and the category of $\Ll^\Hh$-algebras.
\end{proposition}
\begin{proof}
	We can  easily check that $\pi_{J,I}^*$ is a natural transformation
	between the functors $\Mm_I$ and $\Mm_J$
	which commutes with $\omega$. The result follows.
\end{proof}

\begin{remark}\label{rmk:MImoduleoverMJ}
	For $J\subset I$, the space $\Mm_I^\Hh$ has the structure of a module over the $\Ll^\Hh$-algebra $\Mm_J^\Hh$, induced
	by the map $\pi_{J,I}^*\colon\Mm_J^\Hh\to\Mm_I^\Hh$ and the $\Ll^\Hh$-algebra structure on $\Mm_I^\Hh$ (see Example~\ref{ex:moduleinducedbymap}).
\end{remark}

\section{The homotopy coherent bar construction}\label{sec:3}

We wish to generalize the bar construction of section~\ref{sec:prelim-bar-monoid} to the case
where, instead of a monoid $G$ acting on spaces $X$ and $Y$, we have an algebra $A$ over an operad $\Oo$
and modules $M$, $N$ over $A$.

\subsection{Modules}\label{sec:monoidalmodules}

Let $\Oo$ be an operad and let $A$ be a $\Oo$-algebra.
We need to generalize the notion of a (right) module $M$ over $A$ of section~\ref{sec:modoverPalgebras} to the case where the maps
$M\times A^n\to M$ are parametrized, not by $\Oo(n+1)$, but by a certain sequence of spaces $\Ms(n)$.
\out{%
The sequence $\Ms$ will be a right module
over $\Oo$ but, besides the partial composites $\circ_i$
of equation~\eqref{eq:rightmodulepartialcompositions}, we have an extra partial composition
\[
\circ_{k+1}\colon\Ms(k)\times\Ms(j_0)\to\Ms(k+j_0)
\]
which gives $\Ms$ the structure of a graded monoid. Thus we'll call $\Ms$ a monoidal module.
}%
Recall the notation from Definition~\ref{def:blockpermutation}.

\begin{definition}\label{def:monoidalmodule}
  Let $\Oo$ be an operad with composition data $\gamma$ and unit $\Id\in\Oo(1)$.
  A monoidal module $\Ms$ over $\Oo$ is a sequence of spaces $\Ms(n)$ (where $n\geq0$) together with the following data: a right action of $\Sigma_n$
  on $\Ms(n)$; a unit $\Id_\Ms\in\Ms(0)$; and maps
  \[
  \Gamma_R\colon\Ms(k)\times\Ms(j_0)\times\Oo(j_1)\times\dots\times\Oo(j_k) 
  \to\Ms(j)\quad (j=j_0+\ldots+j_{k}),
  \]
  which we also represent by $\Gamma_R(\xi;\xi',\theta_1,\dots,\theta_k)=\xi\circ_R(\xi',\theta_1,\dots,\theta_k)$, satisfying:
\begin{description}
\item[Equivariance] Let $\xi\in\Ms(k)$, $\theta_i\in\Oo(j_i)$, $\xi'\in\Ms(j_{0})$,
$\tau_i\in\Sigma_{j_i}$ and $\sigma\in\Sigma_{k+1}$ with $\sigma(0)=0$. Then
\begin{align*}
\xi\circ_R(\xi'\tau_0,\theta_1\tau_1,\dots,\theta_k\tau_k)&=\bigl(\xi\circ_R(\xi',\theta_1,\dots,\theta_k)\bigr)(\tau_0\oplus\dots\oplus\tau_{k})\\
(\xi\sigma)\circ_R(\xi',\theta_1,\dots,\theta_k)&=\bigl(\xi\circ_R\sigma(\xi',\theta_1,\dots,\theta_k)\bigr)\sigma_{j_0,\dots,j_{k}}
\end{align*}
\item[Associativity]
    \begin{multline*}
    \xi\circ_R\bigl(\xi'\circ_R(\xi'',\theta_{01},\dots,\theta_{0j_0}),
        \theta_1\circ(\theta_{11},\dots,\theta_{1j_1}),\dots,
        \theta_k\circ(\theta_{k1},\dots,\theta_{kj_k})
    \bigr)\\
    =\bigl(\xi\circ_R(\xi',\theta_1,\dots,\theta_k)\bigr)\circ_R(
      \xi'',\theta_{01},\dots,\theta_{0j_0},\dots,\theta_{k1},\dots,\theta_{kj_{k}})
    \end{multline*}
\item[Unit] $\Id_\Ms\circ_R\xi=\xi\circ_R(\Id_\Ms,\Id,\dots,\Id)=\xi$.
\end{description}
We say $\Ms$ is an $E_\infty$-module if the $\Sigma_k$ actions are free and the spaces $\Ms(k)$ are contractible.

A morphism of monoidal modules $\Ms_1\to\Ms_2$ over an operad $\Oo$ is a sequence of equivariant maps $\Ms_1(n)\to\Ms_2(n)$
such that the induced map $\psi\colon\coprod\Ms_1(n)\to\coprod\Ms_2(n)$ satisfies $\psi(\Id_\Ms)=\Id_\Ms$ and
$\psi\bigl(\xi\circ_R(\xi',\theta_1,\dots,\theta_k)\bigr)=(\psi\xi)\circ_R(\psi\xi',\theta_1,\dots,\theta_k)$.
\end{definition}

\begin{remark}
  The name monoidal module comes from $\Ms$ being a right module over $\Oo$ \cite[section 5.1.1]{Fre09} and $\coprod_n\Ms(n)$ being a monoid.
\end{remark}

\begin{example}
  Given spaces $A$ and $M$ the sequence $\End_{A|M}(n)=\Map(M\times A^n,M)$ is a monoidal module over the endomorphisms operad $\End_A$ of Example~\ref{ex:endoperad}.
\end{example}

\out{
\begin{definition}\label{def:pullbackofmonoidalmodules}
Given operads $\Oo_1$, $\Oo_2$, a monoidal module $\Ms$ over
$\Oo_2$ with composition data $\Gamma_R$
and a morphism of operads $\psi\colon\Oo_1\to\Oo_2$,
we represent by $\psi^*\Ms$ the monoidal module over $\Oo_1$
with $\psi^*\Ms(n)=\Ms(n)$ and composition data:
\[
\psi^*\Gamma_R(\xi;\xi',\theta_1,\dots,\theta_n)=
\Gamma_R(\xi;\xi',\psi\theta_1,\dots,\psi\theta_n)\,.
\]
\end{definition}
}

We now generalize the definition of a module over a $\Oo$-algebra from section~\ref{sec:modoverPalgebras}.

\begin{definition}\label{def:MmoduleoverPalgebra}
  Let $\Ms$ be a monoidal module over an operad $\Oo$.
  A $\Ms$-module over a $\Oo$-algebra $A$ is a topological space $M$ together with $\Sigma_k$ equivariant maps 
  $q_R\colon\Ms(k)\times M\times A^k\to M$, 
  such that %
  the following diagram is commutative:
\[
\xymatrix@C=0pt{
  &(M\times A^{j_0})\times A^{j_1}\times\dots\times A^{j_k} 
  =M\times A^{j_0+\dots+j_{k}}
  \ar[rd]^(.6){\quad\xi\circ_R(\xi',\theta_1,\dots,\theta_k)}\ar[ld]_(.6){\xi'\times\theta_1\times\dots\times\theta_k\quad}\\
M\times A^k\ar[rr]^-{\xi}&&M
}\]
\end{definition}

\begin{example}\label{ex:Oo+AsAMonoidalModule}
  Let $\Oo$ be an operad with composition data $\gamma$
  and, for each $n\geq0$, let $\Oo_+(n)=\Oo(n+1)$. Then $\Oo_+$ is a monoidal module over $\Oo$ with data
  $\Gamma_R=\gamma$
  and the $\Oo_+$-modules over a $\Oo$-algebra $A$ are the right modules over $A$ in the sense of section~\ref{sec:modoverPalgebras}.
\end{example}

The main example of a monoidal module in our work is the following:

\begin{example}\label{ex:MainExOfMonoidalModule}
Let $\Ii$ denote the category of linear isometries (Definition~\ref{def:Ii}).
Given universes $\Hh$ and $\Fu$ (Definition~\ref{ex:Ll^H}),
let $\Ms(n)=\Ii(\Hh\oplus \Fu^n,\Hh)$. Then $\Ms(n)$ is a monoidal module over
the linear isometries operad $\Ll^\Fu$. We are interested in the case where $\Fu=\Hh^r$ for some $r>0$. Let $\Oo^r(n)=\prod_{i=1}^r\Ll^\Hh(n)$.
Then $\Oo^r(n)$  is naturally a sub-operad of
$\Ll^{\Hh^r}(n)$ so, in this case, $\Ms$ is a monoidal module over $\Oo^r$. Furthermore, if $A$ in a $\Ll^\Hh$-algebra
then $A^r$ is a $\Oo^r$-algebra and $A$ is a $\Ms$-module over $A^r$.
\end{example}

\out{%
\begin{example}\label{ex:MainExOfMonoidalModule}
  Let $\Oo$ be an operad with composition data $\gamma$, fix an integer $r>0$, let $\Oo^r(n)$ be the cartesian product of $r$ copies of $\Oo(n)$
  and let $\Ms(n)=\Oo(nr+1)$. Then $\Ms$ is a monoidal module over the operad $\Oo^r$. Given a $\sigma\in\Sigma_{n+1}$ which fixes $n+1$, 
  and $\xi\in\Ms(n)=\Oo(nr+1)$, the action of $\sigma$ on $\xi$ is given by $\xi\mapsto\xi\sigma_{r,\dots,r,1}$.
  Composition is given by the composition in $\Oo$:
  \[
  \gamma\colon\Oo(nr+1)\times\Oo(j_1)^r\times\dots\times\Oo(j_k)^r\times\Oo(j_{k+1}r+1)\to\Oo(jr+1)\quad (j=j_1+\dots+j_{k+1}).
  \]
  If $A$ is a $\Oo$-algebra, then $A^r$ is a $\Oo^r$-algebra and the $\Oo$-algebra structure on $A$ makes $A$ into a $\Ms$-module over $A^r$:
  \[
  q_L\colon\Ms(n)\times(A^r)^n\times A=\Oo(nr+1)\times A^{nr+1}\xrightarrow{\ q\ } A\,.
  \]
\end{example}
}%

\begin{remark}\label{rmk:leftmonoidal}
  Definitions~\ref{def:monoidalmodule} and~\ref{def:MmoduleoverPalgebra} correspond to the notion of a right module.
  In a completely analogous way we can also define left monoidal modules $\Ns$ over an operad $\Oo$ and left $\Ns$-modules over a $\Oo$-algebra $A$.
\end{remark}


\out{%
\begin{definition}
  For each $k$ let $\sigma_k\in\Sigma_k$ denote the permutation with $\sigma^k(i)=k+1-i$, $i=1,\dots,k$.
  Given a monoidal module $\Ms$, with data $\Gamma_L$, over an operad $\Oo$ let
  \[
  \Gamma_R\colon\Ms(k)\times\Ms(j_0)\times\Oo(j_1)\times\dots\times\Oo(j_k)\to\Ms(j)\quad(j=j_0+\dots+j_k)
  \]
  be the map defined by 
  \[
  \Gamma_R(\xi;\xi',\theta_1,\dots,\theta_k)= 
  \Gamma_L(\xi;\theta_k\sigma_{j_k},\dots,\theta_1\sigma_{j_1},\xi')\,.
  \]
  We also write $\Gamma_R(\xi;\xi',\theta_1,\dots,\theta_k)=\xi\circ_R(\xi',\theta_1,\dots,\theta_k)$.
  Given a $\Ms$-module $M$ over a $\Oo$-algebra $A$ let
  $q_R\colon\Ms(k)\times M\times A^k\to M$
  be the map defined by
  \[
  q_R(\xi;m,a_1,\dots,a_k)= 
  q_L(\xi;a_k,\dots,a_1,m)\,.
  \]
\end{definition}

From the relations $(\sigma_k)^2=\Id$ and
\[
\sigma_j=(\sigma_k)_{j_1,\dots,j_k}\circ(\sigma_{j_1}\oplus\dots\oplus\sigma_{j_k})\qquad(j=j_0+\dots+j_k)
\]
it is straightforward to see that $\Gamma_R$ and $q_R$ satisfy relations analogous to the ones in Definitions~\ref{def:monoidalmodule} and~\ref{def:MmoduleoverPalgebra}.
}%

\out{%
However a $\Ms$-module $M$ over a $\Oo$-algebra $A$ can be seen as a bimodule:
we can define $q_L\colon\Ms(k)\times M\times A^k\to M$ by
$q_R(\xi;m,a_1,\dots,a_k)=q_L(\xi;a_1,\dots,a_k,m)$ and we easily check that we have a commutative diagram
\[
\xymatrix@C=0pt{
  &(M\times A^{j_{k+1}})\times A^{j_1}\times\dots\times A^{j_k}
  \ar[rd]^(.6){\quad\xi\circ_R(\xi',\theta_1,\dots,\theta_k)}\ar[ld]_(.6){\xi'\times\theta_1\times\dots\times\theta_k\quad}\\
M\times A^k\ar[rr]^-{\xi}&&M
}\]
}%


\out{%
\begin{example}\label{ex:Oo+AsAMonoidalModule}
  Let $\Oo$ be an operad with data $\gamma$
and, for each $n\geq0$, let $\Oo_+(n)=\Oo(n+1)$. Then $\Oo_+$ is a monoidal module over $\Oo$ with data 
$\Gamma_L=\gamma$
and the $\Oo_+$-modules over a $\Oo$-algebra $A$ are the modules over $A$ in the sense of section~\ref{sec:modoverPalgebras}.
  If 
$\theta\in\Oo(k+1)$, $\theta'\in\Oo(j_0)$, $\theta_i\in\Oo(j_i)$ and $j=j_0+\dots+j_k$ then
  \[
  \Gamma_R(\theta;\theta',\theta_1,\dots,\theta_k)=\gamma(\theta\sigma_{k+1};\theta'\sigma_{j_0},\theta_1,\dots,\theta_k)\sigma_j
  \]
and the  left and right actions are related to the total action of equation~\eqref{eq:totalaction} by
  \begin{align*}
  q_L(\theta;a_1,\dots,a_k,m)&=q_T(\theta;a_1,\dots,a_k,m)\\
  q_R(\theta;m,a_1,\dots,a_k)&=q_T(\theta\sigma_{k+1};m,a_1,\dots,a_k)
  \end{align*}
\end{example}
}%

\out{%
\begin{example}\label{ex:M(a,b,c)}
  In appendix~\ref{app:bar} we will use the following constructions.
  Let $\Hh$ be a universe and given non-negative integers $r$, $s$ let $\Mm(a,b,c)=\Ii\bigl((\Hh^r)^a\oplus\Hh\oplus\Hh^b\oplus\Hh\oplus(\Hh^s)^c,\Hh\bigr)$,
  where $a,c\geq0$ and $b\geq-2$.
  For $b$, $c$ fixed, $\Mm$ is a monoidal module over $\Ll^{\Hh^r}$, for $a$ and $b$ fixed $\Mm$ is a monoidal module over $\Ll^{\Hh^s}$ and for $a$ and $c$ fixed
  $\Oo(b)=\Mm(a,b,c)$ satisfies all properties of an operad except that $\Oo(0)\neq\{*\}$. If $A$ is a $\Ll^\Hh$ algebra then we have maps
  \[
  \Mm(a,b,c)\times(A^r)^a\times A\times A^b\times A\times (A^s)^c\to A\,.
  \]
  Furthermore, let $\Ms(n)=\Ii(\Hh\oplus \Fu^n,\Hh)$ (see Example~\ref{ex:MainExOfMonoidalModule}). Then 
\end{example}
}%

\out{%
\begin{example}\label{example:L(W+nV,W)}
Given countably infinite dimensional complex Hermitian vector spaces $\Vv$ and $W$,
let $\Ll^\Vv$ be the linear isometries operad over $\Vv$ and
let $\Ll^{\Vv}_W(n)$ be the space of linear isometries from $W\oplus \Vv^n$ to $W$. Then $\Ll^{\Vv}_W$ is a monoidal module over $\Ll^\Vv$.
Notice that, when $\Vv=W$, we have $\Ll_W^\Vv=\Ll^\Vv_+$.
\end{example}
}%

\subsection{The bar construction}\label{subsec:defofbar}

To define the homotopy coherent bar construction we replace the simplicial category $\Delta$ with a new category,
obtained by expanding the spaces of morphisms to include the operad data. Given an operad $\Oo$, write $\Oo_+(n)=\Oo(n+1)$ (see Example~\ref{ex:Oo+AsAMonoidalModule})

\begin{construction}\label{con:catDelta(P)}
  Let $\Oo$ be an operad with data $\gamma$ and let $\Ms$ be a monoidal module over $\Oo$ with composition data $\Gamma_R$.
  We represent by $\Delta(\Ms,\Oo)$ the category with the same objects as the simplicial category $\Delta$ and
  whose morphisms are defined as follows:
  For each morphism $\mu\in\Delta(m,n)$ let
  \[
  \Delta(\Ms,\Oo)(\mu)=\Ms(\mu_0)\times
  \left(\prod_{\alpha=1}^m\Oo(\mu_\alpha-\mu_{\alpha-1})\right)\times
  \Oo_+(n-\mu_m)\,;
  \]
  Then, the space of morphisms is defined to be:
  \[
  \Delta(\Ms,\Oo)(m,n)=\coprod_{\mu\in\Delta(m,n)}\Delta(\Ms,\Oo)(\mu)\,.
  \]
  Let $\mu\in\Delta(m,n)$ and $\nu\in\Delta(n,p)$.
  Composition of morphisms: $\Delta(\Ms,\Oo)(\mu)\times\Delta(\Ms,\Oo)(\nu)\to\Delta(\Ms,\Oo)(\nu\circ\mu)$
  is defined using the operad data:
  \begin{align*}
    \Gamma_R&:\Ms(\mu_0)\times\Ms(\nu_0)\times\prod_{\beta=1}^{\mu_0}\Oo(\nu_\beta-\nu_{\beta-1})
    \to\Ms(\nu_{\mu_0})\,;
    \\
    \gamma&:\Oo(\mu_\alpha-\mu_{\alpha-1})\times\prod_{\beta=\mu_{\alpha-1}+1}^{\mu_{\alpha}}\Oo(\nu_{\beta}-\nu_{\beta-1})
    \to\Oo(\nu_{\mu_\alpha}-\nu_{\mu_{\alpha-1}})\,;
    \\
    \gamma&:\Oo_+(n-\mu_m)\times\biggl(\,\prod_{\beta=\mu_m+1}^{n}\Oo(\nu_\beta-\nu_{\beta-1})\biggr)\times\Oo_+(p-\nu_n)
    \to\Oo_+(p-\nu_{\mu_m})\,.
  \end{align*}
  From the associativity of the operad data
  it is straightforward to prove that this composition law is associative.
\end{construction}

\begin{remark}
  The above construction could also be done replacing $\Oo_+$ by an arbitrary (left) monoidal module but we won't need that generality.
\end{remark}

We now define the bar construction. Comparing with the bar construction for monoids in section~\ref{sec:prelim-bar-monoid},
the category $\Delta(\Ms,\Oo)$ replaces the simplicial category $\Delta$ and geometric realization is replaced by the homotopy colimit
(section~\ref{sec:hcolim}). By analogy with the notation for geometric realization, we often represent the homotopy colimit of
a $\Delta(\Ms,\Oo)^\op$-diagram $X$ by
\[
\hcolim_{\Delta(\Ms,\Oo)^\op}X=\|X\|\,.
\]

\begin{construction}\label{con:barconstruction}
  Let $\Oo$, $\Ms$ be as in Definition~\ref{con:catDelta(P)}.
  Let $A$ be a $\Oo$-algebra, let $M$ be a right $\Ms$-module over $A$ and let $N$ be a left module over $A$.
  We represent by $\Bb_\bullet(M,A,N)\colon\Delta(\Ms,\Oo)^\op\to \Top$
  the functor defined on objects by
  \[
  \Bb_n(M,A,N)= M\times A^n\times N
  \]
  and defined on morphisms as follows:
  Given $\mu\in\Delta(m,n)$ 
  we can write (compare with equation~\eqref{eq:barconstructionformonoids}):
  \[
  M\times A^n\times N
  =\left(M\times A^{\mu_0}\right)\times
  \left(\prod_{\alpha=1}^n A^{\mu_\alpha-\mu_{\alpha-1}}\right)\times
  \left( A^{n-\mu_m}\times N\right)\,.
  \]
  Then the map $\Delta(\Ms,\Oo)(\mu)\times M\times A^n\times N\to M\times A^m\times N$ is induced by the maps
  \begin{align*}
    q_R&\colon\Ms(\mu_0)\times M\times A^{\mu_0}\longrightarrow M\,;\\
    q&\colon\Oo(\mu_\alpha-\mu_{\alpha-1})\times A^{\mu_\alpha-\mu_{\alpha-1}}\longrightarrow A\,;\\
    q&\colon\Oo_+(n-\mu_m)\times A^{n-\mu_m}\times N\longrightarrow N\,.
  \end{align*}
  We define the bar construction by taking the homotopy colimit:
  \begin{equation}\label{eq:defofbarhcolim}
  \Bb(M,A,N)=\hcolim_{\Delta(\Ms,\Oo)^\op}\Bb_\bullet(M,A,N)\,.
  \end{equation}
\end{construction}

\subsection{Maps from the bar construction}\label{sec:mapsfromthebarconstruction}

If $G$ is a monoid acting on spaces $X$ and $Z$, we can build a map from $\BAR(X,G,Z)$ to a space $Y$ from a sequence of maps $X\times G^n\times Z\to Y$ giving a map to the constant diagram $Y_\bullet$
(see Example~\ref{ex:constantsimplicialspace}). When $X=G$ and $Y=Z$ this map is a homotopy equivalence \cite[Proposition 3.1(5)]{HoVo92}.
In this section we generalize these results to the bar construction of section~\ref{subsec:defofbar}.

Let $\Oo$, $\Ms$, and $A$, $M$, $N$ be as in Definitions~\ref{con:catDelta(P)} and~\ref{con:barconstruction}.
We wish to define maps $\Bb(M,A,N)\to Y$ for some space $Y$.
The idea is to replace $Y$ with a homotopically equivalent space. We will need, for each $n\geq0$, spaces $\mathcal Y(n)$ parametrizing maps $M\times A^n\times N\to Y$.
Although Definition~\ref{con:widetildedelta(P)} below could be made for any spaces $\mathcal Y(n)$ satisfying suitable associativity relations, in this paper we will
confine ourselves to the case where $\mathcal Y(n)=\Ms(n+1)$.

\begin{definition}\label{def:widetildeDelta}
  Let $\widetilde\Delta$ denote the category whose objects are the ordered sets
  $[n]=\{0,1,\ldots,n\}\subset\Zz$ plus the empty set, which we represent by $[-1]\in\widetilde\Delta$, and whose morphisms are the
  order preserving maps.
\end{definition}

The simplicial category $\Delta$ is a full subcategory of $\widetilde\Delta$. We will define a category
$\widetilde\Delta(\Ms,\Oo)$ equivalent to $\widetilde\Delta$ which contains $\Delta(\Ms,\Oo)$ as a full subcategory.

\begin{construction}\label{con:widetildedelta(P)}
\hfill
\begin{enumerate}
\item Given a monoidal module $\Ms$ over an operad $\Oo$, let $\Oo_+(n)=\Oo(n+1)$ and let $\Ms_+(n)=\Ms(n+1)$.
We define a category $\widetilde\Delta(\Ms,\Oo)$ with the same objects as $\widetilde\Delta$ and morphisms defined as follows:
For $m,n\neq[-1]$, the spaces of morphisms coincide with those of $\Delta(\Ms,\Oo)$ and we let
$\widetilde\Delta(\Ms,\Oo)(-1,n)=\Ms_+(n)$. Given $m,n\geq 0$ and $\mu\in\Delta(m,n)$,
composition of morphisms 
\[
\widetilde\Delta(\Ms,\Oo)(-1,m)\times\widetilde\Delta(\Ms,\Oo)(\mu)\to
\widetilde\Delta(\Ms,\Oo)(-1,n)
\]
is defined using the operad data:
\[
\Gamma_R:\Ms_+(m)\times\Ms(\mu_0)\times\left[\prod_{\alpha=1}^m\Oo(\mu_\alpha-\mu_{\alpha-1})\right]\times
\Oo_+(n-\mu_m)\to\Ms_+(n)\,.
\]
\item Let $A$ be a $\Oo$-algebra, let $M$ be a right $\Ms$-module over $A$ and let $N$ be a left module over $A$.
  Let $Y$ be a topological space and suppose we have, for each $n\geq 0$, maps 
  \begin{equation}\label{eq:mMs(n)xMxA^nxN->Y}
  Q\colon\Ms_+(n)\times M\times A^n\times N\to Y
  \end{equation}
  equivariant with respect to the $\Sigma_n$-action on $\Ms_+(n)$ and $A^n$,
  and such that
  the following diagram commutes (where $j=j_0+\dots+j_{k+1}$):
  \[
  \xymatrix@C=-5em{
    &(M\times A^{j_0})\times A^{j_1}\times\dots\times A^{j_k}\times(A^{j_{k+1}}\times N)=M\times A^j\times N
    \ar[rd]^(0.6){\qquad\xi\circ_R(\xi_0,\theta_1,\dots,\theta_k,\theta_{k+1})}\ar[ld]_(0.6){\xi_0\times\theta_1\times\dots\times\theta_{k+1}\qquad}\\
    M\times A^k\times N\ar[rr]^-{\xi}&&Y
  }\]
Let $\widetilde\Bb_\bullet(M,A,N;Y)\colon\widetilde\Delta(\Ms,\Oo)^\op\to\Top$ 
be the functor extending $\Bb_\bullet(M,A,N)$, sending the object 
$[-1]$ to $Y$ and defined on the remaining morphisms by the map in equation~\eqref{eq:mMs(n)xMxA^nxN->Y}.
We let
\begin{equation}\label{eq:tildebarmany}
\widetilde\Bb(M,A,N;Y)=\hcolim\limits_{\Delta(\Ms,\Oo)^\op}\widetilde\Bb_\bullet(M,A,N;Y)\,.
\end{equation}
\end{enumerate}
\end{construction}

\begin{proposition}\label{prop:B->Y}
  The inclusion $[-1]\to\widetilde\Delta$
induces a map $Y\to\widetilde\Bb(M,A,N;Y)$ which  
is a homotopy equivalence. 
\end{proposition}
\begin{proof}
This follows from the cofinality theorem \cite[Proposition~4.4]{HoVo92}
since $[-1]$ is a homotopy initial object.
\end{proof}

\begin{definition}\label{def:themaph}
We represent by $h_{M,A,N;Y}\colon\Bb(M,A,N)\to Y$ the map in the homotopy category $\hTop$ determined by the diagram:
\[
\xymatrix{
\Bb(M,A,N)\ar[r]&\widetilde\Bb(M,A,N;Y)&Y\,.\ar[l]_-{\simeq}
}\]
where the first map is determined by the functoriality 
of the homotopy colimit.
\end{definition}

If $G$ is a monoid acting on the right on a space $Y$ and on the left on itself, then $\BAR(Y,G,G)\simeq Y$ \cite[Proposition 3.1(5)]{HoVo92}.
The analogous result for $\Bb$ is the following:

\begin{proposition}\label{prop3.1}
Let $\Ms$ be a monoidal module over an operad $\Oo$ and let $M$ be a $\Ms$-module over a $\Oo$-algebra $A$.
Then the map $h_{M,A,A;M}\colon\Bb(M,A,A)\to M$ is a homotopy equivalence.
Furthermore, if $\Ms=\Oo_+$ and $N$ is a left module over $A$ then
the map $h_{A,A,N;N}\colon\Bb(A,A,N)\to N$ is a homotopy equivalence.
\end{proposition}
\begin{proof}
The statement for $h_{M,A,A;M}$
follows immediately from Lemma~\ref{theo3.1} in appendix~\ref{app:bar}
and the proof for $h_{A,A,N;N}$ is completely analogous.
\end{proof}

\subsection{Morphisms}

If, for each $i=1,2$, we have a monoid $G_i$ acting on spaces $X_i$ and $Z_i$, then a
homomorphism $f\colon G_1\to G_2$ and equivariant maps $f_X\colon X_1\to X_2$ and $f_Z\colon Z_1\to Z_2$
induce a natural transformation $\BAR_\bullet(X_1,G_1,Z_1)\to\BAR_\bullet(X_2,G_2,Z_2)$
and hence a map between the bar constructions which is a homotopy equivalence if $f$, $f_X$ and $f_Z$ are 
homotopy equivalences \cite[Proposition 3.1(6)]{HoVo92}. 
In this section we prove analogous results for the homotopy coherent bar construction.

Suppose $\psi\colon\Oo_1\to\Oo_2$ is a morphism of operads, 
	and $\Ms$ is a monoidal module over
	$\Oo_2$ with composition data $\Gamma_R$.
	We represent by $\psi^*\!\Ms$ the monoidal module over $\Oo_1$
	with $\psi^*\!\Ms(n)=\Ms(n)$ and composition data:
	\[
	\psi^*\Gamma_R(\xi;\xi',\theta_1,\dots,\theta_n)=
	\Gamma_R(\xi;\xi',\psi\theta_1,\dots,\psi\theta_n)\,.
	\]
For each $i=1,2$ 
let  $\Ms_i$ be a monoidal module over $\Oo_i$
and let $\psi_{\!{}_\Ms}\colon\Ms_1\to \psi^*\!\Ms_2$ be a morphism of monoidal modules over $\Oo_1$.
Then $\psi$, $\psi_{\!{}_\Ms}$ induce a functor
$F\colon\Delta(\Ms_1,\Oo_1)\to\Delta(\Ms_2,\Oo_2)$
and given any $\Delta(\Ms_2,\Oo_2)$-diagram $\Bb$ we have a pullback
$\Delta(\Ms_1,\Oo_1)$-diagram $F^*\Bb=F\circ\Bb$.
Now suppose $A$ is a $\Oo_2$-algebra and $M$ and $N$  are,
respectively, a right $\Ms_2$-module over $A$ and a left module over $A$. 
Then we have a $\Oo_1$-algebra $\psi^*A$, a right $\Ms_1$-module
$\psi^*M$ and a left module $\psi^*N$ over $\psi^*A$
defined in the obvious way. 

\begin{proposition}\label{prop:MapsFBetweenBars}
We have an isomorphism of 
$\Delta(\Ms_1,\Oo_1)^{\mathrm{op}}$-diagrams
\[
\Bb_\bullet(\psi^*M,\psi^*A,\psi^*N)\cong F^*\Bb_\bullet(M,A,N)
\]
and if $\Oo_i$, $\Ms_i$ are $E_\infty$ then the functor $F$ induces a homotopy equivalence
\[
\hcolim_{\Delta(\Ms_1,\Oo_1)^{\mathrm{op}}}F^*\Bb_\bullet(M,A,N)\xrightarrow{\ \simeq\ }\hcolim_{\Delta(\Ms_2,\Oo_2)^{\mathrm{op}}}\Bb_\bullet(M,A,N)\,.
\]
\end{proposition}
\begin{proof}
The isomorphism of 
$\Delta(\Ms_1,\Oo_1)^{\mathrm{op}}$-diagrams can easily be checked by direct inspection. The homotopy equivalence
follows from  \cite[Proposition 3.1(6)]{HoVo92}, since $F$ is an equivalence of categories.
\end{proof}

Let $\Oo$ be an operad and let $\Ms$ be a monoidal module over $\Oo$. For each $i=1,2$ let $A_i$ be a $\Oo$-algebra, let $M_i$ be a $\Ms$-module over $A_i$ and let
$N_i$ be a module over $A_i$. Let
    \[
    f\colon A_1\to A_2\,,\quad f_M\colon M_1\to f^*M_2\,,\quad f_N\colon N_1\to f^*N_2
    \]
    be, respectively, a morphism of $\Oo$ algebras, a morphism of $\Ms$-modules over $A_1$ and a morphism of modules over $A_1$.
    Also suppose we have, for each $i=1,2$, maps $Q_i\colon\Ms_+(k)\times M_i\times A_i^k\times N_i\to Y_i$ satisfying the associativity conditions in Definition~\ref{con:widetildedelta(P)} and
    let $f_Y\colon Y_1\to Y_2$ be a map such that  
\[
f_Y\bigl(Q_1(\xi;m,a_1,\dots,a_k,n)\bigr)=Q_2\bigl(\xi;f_M(m),f(a_1),\dots,f(a_k),f_N(n)\bigr)\,.
\]
    Then $f$, $f_M$, $f_N$ and $f_Y$ induce natural transformations, which we represent by
    \begin{equation}\label{eq:MapsBetweenBars}
\begin{aligned}
    \Bb_\bullet(f_M,f,f_N)&\colon\Bb_\bullet(M_1,A_1,N_1)\to\Bb_\bullet(M_2,A_2,N_2)\,;\\
    \widetilde\Bb_\bullet(f_M,f,f_N;f_Y)&\colon\widetilde\Bb_\bullet(M_1,A_1,N_1;Y_1)\to\widetilde\Bb_\bullet(M_2,A_2,N_2;Y_2)\,.
\end{aligned}
    \end{equation}


\begin{proposition}\label{prop:morphismsbetweenbarsarehtpyequiv}
  If $\Oo$ and $\Ms$ are $E_\infty$ and $f$, $f_M$ and $f_N$ are homotopy equivalences, then $\Bb(f_M,f,f_N)$ is a homotopy equivalence.
\end{proposition}
\begin{proof}
  This follows immediately from the homotopy invariance of the homotopy colimit \cite[Proposition 4.4]{HoVo92}
\end{proof}

\section{The spaces $\|\Bb_I\|$}\label{sec:4}

Let $\Ll^\Hh$ be the linear isometries operad (Definition~\ref{ex:Ll^H}) and let $\Ii$ be the category of linear isometries (Definition~\ref{def:Ii}).
By Proposition~\ref{prop:MisaL-algebra}, the spaces $\Mm_I^\Hh$ are $\Ll^{\Hh}$-algebras.

\begin{definition}\label{def:LlI}
Given a finite set $I\subset\Cc^2$
let $\Hh^I=\smash[b]{\bigoplus\limits_{x\in I}}\Hh$ and,
for each non-negative integer $n$, let
\begin{align*}
\Ll_I^\Hh(n)&=\Ii\bigl(\Hh\oplus(\Hh^I)^n,\Hh\bigr)\,,& \Ll^{\Hh,I}(n)&=\prod_{x\in I}\Ll^\Hh(n)\,.
\end{align*}
Also, let $\Ll_{I+}^\Hh(n)=\Ll_I^\Hh(n+1)$ and $\Ll_+^{\Hh,I}(n)=\Ll^{\Hh,I}(n+1)$.
\end{definition}

In Example~\ref{ex:MainExOfMonoidalModule} we observed that $\Ll^\Hh_I$ is a monoidal module over $\Ll^{\Hh,I}$ and
$\Mm^\Hh_\emptyset$ is a $\Ll_I^\Hh$-module over the $\Ll^{\Hh,I}$-algebra $\prod_{x\in I}\Mm^\Hh_\emptyset$.
%
%
The space $\prod_{x\in I}\Mm^\Hh_x$ is also a $\Ll^{\Hh,I}$-algebra, and the pullback maps $\pi_{\emptyset,x}\colon\Mm^\Hh_\emptyset\to\Mm^\Hh_x$
make $\prod_{x\in I}\Mm^\Hh_x$ into a module over 
$\prod_{x\in I}\Mm^\Hh_\emptyset$ (Example~\ref{ex:moduleinducedbymap}).
Next we apply Definition~\ref{con:widetildedelta(P)}(2) with $Y=\Mm_I^\Hh$. We need a map 
\[
Q\colon\Ll_{I+}^\Hh(n)\times \Mm^\Hh_\emptyset\times\Bigl(\prod_{x\in I}\Mm^\Hh_\emptyset\Bigr)^n\times\prod_{x\in I}\Mm^\Hh_x\to\Mm_I^\Hh
\]
satisfying the associativity conditions in Definition~\ref{con:widetildedelta(P)}. We define this map,
for each $f\in\Ll_{I+}^\Hh(n)$ by the composition:
\[\textstyle
\Mm_\emptyset\times\left(\prod\Mm_\emptyset\right)^n\times\left(\prod\Mm_x\right)\xrightarrow{\pi^*}
\Mm_I\times\left(\prod\Mm_I\right)^n\times\left(\prod\Mm_I\right)=\Mm_I\times\left(\prod\Mm_I\right)^{n+1}\xrightarrow{f}\Mm_I\,,
\]
where the map $\pi^*$ is induced by pullback.

\begin{definition}\label{def:Delta_I,Bb_I,h_I}
We write:
\begin{align*}
  \Delta_I^\Hh&=\Delta(\Ll_I^\Hh,\Ll^{\Hh,I})\,; &
  \widetilde\Delta_I^\Hh&=\widetilde\Delta(\Ll_I^\Hh,\Ll^{\Hh,I})\,; \\
  \Bb_I^\Hh&=\Bb_\bullet\bigl(\Mm_\emptyset^\Hh,\tprod_{x\in I}\Mm_\emptyset^\Hh,\tprod_{x\in I}\Mm_x^\Hh\bigr)\,; &
  \widetilde\Bb_I^\Hh&=\widetilde\Bb_\bullet\bigl(\Mm_\emptyset^\Hh,\tprod_{x\in I}\Mm_\emptyset^\Hh,\tprod_{x\in I}\Mm_x^\Hh;\Mm_I^\Hh\bigr)\,.
\end{align*}
Also, we write $\|\Bb_I^\Hh\|=\hcolim\Bb_I^\Hh$ and we denote the map of Definition~\ref{def:themaph} by
\begin{equation}\label{eq:defofh_I:Bb_I->M_I}
h_I\colon\|\Bb_I^\Hh\|\to\Mm_I^\Hh\,.
\end{equation}
\end{definition}

\begin{proposition}\label{thm:h0hxsimeq}
The maps $h_\emptyset$, $h_x$ are homotopy equivalences.
\end{proposition}
\begin{proof}
	This immediately follows from Proposition~\ref{prop3.1}.
\end{proof}

Let $\cat$ be as in Definition~\ref{def:cat}.
Given a morphism $i\colon J\to I$ in $\cat$,
the projection $\Ll^{\Hh,I}(n)\to\Ll^{\Hh,J}(n)$ and the map
$\Ll_I^\Hh(n)\to\Ll_J^\Hh(n)$ induced by the inclusion $\Hh\oplus(\Hh^J)^{n}\to\Hh\oplus(\Hh^I)^{n}$ induce 
equivalences of categories, which we represent by
\[
\Delta_i\colon\Delta_I^\Hh\to\Delta_J^\Hh\quad\text{and}\quad \widetilde\Delta_i\colon\widetilde\Delta_I^\Hh\to\widetilde\Delta_J^\Hh\,,
\]
and hence, by Proposition~\ref{prop:MapsFBetweenBars}, homotopy equivalences
\begin{equation}\label{eq:whatisDelta_i}
\|\Delta_i^*\Bb^\Hh_J\|\to\|\Bb^\Hh_J\|\quad\text{and}\quad\|\widetilde\Delta_i^*\widetilde\Bb^\Hh_J\|\to\|\widetilde\Bb^\Hh_J\|.
\end{equation}


The inclusions of based spaces:
\begin{align*}
i\colon \prod_{x\in J}\Mm^\Hh_\emptyset&\to\prod_{x\in I}\Mm^\Hh_\emptyset\,,&
i_N\colon\prod_{x\in J}\Mm^\Hh_x&\to\prod_{x\in I}\Mm^\Hh_x\,,
\end{align*}
together with the pullback maps $\pi_{J,I}^*\colon\Mm^\Hh_J\to\Mm^\Hh_I$ induce maps (see equation~\ref{eq:MapsBetweenBars})
\begin{equation}\label{eq:Bb_i}
\begin{aligned}
  \widetilde\Bb(\ident,i,i_N;\pi_{J,I}^*)&\colon\|\widetilde\Delta_i^*\widetilde\Bb^\Hh_J\|\to \|\widetilde\Bb^\Hh_I\|\\
  \Bb(\ident,i,i_N)&\colon\|\Delta_i^*\Bb^\Hh_J\|\to\|\Bb^\Hh_I\|
\end{aligned}
\end{equation}

\begin{remark}
  Although we are not going to need these facts, we observe that
  the assignments $I\mapsto\Delta_I$ and $i\mapsto\Delta_i$ define a functor from $\cat^\op$ to the category of topological categories equivalent to $\Delta$, and
  the assignments $I\mapsto\|\Bb^\Hh_I\|$ and $(i\colon J\to I)\mapsto \|\Bb^\Hh_i\|$ define a functor $\|\Bb^\Hh\|:\cat\to \hTop$. Furthermore
  the maps $h_I\colon\|\Bb^\Hh_I\|\to\Mm^\Hh_I$
  define a natural transformation between the functors $\|\Bb^\Hh\|,\Mm^\Hh:\cat\to \hTop$ (for the functor $\Mm^\Hh$ see Proposition~\ref{prop:Mm:cat->Top}).
\end{remark}


We now show that, for an appropriate choice of universes
(Definition~\ref{ex:Ll^H}), the spaces
$\|\widetilde\Bb_I\|$ and $\|\Bb_I\|$ are modules over the $\Ll$-algebra $\Mm_\emptyset$ (see Proposition~\ref{prop:barLmoduleoverM} below).
Given universes $\Hh_0$, $\Hh_1$, we have canonical operad maps
$i_j\colon\Ll^{\Hh_j}\to\Ll^{\Hh_0\otimes\Hh_1}$ (with $j=0,1$): the map $i_0$ maps
$f\in\Ll^{\Hh_0}(n)$ to the isometry 
$f\otimes \Id:(\Hh_0^{\oplus n})\otimes\Hh_1\to\Hh_0\otimes \Hh_1$ and
similarly for $i_1$. Let $\Hh=\Hh_0\otimes\Hh_1$.
The maps $i_0$, $i_1$ induce equivalences of categories
$F_{i_j}\colon\Delta_I^{\Hh_j}\to\Delta_I^\Hh$ (with $j=0,1$),
and by Proposition~\ref{prop:MapsFBetweenBars}, homotopy equivalences
\[
\hcolim_{\bigl(\Delta_I^{\Hh_j}\bigr)^\op}F_{i_j}^*\Bb_I^\Hh\xrightarrow{\ \simeq\ }\hcolim_{(\Delta_I^\Hh)^\op}\Bb_I^\Hh\,.
\]
Also observe that each $i_j$ gives $\Mm_\emptyset^\Hh$ the structure of
a $\Ll^{\Hh_j}$-algebra.

\out{%
\begin{proof}
  The isomorphism $\alpha$ induces a map of operads $\Ll^{\Hh_1}\to\Ll^{\Hh_2}$ and, for any $J\in\cat$,
  homotopy equivalences $\Mm_J^{\Hh_1}\to\Mm_J^{\Hh_2}$.
  Thus, we get an equivalence of categories $\Delta_\alpha\colon\Delta_I^{\Hh_1}\to\Delta_I^{\Hh_2}$
  and a weak equivalence of functors
  $\Bb_\alpha:\Delta_\alpha^*\Bb_I^{\Hh_2}\to\Bb_I^{\Hh_1}$.
  Thus we have homotopy equivalences
  \[
  \xymatrix{\|\Bb_I^{\Hh_1}\|\ar[r]^-{\|\Bb_\alpha\|}&
    \|\Delta_\alpha^*\Bb_I^{\Hh_1}\|\ar[r]^-{\|\Delta_\alpha\|}&
    \|\Bb_I^{\Hh_2}\|}
  \]
  which concludes the proof that $\|\Bb_I^{\Hh_1}\|\simeq\|\Bb_I^{\Hh_2}\|$. The proof that
  $\|\widetilde\Bb_I^{\Hh_1}\|\simeq\|\widetilde\Bb_I^{\Hh_2}\|$  is completely analogous.
\end{proof}
}%

\begin{proposition}\label{prop:barLmoduleoverM}
Let $\Hh_\emptyset$, $\Hh_I$ be universes, let $\Hh=\Hh_\emptyset\otimes\Hh_I$ and
let $i_I:\Ll^{\Hh_I}\to\Ll^{\Hh}$ be the canonical map of operads.
Then $\|F_{i_I}^*\widetilde\Bb_I^\Hh\|$ and $\|F_{i_I}^*\Bb_I^\Hh\|$ are left modules 
over the $\Ll^{\Hh_\emptyset}$-algebra $\Mm_\emptyset^\Hh$.
\end{proposition}
\begin{proof}
Given $f\in\Ll^{\Hh_\emptyset}_+(n)$ the map $(\Mm_\emptyset^\Hh)^n\times\|F_{i_I}^*\widetilde\Bb_I^\Hh\|\to \|F_{i_I}^*\widetilde\Bb_I^\Hh\|$
is defined by the natural transformation
$(\Mm_\emptyset^\Hh)^n\times F_{i_I}^*\widetilde\Bb_I^\Hh\to F_{i_I}^*\widetilde\Bb_I^\Hh$ 
given, for $m\neq -1$, by the maps:
\begin{align}\label{eq:leftactionofMonBar}\notag
(\Mm_\emptyset^\Hh)^n\times \widetilde\Bb_I^\Hh(m)
&=\Bigl((\Mm_\emptyset^\Hh)^{n}\times\Mm_\emptyset^\Hh\Bigr)\times\Bigl(\,\prod_{x\in I}\Mm_\emptyset^\Hh\Bigr)^{m}\times
\Bigl(\,\prod_{x\in I}\Mm_x^\Hh\Bigr)\\
&\hspace*{-3em}\xrightarrow{i_\emptyset(f)\times\ident\times\ident}
\Mm_\emptyset^\Hh\times\Bigl(\,\prod_{x\in I}\Mm_\emptyset^\Hh\Bigr)^{m}\times\Bigl(\,\prod_{x\in I}\Mm_x^\Hh\Bigr)=\widetilde\Bb_I^\Hh(m)\\ 
\intertext{and for $m=-1$ by the maps}\notag
(\Mm_\emptyset^\Hh)^n\times \widetilde\Bb_I^\Hh(-1)&=(\Mm_\emptyset^\Hh)^n\times\Mm_I^\Hh\xrightarrow{(\pi_{\emptyset,I}^*)^n\times\ident}
(\Mm_I^\Hh)^{n+1}\xrightarrow{i_\emptyset(f)}\Mm_I^\Hh=\widetilde\Bb_I^\Hh(-1)\,.
\end{align}
The fact that this is a natural transformation follows from the commutativity
of the following
diagram, where $f\in\Ll^{\Hh_\emptyset}(n+1)$, $g\in\Ll^{\Hh_I}(k+1)$ and $J=\emptyset$ or $J=I$:
\[\xymatrix{
(\Mm_J^\Hh)^n\times\Mm_J^\Hh\times(\Mm_J^\Hh)^k
\ar[rr]^-{i_\emptyset(f)\times\ident}\ar[d]_-{\ident\times i_I(g)}&&
\Mm_J^\Hh\times(\Mm_J^\Hh)^k\ar[d]_-{i_I(g)}\\
(\Mm_J^\Hh)^n\times\Mm_J^\Hh\ar[rr]^-{i_\emptyset(f)}&&
\Mm_J^\Hh
}\]
The map $(\Mm_\emptyset^\Hh)^n\times\|F_{i_I}^*\Bb_I^\Hh\|\to \|F_{i_I}^*\Bb_I^\Hh\|$
is obtained by restricting the natural transformation defined above.
\end{proof}

\begin{remark}\label{rmk:barLmoduleoverM}
Using equivariance, we can also think of 
$\|F_{i_I}^*\widetilde\Bb_I^\Hh\|$ and $\|F_{i_I}^*\Bb_I^\Hh\|$
as right modules over $\Mm_\emptyset^\Hh$
(see section~\ref{sec:modoverPalgebras}).
These structures are induced by the maps
(compare with equation~\eqref{eq:leftactionofMonBar}):
\begin{multline*}
\Bb_I^\Hh(m)\times(\Mm_\emptyset^\Hh)^n
\xrightarrow{\text{shuffle}}
 \Bigl(\Mm_\emptyset^\Hh\times(\Mm_\emptyset^\Hh)^{n}\Bigr)\times\Bigl(\,\prod_{x\in I}\Mm_\emptyset^\Hh\Bigr)^{m}\times
\Bigl(\,\prod_{x\in I}\Mm_x^\Hh\Bigr)\\
\xrightarrow{i_\emptyset(f)\times\ident\times\ident}
\Mm_\emptyset^\Hh\times\Bigl(\,\prod_{x\in I}\Mm_\emptyset^\Hh\Bigr)^{m}\times\Bigl(\,\prod_{x\in I}\Mm_x^\Hh\Bigr)=\Bb_I^\Hh(m)
\end{multline*}
and also, in the case of $\|F_{i_I}^*\widetilde\Bb_I^\Hh\|$, 
by the maps
\[
\widetilde\Bb_I^\Hh(-1)\times(\Mm_\emptyset^\Hh)^n =\Mm_I^\Hh\times(\Mm_\emptyset^\Hh)^n\xrightarrow{\ident\times(\pi_{\emptyset,I}^*)^n}
(\Mm_I^\Hh)^{n+1}\xrightarrow{i_\emptyset(f)}\Mm_I^\Hh=\widetilde\Bb_I^\Hh(-1)\,.
\]
\end{remark}

Let $I,J\subset\Cc^2$ be finite disjoint sets.
Fix universes $\Hh_{IJ},\Hh_\emptyset$ and let $\Hh=\Hh_{IJ}\otimes\Hh_\emptyset$. Let
$i:\Ll^{\Hh_{IJ}}\to\Ll^\Hh$ be the canonical operad map.
From Definition~\ref{con:barconstruction},
with $\Oo=\Ll^{\Hh_\emptyset}$ and $\Ms=\Oo_+$
(see Example~\ref{ex:Oo+AsAMonoidalModule}),
we get functors
\begin{equation}\label{eq:defofB(B,M,B)}
\begin{aligned}
  \Bb_\bullet^{\Hh_\emptyset}(\|F_i^*\Bb_I^\Hh\|,\Mm_\emptyset^\Hh,\|F_i^*\Bb_J^\Hh\|)&\colon\Delta(\Ll_+^{\Hh_\emptyset},\Ll^{\Hh_\emptyset})\to\Top\,,\\[1ex] 
  \Bb_\bullet^{\Hh_\emptyset}(\|F_i^*\widetilde\Bb_I^\Hh\|,\Mm_\emptyset^\Hh,\|F_i^*\widetilde\Bb_J^\Hh\|)&\colon\Delta(\Ll_+^{\Hh_\emptyset},\Ll^{\Hh_\emptyset})\to\Top\,,
\end{aligned}
\end{equation}
where the superscript in $\Bb_\bullet^{\Hh_\emptyset}(\cdots)$ 
is added to 
stress that we are over the operad $\Ll^{\Hh_\emptyset}$.
From Remark~\ref{rmk:MImoduleoverMJ}, $\Mm_I^\Hh$ and
$\Mm_J^\Hh$ are modules over the $\Ll^\Hh$-algebra $\Mm_\emptyset^\Hh$.
Taking $\Oo=\Ll^{\Hh}$ and $\Ms=\Oo_+$,  we get the functor
\[
\Bb^\Hh_\bullet(\Mm_I^\Hh,\Mm_\emptyset^\Hh,\Mm_J^\Hh)\colon\Delta(\Ll_+^\Hh,\Ll^\Hh)\to\Top\,.
\]
We will use the notation in equation~\eqref{eq:defofbarhcolim} to represent the homotopy colimit of these functors.


\begin{proposition}\label{prop:IcupJ}
With notation as above, 
we have a commutative diagram in $\hTop$:
\[
\xymatrix{
\Bb^{\Hh_\emptyset}(\|F_i^*\Bb_I^\Hh\|,\Mm_\emptyset^\Hh,\|F_i^*\Bb_J^\Hh\|)\ar[r]\ar[d]_-{\simeq}&
\Bb^\Hh(\Mm_I^\Hh,\Mm_\emptyset^\Hh,\Mm_J^\Hh)\ar[d]\\
\|\Bb_{I\cup J}^\Hh\|\ar[r]^-{h_{I\cup J}}&
\Mm_{I\cup J}^\Hh}
\]
where the left vertical map is a homotopy equivalence and
the top horizontal map is induced by $h_I$ and $h_J$.
\end{proposition}

The proof will be given on page \pageref{prop:IcupJinappendix} in appendix~\ref{app:bar}.

\section{Proof of Theorem  \ref{theoI}}\label{sec:theoI}\label{sec:5}

In this section we fix 
a universe $\Hh$, and write $\Mm$, $\Bb$ instead of $\Mm^\Hh$, $\Bb^\Hh$. Fix a finite set $I\subset\Cc^2$.
For each $n=-1,0,\ldots$ and $J\subset I$, the topological space $\widetilde\Bb_J(n)$ is naturally graded
as a product of graded spaces, and given a morphism $f\in\widetilde\Delta_J(m,n)$, the induced map
$\widetilde\Bb_J(n)\to\widetilde\Bb_J(m)$ preserves the grading. 
Denote by $\widetilde\Bb_{J,k}\colon\widetilde\Delta_J\to\Top$ the functor obtained by taking the degree $k$ component of $\widetilde\Bb_J$.
Let $h_J\colon\|\Bb_J\|\to\Mm_J$ be the map in Definition~\ref{def:Delta_I,Bb_I,h_I}.
The objective of this section is to prove Theorem~\ref{theoI}, which we now restate:

{
\renewcommand{\theequation}{\ref{theoI}}
\begin{theorem}
If, for every $J\subset I$ with cardinality
$\#J\leq k$, the map $h_{J}$ is a homotopy equivalence in degree $k$, then $h_{I}$ is a homotopy equivalence in degree $k$.
\end{theorem}
\addtocounter{equation}{-1}
}

Let $\cat$ be as in Definition~\ref{def:cat} and let $\Mm\colon\cat\to\Top$ be the functor introduced in Proposition~\ref{prop:Mm:cat->Top}.

\begin{definition}\label{def:catI,k}
  Let $\cat_{I,k}$ be the full subcategory of $\cat$ whose objects are the subsets $J\subset I$ with $\#J\leq k$.
\end{definition}

We first need the following result (by nerve of an open cover we mean the functor of Example~\ref{ex:nerveofopencover}):

\begin{lemma}\label{lemma:theo1nerve}
  Let $\Mm_k\colon\cat\to\Top$ denote the degree $k$ component of the functor $\Mm$ and write $\Mm_{J,k}$ for its value
  on $J\in\cat$. Then the collection $\mathcal U=\{\pi_{J,I}^*\Mm_{J,k}\}_{J\in\cat_{I,k}}$  is an open cover of $\Mm_{I,k}$
  and the maps $\pi_{J,I}^*$ induce homeomorphisms between the restriction of $\Mm_{k}$ to $\cat_{I,k}$ and the nerve of $\mathcal U$.
\end{lemma}
\begin{proof}
  \cite[Theorem 2.1]{San05}
\end{proof}

We now turn to the proof of Theorem~\ref{theoI}\label{proof:theorem1.3}

\begin{proof}
It is enough to show that the map $\|\Bb_{I,k}\|\to\|\widetilde\Bb_{I,k}\|$ is a homotopy equivalence.
For each morphism $j:J\to I$, 
it will be convenient to replace the functor
$\widetilde\Bb_J$ with the functor $\Delta_j^*\widetilde\Bb_J\colon\widetilde\Delta_I^\op\to\Top$ (see equation~\eqref{eq:whatisDelta_i}).
Let $\Delta^*\widetilde\Bb_k\colon\widetilde\Delta_I^{\mathrm{op}}\times\cat_{I,k}\to\Top$
be the functor defined on objects by 
\[
\Delta^*\widetilde\Bb_k(n,J)=\Delta_j^*\widetilde\Bb_{J,k}(n)\,;
\]
given morphisms $i:(J_1,j_1)\to(J_2,j_2)$ and 
$f\in\widetilde\Delta_I(m,n)$, we define $\Delta^*\widetilde\Bb(f,i)$ by the
diagonal map in the following commutative diagram (for the vertical maps 
in the diagram see equation~\eqref{eq:Bb_i}):
\[
\xymatrix{ \Delta_{j_1}^*\widetilde\Bb_{J_1}(n)\ar[r]^-{f}
\ar[d]_{i}&
\Delta_{j_1}^*\widetilde\Bb_{J_1}(m)\ar[d]^{i}\\
\Delta_{j_2}^*\widetilde\Bb_{J_2}(n)\ar[r]^-{f}&
\Delta_{j_2}^*\widetilde\Bb_{J_2}(m)}
\]
We define the functor $\Delta^*\Bb_k:\Delta_I^{\mathrm{op}}\times\cat_{I,k}\to \Top$ by restricting $\Delta^*\widetilde\Bb_k$.
We claim that the maps 
\begin{align}\label{eq:theorem1}
&\hcolim_{\cat_{I,k}}\Bigl(\hcolim_{\Delta_I^{\mathrm{op}}}\Delta^*\Bb_k\Bigr)\to
\hcolim_{\Delta_I^{\mathrm{op}}}\Bb_{I,k}\,,\\
\tag{\ref{eq:theorem1}a}
&\hcolim_{\cat_{I,k}}\Bigl(\hcolim_{\widetilde\Delta_I^{\mathrm{op}}}\Delta^*\widetilde\Bb_k\Bigr)\to
\hcolim_{\widetilde\Delta_I^{\mathrm{op}}}\widetilde\Bb_{I,k}
\label{eq:theorem1a}
\end{align}
induced by the maps $\Delta_j^*\widetilde\Bb_{J,k}\to\widetilde\Bb_{I,k}$
are homotopical equivalences; the theorem will follow since we then have a commutative diagram:
\[
\entrymodifiers={+!!<0pt,\fontdimen22\textfont2>}
\xymatrix{
\hcolim\limits_{\cat_{I,k}}\|\Delta^*\Bb_k\|\ar[r]\ar[d]_-{\simeq}&
\hcolim\limits_{\cat_{I,k}}\|\Delta^*\widetilde\Bb_k\|\ar[d]^-{\simeq}\\
\|\Bb_{I,k}\|\ar[r]&\|\widetilde\Bb_{I,k}\|}
\]
and by hypothesis (and the homotopy invariance of the homotopy colimit), 
the top horizontal map is a homotopy equivalence. We first prove that the map in equation~\eqref{eq:theorem1}
is a homotopy equivalence. From the commutation of homotopy colimits
\cite[Section~6]{HoVo92} we get
\[
\hcolim_{\cat_{I,k}}\Bigl(\hcolim_{\Delta_I^{\mathrm{op}}}\Delta^*\Bb_k\Bigr)\cong
\hcolim_{\Delta_I^{\mathrm{op}}}\Bigl(\hcolim_{\cat_{I,k}}\Delta^*\Bb_k\Bigr)\,,
\]
so it is enough to show that we have a weak equivalence of $\Delta_I^\op$-diagrams
$\hcolim_{\cat_{I,k}}\Delta^*\Bb_k\simeq\Bb_{I,k}$ which we now prove.
Let $Z_k\subset\Zz\times\Zz^{(n+1)(\#I)}$ be the subset of tuples of non-negative integers whose sum is $k$. We write an element $\mathbf k\in Z_k$ as
$\mathbf k=\bigl(\,k_0\,,\,(k_{\alpha x})_{\begin{subarray}{l}x\in I\\ \alpha=0,\ldots,n\end{subarray}}\,\bigr)$.
Unraveling Definition~\ref{def:Delta_I,Bb_I,h_I} we see that, for $n\geq 0$, we have $\Bb_{I,k}(n)=\coprod_{\mathbf k}\Bb_{I,\mathbf k}(n)$ where
\[
\Bb_{I,\mathbf k}(n)=\Mm_{\emptyset,k_0}\times
\Bigl(\prod_{\substack{\alpha=1,\ldots,n\\x\in I}}\Mm_{\emptyset,k_{\alpha x}}\Bigr)
\times\Bigl(\prod_{x\in I}\Mm_{x,k_{0x}}\Bigr).
\]
Let $\supp\mathbf k\subset I$ be the set of points such that there is an $\alpha$ for which $k_{\alpha x}>0$. If we let
\[
\Bb_{J,\mathbf k}(n)=\begin{cases}\Bb_{I,\mathbf k}(n),&\text{if $\supp\mathbf k\subset J$;}\\ \emptyset,&\text{if $\supp\mathbf k\not\subset J$,}\end{cases}
\]
then $\Bb_{J,k}(n)\cong\coprod_{\mathbf k}\Bb_{J,\mathbf k}(n)$ and under this isomorphism the map $\Bb_{J,k}(n)\to\Bb_{I,k}(n)$ corresponds to inclusion. 
Let $\mathfrak D_{I,\mathbf k}\subset\cat_{I,k}$ denote the full subcategory whose objects $J$ satisfy $\supp\mathbf k\subset J$.
Then, using the properties of the homotopy colimit (Proposition~\ref{prop:hcolim}) we obtain:
\begin{multline*}
  \Bigl(\hcolim_{\cat_{I,k}}\Delta^*\Bb_k\Bigr)(n)
  \cong\coprod_{\mathbf k\in Z_k}\hcolim_{J\in\mathfrak C_{I,\mathbf k}}\Bb_{J,\mathbf k}(n)
  \cong\coprod_{\mathbf k\in Z_k}\hcolim_{J\in\mathfrak D_{I,\mathbf k}}\Bb_{I,\mathbf k}(n)\\
  \cong \coprod_{\mathbf k\in Z_k}B\mathfrak D_{I,\mathbf k}\times\Bb_{I,\mathbf k}(n)\simeq \Bb_{I,k}(n)
\end{multline*}
because $\mathfrak D_{I,\mathbf k}$ has an initial object, namely: $\supp\mathbf k$.
To complete the proof we need to show that the map in equation~\eqref{eq:theorem1a}
is a homotopy equivalence. We just need to show that $\Bigl(\hcolim_{\cat_{I,k}}\Delta^*\widetilde\Bb_k\Bigr)(-1)\simeq\widetilde\Bb_{I,k}(-1)$
that is (see Definition~\ref{def:Delta_I,Bb_I,h_I}), $\hcolim_{\cat_{I,k}}\Mm_{k}\simeq\Mm_{I,k}$.
This immediately follows from Lemma~\ref{lemma:theo1nerve} and \cite[Proposition 4.1]{Seg68}.
\end{proof}

\section{The gluing maps for $k=1,2$}\label{sec:glue}

In this section we define the gluing maps $\boxplus$ used in the statement of Theorem~\ref{theor0}. The definition is based on the monad description of
holomorphic bundles  introduced in \cite{Don84,Kin89} which we now review.

\subsection{The monad description of the moduli spaces}\label{sec:monads}


We first set some notation.
Let $[x_1,x_2,x_3]\in\CP^2$ be homogeneous coordinates,
and let $L_\infty=\{[x_1,x_2,x_3]\in\CP^2:x_3=0\}\subset\CP^2$ be the line at infinity.
We identify the blow up $\cp{\set0}$ of $\CP^2$ at $[0,0,1]$ with the subvariety of
$\CP^2\times\CP^1$ consisting of the pairs $\bigl([x_1,x_2,x_3],[y_1,y_2]\bigr)$ such that $x_1y_1+x_2y_2=0$.

Let $\Vv$ be a finite dimensional complex Hermitian vector space and
let $W$, $W_0$, $W_1$ be complex vector spaces of dimension $k$.
Let $\Rc_W^\Vv$ be the space of 4-tuples $(a_1,a_2,b,c)$ where
$a_i\in{\rm End}(W)$, $b\in{\rm Hom}(\Vv,W)$ and $c\in{\rm Hom}(W,\Vv)$,
obeying the integrability condition 
\begin{equation}\label{eq:monadintegrabilityforP}
a_1a_2-a_2a_1+bc=0\,,
\end{equation}
and let $\tilde\Rc_{W_0,W_1}^\Vv$ be the space of 5-tuples
$(\tilde a_1,\tilde a_2,\tilde d,\tilde b,\tilde c)$ where
$\tilde a_i\in{\rm Hom}(W_1,W_0)$, $\tilde d\in{\rm Hom}(W_0,W_1)$,
$\tilde b\in{\rm Hom}(\Vv,W_0)$ and $\tilde c\in{\rm Hom}(W_1,\Vv)$,
such that 
\begin{equation}\label{eq:monadconditionfortildeP}
\tilde a_1(W_1)+\tilde a_2(W_1)+\tilde b(\Vv)=W_0\,,
\end{equation}
obeying the integrability condition 
\begin{equation}\label{eq:monadintegrabilityfortildeP}
\tilde a_1\tilde d\tilde a_2-\tilde a_2\tilde d\tilde a_1+\tilde b\tilde c=0\,.
\end{equation}

\begin{notation}\label{remark:W=W_0orW=W_1}
	Except for the proof or Proposition~\ref{prop:propertiesofmonaddescription} below, we always take $W=W_1$ and write
        \[
	\Rc_{W_1}^\Vv=\Rc_k\,,\quad \tilde\Rc_{W_0,W_1}^\Vv=\tilde\Rc_k\quad(k=\dim W_0=\dim W_1)\,.
	\]
        Also, we fix for each $k$ an isomorphism $\chi_k\colon W_0\to W_1$.
\end{notation}

For each $r=(a_1,a_2,b,c)\in\mathcal R_k$ and $x=(x_1,x_2,x_3)\in\Cc^3$ consider the linear maps
$A_{r,x}\colon W_1\to W_1^2\oplus \Vv$ and $B_{r,x}\colon W_1^2\oplus \Vv\to W_1$
given in matrix form by
\begin{equation}\label{eq:monadmatricesforP2}
A_{r,x}=\begin{pmatrix}x_1-a_1x_3\\x_2-a_2x_3\\cx_3\end{pmatrix}\,,\quad
B_{r,x}=\begin{pmatrix}-x_2+a_2x_3 & x_1-a_1x_3 & bx_3\end{pmatrix}\,.
\end{equation}
Then the vector spaces $\text{Im}\,A_{r,x}$ and $\text{Ker}\,B_{r,x}$ depend only on $[x_1,x_2,x_3]\in\CP^2$.
The 4-tuple  $r$ is called nondegenerate if:
\begin{equation}\label{eq:monaddegenerateforP}
\text{$A_{r,x}$ and $B_{r,x}$ have maximal rank for every $[x]\in\CP^2$.}
\end{equation}
Let $\Rc_k^{\reg}\subset\Rc_k$ denote the subspace of nondegenerate configurations. To each $r\in\Rc_k^\reg$ we assign a framed holomorphic bundle $(\Ee_r,\phi_r)$ over $\CP^2$ as follows:
Consider the subbundles $\text{Im}\,A_{r}$, $\text{Ker}\,B_{r}$
of the trivial bundle $\bigl(W_1^2\oplus \Vv\bigr)\times\CP^2$. The integrability condition~\eqref{eq:monadintegrabilityforP} is equivalent to
$\text{Im}\,A_{r}\subset\text{Ker}\,B_{r}$
and we define $\Ee_r$ as the quotient:
\[
\Ee_r=\text{Ker}\,B_{r}/\text{Im}\,A_{r}\,.
\]
The inverse of the trivialization $\phi_r$ is the composition
\begin{equation}\label{eq:trivializationphir}
\phi_r^{-1}\colon \Vv\times L_\infty\to\text{Ker}\,B_r\to\Ee_r
\end{equation}
where the first map is induced, for each $[x]\in L_\infty$, by the inclusion $\Vv\to\text{Ker}\,B_{r,x}$.

In a similar way, given $\tilde r=(\tilde a_1,\tilde a_2,\tilde d,\tilde b,\tilde c)\in\tilde \Rc_k$, $x=(x_1,x_2,x_3)\in\Cc^3$ and $y=(y_1,y_2)\in\Cc^2$,
we have linear maps
\[
\tilde A_{\tilde r,x}\colon W_1\oplus W_0\to(W_0\oplus W_1)^2\oplus \Vv\quad\text{and}\quad\tilde B_{\tilde r,x}\colon(W_0\oplus W_1)^2\oplus \Vv\to W_0\oplus W_1
\]
defined by
\begin{equation*}
\tilde A_{\tilde r,x}=\begin{pmatrix} \tilde a_1x_3 & -y_2 \\ x_1-\tilde d\tilde a_1x_3 & 0 \\ \tilde a_2x_3 & y_1 \\ x_2-\tilde d\tilde a_2x_3 & 0 \\ \tilde cx_3 & 0 \end{pmatrix}\,,\quad
\tilde B_{\tilde r,x}=\begin{pmatrix} x_2 & \tilde a_2x_3 & -x_1 & -\tilde a_1x_3 & \tilde bx_3 \\ \tilde dy_1 & y_1 & \tilde dy_2 & y_2 & 0 \end{pmatrix}\,.
\end{equation*}
The subspace $\tilde\Rc_k^{\reg}\subset\tilde\Rc_k$ of non-degenerate 5-tuples is the space of elements $\tilde r\in\tilde\Rc_k^\reg$ such that
\begin{equation}\label{eq:monaddegeneratefortildeP}
\text{$\tilde A_{\tilde r,x}$ and $\tilde B_{\tilde r,x}$ have maximal rank for every $\bigl([x],[y]\bigr)\in\cp{\set0}$}\,.
\end{equation}
The integrability condition~\eqref{eq:monadintegrabilityfortildeP} is equivalent to $\text{Im}\,\tilde A_{\tilde r}\subset\text{Ker}\,\tilde B_{\tilde r}$ (provided $x_1y_1+x_2y_2=0$) and
we define a holomorphic bundle over $\cp{\set0}$ by
$\tilde \Ee_{\tilde r}=\text{Ker}\,\tilde B_{\tilde r}/\text{Im}\,\tilde A_{\tilde r}$; as before (see equation~\eqref{eq:trivializationphir}),
the inverse of the trivialization $\tilde\phi_{\tilde r}$ is induced by the inclusion $\Vv\to\text{Ker}\,\tilde B_{\tilde r}$.

The groups $GL(W_1)$ and $GL(W_0)\times GL(W_1)$ act on $\Rc_k$ and $\tilde\Rc_k$, respectively by
\begin{equation}\label{eq:GLactiononR}
\begin{aligned}
g\cdot(a_1,a_2,b,c)&=(ga_1g^{-1},ga_2g^{-1},gb,cg^{-1})\,,\quad\text{and}\\
(g_0,g_1)\cdot(\tilde a_1,\tilde a_2,\tilde d,\tilde b,\tilde c)&=(g_0\tilde a_1g_1^{-1},g_0\tilde a_2g_1^{-1},g_1\tilde dg_0^{-1},g_0\tilde b,\tilde cg_1^{-1})\,.
\end{aligned}
\end{equation}
The group actions preserve the nondegeneracy conditions and we denote the quotients by
\[
M_{\emptyset,k}^\Vv=\Rc_k^\reg/GL(W_1)\quad\text{and}\quad M_{\set0,k}^\Vv=\tilde\Rc_k^\reg/\bigl(GL(W_0)\times GL(W_1)\bigr)\,.
\]

\begin{theorem*}[Donaldson, King]
	The actions of $GL(W_1)$ on $\Rc_k^\reg$ and $GL(W_0)\times GL(W_1)$ on $\tilde \Rc_k^{\reg}$
	are free and the assignments $r\mapsto(\Ee_r,\phi_r)$ and $\tilde r\mapsto(\tilde\Ee_{\tilde r},\tilde\phi_{\tilde r})$ descend to the quotient defining isomorphisms
	\begin{equation}\label{eq:defofpsi_emptysetpsi_0}
	\psi_\emptyset\colon M_{\emptyset,k}^\Vv\xrightarrow{\cong}\Mm_{\emptyset,k}^\Vv\,,\qquad 
	\psi_{\set0}\colon M_{\set0,k}^\Vv\xrightarrow{\cong}\Mm_{\set0,k}^\Vv\,.
	\end{equation}
\end{theorem*}
\begin{proof}
	The statement for $\CP^2$ is proven in \cite[Proposition 1]{Don84}. The statement for $\cp{\set0}$ is proven in \cite[Theorem 3.4.1]{Kin89}.
\end{proof}

\out{%
Translation in $\Cc^2$ induces an action $\tau\colon\Cc^2\times\CP^2\to\CP^2$ and
a map $\tau_z\colon\cp{\set0}\to\cp{\set z}$ given respectively by
\begin{equation*}
\begin{aligned}
\tau_z\bigl([x_1,x_2,x_3]\bigr)&=[x_1+z_1x_3,x_2+z_2x_3,x_3]\quad\text{and}\\
\tau_z\bigl([x_1,x_2,x_3],[y_1,y_2]\bigr)&=\bigl([x_1+z_1x_3,x_2+z_2x_3,x_3],[y_1,y_2]\bigr)\,.
\end{aligned}
\end{equation*}
The translation map $\tau_z$ allows us to identify,
for any $z\in\Cc^2$, the moduli space $\Mm_{\set z,k}$ with $M_{\set 0,k}$, through the
isomorphism
\begin{equation}\label{eq:defofpsi_z}
\psi_{\set z}=\tau_{-z}^*\circ\psi_{\set0}\colon M_{\set 0,k}\to\Mm_{\set z,k}\,.
\end{equation}
We have a commutative diagram
\begin{equation}\label{eq:pitau=taupi}
\xymatrix{
	M_{\set 0,k}\ar[r]^-{\psi_{\set0}}&
	\Mm_{\set 0,k}\ar[r]^-{\tau_{-z}^*} &
	\Mm_{\set z,k}\\
	M_{\emptyset,k}\ar[r]^-{\psi_\emptyset}&
	\Mm_{\emptyset,k}\ar[u]^{\pi_{\emptyset,\set 0}^*}\ar[r]^-{\tau_{-z}^*} &
	\Mm_{\emptyset,k}\ar[u]_{\pi_{\emptyset,\set z}^*}
}
\end{equation}
}%


\begin{proposition}\label{prop:propertiesofmonaddescription}
	Let $\Vv_1$, $\Vv_2$ be finite dimensional complex Hermitian vector spaces and let $J\subset\Cc^2$ be either $\emptyset$ or $\{0\}$.
	\begin{enumerate}
	\item Given a  linear isometry $\alpha\colon \Vv_1\to \Vv_2$, with dual $\alpha^*:\Vv_2\to \Vv_1$, 
	the isomorphisms in equation~\eqref{eq:defofpsi_emptysetpsi_0} take the induced maps $\Mm_{J,k}^{\Vv_1}\to \Mm_{J,k}^{\Vv_2}$
	to the maps $\alpha\colon M_{J,k}^{\Vv_1}\to M_{J,k}^{\Vv_2}$ given by
		\begin{align*}
			[a_1,a_2,b,c]&\mapsto[a_1,a_2,b\circ\alpha^*,\alpha\circ c] & (J&=\emptyset)\\
			[\tilde a_1,\tilde a_2,\tilde d,\tilde b,\tilde c]&\mapsto[\tilde a_1,\tilde a_2,\tilde d,\tilde b\circ\alpha^*,\alpha\circ \tilde c] & (J&=\{0\})
		\end{align*}
	\item The isomorphisms in equation~\eqref{eq:defofpsi_emptysetpsi_0} take
	Whitney sum $\Mm_{J,k_1}^{\Vv_1}\times \Mm_{J,k_2}^{\Vv_2}\to \Mm_{J,k_1+k_2}^{\Vv_1\oplus \Vv_2}$ to the maps
	$\omega\colon M_{J,k_1}^{\Vv_1}\times M_{J,k_2}^{\Vv_2}\to M_{J,k_1+k_2}^{\Vv_1\oplus \Vv_2}$ induced by direct sum:
		\begin{align*}
			\bigl([a_1,a_2,b,c],[a_1',a_2',b',c']\bigr)&\mapsto[a_1\oplus a_1',a_2\oplus a_2',b\oplus b',c\oplus c']& J&=\emptyset\\
			\bigl([\tilde a_1,\tilde a_2,\tilde d,\tilde b,\tilde c],[\tilde a_1',\tilde a_2',\tilde d',\tilde b',\tilde c']\bigr)
			&\mapsto[\tilde a_1\oplus \tilde a_1',\tilde a_2\oplus \tilde a_2',\tilde d\oplus \tilde d',\tilde b\oplus \tilde b',\tilde c\oplus \tilde c']& J&=\{0\}
		\end{align*}
        \item For each $z=(z_1,z_2)\in\Cc^2$, let $\tau_z\colon\CP^2\to\CP^2$ be the translation map defined by 
        \begin{equation}\label{eq:translationonCP2}
        \tau_z\bigl([x_1,x_2,x_3]\bigr)=[x_1+z_1x_3,x_2+z_2x_3,x_3]\,.
        \end{equation}
        Then the isomorphisms in equation~\eqref{eq:defofpsi_emptysetpsi_0} take
        the pullback map $\tau_z^*\colon \Mm_{\emptyset,k}^\Vv\to \Mm_{\emptyset,k}^\Vv$ to the map
        $\tau_z^*\colon M_{\emptyset,k}^\Vv\to M_{\emptyset,k}^\Vv$ given by 
		\[
		\tau_z^*\bigl([a_1,a_2,b,c]\bigr)=[a_1-z_1,a_2-z_2,b,c]\,.
		\]
        \item Let $\pi\colon\cp{\set0}\to\CP^2$ be the blowup map.
        The pullback map $\pi^*\colon\Mm_{\emptyset,k}^\Vv\to \Mm_{\set0,k}^\Vv$
        is taken by the isomorphisms in equation~\eqref{eq:defofpsi_emptysetpsi_0} to the map
        $\pi^*\colon M_{\emptyset,k}^\Vv\to M_{\set0,k}^\Vv$ given as follows: fix a vector space isomorphism $\chi_k\colon W_0\to W_1$; then
		\[
		\pi^*\bigl([a_1,a_2,b,c]\bigr)=[\chi_k^{-1}a_1,\chi_k^{-1}a_2,\chi_k,\chi_k^{-1}b,c]\,.  
		\]
	\end{enumerate}
\end{proposition}
\begin{proof}
	We prove (1) and (2) only for $J=\emptyset$ since the proof for $J=\set0$ is completely analogous. We'll use the notation $\Rc_W^\Vv$ instead of $\Rc_k$ (see Notation~\ref{remark:W=W_0orW=W_1}). 
	\begin{enumerate}
	\item We assume first that $\dim \Vv_1=\dim \Vv_2$.
          Given $r=(a_1,a_2,b,c)\in\Rc_{W}^{\Vv_1}$ let $\alpha(r)=(a_1,a_2,b\circ\alpha^*,\alpha\circ c)\in\Rc_{W}^{\Vv_2}$.
		We want to show that $[\Ee_{\alpha(r)},\phi_{\alpha(r)}]=[\Ee_r,(\alpha\times\ident)\circ\phi_r]$ (see Definition~\ref{def:omega,pi*,Malpha}(2)).
		We have a commutative diagram:
		\[
		\xymatrix@C=3.5em{
			W_1\ar[d]_-=\ar[r]^-{A_{r,x}}&
			W_1^2\oplus \Vv_1\ar[d]_{\cong}^{\Id\times\alpha}\ar[r]^-{B_{r,x}}&
			W_1\ar[d]_-=\\
			W_1\ar[r]^-{A_{\alpha(r),x}}&
			W_1^2\oplus \Vv_2\ar[r]^-{B_{\alpha(r),x}}&
			W_1
		}
		\]
		which induces an isomorphism $\psi_\alpha\colon\Ee_r\to\Ee_{\alpha(r)}$ and restricting to $L_\infty$ 
                we have $(\alpha\times\ident)\circ\phi_r=\phi_{\alpha(r)}\circ\psi_\alpha$ (see equation~\eqref{eq:trivializationphir}) so                
		\[
		[\Ee_r,(\alpha\times\ident)\circ\phi_r]=[\Ee_r,\phi_{\alpha(r)}\circ\psi_\alpha]=[\Ee_{\alpha(r)},\phi_{\alpha(r)}]\,.
		\]
                Now consider the general case. Any isometry $\alpha\colon \Vv_1\to \Vv_2$ is the composition of an isomorphism $\Vv_1\to\alpha(\Vv_1)$ and an inclusion
                $\alpha(\Vv_1)\to\alpha(\Vv_1)\oplus\alpha(\Vv_1)^\perp=\Vv_2$ whose induced map $M^{\alpha(\Vv_1)}_{J,k}\to\Mm^{\Vv_2}_{J,k}$ is given by direct sum with a trivial bundle so the result follows from (2).               
		\item Given $r\in\Rc_{W}^\Vv$ and $r'\in\Rc_{W'}^{\Vv'}$ let $r\oplus r'\in\Rc_{W\oplus W'}^{\Vv\oplus \Vv'}$ be the 4-tuple obtained by direct sum in each coordinate. Then from the definition
		of the linear maps $A$ and $B$ (equation~\eqref{eq:monadmatricesforP2}) it is straightforward to see that the complex
		\[
		W\oplus W'\xrightarrow{A_{r\oplus r',x}}(W\oplus W')^2\oplus(\Vv\oplus \Vv')\xrightarrow{B_{r\oplus r'}}W\oplus W'
		\]
		gives rise to an isomorphism $\Ee_{r\oplus r'}\cong\Ee_r\oplus\Ee_{r'}$.
	      \item For each $x\in\CP^2$ let $\tau_zx$ be as in equation~\eqref{eq:translationonCP2}. 
	      The map $\tau_z^*\colon\Rc_k\to\Rc_k$ defined by $\tau_z^*(a_1,a_2,b,c)=(a_1-z_1,a_2-z_2,b,c)$
                preserves the nondegeneracy conditions and the framing and from equation~\eqref{eq:monadmatricesforP2} we immediately get
		$A_{r,\tau_zx}=A_{\tau_z^*r,x}$ and $B_{r,\tau_zx}=B_{\tau_z^*r,x}$. It follows that
		$\tau_z^*\Ee_r=\Ee_{\tau_z^*r}$.
	      \item In \cite[Lemma 4.1]{BrSa00} it was shown that, if we take $W=W_0$ (see Notation~\ref{remark:W=W_0orW=W_1}) then pullback is given by the map
                $j(a_1,a_2,b,c)=\bigl(a_1\chi_k^{-1},a_2\chi_k^{-1},\chi_k,b,c\chi_k^{-1}\bigr)$. In this paper we take $W=W_1$. If we let
                $f\colon M_{\emptyset,k}^\Vv=\Rc_{W_1}^\Vv/GL(W_1)\to \Rc_{W_0}^\Vv/GL(W_0)$ be the isomorphism
                defined by \[f(a_1,a_2,b,c)=\bigl(\chi_k^{-1}a_1\chi_k,\chi_k^{-1}a_2\chi_k,\chi_k^{-1}b,c\chi_k\bigr)\] then                 
		$j\circ f(a_1,a_2,b,c)=(\chi_k^{-1}a_1,\chi_k^{-1}a_2,\chi_k,\chi_k^{-1}b,c)$, which finishes the proof.\qedhere
	\end{enumerate}
\end{proof}

\subsection{Completion of the moduli space}\label{subsec:completionofM}

From now on (except in Proposition~\ref{prop:mapsboxplus}) we work with a fixed finite dimensional complex Hermitian vector space $\Vv$, which we will omit from the notation.

We now drop the nondegeneracy conditions~\eqref{eq:monaddegenerateforP} and~\eqref{eq:monaddegeneratefortildeP},
and consider the geometric invariant theory quotients,
formed by identifying orbits whose closures intersect \cite[Remark 8.14]{Kir84}:
\[
\overline M_{\emptyset,k}=\Rc_k/\!/GL(W_1)\,,\qquad\overline M_{\set0,k}=\tilde\Rc_k/\!/\bigl(GL(W_0)\times GL(W_1)\bigr)\,.
\]

\begin{theorem*}[King]
	The spaces $\overline M_{\emptyset,k}$ and $\overline M_{\set0,k}$  
	are isomorphic to the Donaldson-Uhlenbeck completions $\overline\Mm_{\emptyset,k}$ and
	$\overline\Mm_{\set 0,k}$.
\end{theorem*}
\begin{proof}
	The statement for $\CP^2$ is proven in \cite[Theorem 5.2.7]{Kin89}.  The statement for $\cp{\set0}$ is proven in \cite[Theorem 5.3.7]{Kin89}.
\end{proof}

We briefly sketch how the isomorphisms $\psi_\emptyset$, $\psi_{\set0}$  in equation~\eqref{eq:defofpsi_emptysetpsi_0} extend to the completion.
For more details see \cite[section 5]{Kin89}. 
The Donaldson-Uhlenbeck completion
$\overline \Mm_{\emptyset,k}$ consists of pairs $([\Ee,\phi],\ell)$ where $[\Ee,\phi]\in \Mm_{\emptyset,j}$ for some $j\leq k$,
and $\ell\colon\Cc^2\to\Zz$ is a non-negative function which vanishes except at a finite number of points
$p_1,\dots,p_r$ and is such that $\sum_i\ell(p_i)=k-j$. Such a function $\ell$
determines an unordered $(k-j)$-tuple in $\Cc^2$, 
with each point $p_i$ appearing with multiplicity $\ell(p_i)$, which
in turn correspond to a point in the symmetric product $S^{k-j}\Cc^2$. So the Donaldson-Uhlenbeck completion has a natural stratification
\begin{equation}\label{eq:stratificationDonUhl}
\overline \Mm_{\emptyset,k}=\bigcup_{j=0}^k\Mm_{\emptyset,j}\times S^{k-j}\Cc^2\,.
\end{equation}
Any 4-tuple $r\in\Rc_k$ is equivalent to a direct sum $r^\reg\oplus r^\Delta$ where $r^\reg\in\Rc_j^\reg$
for some $0\leq j\leq k$, determining a point in $\Mm_{\emptyset,j}$, and $r^\Delta=(a_1^\Delta,a_2^\Delta,0,0)$ where $a_1^\Delta$, $a_2^\Delta$ are  $(k-j)\times(k-j)$ diagonal matrices.
For each simultaneous eigenvector, the eigenvalues of $a_1$, $a_2$ determine a point $p\in\Cc^2$.
The $k-j$ points thus obtained define an unordered $(k-j)$-tuple in $S^{k-j}\Cc^2$.

\begin{remark}\label{rmk:compactifications}
	If $r=r^{\text{reg}}\oplus r^\Delta$ 
	is a degenerate configuration then 
	$\Ee_r=\text{Ker}\,B_r/\text{Im}\,A_r$ is a coherent sheaf
	which is locally free except at a finite number of points.
	The holomorphic bundle $\Ee_{r^{\text{reg}}}$ is the double 
	dual of the sheaf $\Ee_r$ and 
	$\ell$ is the cycle associated to the zero-dimensional sheaf
	$\Ee_{r^{\text{reg}}}/\Ee_r$:
	at each point $p_i$ where
	$\Ee_r$ is not locally free, $\ell(p_i)=\dim_\Cc(\Ee_{r^{\text{reg}}}/\Ee_r)_{p_i}$.
\end{remark} 

The completion $\overline \Mm_{\set0,k}$ can be described in an analogous way to $\overline\Mm_{\emptyset,k}$: it is the space of pairs
$\bigl([\Ee,\phi],(p_1,\dots,p_{k-j})\bigr)$ with $[\Ee,\phi]\in \Mm_{\set0,j}$ and $(p_1,\dots,p_{k-j})$ an unordered
$(k-j)$-tuple in the blowup $\widetilde\Cc^2$ of $\Cc^2$ at the origin,
which we identify with the pairs of points
$\bigl((\lambda_1,\lambda_2),[\mu_1,\mu_2]\bigr)\in\Cc^2\times\mathbb P^1$
satisfying $\mu_1\lambda_1+\mu_2\lambda_2=0$.
Any 5-tuple $\tilde r\in\tilde\Rc_k$ is equivalent to a direct sum $\tilde r^\reg\oplus\tilde r^\Delta$ with
$\tilde r^\reg\in\tilde\Rc_j^\reg$ for some $j$, determining a point in $\Mm_{\set0,j}$, and
$\tilde r^\Delta=(\tilde a_1^\Delta,\tilde a_2^\Delta,\tilde d^\Delta,0,0)$ 
where $\tilde a_1^\Delta$, $\tilde a_2^\Delta$, $\tilde d^\Delta$ are  $(k-j)\times(k-j)$ diagonal matrices.
For each simultaneous eigenvector $v$ of 
$\tilde d^\Delta\tilde a^\Delta_1$ and $\tilde d^\Delta \tilde a^\Delta_2$, 
there is a unique $[\mu_1,\mu_2]\in\mathbb P^1$ such that
$\mu_1\tilde a^\Delta_1v+\mu_2\tilde a^\Delta_2v=0$ 
(this follows from condition~\eqref{eq:monadconditionfortildeP}) and
the eigenvalues $\lambda_i$ of $\tilde d^\Delta \tilde a^\Delta_i$ satisfy
$\mu_1\lambda_1+\mu_2\lambda_2=0$, determining a point
$\bigl((\lambda_1,\lambda_2),[\mu_1,\mu_2]\bigr)\in\widetilde\Cc^2$.
The $k-j$ points 
thus obtained define an unordered $(k-j)$-tuple in 
$S^{k-j}\widetilde\Cc^2$.

\subsection{Blowup}

Let $\pi\colon\tilde X\to X$ be the blowup of a complex surface $X$ at a point $z\in X$ and let $L=\pi^{-1}(z)$ be the exceptional divisor.
A holomorphic bundle on $\tilde X$ determines a holomorphic bundle on $X$: 

\begin{proposition}\label{prop:pi*veeveeexistsandisunique}
	Given a holomorphic bundle $\tilde \Ee$ over $\tilde X$
	there is a unique (up to isomorphism) holomorphic bundle $\Ee$ over $X$ such that $(\pi^*\Ee)|_{X\setminus L}\cong\tilde \Ee|_{\tilde X\setminus L}$.
\end{proposition}
\begin{proof}
	Uniqueness follows from Hartogs' Theorem: any bundle isomorphism over $X\setminus\set z$ can be extended to an isomorphism over $X$. To show existence
	we take $\Ee$ to be the double dual of the direct image sheaf: $\Ee=(\pi_*\tilde\Ee)^{\vee\vee}$.
	This sheaf is locally free \cite[Chapter 2, Proposition 25]{Fri98} and hence it is a holomorphic bundle.
	Moreover, we have $\bigl(\pi_*\tilde\Ee\bigr)^{\vee\vee}|_{X\setminus\set z}\cong\tilde \Ee|_{\tilde X\setminus L}$ as required.
\end{proof}

We represent the double dual of the direct image sheaf by
\begin{equation*}
\pi_*^{\vee\vee}\tilde\Ee=\bigl(\pi_*\tilde\Ee\bigr)^{\vee\vee}\,.
\end{equation*}

\begin{proposition}\label{prop:pibuletpi*=id}
	For any holomorphic bundle $\Ee$ over $X$ we have $\Ee\cong\pi_*\pi^*\Ee\cong\pi_*^{\vee\vee}\pi^*\Ee$.
\end{proposition}
\begin{proof}
	From \cite[Lemma 2.2(a)]{Buc00} it follows that $\pi_*\pi^*\Ee$ is locally free so $\pi_*\pi^*\Ee\cong\pi_*^{\vee\vee}\pi^*\Ee$.
	Since $\pi|_{\tilde X\setminus L}$ is an isomorphism we have
	$\Ee|_{X\setminus\set z}\cong\pi_*\pi^*\Ee|_{X\setminus \set z}$ 
	and hence, by uniqueness (Proposition~\ref{prop:pi*veeveeexistsandisunique}), we get $\Ee\cong \pi_*\pi^*\Ee$.
\end{proof}  

Now fix a finite set $J\subset \Cc^2$ and a point $x\notin J$, let $I=J\cup\set x$ and let $\pi\colon\cp{I}\to\cp{J}$
be the blowup at $x$.
We define a map $\pi_\bullet\colon \Mm_{I,k}\to\overline \Mm_{J,k}$ as follows: Let $\tilde\Ee$ be 
a holomorphic bundle over $\cp{I}$;
since $\pi|_{L_\infty}$ is an isomorphism,
a trivialization
$\tilde\phi$ of $\tilde\Ee|_{L_\infty}$ induces a trivialization
of $(\pi_*^{\vee\vee}\tilde\Ee)|_{L_\infty}$, 
which we represent by $\pi_*(\tilde\phi)$;
we define
\begin{equation}\label{eq:defofpibullet}
\pi_\bullet[\tilde\Ee,\tilde\phi]=\bigl([\pi_*^{\vee\vee}\tilde\Ee,\pi_*(\tilde\phi)],\ell\bigr)
\end{equation}
where $\ell(p)=0$ for $p\neq x$ and $\ell(x)=c_2(\tilde\Ee)-c_2(\pi_*^{\vee\vee}\tilde\Ee)\geq0$ (see \cite[Equation (2.10)]{Buc00}).
From Proposition~\ref{prop:pibuletpi*=id} it immediately follows that
\begin{equation}\label{eq:pibuletpi*=id}
\pi_\bullet\pi^*[\Ee,\phi]=[\Ee,\phi]\,.
\end{equation}

\begin{remark}
	The map $\pi_\bullet$ sends a holomorphic bundle $\tilde\Ee$
	to the point in the Donaldson-Uhlenbeck compactification 
	determined by the direct image sheaf 
	$\pi_*\tilde\Ee$: see Remark~\ref{rmk:compactifications}.
\end{remark}

\begin{proposition}\label{prop:caracterizationofpi*M}
	Let $[\tilde\Ee,\tilde\phi]\in\Mm_{I,k}$.
	The following are equivalent:
	\begin{enumerate}
		\item $\tilde\Ee|_L$ is the trivial bundle.
		\item $\pi_\bullet\bigl([\tilde\Ee,\tilde\phi]\bigr)\in\Mm_{J,k}$.
		\item $\tilde\Ee=\pi^*\bigl(\pi_*^{\vee\vee}\tilde\Ee\bigr)$.
	\end{enumerate}
\end{proposition}
\begin{proof}
	Statement (2) is equivalent to $c_2(\tilde\Ee)=c_2\bigl(\pi_*^{\vee\vee}\tilde\Ee\bigr)$ so it is clear that $(3)\Rightarrow(2)$
	and from \cite[Equation (2.10)]{Buc00} it follows that $(2)\Rightarrow(1)$. Finally,
	from \cite[Lemma 2.2(a)(c)]{Buc00}, it follows that $(1)\Rightarrow(3)$.
\end{proof}

\begin{corollary}\label{coro:pi_*pi^*}
	Let $I=\{x,y\}\subset\Cc^2$ and let $\Ee$ be a holomorphic
	bundle over $\cp{\set{y}}$. Then, with notation as in 
	Definition~\ref{def:omega,pi*,Malpha}(4),
	we have
	\[
	(\pi_{x,I})_*^{\vee\vee}\circ\pi_{y,I}^*\Ee\cong
	\pi_{\emptyset,x}^*\circ(\pi_{\emptyset,y})_*^{\vee\vee}\Ee\,.
	\]
\end{corollary}
\begin{proof}
	From Proposition~\ref{prop:pi*veeveeexistsandisunique} it follows that,
	given any holomorphic bundle $\tilde\Ee$ over $\cp{I}$, we have
	\[
	(\pi_{\emptyset,y})_*^{\vee\vee}\circ (\pi_{y,I})_*^{\vee\vee}\tilde\Ee
	\cong(\pi_{\emptyset,I})_*^{\vee\vee}\tilde\Ee\cong
	(\pi_{\emptyset,x})_*^{\vee\vee}\circ (\pi_{x,I})_*^{\vee\vee}\tilde\Ee\,.
	\]
	Since the bundle $(\pi_{x,I})_*^{\vee\vee}\circ\pi_{y,I}^*\Ee$
	is trivial on the exceptional divisor we get
	\begin{align*}
	(\pi_{x,I})_*^{\vee\vee}\circ\pi_{y,I}^*\Ee
	&\cong\pi_{\emptyset,x}^*\circ(\pi_{\emptyset,x})_*^{\vee\vee}
	\circ (\pi_{x,I})_*^{\vee\vee}\circ\pi_{y,I}^*\Ee&
	&\text{(Proposition~\ref{prop:caracterizationofpi*M})} \\
	&\cong\pi_{\emptyset,x}^*\circ(\pi_{\emptyset,y})_*^{\vee\vee}
	\circ (\pi_{y,I})_*^{\vee\vee}\circ\pi_{y,I}^*\Ee \\
	&\cong \pi_{\emptyset,x}^*\circ(\pi_{\emptyset,y})_*^{\vee\vee}\Ee &
	&\text{(Proposition~\ref{prop:pibuletpi*=id})} \qedhere
	\end{align*}
\end{proof}


\begin{proposition}\label{prop:formulaforpibullet}
  Let $\pi_\bullet\colon M_{\set0,k}\to \overline M_{\emptyset,k}$ be the map induced by the map in equation~\eqref{eq:defofpibullet} and by the isomorphisms in equation~\eqref{eq:defofpsi_emptysetpsi_0}.
  Then
  we have \[\pi_\bullet[a_1,a_2,d,b,c]=[da_1,da_2,db,c]\,.\] 
\end{proposition}
\begin{proof}
	For any $\tilde r\in\tilde\Rc_k^{\text{reg}}$, $\pi_\bullet\tilde r$ is nondegenerate
	away from the blowup point and
	$\pi^*\Ee_{\pi_\bullet\tilde r}$ restricted to 
	$\cp{\set0}\setminus L$ is isomorphic to 
	$\tilde\Ee_{\tilde r}$ restricted to $\cp{\set0}\setminus L$
	(see \cite{Kin89}, Proposition~6.1.1 and subsequent discussion).
	By Proposition~\ref{prop:pi*veeveeexistsandisunique},
	this implies that $\Ee_{\pi_\bullet\tilde r}$
	is isomorphic to $\pi_*^{\vee\vee}\tilde\Ee$, 
	which finishes the proof.
\end{proof}

\begin{definition}\label{def:S_0M_1}
  Let $S_0M_{\set0,1}=\bigl\{[\tilde r]\in M_{\set0,1}:\text{$\pi_*^{\vee\vee}\tilde\Ee_{\tilde r}$ is the trivial bundle}\bigr\}$.
\end{definition}

\begin{proposition}\label{prop:propertiesofS_0M_1}
	Let $\tilde r\in\tilde \Rc_1^\reg$. The following are equivalent:
	\begin{enumerate}
		\item $[\tilde r]\in S_0M_{\set 0,1}$\,.
		\item $\pi_{\bullet}[\tilde r]=0\in\Cc^2=\overline M_{\emptyset,1}\setminus M_{\emptyset,1}$.
		\item $[\tilde r]\notin\pi^*M_{\emptyset,1}$.
		\item $\tilde r=(a_1,a_2,0,b,c)$.
	\end{enumerate}
\end{proposition}
\begin{proof}
	The equivalence $(1)\Leftrightarrow(2)$ follows immediately from the definition of $S_0M_{\set 0,1}$ and the definition of $\pi_{\bullet}$.
	To show that $(1)\Leftrightarrow(3)$ note that $[\tilde r]\notin S_0M_{\set 0,1}$ is equivalent to
	$\pi_\bullet\tilde\Ee_{\tilde r}=\pi_*^{\vee\vee}\tilde\Ee_{\tilde r}\in\Mm_{\emptyset,1}$. 
	From Proposition~\ref{prop:pibuletpi*=id} it follows that $\tilde\Ee\in\pi^*\Mm_{\emptyset,1}\Rightarrow\pi_*^{\vee\vee}\tilde\Ee\in\Mm_{\emptyset,1}$
        so $(1)\Rightarrow(3)$; from 
	Proposition~\ref{prop:caracterizationofpi*M} it follows that $\pi_\bullet\tilde\Ee\in\Mm_{\emptyset,1}\Rightarrow\tilde\Ee\in\pi^*\Mm_{\emptyset,1}$
        so $(3)\Rightarrow(1)$. Now let $\tilde r=(a_1,a_2,d,b,c)$. Then
	$[\tilde r]\in\pi^*M_{\emptyset,1}$ is equivalent to $d$ being an isomorphism (see the beginning of section 3.6 in \cite{Kin89})
	which shows that $(3)\Leftrightarrow(4)$.
\end{proof}


\subsection{Blowups at points \boldmath $z\neq0$}

We identify the blow up $\cpp z$ of $\CP^2$ at a point $z=(z_1,z_2)\in\Cc^2$ with the subvariety of
$\CP^2\times\CP^1$ consisting of the pairs $\bigl([x_1,x_2,x_3],[y_1,y_2]\bigr)$ such that $(x_1-z_1)y_1+(x_2-z_2)y_2=0$.
Translation in $\Cc^2$ induces an action $\tau\colon\Cc^2\times\CP^2\to\CP^2$ and
an isomorphism $\tau_z\colon\cp{\set0}\to\cp{\set z}$ 
given respectively by
\begin{equation*}
  \begin{aligned}
    \tau_z\bigl([x_1,x_2,x_3]\bigr)&=[x_1+z_1x_3,x_2+z_2x_3,x_3]\quad\text{and}\\
    \tau_z\bigl([x_1,x_2,x_3],[y_1,y_2]\bigr)&=\bigl([x_1+z_1x_3,x_2+z_2x_3,x_3],[y_1,y_2]\bigr)\,.
  \end{aligned}
\end{equation*}
Let $\psi_{\set0}\colon M_{\set0,k}\to\Mm_{\set0,k}$ be as in equation~\eqref{eq:defofpsi_emptysetpsi_0}.
The translation map $\tau_z$ allows us to identify,
for any $z\in\Cc^2$, the moduli space $\Mm_{\set z,k}$ with $M_{\set 0,k}$, through the
isomorphism
\begin{equation}\label{eq:defofpsi_z}
  \psi_{\set z}=(\tau_{z}^{-1})^*\circ\psi_{\set0}\colon M_{\set 0,k}\to\Mm_{\set z,k}\,.
\end{equation}

Let $\pi_z\colon\cp{z}\to\CP^2$ be the blowup at $z\in\Cc^2$ and let $\pi=\pi_0$ be the blowup at $0$. Then
\[
\pi_{z}^*\circ\tau_{-z}^*=(\tau_{z}^{-1})^*\circ\pi^*\quad \text{and}\quad
\pi_{z*}^{\vee\vee}\circ(\tau_{z}^{-1})^*=\tau_{-z}^*\circ\pi_*^{\vee\vee}\,.
\]



\subsection{The monad description of the gluing maps}

For $k=1$ we have $\dim W_0=\dim W_1=1$ so there is a canonical isomorphism $\text{End}(W_1)\cong\Cc$. Consider the maps
$f\colon \overline M_{\emptyset,1}\to\Cc$ and $\tilde f\colon \overline M_{\set0,1}\to\Cc$ defined by
\begin{equation}\label{eq:f:M->C^2}
f\bigl([a_1,a_2,b,c]\bigr)=a_1\quad\text{and}\quad \tilde f\bigl([a_1,a_2,d,b,c]\bigr)=da_1\,.
\end{equation}
It follows from Proposition~\ref{prop:propertiesofmonaddescription}(4)
that $\tilde f\circ\pi^*=f$.

\out{%
These maps extend to maps $\Mm_{I,1}\to\Cc$
in the following sense:

\begin{proposition}\label{prop:f_I:M->C}
	There is, for each finite set $I\subset\Cc^2$, a unique map $f_I\colon\Mm_{I,1}\to\Cc$ with the property that $f_\emptyset=f\circ\psi_{\emptyset}^{-1}$ and
	$f_J=f_I\circ\pi_{J,I}^*$ for any $J\subset I$.
\end{proposition}
\begin{proof}
  Uniqueness follows because $\pi_{\emptyset,I}^*\Mm_{\emptyset,1}$ is dense in $\Mm_{I,1}$. Let $z=(z_1,z_2)\in\Cc^2$ and
  let $\psi_\emptyset,$ $\psi_{\set z}$ be as in equations~\eqref{eq:defofpsi_emptysetpsi_0} and~\eqref{eq:defofpsi_z}.
  We define $f_\emptyset=f\circ\psi_{\emptyset}^{-1}$ and $f_{z}=\bigl(\tilde f\circ\psi_{\set z}^{-1}\bigr)+z_1$. Then $f_{z}\circ\pi_{\emptyset,z}^*=f_\emptyset$.
  Since the open sets $\{\pi_{z,I}^*\Mm_{z,1}^V\}_{z\in I}$
  form an open cover of $\Mm_{I,1}^V$ (Lemma~\ref{lemma:theo1nerve}), 
  the collection of maps $\{f_{z}\}_{z\in I}$ determines the map $f_I$.
\end{proof}
}%


\begin{definition}\label{def:boxplustboxplus}
	Let $D=\{(m,m')\in\overline M_{\emptyset,1}\times\overline M_{\emptyset,1}:f(m)\neq f(m')\}$.
	We define $\boxplus\colon D\to\overline M_{\emptyset,2}$ by
	\begin{multline}\label{eq:defOfboxplus}
		[a_1,a_2,b,c]\boxplus[a_1',a_2',b',c']\\
		=\left[
		\begin{pmatrix}a_{1}&0\\0&a_{1}'\end{pmatrix},
		\begin{pmatrix}a_{2}&\frac{bc'}{a_{1}'-a_{1}}\\ \frac{b'c}{a_{1}-a_{1}'}&a_{2}'\end{pmatrix},
		\begin{pmatrix}b\\b'\end{pmatrix},
		\begin{pmatrix}c&c'\end{pmatrix}
		\right]\,.
	\end{multline}
        Let
	\begin{align*}
		\overline M_{\emptyset,1}^{\neq 0}&=\{[a_1,a_2,b,c]\in\overline M_{\emptyset,1}:[a_1,a_2,b,c]\in M_{\emptyset,1}\text{ or }(a_1,a_2)\neq(0,0)\}\,,\\
		\tilde D&=\{(\tilde m,m)\in \overline M_{\set0,1}\times\overline M_{\emptyset,1}^{\neq 0}:\tilde f(\tilde m)\neq f(m)\}\,.
	\end{align*}
	Let $\tboxplus\colon\tilde D\to\overline M_{\set0,2}$  be defined as follows (where $\chi_1$ is as in Notation~\ref{remark:W=W_0orW=W_1}): 
	\begin{multline}\label{eq:defOftboxplus}
		[\tilde a_1,\tilde a_2,\tilde d,\tilde b,\tilde c]\tboxplus[a_{1},a_{2},b,c]\\
		=\left[
		\begin{pmatrix}\tilde a_{1}&0\\0&\chi_1^{-1}a_1\end{pmatrix},
		\begin{pmatrix}\tilde a_{2}&\frac{\tilde bc}{a_1-\tilde d\tilde a_{1}}\\ \frac{\chi_1^{-1}b\tilde c}{\tilde d\tilde a_{1}-a_{1}}&\chi_1^{-1}a_2\end{pmatrix},
		\begin{pmatrix}\tilde d&0\\0&\chi_1\end{pmatrix},
		\begin{pmatrix}\tilde b\\\chi_1^{-1}b\end{pmatrix},
		\begin{pmatrix}\tilde c&c\end{pmatrix}
		\right]\,.
	\end{multline}
\end{definition}

\begin{proposition}\label{prop:propertiesofboxplus}
	The maps $\boxplus$ and $\tboxplus$ are well defined and we have:
	\begin{enumerate}
		\item The map $\boxplus$ is commutative.
		\item Let $\widetilde\Cc^2$ be the blowup of $\Cc^2$ at the origin. If we write $\overline M_{\emptyset,1}=M_{\emptyset,1}\cup \Cc^2$,
                  $\overline M_{\emptyset,1}^{\neq0}=M_{\emptyset,1}\cup (\Cc^2\setminus\{0\})$ and 
		$\overline M_{\set0,1}=M_{\set0,1}\cup\widetilde\Cc^2$
		(see the stratification in equation~\eqref{eq:stratificationDonUhl}) then the restriction of $\boxplus$ and $\tboxplus$ to the strata indicated below is given by the canonical maps:
		\begin{align*}
			\boxplus&\colon \bigl(M_{\emptyset,1}\times\Cc^2\bigr)\cap D\to M_{\emptyset,1}\times\Cc^2\subset\overline M_{\emptyset,2}\\
			\boxplus&\colon\bigl(\Cc^2\times\Cc^2\bigr)\cap D\to S^2\Cc^2\subset\overline M_{\emptyset,2}\\
			\tboxplus&\colon \bigl(M_{\set0,1}\times\Cc^2\setminus\set0\bigr)\cap \tilde D\to M_{\set0,1}\times\widetilde\Cc^2\subset \overline M_{\set0,2}\\
			\tboxplus&\colon \bigl(\widetilde\Cc^2\times M_{\emptyset,1}\bigr)\cap\tilde D\xrightarrow{\ident\times\pi^*}\widetilde\Cc^2\times M_{\set0,1}\subset\overline M_{\set0,2}\\
			\tboxplus&\colon \bigl(\widetilde\Cc^2\times \Cc^2\setminus\set0\bigr)\cap \tilde D\to \widetilde\Cc^2\times\widetilde\Cc^2\to S^2\widetilde\Cc^2\subset\overline M_{\set0,2}
		\end{align*}
		\item Given $(m,m')\in D$ we have $\tau_z^*(m\boxplus m')=(\tau_z^* m)\boxplus(\tau_z^*m')$.
		\item Given $(m,m')\in D$ we have $\pi^*(m\boxplus m')=(\pi^*m)\tboxplus m'$.
		\item Given $(\tilde m,m)\in\tilde D$ we have $\pi_{\bullet}(\tilde m\tboxplus m)=(\pi_{\bullet}\tilde m)\boxplus m$.
	\end{enumerate}
\end{proposition}
\begin{proof}
  First we check that $\tboxplus$ is well defined. Let $\tilde r\in\tilde\Rc_1$ and $r\in\Rc_1$ be such that $\bigl([\tilde r],[r]\bigr)\in\tilde D$, let
  $g,g_1\in GL(W_1)$ and let $g_0\in GL(W_0)$. Then, by equations~\eqref{eq:GLactiononR} and~\eqref{eq:defOftboxplus} we get
  \begin{equation}\label{eq:tboxplusisequivariant}
    \bigl((g_0,g_1)\cdot\tilde r\bigr)\tboxplus(g\cdot r)
    =\Bigl(\bigl(\begin{smallmatrix}g_0&0\\0&\chi_1^{-1} g\chi_1\end{smallmatrix}\bigr),\bigl(\begin{smallmatrix}g_1&0\\0&g\end{smallmatrix}\bigr)\Bigr)\cdot(\tilde r\tboxplus r)
  \end{equation}
  which shows $\tboxplus$ is well defined in the quotient. 
  We also need to check that $[\tilde r]\tboxplus [r]$ satisfies 
  conditions~\eqref{eq:monadconditionfortildeP} and~\eqref{eq:monadintegrabilityfortildeP}. The integrability condition is straightforward. Condition~\eqref{eq:monadconditionfortildeP} is trivially satisfied if
  $\tilde a_1a_1\neq0$ so assume $\tilde a_1a_1=0$. Since 
  $\tilde f(\tilde m)\neq f(m)$ we cannot have $\tilde a_1=a_1=0$.
  If $\tilde a_1=0$ then $a_1\neq 0$ and since $[\tilde r]$
  satisfies condition~ \eqref{eq:monadconditionfortildeP},
  either $\tilde a_2\neq0$ or $\tilde b\neq0$. In either case
  $[\tilde r]\tboxplus [r]$ satisfies 
  condition~\eqref{eq:monadconditionfortildeP}.
  If $a_1=0$ then $\tilde a_1\neq0$ and since
  $[r]\in \overline M_{\emptyset,1}^{\neq 0}$,
  either $a_2\neq0$ or $b\neq0$. In either case 
  condition~\eqref{eq:monadconditionfortildeP} is satisfied.
  This shows that $\tboxplus$ is well defined.
  The proof that $\boxplus$ is well defined is completely analogous.
	Statements (3), (4) and (5) follow easily from Propositions~\ref{prop:propertiesofmonaddescription}(3)(4) and Proposition~\ref{prop:formulaforpibullet} if
	$m,m'\in M_{\emptyset,1}$ and $\tilde m\in M_{\{0\},1}$,
	  and the general case follows by continuity since 
	  $M_{\emptyset,1}\subset\overline M_{\emptyset,1}$
	  and $M_{\{0\},1}\subset\overline M_{\{0\},1}$ are dense.
	We prove (1) and (2):
	\begin{enumerate}
	\item Let $r,r'\in\Rc_1$ be such that $[r],[r']\in D$ and let $g=\bigl(\begin{smallmatrix}0&1\\1&0\end{smallmatrix}\bigr)\in GL(W_1)$.
                Then $g^{-1}(r\boxplus r')g=r'\boxplus r$
		so, in the quotient, $[r]\boxplus [r']=[r']\boxplus [r]$.
		\item We only prove the statement for $\boxplus$ since the proof for $\tboxplus$ is entirely analogous.
		Suppose $r'=(a_1',a_2',,b',c')\notin\Rc_1^\reg$. Then we must have either $b'=0$ or $c'=0$. Assume $c'=0$
(the proof for the case $b'=0$ is completely analogous). Let $r=(a_1,a_2,b,c)$
		and for each $\varepsilon>0$ let $g=\bigl(\begin{smallmatrix}1&0\\0&\varepsilon\end{smallmatrix}\bigr)\in GL(W_1)$. Then
		\[
		g\cdot(r\tboxplus r')
		=\left(
		\begin{pmatrix} a_{1}&0\\0&a_1'\end{pmatrix},
		\begin{pmatrix} a_{2}&0\\ \varepsilon\frac{b' c}{ a_1-a_1'}&a_2'\end{pmatrix},
		\begin{pmatrix} b\\\varepsilon b'\end{pmatrix},
		\begin{pmatrix} c&0\end{pmatrix}
		\right)\,.
		\]
		Taking the limit when $\varepsilon\to0$ we get $r\boxplus r'\sim(a_1,a_2,b,c)\oplus(a_1',a_2',0,0)$
		which corresponds to $\bigl([r],(a_1',a_2')\bigr)\in\overline M_{\emptyset,2}$ if $r\in\Rc_1^\reg$ or $\bigl((a_1,a_2),(a_1',a_2')\bigr)\in S^2\Cc^2$ if $r\notin\Rc_1^\reg$
                (see section~\ref{subsec:completionofM}).\qedhere
	\end{enumerate}
\end{proof}

\subsection{The spaces \boldmath $\Mm_I^{\Vv,U}$}

In order to extend the maps $\boxplus$, $\tboxplus$ in Definition~\ref{def:boxplustboxplus} to the moduli spaces $\Mm_{I,1}$
we will replace these moduli spaces by homeomorphic subspaces.
Let $\pi_{J,I}^*$ be as in Definition~\ref{def:omega,pi*,Malpha}(4). Identify the spaces $\Mm_\emptyset$ and $M_\emptyset$ through the isomorphism
$\psi_\emptyset$ (see equation~\eqref{eq:defofpsi_emptysetpsi_0}).
  The maps $f$ and $\tilde f$ in equation~\eqref{eq:f:M->C^2} extend to maps $\Mm_{I,1}\to\Cc$
  in the following sense:

  \begin{proposition}\label{prop:f_I:M->C}
    There is, for each finite set $I\subset\Cc^2$, a unique continuous map $f_I\colon\Mm_{I,1}\to\Cc$ with the property that $f_\emptyset=f$ and
    $f_J=f_I\circ\pi_{J,I}^*$ for any $J\subset I$.
  \end{proposition}
  \begin{proof}
    Uniqueness follows because $\pi_{\emptyset,I}^*\Mm_{\emptyset,1}$ is dense in $\Mm_{I,1}$. Let $z=(z_1,z_2)\in\Cc^2$ and
    let $\psi_{\set z}$ be as in equation~\eqref{eq:defofpsi_z}.
    We define $f_{z}=\bigl(\tilde f\circ\psi_{\set z}^{-1}\bigr)+z_1$. Then $f_{z}\circ\pi_{\emptyset,z}^*=f_\emptyset$.
    The open sets $\{\pi_{z,I}^*\Mm_{z,1}\}_{z\in I}$
    form an open cover of $\Mm_{I,1}$ and 
    the collection of maps $\{f_{z}\}_{z\in I}$ is compatible 
    in the intersections (see Lemma~\ref{lemma:theo1nerve})
    so $\{f_{z}\}_{z\in I}$ defines the required map $f_I$.
  \end{proof}


\begin{definition}\label{def:UgoodrelI}
  We say an open set $U\subset \Cc$ is a Jordan open set if $U$ is the bounded region inside a Jordan curve.

  Let $U\subset\Cc$ be a Jordan open set and let $I\subset U\times\Cc$ be a finite set. Let $\Vv$ be a finite dimensional complex Hermitian vector space. Then we define:
  \begin{equation}\label{eq:DefOfM^U}
    \Mm_{I,1}^{\Vv,U}=f_I^{-1}(U)\,,\qquad
    \Mm_{I,k}^{\Vv,U}=\Mm_{I,k}^\Vv\quad\text{(for $k\neq 1$).}
  \end{equation}
\end{definition}

Except for Proposition~\ref{prop:mapsboxplus},
we omit the vector space $\Vv$ from the notation.

\begin{proposition}\label{prop:MUishomeoandmoretoM}
Let $U$, $I$ be as in Definition~\ref{def:UgoodrelI}.
	Then there is a homeomorphism $\Mm_I^{U}\xrightarrow{\cong}\Mm_I$ which is homotopic to the inclusion map.
\end{proposition}
\begin{proof}
	The statement is a tautology in degrees $k\neq 1$ so let $k=1$. 
	Let $p\colon\Cc^2\to\Cc$ be projection onto the first factor.
        First we show that there is a neighbourhood $U'$ of $p(I)$ in $U$ and a homotopy $H\colon U\times [0,1]\to\Cc$
        between the inclusion $U\to \Cc$ and a homeomorphism $U\cong\Cc$ such that $H(z,t)=z$ for all $z\in U'$.
        By the Schoenflies Theorem there is a homeomorphism $h\colon\Cc\to\Cc$ such that $h(U)=D$ is the open unit disk.
        We let $H=h^{-1}\circ H_D\circ(h\times\ident)$ where $H_D\colon D\times[0,1]\to\Cc$ is a homotopy between the inclusion and a homeomorphism $D\cong\Cc$ which is the identity
        on a smaller disk $D'\subset D$ containing $h\bigl(p(I)\bigr)$. Then $H$ has the required properties, with $U'=h^{-1}(D')$.

        Identifying $M_{\emptyset,1}$ with $\Mm_{\emptyset,1}$,
        we define the homotopy $H_\emptyset\colon \Mm_{\emptyset,1}^{U}\times[0,1]\to \Mm_{\emptyset,1}$ by $H_\emptyset\bigl([a_1,a_2,b,c],t\bigr)=\bigl[H(a_1,t),a_2,b,c\bigr]$.
        Then:
        \begin{equation}\label{eq:Hemptyset=idonU'}
          m\in \Mm_{\emptyset,1}^{U'}\ \Rightarrow\ H_\emptyset(m,t)=m\,.
        \end{equation}
        For each $z\in I$ let $\pi_z^*\colon \Mm_{\emptyset,1}\to\Mm_{z,1}$ and $\pi_{z\bullet}\colon\Mm_{z,1}\to \overline \Mm{}_{\emptyset,1}$ be the maps induced by the projection.
        From Proposition~\ref{prop:propertiesofS_0M_1}(4), if 
        $[\tilde r]\in S_0 M_{\set 0,1}$ then $f_{\{0\}}([\tilde r])=0$ so
        $S_0 M_{\set 0,1}\subset M_{\set 0,1}^{U'}$.
        Let $S_0\Mm_{z,1}\subset\Mm_{z,1}$ be the image of 
        $S_0 M_{\set 0,1}$ under the isomorphism 
        $\psi_{\set z}\colon M_{\set 0,1}\to \Mm_{z,1}$. Then
        $S_0\Mm_{z,1}\subset \Mm_{z,1}^{U'}$ and
        from Proposition~\ref{prop:propertiesofS_0M_1}(3) we have $\pi_z^*\Mm_{\emptyset,1}^{U}\cup S_0\Mm_{z,1}=\Mm_{z,1}^{U}$ so:
        \[
        \pi_z^*\Mm_{\emptyset,1}^{U}\cup \Mm_{z,1}^{U'}=\Mm_{z,1}^{U}\,.
        \]
        Define $\tilde H_z\colon \Mm_{z,1}^{U}\times[0,1]\to \Mm_{z,1}$ by
	\[
\tilde H_z(\tilde m,t)=
\begin{cases}\pi_z^*\circ H_\emptyset(\pi_{z\bullet}\tilde m,t)& \tilde m\in\pi_z^*\Mm_{\emptyset,1}^{U}\,;\\ \tilde m & \tilde m\in \Mm_{z,1}^{U'}\,.\end{cases}
	\]
	We need to check that $\tilde H_z$ is well defined and continuous. 
        From equation~\eqref{eq:pibuletpi*=id} it follows that, for any $m\in \Mm_{\emptyset,1}^{U}$ we have:
        \begin{equation}\label{eq:tildeHcircpi*=pi*circH}
          \tilde H_z(\pi_z^*m,t)=\pi_z^*H_\emptyset(m,t)\,.
        \end{equation}
        If $\tilde m=\pi_z^*m\in\pi_z^*\Mm_{\emptyset,1}^{U}\cap \Mm_{z,1}^{U'}$ then $m\in \Mm_{\emptyset,1}^{U'}$ so,
        by equations~\eqref{eq:Hemptyset=idonU'} and~\eqref{eq:tildeHcircpi*=pi*circH}:
        \[
        \tilde H_z(\pi_z^*m,t)=\pi_z^*H_\emptyset(m,t)=\pi_z^*m
        \]
        and $\tilde H_z$ is well defined.
        By Lemma~\ref{lemma:theo1nerve} the set $\pi_z^*\Mm_{\emptyset,1}^{U}$ is open, and $\Mm_{\set0,1}^{U'}$ is clearly open so $\tilde H_z$ is continuous.
        From equation~\eqref{eq:tildeHcircpi*=pi*circH} it follows that the homotopies $\{H_z\}_{z\in I}$  patch together to define the homotopy $H_I\colon\Mm_{I,1}^{U}\times[0,1]\to\Mm_{I,1}$.
\end{proof}

\begin{proposition}\label{prop:boxplusisembedding}
	Let $U_1, U_2\subset\Cc$ be disjoint open sets. Then: 
	\begin{enumerate}
		\item The map $\boxplus\colon \overline M_{\emptyset,1}^{U_1}\times \overline M_{\emptyset,1}^{U_2}\to \overline M_{\emptyset,2}$ is a homeomorphism onto the subspace
		\[
		N_0=\{[a_1,a_2,b,c]\in \overline M_{\emptyset,2}:\text{each $U_i$ contains exactly one eigenvalue of $a_1$}\}\,.
		\]
		\item If $0\in U_1$ then the map $\tboxplus\colon \overline M_{\set0,1}^{U_1}\times \overline M_{\emptyset,1}^{U_2}\to \overline M_{\set0,2}$ is a homeomorphism onto the subspace
		\[
		N_1=\{(a_1,a_2,d,b,c)\in \overline M_{\set0,2}:\text{each $U_i$ contains exactly one eigenvalue of $da_1$}\}\,.
		\]
	\end{enumerate}
\end{proposition}
\begin{proof}
	This is proven in \cite[Proposition 4.5(1)]{San05} in the non-degenerate case and when $U_1$ and $U_2$ are disjoint balls but the exact same proof works in our case.
	We sketch how the inverse of the map $\tboxplus$ is constructed. Let
	$\tilde m=[\tilde a_1,\tilde a_2,\tilde d,\tilde b,\tilde c]\in N_1$. Then the eigenvalues of $\tilde d\tilde a_1$ are distinct. Fix an eigenvector basis for $\tilde d\tilde a_1$. Then each eigenvector $v$ determines, up to rescaling, an eigenvector of
	 $\tilde a_1\tilde d$ with the same eigenvalue: if $\tilde a_1v\neq0$, then we can take $w=\tilde a_1v$; if 
	 $\tilde a_1v=0$ then we choose $w$ generating the kernel of 
	 $\tilde a_1\tilde d$.
	 With this choice of basis, $\tilde m$ is of the form
        \begin{equation}\label{eq:proofofembedding}
        \begin{pmatrix}\tilde a_{1}'&0\\0&\chi_1^{-1}a_1'\end{pmatrix},\ 
          \begin{pmatrix}\tilde a_{2}'&\frac{\tilde b'c'}{a_1'-\tilde d'\tilde a_{1}'}\\ \frac{\chi_1^{-1}b'\tilde c'}{\tilde d'\tilde a_{1}'-a_{1}'}&\chi_1^{-1}a_2'\end{pmatrix},\ 
            \begin{pmatrix}\tilde d'&0\\0&\chi_1\end{pmatrix},\ 
              \begin{pmatrix}\tilde b'\\\chi_1^{-1}b'\end{pmatrix},\ 
                \begin{pmatrix}\tilde c'&c'\end{pmatrix}                  
        \end{equation}
        (compare with equation~\eqref{eq:defOftboxplus}) and the action of
	$GL(W_0)\times GL(W_1)$ reduces to the action of pairs of the form
	$\bigl(\bigl(\begin{smallmatrix}g_0&0\\0&\chi_1^{-1}g\chi_1\end{smallmatrix}\bigr),\bigl(\begin{smallmatrix}g_1&0\\0&g\end{smallmatrix}\bigr)\bigr)$
            (compare with equation~\eqref{eq:tboxplusisequivariant}). From the matrices in equation~\eqref{eq:proofofembedding}
            we can recover $[\tilde a_1',\tilde a_2',\tilde d',\tilde b',\tilde c']\in\overline M_{\set0,1}$ and
        $[a_1',a_2',b',c']\in\overline M_{\emptyset,1}$.
\end{proof}

\subsection{The product maps}

Given disjoint finite sets $I_1$, $I_2$ and disjoint open sets $U_1$, $U_2$ containing them
we wish to define maps $\boxplus_{I_1,I_2}\colon\Mm_{I_1,1}^{U_1}\times\Mm_{I_2,1}^{U_2}\to\Mm_{I_1\cup I_2,2}$.
Since the space $\Mm_{I,1}$ has an open cover $\{\pi_{x,I}^*\Mm_{x,1}\}_{x\in I}$,
we only need to define the map in the cases when
$I_1$ and $I_2$ have 1 or less elements and check the compatibility
on the intersections.
Let $\psi_{\set z}$ be as in equation~\eqref{eq:defofpsi_z}. Identifying $M_\emptyset$ with $\Mm_\emptyset$, we define:
\begin{equation}\label{eq:defofboxplus_I}
\boxplus_{\emptyset,\emptyset}=\boxplus\,,\quad
\boxplus_{z,\emptyset}=\psi_{\set z}\circ\tboxplus\circ(\psi_{\set z}^{-1}\times\tau_z^*)\,,\quad
\boxplus_{\emptyset,z}=\boxplus_{z,\emptyset}\circ\sigma
\end{equation}
where $\sigma$ is the permutation $\sigma(m_\emptyset,m_z)=(m_z,m_\emptyset)$.
Then it follows from Proposition~\ref{prop:propertiesofboxplus}(3)(4) that $(\pi_{\emptyset,z}^* m)\boxplus_{z,\emptyset} m'=\pi_{\emptyset,z}^*(m\boxplus_{\emptyset,\emptyset}m')$.
Next we consider the case where $I_1$ and $I_2$ are both singletons:

\begin{proposition}\label{app:thm:boxplusI-newversion}
	Let $U_1,U_2\subset\Cc$ be disjoint open sets, let $x\in U_1\times\Cc$, $y\in U_2\times\Cc$ and let $I=\{x,y\}$.
	There is an open embedding $\boxplus_{x,y}:\Mm_{x,1}^{U_1}\times \Mm_{y,1}^{U_2}\to \Mm_{I,2}$ such that the following diagram commutes:
	\[
	\xymatrix@C=4em{
		\Mm_{x,1}^{U_1}\times\Mm_{\emptyset,1}^{U_2}\ar[r]^-{\ident\times\pi_{\emptyset,y}^*}\ar[d]^{\boxplus_{x,\emptyset}}&
		\Mm_{x,1}^{U_1}\times\Mm_{y,1}^{U_2}\ar[d]^{\boxplus_{x,y}}&
		\Mm_{\emptyset,1}^{U_1}\times\Mm_{y,1}^{U_2}\ar[l]_-{\pi_{\emptyset,x}^*\times\ident}\ar[d]^{\boxplus_{\emptyset,y}}\\
		\Mm_{x,2}\ar[r]^-{\pi_{x,I}^*}&
		\Mm_{I,2}&
		\Mm_{y,2}\ar[l]_-{\pi_{y,I}^*}
	}
	\]    
        Furthermore, 
        the image of $\boxplus_{x,y}$ is given by
        \begin{equation*}
        \mathrm{Im}\,\boxplus_{x,y}=\bigl(\pi_{x,I}^*N_x\cup\pi_{y,I}^*N_y\cup C\bigr)\cap\Mm_{I,2}
        \end{equation*}
        where $N_x$, $N_y$ are the images of $\boxplus_{x,\emptyset}$ and $\boxplus_{\emptyset,y}$ respectively,
        and $C$ is the complement of $\pi_{x,I}^*\Mm_{x,2}\cup\pi_{y,I}^*\Mm_{y,2}$ in $\Mm_{I,2}$.
\end{proposition}
\begin{proof}
	The proof follows the same lines as the proof of Proposition~4.9 in \cite{San05}. 
	Given $\tilde m_x\in\Mm_{x,1}^{U_1}$ and $\tilde m_y\in\Mm_{y,1}^{U_2}$ we define $m=\tilde m_x\boxplus_{x,y} \tilde m_y$ 
	as the unique solution of the system of equations:
	\begin{align}
		(\pi_{x,I})_\bullet m
		&=\tilde m_x\boxplus_{x,\emptyset}(\pi_{\emptyset,y})_\bullet\tilde m_y\label{eq:defofboxplusxy-eq1}\\
		(\pi_{y,I})_\bullet m
		&=(\pi_{\emptyset,x})_\bullet\tilde m_x\boxplus_{\emptyset,y}\tilde m_y\label{eq:defofboxplusxy-eq2}
	\end{align}
	We need to prove existence and uniqueness of solution. We have two cases (see Proposition~\ref{prop:propertiesofS_0M_1}(3)):
	\begin{enumerate}
		\item Suppose that 
		$\tilde m_x=\pi_{\emptyset,x}^*m_x$ for some $m_x\in \Mm_{\emptyset,1}$. Then, by Proposition~\ref{prop:pibuletpi*=id},
		  equation~\eqref{eq:defofboxplusxy-eq2} becomes $(\pi_{y,I})_\bullet m=m_x\boxplus_{\emptyset,y}\tilde m_y\in \Mm_{y,2}$ which,
          by Proposition~\ref{prop:caracterizationofpi*M},
          is equivalent to:        
		\begin{equation}\label{eq:boxplusxycommuteswithpi*}
		m=\pi_{y,I}^*(\pi_{y,I})_\bullet m=\pi_{y,I}^*(m_x\boxplus_{\emptyset,y} \tilde m_y)\,.
		\end{equation}
		This shows uniqueness of the solution. To show existence
		we must check that the solution of
		equation~\eqref{eq:defofboxplusxy-eq2} given in equation~\eqref{eq:boxplusxycommuteswithpi*} satisfies also equation~\eqref{eq:defofboxplusxy-eq1}. Applying
		Corollary~\ref{coro:pi_*pi^*} and
		Proposition~\ref{prop:propertiesofboxplus}(4)(5) we get
		\begin{multline*}
			(\pi_{x,I})_\bullet(\pi_{y,I})^*(m_x\boxplus_{\emptyset,y}\tilde m_y)=(\pi_{\emptyset,x})^*(\pi_{\emptyset,y})_\bullet(m_x\boxplus_{\emptyset,y}\tilde m_y)\\
			=(\pi_{\emptyset,x})^*\bigl(m_x\boxplus_{\emptyset,\emptyset}(\pi_{\emptyset,y})_\bullet\tilde m_y\bigr)
			=(\pi_{\emptyset,x}^*m_x)\boxplus_{x,\emptyset}(\pi_{\emptyset,y})_\bullet\tilde m_y\,.
		\end{multline*}
		The case where $\tilde m_y=\pi_{\emptyset,y}^*m_y$ for some $m_y\in \Mm_{\emptyset,1}$ is completely analogous.
		\item Suppose $\tilde m_x\in S_0\Mm_{x,1}$ and $\tilde m_y\in S_0\Mm_{y,1}$ (see Definition~\ref{def:S_0M_1}). Then, by Proposition~\ref{prop:propertiesofS_0M_1}(2) we have
		$\pi_{\emptyset,x\bullet}\tilde m_x=x$ and $\pi_{\emptyset,y\bullet}\tilde m_y=y$ so
		by Proposition~\ref{prop:propertiesofboxplus}(2)
		\begin{align*}
			(\pi_{x,I})_\bullet m&=\tilde m_x\boxplus_{x,\emptyset}(\pi_{\emptyset,y\bullet}\tilde m_y)=(\tilde m_x,y)\in \Mm_{x,1}\times\Cc^2\subset\overline\Mm_{x,2}\\
			(\pi_{y,I})_\bullet m&=(\pi_{\emptyset,x\bullet}\tilde m_x)\boxplus_{\emptyset,y}\tilde m_y=(\tilde m_y,x)\in \Mm_{y,1}\times\Cc^2\subset\overline\Mm_{y,2}
		\end{align*}
		and these equations are equivalent to
		\[
		(\pi_{x,I})_*^{\vee\vee} m=\tilde m_x\,, \quad
		(\pi_{y,I})_*^{\vee\vee} m=\tilde m_y\,.
		\]
		But the map $(\pi_{x,I})_*^{\vee\vee}\times(\pi_{y,I})_*^{\vee\vee}$ is a homeomorphism
		by \cite[Proposition 4.3]{San05}, so this system of equations has a unique solution, namely:
		\[
		m=\Bigl((\pi_{x,I})_*^{\vee\vee}\times(\pi_{y,I})_*^{\vee\vee}\Bigr)^{-1}(\tilde m_x,\tilde m_y)\,.
		\]
	\end{enumerate}
        The statement about the image of $\boxplus_{x,y}$ easily follows from (1) and (2) above.
	Now we must show continuity. Consider the Gieseker compactification $\overline{\Mm_{I,2}}{}^{\text{Gie}}$, defined as the closure of $\Mm_{I,2}$ inside
	the moduli space of framed sheaves \cite[section 2.3]{San02}, \cite[Theorem 1.21]{HuLe95}.
	Pick sequences $\tilde m_k\in\Mm_{x,1}$, $\tilde m'_k\in \Mm_{y,1}$ converging respectively to $\tilde m$, $\tilde m'$.
	If $\tilde m_k\boxplus_{x,y}\tilde m_k'$ doesn't converge to $\tilde m\boxplus_{x,y}\tilde m'$ then, by compactness, there is a subsequence
	$\tilde m_{k_n}\boxplus_{x,y}\tilde m_{k_n}'$ converging to a framed sheaf $(\Ee,\phi)$ over $\cp{I}$, with $(\Ee,\phi)\neq \tilde m\boxplus_{x,y}\tilde m'$.
	By the continuity of the direct image maps
	(see \cite[appendix D]{San02}) and pullback maps, equations~\eqref{eq:defofboxplusxy-eq1} and~\eqref{eq:defofboxplusxy-eq2} imply that
	\[
	(\pi_{x,I})_\bullet\Ee=\tilde m\boxplus_{x,\emptyset}(\pi_{\emptyset,y})_\bullet\tilde m'\,,\quad
	(\pi_{y,I})_\bullet\Ee=(\pi_{\emptyset,x})_\bullet\tilde m\boxplus_{\emptyset,y}\tilde m'\,.
	\]
	The first equation implies that the sheaf $\Ee$ is locally free
	away from $y$ and from the second equation it follows that $\Ee$
	is locally free away from $x$, so
$\Ee$ is locally free, and hence a holomorphic vector bundle. Uniqueness of solution implies that $\Ee=\tilde m\boxplus_{x,y}\tilde m'$ which finishes the proof of continuity.
	
	It remains to be shown that the inverse is also continuous. Consider sequences $\tilde m_k\in\Mm_{x,1}$ and $\tilde m'_k\in \Mm_{y,1}$ such that
	$\tilde m_k\boxplus_{x,y}\tilde m_k'\to \tilde m\boxplus_{x,y}\tilde m'$. Then equations~\eqref{eq:defofboxplusxy-eq1} and~\eqref{eq:defofboxplusxy-eq2} imply that
	\begin{align*}
		\lim_{k\to+\infty} \bigl(\tilde m_k\boxplus_{x,\emptyset}(\pi_{\emptyset,y})_\bullet\tilde m_k'\bigr)&=\tilde m\boxplus_{x,\emptyset}(\pi_{\emptyset,y})_\bullet\tilde m'\\
		\lim_{k\to+\infty} \bigl((\pi_{\emptyset,x})_\bullet\tilde m_k\boxplus_{\emptyset,y}\tilde m_k'\bigr)&=(\pi_{\emptyset,x})_\bullet\tilde m\boxplus_{\emptyset,y}\tilde m'
	\end{align*}
	By Proposition~\ref{prop:boxplusisembedding} the maps $\boxplus_{x,\emptyset}$ and $\boxplus_{\emptyset,y}$ are embeddings so
        $\tilde m_k\to \tilde m$ and $\tilde m_k'\to \tilde m'$, which finishes the proof.
\end{proof}

\begin{corollary}\label{cor:opencoverofM_2}
  The collection $\bigl\{\pi_{x,I}^*\Mm_{x,2},\pi_{y,I}^*\Mm_{y,2},\mathrm{Im}\,\boxplus_{x,y}\bigr\}$ is an open cover of $\Mm_{I,2}$.
\end{corollary}
\begin{proof}
  This is \cite[Theorem 4.1]{San05} if we take $U_1$ and $U_2$ to be open disks around $x$ and $y$ respectively (see \cite[Definition 4.1]{San05}): with the notation in \cite{San05} we have
  $A_L=\pi_{x,I}^*\Mm_{x,2}$, $A_R=\pi_{y,I}^*\Mm_{y,2}$ and $N_2=\text{Im}\,\boxplus_{x,y}$.
  For general $U_1$ and $U_2$ the only thing that needs to be checked is that $\text{Im}\,\boxplus_{x,y}$ is open. This is \cite[Proposition 4.8]{San05}, whose proof 
  only uses the fact that $x\in U_1\times\Cc$ and $y\in U_2\times\Cc$.
\end{proof}

\begin{proposition}\label{prop:mapsboxplus}\label{app:prop:boxplus}
	Let $k_1,\dots,k_n$ be nonnegative integers such that $k=\sum k_i\leq 2$ and
	let $U_1,\dots,U_n\subset \Cc$ be pairwise disjoint open sets all contained in an open set $U\subset\Cc$. For each $n$-tuple of pairwise disjoint finite sets
        $\mathbf I=(I_1,\dots,I_n)$ such that $I_i\subset U_i\times\Cc$ 
	there is a map $\boxplus_{\mathbf I}\colon \Mm_{I_1,k_1}^{\Vv,U_1}\times\dots\times \Mm_{I_n,k_n}^{\Vv,U_n}\to \Mm_{I,k}^{\Vv,U}$ (with $I=\bigcup_iI_i$) such that:
	\begin{enumerate}
		\item Given any $\mathbf J=(J_1,\dots,J_n)$ with $J_i\subset I_i$ for any $i$, we have
		\[
		\pi_{J,I}^*\circ\boxplus_{\mathbf J}=\boxplus_{\mathbf I}\circ\bigl(\pi_{J_1,I_1}^*\times\dots\times\pi_{J_n,I_n}^*\bigr)\qquad
		\text{(where $J=\textstyle\bigcup_iJ_i$).}
		\]
		\item Given a linear isometry $\alpha\colon \Vv\to \Vv'$, the following diagram is commutative:
		\[
		\xymatrix{
			\prod\Mm_{I_i,k_i}^{\Vv,U_i}\ar[r]^-{\boxplus_{\mathbf I}}\ar[d]_{\alpha\times\dots\times\alpha}&
			\Mm_{I,k}^{\Vv,U}\ar[d]^{\alpha}\\
			\prod\Mm_{I_i,k_i}^{\Vv',U_i}\ar[r]^-{\boxplus_{\mathbf I}}&
			\Mm_{I,k}^{\Vv',U}
		}
		\]
		\item Let $\imath_i:\Vv\to \Vv^m$ be the inclusion onto the $i$-th factor and
		let $\omega$ be the map induced by pullback and Whitney sum.
		Then we have a commutative diagram
		\[
		\xymatrix{
			\prod \Mm^{\Vv,U_i}_{I_i,k_i}\ar[r]^-{\omega}\ar[d]_-{\prod\imath_i}&\Mm_{I,k}^{\Vv^m,U}\\
			\prod \Mm^{\Vv^m,U_i}_{I_i,k_i}\ar[ru]_-{\boxplus_{\mathbf I}}}
		\]
		\item Given a permutation $\sigma\in\Sigma_n$, 
		let $\sigma\mathbf I=(I_{\sigma(1)},\dots,I_{\sigma(n)})$.
		Then the following  diagram is commutative:
		\[
		\xymatrix{
			\prod\Mm_{I_i,k_i}^{\Vv,U_i}\ar[r]^-{\boxplus_{\mathbf I}}\ar[d]_{\sigma}&
			\Mm_{I,k}^{\Vv,U}\\
			\prod\Mm_{I_{\sigma(i)},k_{\sigma(i)}}^{\Vv,U_{\sigma(i)}}
			\ar[ru]_{\boxplus_{\sigma\mathbf I}}}
		\]
		\item The maps $\boxplus_{\mathbf I}$ are associative, that is,
		if $m<n$ then
		\[
		\boxplus_{\mathbf I}=\boxplus_{I_1\cup\dots\cup I_m,I_{m+1},\dots,I_n}
		\circ\bigl(\boxplus_{I_1,\dots,I_m}\times\ident\bigr)
		\]
	\end{enumerate}
\end{proposition}
\begin{proof}
	If $k_i=k$ and $k_j=0$ for $j\neq i$ then $\boxplus_{\mathbf I}$ is defined by pullback:
	\[
	\boxplus_{\mathbf I}=\pi_{I_i,I}^*\colon\Mm_{I_i,k}^{\Vv,U_i}\to \Mm_{I,k}^U\,.
	\] 
	The only other case is the one where $k_i=k_j=1$ for some $i\neq j$ and $k=2$.
	The map $\boxplus_{\mathbf I}$ will be defined as the composition of a certain map $\boxplus_{I_i,I_j}\colon\Mm_{I_i,1}^{\Vv,U_i}\times\Mm_{I_j,1}^{\Vv,U_j}\to\Mm_{I_i\cup I_j,2}^{\Vv,U}$
	with pullback $\pi_{I_i\cup I_j,I}^*\colon \Mm_{I_i\cup I_j,2}^{\Vv,U}\to\Mm_{I,2}^{\Vv,U}$. It remains to define the maps $\boxplus_{I_i,I_j}$.
	Using the open cover from Lemma~\ref{lemma:theo1nerve} we can reduce to the case where $I_i$ and $I_j$ are either empty or have only one element. These cases were taken care of
	in equation~\eqref{eq:defofboxplus_I} and Proposition~\ref{app:thm:boxplusI-newversion}. Property (1) is clear and 
	property (5) follows easily from (1).
    It remains to prove properties (2), (3) and (4).
	First we consider the case when $I=\emptyset$. 
	\begin{enumerate}
		\item[(2)] Follows immediately from Proposition~\ref{prop:propertiesofmonaddescription}(1) and the definition of $\boxplus$ (equation~\eqref{eq:defOfboxplus}).
		\item[(3)] Let $m=[a_1,a_2,b,c]\in M_{\emptyset,1}^{\Vv,U_i}$ and $m'=[a_1',a_2',b',c']\in M_{\emptyset,1}^{\Vv,U_j}$. By Proposition~\ref{prop:propertiesofmonaddescription}(1)
		we have $\imath_i(m)=[a_1,a_2,b\,\imath_i^*,\imath_i c]$ and $\imath_j(m')=[a_1',a_2',b'\imath_j^*,\imath_j c']$.
		But now $b\,\imath_i^*\imath_j c'=0$ and $b'\imath_j^*\imath_i c=0$ from which it follows that the matrices in equation~\eqref{eq:defOfboxplus}
		are block diagonal and hence $\imath_i(m)\boxplus\imath_j(m')=\imath_i(m)\oplus\imath_j(m')$ which, by Proposition~\ref{prop:propertiesofmonaddescription}(2), corresponds to Whitney sum.
		\item[(4)] This was
		proven in Proposition~\ref{prop:propertiesofboxplus}(1).
	\end{enumerate}
	The results extend to the general case by continuity, since $\pi_{\emptyset,I_i}^*\Mm_{\emptyset,1}^{\Vv,U_i}$ 
	is dense in $\Mm_{I_i,1}^{\Vv,U_i}$.
\end{proof}

\section{The bar constructions in finite rank}\label{app2}\label{sec:6}





\out{
	\begin{proposition}
		If $\#I=\#J$ then the moduli spaces $\Mm_I$ and $\Mm_J$ are isomorphic.
	\end{proposition}
	\begin{proof}
		Either use moduli of instantons or gluing of holomorphic bundles.
	\end{proof}
}%

The objective of this section is to prove Theorem~\ref{theor0}.
Up until section~\ref{sec:8.5proof1.1}
we work with a fixed finite dimensional complex Hermitian vector space
$\Vv$ which we omit from the notation.
First we define,
for each finite set $L\subset \Cc^2$, the degree 1 and 2 components of the bar construction (see section~\ref{sec:prelim-bar-monoid}):
\[
B_L(\bullet)=\BAR_\bullet\Bigl(\Mm_{\emptyset},\tprod_{x\in L}\Mm_{\emptyset},\tprod_{x\in L}\Mm_{x}\Bigr)
\]
and, given disjoint finite sets $J$ and $K$, the degree 1 and 2 components of the bar constructions and the maps appearing in the following diagram:
\begin{equation}\label{eq:diagramofbarsinfiniterank}
\xymatrix{
  \BAR\bigl(|B_J|,\Mm_\emptyset,|B_K|\bigr)\ar[r]\ar[d] &
  |B_I|\ar[d] \\
  \BAR(\Mm_J,\Mm_\emptyset,\Mm_K)\ar[r] &
  \Mm_I
}
\end{equation}
For each Jordan open set $U$ (see Definition~\ref{def:UgoodrelI})
and each finite set $I\subset U\times\Cc$, let $\Mm_I^{U}$ be as in equation~\eqref{eq:DefOfM^U}. 
The face maps will be defined using the gluing maps $\boxplus_{\mathbf I}$
from Proposition~\ref{prop:mapsboxplus}. 
In degree $k=1$ all non-trivial gluing maps
are pullback maps.
In degree $k=2$ we will need:
\begin{itemize}
	\item The pullback maps $\pi_{J,I}^*\colon\Mm_J^{U_1}\to\Mm_I^{U_2}$, with $J\subset I$ and $U_1\subset U_2$, obtained by restriction 
	of the pullback maps in Definition~\ref{def:omega,pi*,Malpha}(4);
	\item The gluing maps $\boxplus_{I_1,I_2}\colon \Mm_{I_1,1}^{U_1}\times\Mm_{I_2,1}^{U_2}\to \Mm_{I_1\cup I_2,2}^{U}$,  with $U_1,U_2\subset U$ and
	$U_1\cap U_2=\emptyset$;
	\item The inclusion maps $i\colon\Mm_I^{U_1}\to\Mm_I^{U_2}$, with
	$U_1\subset U_2$.
\end{itemize}
The maps $\boxplus_{I_1,I_2}$ require a choice of pairs of 
disjoint Jordan open sets $U_1$, $U_2$, and for the inclusion maps 
and pullback maps to be well defined 
we also need larger open sets containing certain families of open sets.
We will now  make these observations precise.
Fix disjoint finite sets $J,K\subset\Cc^2$ and let $I=J\cup K$.
	Let $p\colon\Cc^2\to\Cc$ be projection onto the first factor. Given $z\in\Cc$ let $\Re z$, $\Im z$ be respectively the real and imaginary parts of $z$.
Since a biholomorphic map $f\colon\Cc^2\to\Cc^2$ induces isomorphisms $\Mm_I\cong\Mm_{f(I)}$, we may assume
without loss of generality that: 
\begin{equation}\label{eq:Iisgeneric}
\text{The assignment $x\mapsto\Re p(x)$ (with $x\in I$) is one to one.}
\end{equation}
In degree $k=2$, our constructions depend on a choice of Jordan open sets:
\[
C,\,D_x,\,U_L,\,C_L\subset\Cc\quad (\text{for each $x\in I$ and each $L\subset I$, $L\neq\emptyset$})
\]
satisfying the following properties, for all $L',L\subset I$ and all $x,y\in I$ with $x\neq y$: 
\begin{equation}\label{eq:conditions1}
\begin{aligned}
&C_L\subset U_L\cap C\,,\quad p(x)\in D_x\subset U_{x}\,,\quad L'\subset L\Rightarrow U_{L'}\subset U_L\,;\\
&D_x\cap C=\emptyset\,,\quad U_{x}\cap U_{y}=\emptyset\,,\quad U_J\cap U_K=\emptyset\,.
\end{aligned}
\end{equation}

\begin{remark}
  Theorem~\ref{theor0}, proved below, shows that 
  the homotopy type of the spaces $|B_L|$ is independent of the choice of open sets.
\end{remark}

\begin{proposition}
	Given disjoint finite sets $J$ and $K$ such that $I=J\cup K$ 
	satisfies condition~\eqref{eq:Iisgeneric},
	there are Jordan open sets satisfying conditions~\eqref{eq:conditions1}.
\end{proposition}
\begin{proof}
	Fix $\delta<\frac12\min\{|\Re p(x)-\Re p(y)|:x,y\in I,\ x\neq y\}$ and fix $r$ so that $|p(w)|<r-3\delta$ for any $w\in I$. Consider the rectangle
	\[
	R=\bigl\{z\in\Cc:-\delta<\Re z<\delta,\ -r+2\delta<\Im z<r+\delta\bigr\}\,.
	\]
	For each $U\subset \Cc$ let $U^*=\{\overline z:z\in U\}$ be reflexion on the real axis. Consider the following open sets,
        where $x\in J$, $y\in K$ and $w\in I$:
	\begin{align*}
		C_J&=\bigl\{z\in\Cc:-r-\delta<\Re z<r+\delta,\ r-\delta<\Im z<r+\delta\bigr\},\quad C_K=C_J^*\\
		C&=\bigl\{z\in\Cc:r-\delta<\Re z<r+\delta,\ -r-\delta<\Im z<r+\delta\bigr\}\cup C_J\cup C_K\\
		U_x&=R+\Re p(x),\quad U_y=R^*+\Re p(y)\\
		D_w&=\bigl\{z\in\Cc:|z-p(w)|<\delta\bigr\}\\
		U_J&=C_J\cup\Bigl({\textstyle\bigcup_{x\in J}}U_x\Bigr),\quad U_K=C_K\cup\Bigl({\textstyle\bigcup_{y\in K}}U_y\Bigr),\quad U_I=C\cup\Bigl({\textstyle\bigcup_{w\in I}}U_w\Bigr)
	\end{align*}
        We define $U_L$ for $\#L>1$ as follows: If $L\subset J$ we take $U_{L}=U_J$; if $L\subset K$ we take $U_{L}=U_K$; otherwise $U_L=U_I$. Finally, we take $C_L=U_L\cap C$.
\end{proof}

\subsection{Structures in degrees up to \boldmath $k$}

Although we are only able to define gluing maps up to degree 2, we can still form a bar construction in degrees up to 2 by formally truncating the moduli spaces as follows:

\begin{definition}\label{def:truncated}
Given a graded space  $A=\coprod_{n\geq0} A_n$ with $A_0=\{*\}$,
let $A_{\leq k}=\{*\}\amalg A_1\amalg \cdots\amalg A_k\amalg \{*\}\amalg\cdots$ be the graded space obtained by collapsing to a point
each summand $A_n$ where $n>k$. 

Give the cartesian product 
$A_{\leq k}\times A_{\leq k}$ the natural grading.
A monoid structure on $A$ in degrees up to $k$ is a 
monoid structure on the truncated graded
space $A_{\leq k}$. 
For $k=2$ this is equivalent to a map 
\[q\colon A_1\times A_1\to A_2\,.\]

A left action of $A$ on a graded space $N=\{*\}\amalg N_1\amalg N_2\amalg\cdots$ in degrees up to $k$ is a left action of $A_{\leq k}$ on $N_{\leq k}$.
For $k=1$ this is just a map $Q_1^{\ell}\colon A_1\to N_1$. For $k=2$ giving such an action is equivalent to giving maps
\[
Q^{\ell}_{11}\colon A_1\times N_1\to N_2\,,\quad Q^{\ell}_1\colon A_1\to N_1\quad\text{and}\quad Q^{\ell}_2\colon A_2\to N_2
\]
satisfying the relation
\begin{equation}\label{eq:Q2circq=Q11circ1xQ1}
Q^{\ell}_2\bigl(q(r_1,r_2)\bigr)=Q^{\ell}_{11}\bigl(r_1,Q^{\ell}_1(r_2)\bigr)\,.
\end{equation}
We similarly define a right action of $A$ on a graded space $M$ in degrees up to $k$ as a right action of $A_{\leq k}$ on $M_{\leq k}$.
For $k=2$ this action is given by maps
\[
Q^r_{11}\colon M_1\times A_1\to M_2\,,\quad Q^r_1\colon A_1\to M_1\quad\text{and}\quad Q^r_2\colon A_2\to M_2
\]
satisfying $Q^r_2\circ q=Q^r_{11}\circ(Q^r_1\times\ident )$.
We represent the bar construction induced by the actions of $A$ on $M$ and $N$ by:
\[
\BAR(M,A,N)_{\leq k}=\BAR(M_{\leq k},A_{\leq k},N_{\leq k})_{\leq k}\,.
\]
\end{definition}

A map from $\BAR(M,A,N)$ to a graded space $Y$ can be defined by
a map $\BAR_0(M,A,N)\to Y$ coequalizing the face maps
$\BAR_1(M,A,N)\rightrightarrows\BAR_0(M,A,N)$.
In degree one such a map is defined by maps
$W_M\colon M_1\to Y_1$ and $W_N\colon N_1\to Y_1$
satisfying the relation
\begin{equation}\label{eq:WcircQ=WcircQdeg1}
  W_N\circ Q_1^{\ell}=W_M\circ Q_1^r\,.
\end{equation}
In degree 2 a map $\BAR(M,A,N)\to Y$ is defined by maps
\[
W_M\colon M_2\to Y_2,\quad W_N\colon N_2\to Y_2,\quad W_{11}\colon M_1\times N_1\to Y_2,
\]
satisfying the relations 
\begin{equation}\label{eq:WcircQ=WcircQ}
\begin{aligned}
W_M\circ Q_{11}^r&=W_{11}\circ(\ident\times Q_1^{\ell})\,,\\
W_N\circ Q_{11}^{\ell}&=W_{11}\circ(Q_1^r\times\ident)\,,\\
W_M\circ Q_2^r&=W_N\circ Q_2^{\ell}\,.
\end{aligned}
\end{equation}

\subsection{Bar constructions in degree 1}\label{sec:barfork=1}

In degree $k=1$, a choice of Jordan open sets is not necessary. 
However, we still need to be careful in how we define the degree 1
bar constructions because they will be used in degree 2 
to define the upper row in
diagram~\eqref{eq:diagramofbarsinfiniterank}.

Let $J,K\subset\Cc^2$ be disjoint finite sets and let $I=J\cup K$.
Given $L\subset I$ and $x\in L$ pick open sets
$D_x$ and $U_L$ such that $p(x)\in D_x$ and $D_x\subset U_L$
whenever $x\in L$
(compare with conditions~\eqref{eq:conditions1}: in degree $k=1$ we don't  
need to require that $D_x\cap D_y=\emptyset$ for $x\neq y$).
Given a finite set $L\subset I$ the maps
\[
Q_1^{\ell}\colon\bigl(\tprod_{x\in L}\!\Mm_\emptyset^{D_x}\bigr)_{\!1}
\xrightarrow{\prod\pi_{\emptyset,x}^*}
\bigl(\tprod_{x\in L}\!\Mm_x^{D_x}\bigr)_{\!1}
\quad\text{and}\quad
Q_1^r\colon\bigl(\tprod_{x\in L}\!\Mm_\emptyset^{D_x}\bigr)_{\!1}
=\tcoprod_{x\in L}\!\Mm_{\emptyset,1}^{D_x}
\xrightarrow{i}
\Mm_{\emptyset,1}^{U_L}
\]
(where $\bigl(\prod\Mm\bigr)_1$ denotes the degree $1$ component)
define, in degree 1, a left and a right action of $\prod\Mm_\emptyset^{D_x}$
on $\prod\Mm_x^{D_x}$ and $\Mm_\emptyset^{U_L}$ respectively, giving rise to
a bar construction in degree 1 which we represent by
\begin{equation}\label{eq:defofBL1}
B_{L,1}(\bullet)=\BAR_\bullet\bigl(\Mm_\emptyset^{U_L},\tprod_{x\in L}\!\Mm_\emptyset^{D_x},\tprod_{x\in L}\!\Mm_x^{D_x}\bigr)_1\,.
\end{equation}

\begin{notation}\label{nota:Ax*vs*xA}
In the proof of Proposition~\ref{prop:barbarbardiagramcommutesinfiniterank}
it will be important to distinguish between the following two 
homeomorphic summands in $B_{L,1}(2)$:
\[
\biggl(\bigl(\tprod_{x\in L}\!\Mm_\emptyset^{D_x}\bigr)_{\!1}\times*\biggr)
\,\amalg\,
\biggl(*\times\bigl(\tprod_{x\in L}\!\Mm_\emptyset^{D_x}\bigr)_{\!1}\biggr)
=\bigl(\tprod_{x\in L}\!\Mm_\emptyset^{D_x}\times\tprod_{x\in L}\!\Mm_\emptyset^{D_x}\bigr)_{\!1}
\subset B_{L,1}(2)\,.
\]
We will write:
\[
\begin{aligned}
\bigl(\tprod_{x\in L}\Mm_\emptyset^{D_x}\bigr)_{\!1}^\ell
&=\tcoprod_{x\in L}\Mm_{\emptyset,1}^{D_x,\ell}
=\bigl(\tprod_{x\in L}\Mm_\emptyset^{D_x}\bigr)_{\!1}\times*\,;\\
\bigl(\tprod_{x\in L}\Mm_\emptyset^{D_x}\bigr)_{\!1}^r
&=\tcoprod_{x\in L}\Mm_{\emptyset,1}^{D_x,r}
=*\times\bigl(\tprod_{x\in L}\Mm_\emptyset^{D_x}\bigr)_{\!1}\,.
\end{aligned}
\]
\end{notation}

The maps
\[
W_M\colon \Mm_{\emptyset,1}^{U_L}\xrightarrow{\pi_{\emptyset, L}^*} \Mm_{L,1}^{U_L}
\quad\text{and}\quad
W_N\colon  \bigl(\tprod_{x\in L}\!\Mm_x^{D_x}\bigr)_{\!1}
=\tcoprod_{x\in L}\!\Mm_{x,1}^{D_x}
\xrightarrow{\pi_{x,L}^*}\Mm_{L,1}^{U_L}
\]
satisfy relation~\eqref{eq:WcircQ=WcircQdeg1} so they define a map
\begin{equation}\label{eq:hindegree1}
h_\boxplus\colon|B_{L,1}|\to\Mm_{L,1}^{U_L}\,.
\end{equation}

\begin{remark}\label{rmk:changeopensetsk=1}
  Let $\hat B_{L,1}$ be the simplicial space obtained
  by taking $D_x=U_L=\Cc$. Then for any other choice of open sets  
  $D_x$ and $U_L$ satisfying $p(x)\in D_x\subset U_L$,
  the maps of monoids and modules determined by the inclusions 
  $D_x,U_L\subset \Cc$ induce a weak equivalence
  $B_{L,1}\to\hat B_{L,1}$ (see Proposition~\ref{prop:MUishomeoandmoretoM}).  
\end{remark}

Now let $L_1,L_2\subset I$ be disjoint finite sets and let $L=L_1\cup L_2$.
Let $U_1,U_2, C_L$ be Jordan open sets such that
\begin{equation}\label{eq:jordanink=1}
L_i\subset U_i\times\Cc\quad \text{(with $i=1,2$),}\quad
C_L\subset U_1\cap U_2\quad\text{and}\quad
U_1\cup U_2\subset U_L\,.
\end{equation}
The maps
\[
Q_1^{\ell}\colon\Mm_{\emptyset,1}^{C_L}\xrightarrow{\pi_{\emptyset,L_2}^*}
\Mm_{L_2,1}^{U_2}
\quad\text{and}\quad
Q_1^r\colon\Mm_{\emptyset,1}^{C_L}\xrightarrow{\pi_{\emptyset,L_1}^*}
\Mm_{L_1,1}^{U_1}
\]
define,  in degree 1,
left and right actions of $\Mm_{\emptyset}^{C_L}$
on $\Mm_{L_2}^{U_2}$ and $\Mm_{L_1}^{U_1}$ respectively, giving rise to a bar construction $\BAR(\Mm_{L_1}^{U_1},\Mm_\emptyset^{C_L},\Mm_{L_2}^{U_2})$.
The maps
\[
W_M\colon \Mm_{L_1,1}^{U_1}\xrightarrow{\pi_{L_1,L}^*}\Mm_{L,1}^{U_L}
\quad\text{and}\quad
W_N\colon \Mm_{L_2,1}^{U_2}\xrightarrow{\pi_{L_2,L}^*}\Mm_{L,1}^{U_L}
\]
define a map:
\begin{equation}\label{eq:mapbarMMM->Mdeg1}
f\colon\BAR(\Mm_{L_1}^{U_1},\Mm_\emptyset^{C_L},\Mm_{L_2}^{U_2})_1
\to\Mm_{L,1}^{U_L}
\end{equation}
Consider now the case when $L_1=L_2=\emptyset$.

\begin{proposition}\label{eq:f:B(M,Mcap,M)->Mcup}
	Let $C_L, U_1,U_2,U_L$
	be Jordan open sets satisfying relations~\eqref{eq:jordanink=1}.
	Then the map 
	\[
	f\colon \BAR\bigl(\Mm_{\emptyset,1}^{U_1},\Mm_{\emptyset,1}^{C_L},\Mm_{\emptyset,1}^{U_2}\bigr)\to\Mm_{\emptyset,1}^{U_L}
	\]
	from equation~\eqref{eq:mapbarMMM->Mdeg1} is a homotopy equivalence.
\end{proposition}
\begin{proof}
    The maps of monoids and modules determined by the inclusions 
	$C_{L},U_1,U_2\subset U_{L}$ induce a weak equivalence
	\[\BAR\bigl(\Mm_{\emptyset,1}^{U_1},\Mm_{\emptyset,1}^{C_L},\Mm_{\emptyset,1}^{U_2}\bigr)
	\simeq
	\BAR(\Mm_\emptyset^{U_L},\Mm_\emptyset^{U_L},\Mm_\emptyset^{U_L})_1\]
	and $f$ is the composition
	\[
	\BAR\bigl(\Mm_{\emptyset,1}^{U_1},\Mm_{\emptyset,1}^{C_L},\Mm_{\emptyset,1}^{U_2}\bigr)
	\xrightarrow{\simeq}
	\BAR(\Mm_\emptyset^{U_L},\Mm_\emptyset^{U_L},\Mm_\emptyset^{U_L})_1
	\to\Mm_{\emptyset,1}^{U_L}
	\]
	which is a homotopy equivalence (Proposition~\ref{prop:B(G,G,Y)=Y}).
\end{proof}

\subsection{The spaces \boldmath $|B_{L,2}|$}\label{subsec:B_L,2}

Fix open sets satisfying conditions~\eqref{eq:conditions1}, and for each open set $U$ in the chosen collection, 
fix a homeomorphism
$\psi_U\colon \Mm_L^{U}\xrightarrow{\cong}\Mm_L$ homotopic to the inclusion $\Mm_L^{U}\to\Mm_L$
(see Proposition~\ref{prop:MUishomeoandmoretoM}).
Choose some order for the elements of $I$.
For each finite set
$L\subset I$ we define a monoid structure on $\smash[b]{\prod\limits_{x\in L}}\Mm_\emptyset$ in degrees up to 2 by: 
\begin{align}\label{eq:monoidprodM}
  q&\colon\bigl(\tprod\limits_{x\in L}\!\Mm_{\emptyset}\bigr)_{\!1}\times\bigl(\tprod\limits_{x\in L}\!\Mm_{\emptyset}\bigr)_{\!1}
 \xrightarrow{\psi_{C_x}^{-1}\times\psi_{D_x}^{-1}}\bigl(\tprod\limits_{x\in L}\!\Mm_{\emptyset}^{C_x}\bigr)_{\!1}\times\bigl(\tprod\limits_{x\in L}\!\Mm_{\emptyset}^{D_x}\bigr)_{\!1}
\notag \\
 &\qquad\qquad=\tcoprod_{x\in L}\!\bigl(\Mm_{\emptyset,1}^{C_x}\times\Mm_{\emptyset,1}^{D_x}\bigr)
 \amalg
 \tcoprod\limits_{\substack{x,y\in L\\ x\neq y}}\!\bigl(\Mm_{\emptyset,1}^{C_y}\times\Mm_{\emptyset,1}^{D_x}\bigr)
 \\ \notag
 &\xrightarrow{\boxplus_{\emptyset,\emptyset}\,\amalg\, i}
 \bigl(\tcoprod_{x\in L}\!\Mm_{\emptyset,2}^{U_x}\bigr)
 \amalg
 \tcoprod\limits_{\substack{x,y\in L\\y<x}}\!\bigl(\Mm_{\emptyset,1}^{U_y}\times\Mm_{\emptyset,1}^{U_x}\bigr) 
 =\bigl(\tprod\limits_{x\in L}\!\Mm_{\emptyset}^{U_x}\bigr)_{\!2}
  \xrightarrow{\psi_{U_x}}\bigl(\tprod\limits_{x\in L}\!\Mm_{\emptyset}\bigr)_{\!2}\,.
\end{align}
We define a left action of  $\prod\Mm_\emptyset$ on $\prod\Mm_x$ in degrees up to 2  by conjugating the following maps with the homeomorphisms $\Mm_{\emptyset,k}^U\cong\Mm_{\emptyset,k}$ and $\Mm_{x,k}^U\cong\Mm_{x,k}$:
\begin{equation}\label{eq:leftactionofprodM}
\begin{aligned}
Q_{11}^{\ell}&\colon \bigl(\tprod\limits_{x\in L}\!\Mm_{\emptyset}^{C_x}\bigr)_{\!1}\times\bigl(\tprod\limits_{x\in L}\!\Mm_{x}^{D_x}\bigr)_{\!1}
\xrightarrow{\boxplus_{\emptyset,x}\,\amalg\,i}
\bigl(\tprod\limits_{x\in L}\!\Mm_{x}^{U_x}\bigr)_{\!2}\,;\\
Q_1^{\ell}&\colon \bigl(\tprod\limits_{x\in L}\!\Mm_{\emptyset}^{D_x}\bigr)_{\!1}
\xrightarrow{\prod\pi_{\emptyset,x}^*}
\bigl(\tprod\limits_{x\in L}\!\Mm_{x}^{D_x}\bigr)_{\!1}\,;\\
Q_2^{\ell}&\colon \bigl(\tprod\limits_{x\in L}\!\Mm_{\emptyset}^{U_x}\bigr)_{\!2}
\xrightarrow{\prod\pi_{\emptyset,x}^*}
\bigl(\tprod\limits_{x\in L}\!\Mm_{x}^{U_x}\bigr)_{\!2}\,,
\end{aligned}
\end{equation}
where $Q_{11}^{\ell}$ is defined in an analogous way to the map $q$ is equation~\eqref{eq:monoidprodM}.
\out{
\[
\begin{aligned}
\boxplus_{\emptyset,x}&\colon
\Mm_{\emptyset,1}^{C_x}\times\Mm_{x,1}^{D_x}\to
\Mm_{x,2}^{U_x}\,,& \quad&(x\in L);\\
\pi_{\emptyset,x}^*\times i&\colon
\Mm_{\emptyset,1}^{C_y}\times\Mm_{x,1}^{D_x}\to
\Mm_{y,1}^{U_y}\times\Mm_{x,1}^{U_x}\,,
& &(x,y\in L, x\neq y).
\end{aligned}
\]}
Relation~\eqref{eq:Q2circq=Q11circ1xQ1} follows from Proposition~\ref{prop:mapsboxplus}(1).

\begin{notation}
	We'll abuse notation and denote by $Q_1^\ell$, $Q_2^\ell$ and $Q_{11}^\ell$ both 
	the maps in equation~\eqref{eq:leftactionofprodM} and the maps obtained
	by conjugating them with the homeomorphisms $\psi_U$.
\end{notation}

\begin{remark}
	The monoid structure on $\prod\Mm_\emptyset$
	and its left action on $\prod\Mm_x$ are obtained by 
	conjugating the maps (in degree 2):
	\begin{align*}
	\tprod_{x\in L}\boxplus_{\emptyset,\emptyset}&\colon\tprod_{x\in L}\bigl(\Mm_\emptyset^{C_x}\times\Mm_\emptyset^{D_x}\bigr)\to\tprod_{x\in L}\Mm_\emptyset^{U_x}\,,\\
	\tprod_{x\in L}\boxplus_{\emptyset,x}&\colon\tprod_{x\in L}\bigl(\Mm_\emptyset^{C_x}\times\Mm_x^{D_x}\bigr)\to\tprod_{x\in L}\Mm_x^{U_x}
	\end{align*}
	with the homeomorphisms $\psi_U$.
\end{remark}


We define a right action of  $\bigl(\prod\Mm_\emptyset\bigr)_{\leq2}$ on $\Mm_{\emptyset,\leq2}$ using the maps
\begin{equation}\label{eq:rightactionofprodM}
\begin{aligned}
Q_{11}^r&\colon \Mm_{\emptyset,1}^C\times\bigl(\tprod\limits_{x\in L}\!\Mm_{\emptyset}^{D_x}\bigr)_{\!1}\xrightarrow{\boxplus_{\emptyset,\dots,\emptyset}}\Mm_{\emptyset,2}\,;\\
Q_1^r&\colon \bigl(\tprod\limits_{x\in L}\!\Mm_{\emptyset}^{C_x}\bigr)_{\!1}\xrightarrow{\boxplus_{\emptyset,\dots,\emptyset}}\Mm_{\emptyset,1}^{C}\,;\\
Q_2^r&\colon \bigl(\tprod\limits_{x\in L}\!\Mm_{\emptyset}^{U_x}\bigr)_{\!2}\xrightarrow{\boxplus_{\emptyset,\dots,\emptyset}}\Mm_{\emptyset,2}\,.
\end{aligned}
\end{equation}
Note that, from Proposition~\ref{prop:propertiesofboxplus}(1), the map $Q_{2}^r$ doesn't depend on the order of 
the factors in the product $\tprod_{x\in L}\Mm_{\emptyset}^{U_x}$.
Relation~\eqref{eq:Q2circq=Q11circ1xQ1} follows from 
this observation and the naturality of
the maps $\boxplus_{\emptyset,\emptyset}$ with respect to inclusions.
From the monoid structure and the actions just defined we get a bar construction
 $\BAR_\bullet(\Mm_\emptyset,\tprod\Mm_\emptyset,\tprod\Mm_x)$
in degree 2.

\begin{remark}
	We are not interested in the degree $k=1$ summand of this bar construction, which is not the same as the one in equation~\eqref{eq:defofBL1}.
\end{remark}

Write $L=\{x_1,\dots,x_j\}$ and let $\boxplus_{\mathbf L}=\boxplus_{x_1,\dots,x_j}$.
The maps
\begin{align}\label{eq:maphboxplusW11}\notag
W_M&\colon\Mm_{\emptyset,2}\xrightarrow{\pi_{\emptyset,L}^*}\Mm_{L,2}
\\ \notag
W_N&\colon\bigl(\tprod\limits_{x\in L}\!\Mm_x\bigr)_2\cong
\bigl(\tprod\limits_{x\in L}\!\Mm_x^{U_x}\bigr)_2
\xrightarrow{\boxplus_{\mathbf L}}\Mm_{L,2}\\
W_{11}&\colon\Mm_{\emptyset,1}\times\bigl(\tprod\limits_{x\in L}\!\Mm_x\bigr)_1\cong
\textstyle\coprod\limits_{x\in L}\!\bigl(\Mm_{\emptyset,1}^C\times\Mm_{x,1}^{D_x}\bigr)
\xrightarrow{\pi_{x,I}^*\,\circ\,\boxplus_{\emptyset,x}}\Mm_{L,2}
\end{align}
satisfy relations~\eqref{eq:WcircQ=WcircQ}, so they define a map 
\begin{equation}\label{eq:maphboxplus}
h_\boxplus\colon
\BAR(\Mm_\emptyset,\tprod\Mm_\emptyset,\tprod\Mm_x)_{2}
\to\Mm_{L,2}\,.
\end{equation}

\out{
\begin{remark}\label{rmk:rightactionofMonprodM}                        
	By taking the opposite monoid structure on $\prod\Mm_\emptyset$       
	defined by the map 
	$\bigl(\tprod\Mm_{\emptyset}^{D_x}\bigr)_{\!1}\times\bigl(\tprod\Mm_{\emptyset}^{C_x}\bigr)_{\!1}\to \bigl(\tprod\Mm_{\emptyset}^{U_x}\bigr)_{\!2}$
	we get in a standard way,
        a right action of
	$\prod\Mm_x$
	on $\prod\Mm_x$ and a left action on $\Mm_\emptyset$ in degrees up to 2.
	The resulting bar construction
	$\BAR(\prod\Mm_x,\prod\Mm_\emptyset,\Mm_\emptyset)_{2}$
	is homeomorphic to
	$\BAR(\Mm_\emptyset,\prod\Mm_\emptyset,\prod\Mm_x)_{2}$.        
\end{remark}                                                     
}

The bar construction $\BAR_\bullet(\Mm_\emptyset,\prod\Mm_\emptyset,\prod\Mm_x)_2$
is independent of the choice of homeomorphisms $\psi_U$. To make this independence explicit
we will introduce a simplicial space $B_{L,2}(\bullet)$ 
homeomorphic to $\BAR_\bullet(\Mm_\emptyset,\prod\Mm_\emptyset,\prod\Mm_x)_2$. For each $n$ let
\begin{alignat}{2}\label{eq:defofB_J,2}
B_{L,2}(n)=\Mm_{\emptyset,2}\notag
\,\amalg\,
\bigl(\tprod_{x\in L}\!\Mm_{x}^{U_x}\bigr)_{\!2}
\,\amalg\,
\Mm_{\emptyset,1}^{C}&\times\bigl(\tprod_{x\in L}\!\Mm_{x}^{D_x}\bigr)_{\!1}\\
\amalg
\tcoprod_{i=1}^n\biggl(
\bigl(\tprod_{x\in L}\!\Mm_\emptyset^{U_x}\bigr)_{\!2}
\,\amalg\,
\Mm_{\emptyset,1}^{C}&\times \bigl(\tprod_{x\in L}\!\Mm_{\emptyset}^{D_x}\bigr)_{\!1}
\,\amalg\,&
\bigl(\tprod_{x\in L}\!\Mm_{\emptyset}^{C_x}\bigr)_{\!1}&\times\bigl(\tprod_{x\in L}\!\Mm_{x}^{D_x}\bigr)_{\!1}\biggr)
\\
&&\llap{$\displaystyle\amalg\tcoprod_{\substack{i,j=1 \\ i<j}}^n$}
\bigl(\tprod_{x\in L}\!\Mm_\emptyset^{C_x}\bigr)_{\!1}
&\times
\bigl(\tprod_{x\in L}\!\Mm_\emptyset^{D_x}\bigr)_{\!1}\notag
\end{alignat}
\out{
\begin{equation}\label{eq:defofB_J,2}          
\begin{aligned}
	B_{L,2}(0)&=\bigl(\tprod_{x\in L}\!\Mm_{x}^{U_x}\bigr)_{\!2}\amalg \Bigl(\Mm_{\emptyset,1}^{C}\times\bigl(\tprod_{x\in L}\!\Mm_{x}^{D_x}\bigr)_{\!1}\Bigr)\amalg \Mm_{\emptyset,2}\\
	B_{L,2}(1)&=\bigl(\tprod_{x\in L}\!\Mm_\emptyset^{U_x}\bigr)_{\!2}\amalg
	\Bigl(\Mm_{\emptyset,1}^{C}\times \bigl(\tprod_{x\in L}\!\Mm_{\emptyset}^{D_x}\bigr)_{\!1}\Bigr)\\	
	&\amalg 
	\Bigl(\bigl(\tprod_{x\in L}\!\Mm_{\emptyset}^{C_x}\bigr)_{\!1}\times\bigl(\tprod_{x\in L}\!\Mm_{x}^{D_x}\bigr)_{\!1}\Bigr)
	\amalg\text{(degenerate summands)}\\
	B_{L,2}(2)&=\Bigl(\bigl(\tprod_{x\in L}\!\Mm_\emptyset^{C_x}\bigr)_{\!1}\times\bigl(\tprod_{x\in L}\!\Mm_\emptyset^{D_x}\bigr)_{\!1}\amalg\text{(degenerate summands)}
\end{aligned}                                                             
\end{equation}}
The maps $\psi_U$  
induce, for each $n$, a homeomorphism
$B_{L,2}(n)\cong\BAR_n(\Mm_\emptyset,\prod\Mm_\emptyset,\prod\Mm_x)_2$. These homeomorphisms
induce a simplicial space structure on $B_{L,2}$.
Note that on morphisms $B_{L,2}$ is defined using 
only the inclusions $i$, pullbacks and the gluing maps $\boxplus_{\emptyset,\emptyset}$ and $\boxplus_{\emptyset,x}$.
%
%

In the proof of Proposition~\ref{prop:barbarbardiagramcommutesinfiniterank}
we will need another simplicial space $B_{L,2}^{\bg}$
homotopically equivalent
to $B_{L,2}$, which we now define.
For each $x\in L$ let 
\[
B_x=C\cup\bigl(\textstyle\bigcup\limits_{\substack{y\in I\\y\neq x}}\!U_y\bigr)
\]
We define, for each $n$, a space $B_{L,2}^\bg(n)$ 
homeomorphic to $B_{L,2}(n)$ as follows:
\begin{alignat}{2}\label{eq:BL2big}
B_{L,2}^\bg(n)= \\ \notag 
\Mm_{\emptyset,2}
 \amalg
 \bigl(\tprod_{\llap{$\scriptstyle x\in L$}}\Mm_x^{U_x}\bigr)_{\!2}
 &\amalg
 \bigl(\tcoprod_{\llap{$\scriptstyle x\in L$}}\Mm_{\emptyset,1}^{B_x}\times\Mm_{x,1}^{D_x}\bigr)
 \\ \notag
\amalg\tcoprod_{i=1}^n\biggl(
       \bigl(\tprod_{\llap{$\scriptstyle x\in L$}}\Mm_\emptyset^{U_x}\bigr)_{\!2}
 &\amalg\bigl(\tcoprod_{\llap{$\scriptstyle x\in L$}}\Mm_{\emptyset,1}^{B_x}\times\Mm_{\emptyset,1}^{D_x}\bigr)&
 \amalg\bigl(\tcoprod_{\llap{$\scriptstyle x\in L$}}\Mm_{\emptyset,1}^{C_x}&\times\Mm_{x,1}^{D_x}\bigr)
 \amalg\bigl(\!\!\tcoprod_{\substack{x,y\in L\\y\neq x}}\!\!\Mm_{\emptyset,1}^{U_y}\times\Mm_{x,1}^{D_x}\bigr)
 \biggr) 
 \\ \notag && 
 \llap{$\displaystyle\amalg\tcoprod_{\substack{i,j=1\\i\neq j}}^n$}\biggl(
 \bigl(\tcoprod_{\llap{$\scriptstyle x\in L$}}\Mm_{\emptyset,1}^{C_x}&\times\Mm_{\emptyset,1}^{D_x}\bigr)
 \amalg
 \bigl(\tcoprod_{y\neq x}\Mm_{\emptyset,1}^{U_y}\times\Mm_{\emptyset,1}^{D_x}\bigr)
 \biggr)
\end{alignat}
For each $n$ we have a homotopy equivalence
$B_{L,2}(n)\to B_{L,2}^\bg(n)$,
defined as the identity on the
$\Mm_{\emptyset,2}$, $(\prod\Mm_x^{U_x})_2$ and
$(\prod\Mm_\emptyset^{U_x})_2$ summands (see equation~\eqref{eq:defofB_J,2}) and defined on the remaining summands 
by the inclusion maps
\begin{align}\label{eq:BLsubsetBLbig}
\Mm_{\emptyset,1}^C\times 
\bigl(\tprod_{x\in L}\!\Mm_x^{D_x}\bigr)_{\!1}=
\tcoprod_{x\in L}\!\bigl(\Mm_{\emptyset,1}^C\times \Mm_{x,1}^{D_x}\bigr)
\to
\tcoprod_{x\in L}\!\bigl(\Mm_{\emptyset,1}^{B_x}&\times \Mm_{x,1}^{D_x}\bigr)
\\ \notag
\Mm_{\emptyset,1}^C\times 
\bigl(\tprod_{x\in L}\!\Mm_{\emptyset}^{D_x}\bigr)_{\!1}=
\tcoprod_{x\in L}\!\bigl(\Mm_{\emptyset,1}^C\times \Mm_{\emptyset,1}^{D_x}\bigr)
\to
\tcoprod_{x\in L}\!\bigl(\Mm_{\emptyset,1}^{B_x}&\times \Mm_{\emptyset,1}^{D_x}\bigr)
\\ \notag
\bigl(\tprod_{x\in L}\!\Mm_\emptyset^{C_x}\bigr)_{\!1}\times\bigl(\tprod_{x\in L}\!\Mm_x^{D_x}\bigr)_{\!1}=
\bigl(\tcoprod_{x\in L}\!\Mm_{\emptyset,1}^{C_x}\times\Mm_{x,1}^{D_x}\bigr)\amalg
\bigl(\tcoprod_{x\neq y}\Mm_{\emptyset,1}^{C_y}&\times\Mm_{x,1}^{D_x}\bigr)
\\ \notag
\to
\bigl(\tcoprod_{x\in L}\!\Mm_{\emptyset,1}^{C_x}\times\Mm_{x,1}^{D_x}\bigr)\amalg
\bigl(\tcoprod_{x\neq y}\Mm_{\emptyset,1}^{U_y}&\times\Mm_{x,1}^{D_x}\bigr)
\\ \notag
\bigl(\tprod_{x\in L}\!\Mm_\emptyset^{C_x}\bigr)_{\!1}\times\bigl(\tprod_{x\in L}\!\Mm_\emptyset^{D_x}\bigr)_{\!1}=
\bigl(\tcoprod_{x\in L}\!\Mm_{\emptyset,1}^{C_x}\times\Mm_{\emptyset,1}^{D_x}\bigr)\amalg
\bigl(\tcoprod_{x\neq y}\Mm_{\emptyset,1}^{C_y}&\times\Mm_{\emptyset,1}^{D_x}\bigr)
\\ \notag
\to
\bigl(\tcoprod_{x\in L}\!\Mm_{\emptyset,1}^{C_x}\times\Mm_{\emptyset,1}^{D_x}\bigr)\amalg
\bigl(\tcoprod_{x\neq y}\Mm_{\emptyset,1}^{U_y}&\times\Mm_{\emptyset,1}^{D_x}\bigr)
\end{align}
The face and degeneracy
maps on $B_{L,2}$ are defined using inclusions, pullbacks and
the gluing maps. All these maps 
can be naturally extended to define a simplicial structure on $B_{L,2}^\bg$:
if we write
\begin{align*}
\bigl(\tprod_{x\in L}\!\Mm_x^{U_x}\bigr)_{\!2}&=
  \bigl(\tcoprod_{x\in L}\!\Mm_{x,2}^{U_x}\bigr)\amalg
  \bigl(\tcoprod_{\substack{x,y\in L\\y<x}}\!\Mm_{y,1}^{U_y}\times\Mm_{x,1}^{U_x}\bigr)\\
\bigl(\tprod_{x\in L}\!\Mm_\emptyset^{U_x}\bigr)_{\!2}&=
\bigl(\tcoprod_{x\in L}\!\Mm_{\emptyset,2}^{U_x}\bigr)\amalg
\bigl(\tcoprod_{\substack{x,y\in L\\y<x}}\!\Mm_{\emptyset,1}^{U_y}\times\Mm_{\emptyset,1}^{U_x}
\bigr)
\end{align*}
all the face maps in $B_{L,2}^\bg$
between non-degenerate summands are represented
in the diagram below:
\out{
\[
\xymatrix@C=-1em{
\Mm_{\emptyset,2} & & 
\smash[b]{\coprod\limits_x}\,\Mm_{\emptyset,2}
  \ar[ll]
  \ar[rr]^-{\pi^*_{\emptyset,x}} & &
\smash[b]{\coprod\limits_x}\,\Mm_{x,2} 
\\ & & &
\smash[b]{\coprod\limits_x}\,\Mm_{\emptyset,1}^{C_x}\times\Mm_{\emptyset,1}^{D_x}
  \ar[lu]_{\boxplus_{\emptyset,\emptyset}}
  \ar[ld]^{i\times\ident}
  \ar[rd]^{\ident\times\pi^*_{\emptyset,x}} 
\\
\smash[b]{\coprod\limits_{y<x}}\Mm_{\emptyset,1}^{U_y}\times\Mm_{\emptyset,1}^{U_x}
  \ar[uu]^{\boxplus_{\emptyset,\emptyset}}
  \ar[dd]_{\pi^*_{\emptyset,y}\times\pi^*_{\emptyset,x}} & &
\smash[b]{\coprod\limits_x}\,\Mm_{\emptyset,1}^{B_x}\times\Mm_{\emptyset,1}^{D_x}
  \ar[lluu]_-{\boxplus_{\emptyset,\emptyset}}
  \ar[rrdd]^-{\pi^*_{\emptyset,x}\times\pi^*_{\emptyset,x}} & &
\smash[b]{\coprod\limits_x}\,\Mm_{\emptyset,1}^{C_x}\times\Mm_{x,1}^{D_x}
  \ar[uu]_{\boxplus_{\emptyset,x}}
  \ar[dd]^{i\times\ident} 
\\ &
\smash[b]{\coprod\limits_{y\neq x}}\Mm_{\emptyset,1}^{U_y}\times\Mm_{\emptyset,1}^{D_x}
  \ar[lu]_{\ident\times i}
  \ar[ru]_{i\times\ident}
  \ar[rd]^{\ident\times\pi_{\emptyset,x}^*} 
\\
\smash[b]{\coprod\limits_{y<x}}\Mm_{y,1}^{U_y}\times\Mm_{x,1}^{U_x} & &
\smash[b]{\coprod\limits_{y\neq x}}\Mm_{\emptyset,1}^{U_y}\times\Mm_{x,1}^{D_x}
  \ar[ll]_-{\pi_{\emptyset,y}^*\times i}
  \ar[rr]^-{i\times\ident} & &
\smash[b]{\coprod\limits_x}\,\Mm_{\emptyset,1}^{B_x}\times\Mm_{x,1}^{D_x}
}\]
}
\begin{equation}\label{diagram:BL2big}
\xymatrix@C=-1em{
	\smash[b]{\coprod\limits_{x\in L}}\Mm_{\emptyset,1}^{B_x}\times\Mm_{x,1}^{D_x} & & 
	\smash[b]{\coprod\limits_{x\in L}}\Mm_{\emptyset,1}^{C_x}\times\Mm_{x,1}^{D_x}
	\ar[rr]^{\boxplus_{\emptyset,x}}
	\ar[ll]_{i\times\ident}  & &
	\smash[b]{\coprod\limits_{x\in L}}\Mm_{x,2}\phantom{M}
\\ & & &
	\smash[b]{\coprod\limits_{x\in L}}\Mm_{\emptyset,1}^{C_x}\times\Mm_{\emptyset,1}^{D_x}\phantom{M}
	\ar[rd]_{\boxplus_{\emptyset,\emptyset}}
	\ar[ld]^{i\times\ident}
	\ar[lu]_{\ident\times\pi^*_{\emptyset,x}} 
\\
	\smash[b]{\coprod\limits_{\substack{x,y\in L\\y\neq x}}}\!\Mm_{\emptyset,1}^{U_y}\times\Mm_{x,1}^{D_x}
	\ar[dd]_-{\pi_{\emptyset,y}^*\times i}
	\ar[uu]^-{i\times\ident} & &
	\smash[b]{\coprod\limits_{x\in L}}\Mm_{\emptyset,1}^{B_x}\times\Mm_{\emptyset,1}^{D_x}
	\ar[rrdd]^-{\boxplus_{\emptyset,\emptyset}}
	\ar[lluu]_-{\ident\times\pi^*_{\emptyset,x}}\phantom{M} & &
	\smash[b]{\coprod\limits_{x\in L}}\Mm_{\emptyset,2}\phantom{M}
	\ar[dd]
	\ar[uu]_-{\pi^*_{\emptyset,x}}
\\ &
	\smash[b]{\coprod\limits_{\substack{x,y\in L\\y\neq x}}}\!\Mm_{\emptyset,1}^{U_y}\times\Mm_{\emptyset,1}^{D_x}
	\ar[rd]^{\ident\times i}
	\ar[ru]_{i\times\ident}
	\ar[lu]_{\ident\times\pi_{\emptyset,x}^*} 
\\
	\coprod\limits_{\substack{x,y\in L\\y<x}}\!\Mm_{y,1}^{U_y}\times\Mm_{x,1}^{U_x} & &
	\coprod\limits_{\substack{x,y\in L\\y<x}}\!\Mm_{\emptyset,1}^{U_y}\times\Mm_{\emptyset,1}^{U_x}
	\ar[rr]^{\boxplus_{\emptyset,\emptyset}}
	\ar[ll]_{\pi^*_{\emptyset,y}\times\pi^*_{\emptyset,x}} & &
	\Mm_{\emptyset,2}
}\end{equation}
The inclusion maps in equation~\eqref{eq:BLsubsetBLbig} 
induce a weak equivalence  $B_{L,2}\to B_{L,2}^\bg$.

\out{%

We can visualize all the maps together in the cube diagram:
\[
\xymatrix@C=0pt{
	& \Mm_{\emptyset,2}\ar[rr] &&
	\Mm_{J,2} \\
	\Mm_{\emptyset,1}^{C}\times \Bigl(\prod \Mm_{\emptyset}^{D_x}\Bigr)_{\!1}\ar[ru]^-{\boxplus}\ar[rr]^(0.65){\Id\times\prod\pi_i^*} &&
	\Mm_{\emptyset,1}^{C}\times \Bigl(\prod \Mm_{x}^{D_x}\Bigr)_{\!1}\ar[ru]\\
	& \Bigl(\prod \Mm_\emptyset^{U_x}\Bigr)_{\!2}\ar'[r]_-{\prod\pi_i^*}[rr]\ar'[u][uu]_-{\boxplus} &&
	\Bigl(\prod \Mm_{x}^{U_x}\Bigr)_{\!2} \ar[uu]\\
	\Bigl(\prod \Mm_\emptyset^{C_x}\Bigr)_{\!1}\times\Bigl(\prod \Mm_\emptyset^{D_x}\Bigr)_{\!1}\ar[ru]^(0.6){\prod\boxplus}\ar[rr]^{\Id\times\prod\pi_i^*}\ar[uu]_(0.6){\boxplus\times\Id} &&
	\Bigl(\prod \Mm_{\emptyset}^{C_x}\Bigr)_{\!1}\times\Bigl(\prod \Mm_{x}^{D_x}\Bigr)_{\!1}\ar[ru]^-{\prod\tboxplus_i}\ar[uu]_(0.65){\boxplus\times\Id}
}
\]
}%

The map $W_{11}$ in equation~\eqref{eq:maphboxplusW11} extends to a map
$\tcoprod\Mm_{\emptyset,1}^{B_x}\times\Mm_{x,1}^{D_x}\to\Mm_{I,2}$
so the map $h_\boxplus$ in equation~\eqref{eq:maphboxplus} 
extends to a map 
\begin{equation}\label{eq:hboxplusbig}
h_\boxplus\colon|B_{L,2}^\bg|\to\Mm_{L,2}\,.
\end{equation}

\subsection{The space \boldmath $\BAR(\Mm_{J},\Mm_\emptyset,\Mm_K)$ in degree 2}

In this section and the next we fix finite sets $L_1,L_2\subset I$ such that $U_{L_1}\cap U_{L_2}=\emptyset$ (we are interested in 2 cases: the case where $\#L_1=\#L_2=1$
and the case where $L_1=J$ and $L_2=K$).
We give $\Mm_{\emptyset}$ the structure of a monoid in degrees up to 2 by the following map:
\begin{equation}\label{eq:monoidM}
q\colon\Mm_{\emptyset,1}\times\Mm_{\emptyset,1}\cong
\Mm_{\emptyset,1}^{C_{L_1}}\times\Mm_{\emptyset,1}^{C_{L_2}}\xrightarrow{\,\boxplus_{\emptyset,\emptyset}\,}\Mm_{\emptyset,2}\,.
\end{equation}
We define a left action of $\Mm_\emptyset$ on $\Mm_{L_2}$ 
using the maps:
\begin{equation}\label{eq:leftactionofMonM}
  \begin{aligned}
    Q_{11}^{\ell}&\colon\Mm_{\emptyset,1}\times\Mm_{L_2,1}\cong
    \Mm_{\emptyset,1}^{C_{L_1}}\times\Mm_{L_2,1}^{U_{L_2}}\xrightarrow{\boxplus_{\emptyset,L_2}}\Mm_{L_2,2}\,,\\
    Q_1^{\ell}&\colon\Mm_{\emptyset,1}\cong
    \Mm_{\emptyset,1}^{C_{L_2}}\xrightarrow{\pi_{\emptyset,L_2}^*}\Mm_{L_2,1}^{U_{L_2}}\cong\Mm_{L_2,1}\,,\\
    Q_2^{\ell}&\colon
    \Mm_{\emptyset,2}\xrightarrow{\pi_{\emptyset,L_2}^*}\Mm_{L_2,2}\,,
\end{aligned}
\end{equation}
and we define in an analogous way a right action of $\Mm_\emptyset$ on $\Mm_{L_1}$ by
\begin{equation}\label{eq:rightactionofMonM}
\Mm_{L_1,1}^{U_{L_1}}\times\Mm_{\emptyset,1}^{C_{L_2}}\to\Mm_{L_1,2}\,,\quad
\Mm_{\emptyset,1}^{C_{L_1}}\to\Mm_{L_1,1}^{U_{L_1}}\,,\quad
\Mm_{\emptyset,2}\to\Mm_{L_1,2}\,.
\end{equation}
These actions give rise to a bar construction $\BAR(\Mm_{L_1},\Mm_\emptyset,\Mm_{L_2})$. Let $L=L_1\cup L_2$.
The maps
\[
	\Mm_{L_1,2}\xrightarrow{\pi_{L_1,L}^*}\Mm_{L,2}\,,\quad
	\Mm_{L_1,1}^{U_{L_1}}\times\Mm_{L_2,1}^{U_{L_2}}\xrightarrow{\boxplus_{L_1,L_2}}\Mm_{L,2}\,,\quad
	\Mm_{L_2,2}\xrightarrow{\pi_{L_2,L}^*}\Mm_{L,2}
\]
induce a map in degree 2:
\begin{equation}\label{eq:mapbarMMM->M}
  f\colon\BAR(\Mm_{L_1},\Mm_\emptyset,\Mm_{L_2})\to\Mm_{L}\,.
\end{equation}

\begin{remark}\label{remark:C_L=U_L}
  The monoid structure and the actions defined in this section 
are well defined if we take $C_{L_1}=U_{L_1}$ and $C_{L_2}=U_{L_2}$. We represent the bar construction thus obtained by
  $\hatBAR(\Mm_{L_1},\Mm_\emptyset,\Mm_{L_2})$. By Proposition~\ref{prop:MUishomeoandmoretoM},
  the maps of monoids and modules determined by the inclusions 
  $C_{L_1}\subset U_{L_1}$ and $C_{L_2}\subset U_{L_2}$ induce a homotopy equivalence:
  \[
  \BAR(\Mm_{L_1},\Mm_\emptyset,\Mm_{L_2})\xrightarrow{\simeq}\hatBAR(\Mm_{L_1},\Mm_\emptyset,\Mm_{L_2})\,.
  \]
\end{remark}

In the proof of Proposition~\ref{prop:barbarbardiagramcommutesinfiniterank}
we will also need to consider the bar construction induced by the 
left and right actions of $\Mm_\emptyset$ on itself defined using the maps
$Q_2^{\ell}=Q_2^r=\ident\colon\Mm_{\emptyset,2}\to\Mm_{\emptyset,2}$ and
\begin{align*}
Q_{11}^{\ell}&\colon\Mm_{\emptyset,1}^{C_{L_1}}\times\Mm_{\emptyset,1}^{U_{L_2}}
\xrightarrow{\boxplus_{\emptyset,\emptyset}}\Mm_{\emptyset,2} &
Q_1^{\ell}&\colon\Mm_{\emptyset,1}^{C_{L_2}}\xrightarrow{i}\Mm_{\emptyset,1}^{U_{L_2}}
\\
Q_{11}^r&\colon\Mm_{\emptyset,1}^{U_{L_1}}\times\Mm_{\emptyset,1}^{C_{L_2}}
\xrightarrow{\boxplus_{\emptyset,\emptyset}}\Mm_{\emptyset,2} &
Q_1^r&\colon\Mm_{\emptyset,1}^{C_{L_1}}\xrightarrow{i}\Mm_{\emptyset,1}^{U_{L_1}}
\end{align*}
(compare with equations~\eqref{eq:leftactionofMonM} 
and~\eqref{eq:rightactionofMonM}).
We represent the bar construction thus obtained by 
\begin{equation}\label{eq:Bemptyset2}
B_{\emptyset,2}=\BAR(\Mm_\emptyset,\Mm_\emptyset,\Mm_\emptyset)_2\,.
\end{equation}
The maps $W_M=W_N=\ident\colon\Mm_{\emptyset,2}\to\Mm_{\emptyset,2}$ and 
\[
W_{11}\colon \Mm_{\emptyset,1}^{U_{L_1}}\times\Mm_{\emptyset,1}^{U_{L_2}}
\xrightarrow{\boxplus_{\emptyset,\emptyset}}\Mm_{\emptyset,2}
\]
induce a map $f\colon B_{\emptyset,2}\to\Mm_{\emptyset,2}$.

\begin{proposition}\label{prop:B(M,M,M)->Mk=2}
  The map $f\colon B_{\emptyset,2}\to\Mm_{\emptyset,2}$
  is a homotopy equivalence.
\end{proposition}
\begin{proof}
If we let $\hat B_{\emptyset,2}$ denote the bar construction obtained 
by taking $C_{L_1}=U_{L_1}$ and $C_{L_2}=U_{L_2}$ 
(see remark~\ref{remark:C_L=U_L}), the inclusion map
$B_{\emptyset,2}\to\hat B_{\emptyset,2}$ is  a homotopy equivalence
and the map $f$ is the composition 
$B_{\emptyset,2}\xrightarrow{\simeq}\hat B_{\emptyset,2}\to\Mm_{\emptyset,2}$,
which is a homotopy equivalence (Proposition~\ref{prop:B(G,G,Y)=Y}).
\end{proof}

\subsection{The space \boldmath $\BAR\bigl(|B_J|,\Mm_\emptyset,|B_K|\bigr)$ in degree 2}

Consider the monoid structure on $\Mm_{\emptyset,\leq2}$
defined in equation~\eqref{eq:monoidM}. We define a
left action in degrees up to 2
of the monoid $\Mm_{\emptyset,\leq2}$, seen as a constant simplicial space,
on the graded simplicial space
$B_{L_2}=\star\amalg B_{L_2,1}\amalg B_{L_2,2}\amalg\star\amalg\cdots$
(see equations~\eqref{eq:defofB_J,2} and~\eqref{eq:defofBL1})
by the natural transformations:
\begin{equation}\label{eq:simplicialaction1}
Q_1^{\ell}\colon\Mm_{\emptyset,1}^{C_{L_2}}\xrightarrow{i}
\Mm_{\emptyset,1}^{U_{L_2}}\subset B_{L_2,1}(n)
\quad\text{and}\quad
Q_2^{\ell}\colon\Mm_{\emptyset,2}\xrightarrow{\subset} B_{L_2,2}(n)\,,
\end{equation}
and by the natural transformation
$Q^{\ell}_{11}\colon\Mm_{\emptyset,1}^{C_{L_1}}\times B_{L_2,1}\to B_{L_2,2}$ 
defined by
\begin{align}\label{eq:simplicialaction2}
\Mm_{\emptyset,1}^{C_{L_1}}&\times B_{L_2,1}(n)\\ \notag
=\Mm_{\emptyset,1}^{C_{L_1}}\times\Mm_{\emptyset,1}^{U_{L_2}}
&\amalg
\tcoprod_{i=1}^n\biggl(\Mm_{\emptyset,1}^{C_{L_1}}\times\bigl(\tprod_{x\in L_2}\!\!\Mm_\emptyset^{D_x}\bigr)_{\!1}\biggr)
\amalg\,
\Mm_{\emptyset,1}^{C_{L_1}}\times\bigl(\tprod_{x\in L_2}\!\!\Mm_x^{D_x}\bigr)_{\!1}
\\ \notag
\xrightarrow{\boxplus_{\emptyset,\emptyset}\amalg\,i\,\amalg\,i}
\Mm_{\emptyset,2}
&\amalg\,
\tcoprod_{i=1}^n\biggl(\Mm_{\emptyset,1}^{C}\times\bigl(\tprod_{x\in L_2}\!\!\Mm_\emptyset^{D_x}\bigr)_{\!1}\biggr)
\amalg\,
\Mm_{\emptyset,1}^{C}\times\bigl(\tprod_{x\in L_2}\!\!\Mm_x^{D_x}\bigr)_{\!1}\subset B_{L_2,2}(n)
\end{align}
To check that these maps define a natural transformation it is enough
to observe that the following diagram is commutative:
\begin{equation}\label{eq:MemptyactsonB_J}
\xymatrix@C=4em{
	\Mm_{\emptyset,1}^{C_{L_1}}\times \Mm_{\emptyset,1}^{U_{L_2}}\ar[d]^{\boxplus}&
	\Mm_{\emptyset,1}^{C_{L_1}}\times\bigl(\prod\limits_{x\in L_2}\!\! \Mm_\emptyset^{D_x}\bigr)_{\!1}\ar[r]^-{\ident\times\prod\pi_i^*}\ar[l]_-{\ident\times\boxplus}\ar[d]^{i\times\ident}&
	\Mm_{\emptyset,1}^{C_{L_1}}\times\bigl(\prod\limits_{x\in L_2}\!\! \Mm_{x}^{D_x}\bigr)_{\!1}\ar[d]^{i\times\ident}\\
	\Mm_{\emptyset,2}&
	\Mm_{\emptyset,1}^{C}\times \bigl(\prod\limits_{x\in L_2}\!\! \Mm_{\emptyset}^{D_x}\bigr)_{\!1}\ar[l]_-{\boxplus}\ar[r]^-{\ident\times\prod\pi_i^*}&
	\Mm_{\emptyset,1}^{C}\times \bigl(\prod\limits_{x\in L_2}\!\! \Mm_{x}^{D_x}\bigr)_{\!1}
}
\end{equation}
Relation~\eqref{eq:Q2circq=Q11circ1xQ1} follows from the commutativity
of the diagram:
\[
\xymatrix{
	\Mm_{\emptyset,1}^{C_{L_1}}\times\Mm_{\emptyset,1}^{C_{L_2}}
	\ar[r]^-{\ident\times i}\ar[d]^{q} &
	\Mm_{\emptyset,1}^{C_{L_1}}\times\Mm_{\emptyset,1}^{U_{L_2}}
	\ar[r]^-{\subset}\ar[d]^{\boxplus_{\emptyset,\emptyset}} &
	\Mm_{\emptyset,1}^{C_{L_1}}\times B_{L_2,1}(n)
	\ar[d]^{Q_{11}^{\ell}}\\
	\Mm_{\emptyset,2}\ar[r]^-{=} & \Mm_{\emptyset,2}\ar[r]^-{Q_2^{\ell}} & B_{L_2,2}(n) }
\]
where the map in the top row equals $\ident\times Q_1^{\ell}$.

To define a right action of $\Mm_\emptyset$ on $|B_{L_1}|$ 
in degrees up to 2 we proceed as follows: 
By taking the opposite monoid structure on $\Mm_\emptyset$,
defined using the map 
$\boxplus_{\emptyset,\emptyset}\colon\Mm_{\emptyset,1}^{C_{L_2}}\times\Mm_{\emptyset,1}^{C_{L_1}}\to\Mm_{\emptyset,2}$,
we get, in a standard way, a right action of
$\Mm_\emptyset$ on $B_{L_2}$.
If we then switch every occurrence of the finite sets 
$L_1$ and $L_2$ in the definitions above, we recover the original monoid structure on $\Mm_\emptyset$ (equation~\eqref{eq:monoidM}),
together with a right action on $B_{L_1}$. Explicitly, we have $Q_2^r=Q_2^{\ell}$,
the map $Q_1^r$ is induced by the inclusion 
$\Mm_{\emptyset,1}^{C_{L_1}}\subset\Mm_{\emptyset,1}^{U_{L_1}}$
and the map $Q_{11}^r$ is induced by the inclusion
$\Mm_{\emptyset,1}^{C_{L_2}}\to\Mm_{\emptyset,1}^C$
and by the map
\[
\boxplus_{\emptyset,\emptyset}\colon
\Mm_{\emptyset,1}^{U_{L_1}}\times\Mm_{\emptyset,1}^{C_{L_2}}
\to\Mm_{\emptyset,2}
\]
(compare with equations~\eqref{eq:simplicialaction1} and~\eqref{eq:simplicialaction2}). These actions give rise to a bar construction
$\BAR\bigl(|B_{L_1}|,\Mm_\emptyset,|B_{L_2}|\bigr)$.

Consider the map $h_\boxplus\colon |B_{L_j}|\to\Mm_{L_j}$
(with $j=1,2$) 
defined in degree 1 by composing the map in equation~\eqref{eq:hindegree1}
with the homeomorphism $\Mm_{L_j}^{U_{L_j}}\cong\Mm_{L_j}$
and defined in degree 2 in equation~\eqref{eq:maphboxplus}.

\begin{proposition}\label{prop:hboxplusequivariant}
The map $h_\boxplus$ is equivariant with respect
to both the left actions of the monoid $\Mm_\emptyset$ on $|B_{L_2}|$ and
$\Mm_{L_2}^{U_{L_2}}$ (equation~\eqref{eq:leftactionofMonM}),
and the right actions of $\Mm_\emptyset$ on $|B_{L_1}|$ and $\Mm_{L_1}^{U_{L_1}}$
(equation~\eqref{eq:rightactionofMonM}).
\end{proposition}
\begin{proof}
Consider first the left actions.
We need to show that the following diagrams are commutative:
\[
\xymatrix{
\Mm_{\emptyset,1}^{C_{L_2}}\ar[r]^-{Q_1^\ell}\ar[rd]_{\pi_{\emptyset,L_2}^*} &
|B_{L_2,1}|\ar[d]^{h_\boxplus} \\ &
\Mm_{L_2,1}^{U_{L_2}} }	\qquad
\xymatrix{
	\Mm_{\emptyset,2}\ar[r]^-{Q_2^\ell}\ar[rd]_{\pi_{\emptyset,L_2}^*} &
	|B_{L_2,2}|\ar[d]^{h_\boxplus} \\ &
	\Mm_{L_2,2} }  \qquad
\xymatrix{
	\Mm_{\emptyset,1}^{C_{L_1}}\times|B_{L_2,1}|
	\ar[r]^-{Q_{11}^\ell}\ar[d]_{\ident\times h_\boxplus} &
	|B_{L_2,2}|\ar[d]^{h_\boxplus} \\
	\Mm_{\emptyset,1}^{C_{L_1}}\times\Mm_{L_2,1}^{U_{L_2}}
	\ar[r]^-{\boxplus_{\emptyset,L_2}} &
	\Mm_{L_2,2} }
\]
That is clearly the case for the first 2 diagrams.
The commutativity of the third diagram
amounts to the commutativity of the following diagrams:
\[
\xymatrix{
\Mm_{\emptyset,1}^{C_{L_1}}\times\Mm_{\emptyset,1}^{U_{L_2}}
  \ar[r]^-{\boxplus_{\emptyset,\emptyset}}\ar[d]_{\ident} &
\Mm_{\emptyset,2}\ar[d]_{\pi_{\emptyset,L_2}^*} \\
	\Mm_{\emptyset,1}^{C_{L_1}}\times\Mm_{L_2,1}^{U_{L_2}}
	\ar[r]^-{\boxplus_{\emptyset,L_2}} &
	\Mm_{L_2,2} }\qquad
\xymatrix{
\Mm_{\emptyset,1}^{C_{L_1}}\times\bigl(\prod\limits_{x\in L_2}\!\!\Mm_x^{D_x}\bigr)_{\!1}
  \ar[r]^-{i\times\ident}\ar[d]_{\ident\times\boxplus_{\mathbf {L_2}}} &
\Mm_{\emptyset,1}^{C}\times\bigl(\prod\limits_{x\in L_2}\!\!\Mm_x^{D_x}\bigr)_{\!1}
  \ar[d]_{\boxplus_{\emptyset,\mathbf{L_2}}} \\
\Mm_{\emptyset,1}^{C_{L_1}}\times\Mm_{L_2,1}^{U_{L_2}}
\ar[r]^-{\boxplus_{\emptyset,L_2}} &
\Mm_{L_2,2} }
\]
The proof for the right actions is completely analogous.
\end{proof}

By Proposition~\ref{prop:hboxplusequivariant}, the map $h_\boxplus$ induces
a map in degree 2:
\begin{equation}\label{eq:mapbarbarbartoBI}
\BAR\bigl(|B_{L_1}|,\Mm_\emptyset,|B_{L_2}|\bigr)\to
\BAR(\Mm_{L_1},\Mm_\emptyset,\Mm_{L_2})\,.
\end{equation}

\subsection{The Diagram~(\textbf{\ref{eq:diagramofbarsinfiniterank}})}

In the previous sections we defined, for $k=2$, the bar constructions and 
all maps except the top one appearing in diagram~\eqref{eq:diagramofbarsinfiniterank}
(see equations~\eqref{eq:maphboxplus}, \eqref{eq:mapbarMMM->M} and~\eqref{eq:mapbarbarbartoBI}). 

We now consider the case $k=1$.
We take all Jordan open sets to be $\Cc$ 
(see Remark~\ref{rmk:changeopensetsk=1}). The maps
$h_\boxplus$ and $f$ are defined in equations~\eqref{eq:hindegree1}
and~\eqref{eq:mapbarMMM->Mdeg1} respectively. 
A left action
of $\Mm_\emptyset$ on the simplicial space
$B_{L_2}$ in degree 1 
is defined by the natural transformation
$Q_1^\ell\colon\Mm_{\emptyset,1}\xrightarrow{\subset}B_{L_2,1}(n)$
and a left action of $\Mm_\emptyset$ on the space $\Mm_{L_2}$ in degree 1
is defined by the map 
$Q_1^\ell=\pi_{\emptyset,L_2}^*\colon\Mm_{\emptyset,1}\to\Mm_{L_2,1}$.
The corresponding right actions are defined similarly.
The map $h_\boxplus$ is easily seen to be equivariant
with respect to these actions so
it induces a map in degree 1:
\[
\BAR\bigl(|B_{L_1}|,\Mm_\emptyset,|B_{L_2}|\bigr)\to
\BAR(\Mm_{L_1},\Mm_\emptyset,\Mm_{L_2})\,.
\]

\begin{proposition}\label{prop:barbarbardiagramcommutesinfiniterank}
  Let $L_1,L_2\subset I$ be disjoint finite sets such that $U_{L_1}\cap U_{L_2}=\emptyset$ and let $L=L_1\cup L_2$. Then 
  there is a map $\BAR\bigl(|B_{L_1}|,\Mm_\emptyset,|B_{L_2}|\bigr)\to|B_L|$ in degrees 1 and 2 which is a homotopy equivalence and such that the following diagram commutes up to homotopy:
  \begin{equation}\label{eq:diagram-barbarbar}
  \xymatrix{
    \BAR\bigl(|B_{L_1}|,\Mm_\emptyset,|B_{L_2}|\bigr)\ar[r]^-{\simeq}\ar[d] &
    |B_L| \ar[d]^{h_\boxplus} \\
    \BAR(\Mm_{L_1},\Mm_\emptyset,\Mm_{L_2})\ar[r]^-{f} &
    \Mm_L
  }
  \end{equation}
\end{proposition}


The method of proof of this proposition will be generalized when we prove Proposition~\ref{prop:IcupJinappendix} on page~\pageref{prop:IcupJinappendix}.

\begin{proof}
	Consider first the degree $k=2$ case. 
	We begin by defining a trisimplicial space
  $F_0\colon(\Delta\times\Delta\times\Delta)^\op\to\Top$ 
  homeomorphic to 
  $\BAR_\bullet\bigl(B_{L_1}(\bullet),\Mm_\emptyset,B_{L_2}(\bullet)\bigr)$.
For each non-negative integer $n$ let
\[
  A_2(n)=\bigl(\tcoprod_{i=1}^n\!\Mm_{\emptyset,2}\bigr)\amalg
    \bigl(\tcoprod_{\substack{i,j=1 \\i<j}}^n\!\!\Mm_{\emptyset,1}^{C_{L_1}}\times\Mm_{\emptyset,1}^{C_{L_2}}\bigr)\,,\qquad
  A_{L_j,1}(n)=\tcoprod_{i=1}^n\!\Mm_{\emptyset,1}^{C_{L_j}}\quad (j=1,2).
  \]
  Note that $A_2(n)$ and $A_{L_j,1}(n)$
  are homeomorphic respectively to the degree 1 and 2 summands of
  $\bigl(\Mm_\emptyset\bigr)^n$.
  Then let
  \begin{multline*}
  F_0(n_1,n,n_2)=B_{L_1,2}(n_1)\amalg A_2(n)\amalg B_{L_2,2}(n_2)\\
  \amalg B_{L_1,1}(n_1)\times A_{L_2,1}(n) \amalg
  B_{L_1,1}(n_1)\times B_{L_2,1}(n_2) \amalg
  A_{L_1,1}(n)\times B_{L_2,1}(n_2)\,.
  \end{multline*}
  The homeomorphisms $\psi_U$ induce homeomorphisms
  \[
  F_0(n_1,n,n_2)\cong \BAR_n\bigl(B_{L_1}(n_1),\Mm_\emptyset,B_{L_2}(n_2)\bigr)
  \]
  which give $F_0$ the structure of a trisimplicial space.
        Realizing $F_0$ with respect to $n_1$ and $n_2$ we obtain the simplicial space $\BAR_\bullet\bigl(|B_{L_1}|,\Mm_\emptyset,|B_{L_2}|\bigr)$.

	Let $F_1$ be the bisimplicial space obtained by realizing 
	$F_0$ with respect to $n$.
	We write $|F_1|$ for its geometric realization, which may be described in 3 homeomorphic ways
	(see \cite[Lemma on page 94]{Qui73}):
	realizing with respect to one of the $n_i$'s and then with respect to the other, or realizing the diagonal simplicial space
	(if we let $d\colon\Delta\to\Delta\times\Delta$ be the diagonal functor, then $|d^*F_1|\cong|F_1|$). Let
	\begin{align*}
	B_{\emptyset,1}^{L_1}
	&=\BAR(\Mm_\emptyset^{U_{L_1}},\Mm_\emptyset^{C_{L_1}},\Mm_\emptyset^C)_1\,, &
	B_{\emptyset,1}^{L_2}
	&=\BAR(\Mm_\emptyset^C,\Mm_\emptyset^{C_{L_2}},\Mm_\emptyset^{U_{L_2}})_1 &
	\end{align*}
	(see Proposition~\ref{eq:f:B(M,Mcap,M)->Mcup})
	and let $B_{\emptyset,2}$ be as in equation~\eqref{eq:Bemptyset2}.
	Also recall Notation~\ref{nota:Ax*vs*xA}.
	All the face maps in $d^*F_1$
	between non-degenerate summands are represented
	in the diagram below 
	(in the lower half 
	of the diagram the order of the factors in some of the products was switched to make the maps clearer):
\[
\xymatrix@C=-4em@R=3ex{
	B_{\emptyset,1}^{L_1}\times
	\bigl(\prod\limits_{L_2}\Mm_{x}^{D_x}\bigr)_{\!1} 
	& & 
	\bigl(\prod\limits_{L_2}\Mm_{\emptyset}^{C_x}\bigr)_{\!1}\times
	\bigl(\prod\limits_{L_2}\Mm_{x}^{D_x}\bigr)_{\!1}
	\ar[rrr]\ar[ll]  
	& & &
	\bigl(\prod\limits_{L_2}\Mm_{x}^{U_x}\bigr)_{\!2}
	\\ & & &
	\bigl(\prod\limits_{L_2}\Mm_{\emptyset}^{C_x}\bigr)_{\!1}\times
	\bigl(\prod\limits_{L_2}\Mm_{\emptyset}^{D_x}\bigr)_{\!1}
	\ar[rrd]\ar[ld]\ar[lu] & \text{\phantom{MMMMMM}}
	\\
	\bigl(\prod\limits_{L_1}\Mm_{\emptyset}^{D_x}\bigr)_{\!1}\times
	\bigl(\prod\limits_{L_2}\Mm_{x}^{D_x}\bigr)_{\!1}
	\ar[dd]\ar[uu] 
	& &
	B_{\emptyset,1}^{L_1}\times
	\bigl(\prod\limits_{L_2}\Mm_{\emptyset}^{D_x}\bigr)_{\!1}
	\ar[rrrdd]\ar[lluu] 
	& & &
	\bigl(\prod\limits_{L_2}\Mm_{\emptyset}^{U_x}\bigr)_{\!2}
	\ar[dd]	\ar[uu]	
	\\ &
	\bigl(\prod\limits_{L_1}\Mm_{\emptyset}^{D_x}\bigr)_{\!1}^\ell\times
	\bigl(\prod\limits_{L_2}\Mm_{\emptyset}^{D_x}\bigr)_{\!1}^r
	\ar[rd]\ar[ru]\ar[lu]	
	\\
	\bigl(\prod\limits_{L_1}\Mm_{x}^{D_x}\bigr)_{\!1}\times
	\bigl(\prod\limits_{L_2}\Mm_{x}^{D_x} \bigr)_{\!1}
	& &
	\bigl(\prod\limits_{L_1}\Mm_{\emptyset}^{D_x}\bigr)_{\!1}\times
	\bigl(\prod\limits_{L_2}\Mm_{\emptyset}^{D_x}\bigr)_{\!1}
	\ar[rrr]\ar[ll] & & &
	B_{\emptyset,2}
	\\ &
	\bigl(\prod\limits_{L_1}\Mm_{\emptyset}^{D_x}\bigr)_{\!1}^r\times
	\bigl(\prod\limits_{L_2}\Mm_{\emptyset}^{D_x}\bigr)_{\!1}^\ell
	\ar[rd]\ar[ru]\ar[ld]
	\\
	\bigl(\prod\limits_{L_1}\Mm_{x}^{D_x}\bigr)_{\!1}\times
	\bigl(\prod\limits_{L_2}\Mm_{\emptyset}^{D_x}\bigr)_{\!1}
	\ar[dd]\ar[uu] 
	& &
	\bigl(\prod\limits_{L_1}\Mm_{\emptyset}^{D_x}\bigr)_{\!1}\times	
	B_{\emptyset,1}^{L_2}
	\ar[rrruu]\ar[lldd] 
	& & &
	\bigl(\prod\limits_{L_1}\Mm_{\emptyset}^{U_x}\bigr)_{\!2}
	\ar[dd]	\ar[uu]			
	\\ & & &
	\bigl(\prod\limits_{L_1}\Mm_{\emptyset}^{D_x}\bigr)_{\!1}\times
	\bigl(\prod\limits_{L_1}\Mm_{\emptyset}^{C_x}\bigr)_{\!1}
	\ar[rru]\ar[ld]\ar[lu]
	\\
	\bigl(\prod\limits_{L_1}\Mm_{x}^{D_x}\bigr)_{\!1}\times
	B_{\emptyset,1}^{L_2}
	& & 
	\bigl(\prod\limits_{L_1}\Mm_{x}^{D_x}\bigr)_{\!1}\times
	\bigl(\prod\limits_{L_1}\Mm_{\emptyset}^{C_x}\bigr)_{\!1}
	\ar[rrr]\ar[ll]  
	& & &
	\bigl(\prod\limits_{L_1}\Mm_{x}^{U_x}\bigr)_{\!2}	
}\]
	We now describe a weak equivalence $d^*F_1\to B_{L,2}^\bg$
	(see diagram~\eqref{diagram:BL2big}). We only need to define it 
	on $d^*F_1(n)$ for $n\leq 2$ since
	all summands in $d^*F_1(n)$ are degenerate for $n>2$.
	The map $d^*F_1(0)\to B_{L,2}^\bg(0)$ is defined by the 
	homotopy equivalence
	$B_{\emptyset,2}\to\Mm_{\emptyset,2}$ of 
	Proposition~~\ref{prop:B(M,M,M)->Mk=2},
by the map 
\begin{multline}\label{eq:dF1->Bbig1}
\bigl(\tprod_{x\in L_1}\!\!\Mm_x^{U_x}\bigr)_{\!2}
\amalg	 
\bigl(\tprod_{x\in L_2}\!\!\Mm_x^{U_x}\bigr)_{\!2}
\amalg
\bigl(\tprod_{x\in L_1}\!\!\Mm_x^{D_x}\bigr)_{\!1}
\times
\bigl(\tprod_{y\in L_2}\!\!\Mm_x^{D_x}\bigr)_{\!1}\\
\xrightarrow{D_x\subset U_x}
\bigl(\tprod_{x\in L_1}\!\!\Mm_x^{U_x}\bigr)_{\!2}
\amalg	 
\bigl(\tprod_{x\in L_2}\!\!\Mm_x^{U_x}\bigr)_{\!2}
\amalg
\bigl(\tprod_{x\in L_1}\!\!\Mm_x^{U_x}\bigr)_{\!1}
\times
\bigl(\tprod_{x\in L_2}\!\!\Mm_x^{U_x}\bigr)_{\!1}
=\bigl(\tprod_{x\in L}\!\Mm_x^{U_x}\bigr)_{\!2}
\end{multline}
and by the map
\begin{multline}\label{eq:dF1->Bbig2}
\bigl(\tprod_{x\in L_1}\!\!\Mm_x^{D_x}\bigr)_{\!1}\times B_{\emptyset,1}^{L_2}
\,\amalg\,
B_{\emptyset,1}^{L_1}\times\bigl(\tprod_{x\in L_2}\!\!\Mm_x^{D_x}\bigr)_{\!1}
\\
\cong
\bigl(\tcoprod_{x\in L_1}\!B_{\emptyset,1}^{L_2}\times\Mm_{x,1}^{D_x}\bigr)
\amalg
\bigl(\tcoprod_{x\in L_2}\!B_{\emptyset,1}^{L_1}\times\Mm_{x,1}^{D_x}\bigr)
\xrightarrow{f\times\ident}
\tcoprod_{x\in L}\!\Mm_{\emptyset,1}^{B_x}\times\Mm_{x,1}^{D_x}
\end{multline}
where $f$ is the homotopy equivalence 
from Proposition~\ref{eq:f:B(M,Mcap,M)->Mcup},
which is well defined since
$U_{L_2}\subset B_x$ for $x\in L_1$ and
$U_{L_1}\subset B_x$ for $x\in L_2$.

The map $d^*F_1(1)\to B_{L,2}^\bg(1)$
is defined on non-degenerate summands by the map
\[
\bigl(\tprod_{x\in L_1}\!\!\Mm_\emptyset^{U_x}\bigr)_{\!2}
\amalg	 
\bigl(\tprod_{x\in L_2}\!\!\Mm_\emptyset^{U_x}\bigr)_{\!2}
\amalg
\bigl(\tprod_{x\in L_1}\!\!\Mm_\emptyset^{D_x}\bigr)_{\!1}
\times
\bigl(\tprod_{y\in L_2}\!\!\Mm_\emptyset^{D_x}\bigr)_{\!1}
\to\bigl(\tprod_{x\in L}\!\Mm_\emptyset^{U_x}\bigr)_{\!2}
\]
(compare with equation~\eqref{eq:dF1->Bbig1}), the map
\[
\bigl(\tprod_{x\in L_1}\!\Mm_\emptyset^{D_x}\bigr)_{\!1}
\times B_{\emptyset,1}^{L_2}
\,\amalg\,
B_{\emptyset,1}^{L_1}\times
\bigl(\tprod_{x\in L_2}\!\!\Mm_\emptyset^{D_x}\bigr)_{\!1}
\to \tcoprod_{x\in L}\Mm_{\emptyset,1}^{B_x}\times\Mm_{\emptyset,1}^{D_x}
\]
(compare with equation~\eqref{eq:dF1->Bbig2}), the map
\begin{multline*}
\bigl(\tprod_{x\in L_1}\!\!\Mm_x^{D_x}\bigr)_{\!1}
\times
\bigl(\tprod_{x\in L_1}\!\Mm_\emptyset^{C_x}\bigr)_{\!1}
\,\amalg\,
\bigl(\tprod_{x\in L_2}\!\Mm_\emptyset^{C_x}\bigr)_{\!1}
\times
\bigl(\tprod_{x\in L_2}\!\!\Mm_x^{D_x}\bigr)_{\!1}
\\ \cong
\tcoprod_{x\in L}\!
\Mm_{\emptyset,1}^{C_x}\times
\Mm_{x,1}^{D_x}
\,\amalg\,
\tcoprod_{\substack{x,y\in L_1\\ x\neq y}}\!
\Mm_{\emptyset,1}^{C_y}\times\Mm_{x,1}^{D_x}
\,\amalg\,
\tcoprod_{\substack{x,y\in L_2\\ x\neq y}}\!
\Mm_{\emptyset,1}^{C_y}\times\Mm_{x,1}^{D_x}
\\ \xrightarrow{C_y\subset U_y}
\tcoprod_{x\in L}\!
\Mm_{\emptyset,1}^{C_x}\times
\Mm_{x,1}^{D_x}
\,\amalg\,
\tcoprod_{\substack{x,y\in L_1\\ x\neq y}}\!
\Mm_{\emptyset,1}^{U_y}\times\Mm_{x,1}^{D_x}
\,\amalg\,
\tcoprod_{\substack{x,y\in L_2\\ x\neq y}}\!
\Mm_{\emptyset,1}^{U_y}\times\Mm_{x,1}^{D_x}
\end{multline*}
and the map 
\begin{multline*}
\bigl(\tprod_{x\in L_1}\!\!\Mm_x^{D_x}\bigr)_{\!1}
\times
\bigl(\tprod_{x\in L_2}\!\!\Mm_\emptyset^{D_x}\bigr)_{\!1}
\,\amalg\,
\bigl(\tprod_{x\in L_1}\!\!\Mm_\emptyset^{D_x}\bigr)_{\!1}
\times
\bigl(\tprod_{x\in L_2}\!\!\Mm_x^{D_x}\bigr)_{\!1}
\\ \cong
\tcoprod_{\substack{x\in L_1 \\ y\in L_2}}
\Mm_{\emptyset,1}^{D_y}\times\Mm_{x,1}^{D_x}
\,\amalg\,
\tcoprod_{\substack{y\in L_1 \\ x\in L_2}}
\Mm_{\emptyset,1}^{D_y}\times\Mm_{x,1}^{D_x}
\\ \xrightarrow{D_y\subset U_y}
\tcoprod_{\substack{x\in L_1 \\ y\in L_2}}
\Mm_{\emptyset,1}^{U_y}\times\Mm_{x,1}^{D_x}
\,\amalg\,
\tcoprod_{\substack{y\in L_1 \\ x\in L_2}}
\Mm_{\emptyset,1}^{U_y}\times\Mm_{x,1}^{D_x}
\end{multline*}
The map $d^*F_1(2)\to B_{L,2}^\bg(2)$ is defined on non-degenerate summands 
by the map
\begin{multline*}
\bigl(\tprod_{x\in L_1}\!\!\Mm_\emptyset^{D_x}\bigr)_{\!1}
\times
\bigl(\tprod_{x\in L_1}\!\Mm_\emptyset^{C_x}\bigr)_{\!1}
\,\amalg\,
\bigl(\tprod_{x\in L_2}\!\Mm_\emptyset^{C_x}\bigr)_{\!1}
\times
\bigl(\tprod_{x\in L_2}\!\!\Mm_\emptyset^{D_x}\bigr)_{\!1}
\\ \xrightarrow{C_y\subset U_y}
\tcoprod_{x\in L}\!
\Mm_{\emptyset,1}^{C_x}\times
\Mm_{\emptyset,1}^{D_x}
\,\amalg\,
\tcoprod_{\substack{x,y\in L_1\\ x\neq y}}\!
\Mm_{\emptyset,1}^{U_y}\times\Mm_{\emptyset,1}^{D_x}
\,\amalg\,
\tcoprod_{\substack{x,y\in L_2\\ x\neq y}}\!
\Mm_{\emptyset,1}^{U_y}\times\Mm_{\emptyset,1}^{D_x}
\end{multline*}
and the map 
\begin{multline*}
\bigl(\tprod_{x\in L_1}\!\!\Mm_\emptyset^{D_x}\bigr)^r_{\!1}
\times
\bigl(\tprod_{x\in L_2}\!\!\Mm_\emptyset^{D_x}\bigr)^\ell_{\!1}
\,\amalg\,
\bigl(\tprod_{x\in L_1}\!\!\Mm_\emptyset^{D_x}\bigr)^\ell_{\!1}
\times
\bigl(\tprod_{x\in L_2}\!\!\Mm_\emptyset^{D_x}\bigr)^r_{\!1}
\\ \cong
\tcoprod_{\substack{x\in L_1 \\ y\in L_2}}
\Mm_{\emptyset,1}^{D_y,\ell}\times\Mm_{\emptyset,1}^{D_x,r}
\,\amalg\,
\tcoprod_{\substack{y\in L_1 \\ x\in L_2}}
\Mm_{\emptyset,1}^{D_y,\ell}\times\Mm_{\emptyset,1}^{D_x,r}
\\ \xrightarrow{D_y\subset U_y}
\tcoprod_{\substack{x\in L_1 \\ y\in L_2}}
\Mm_{\emptyset,1}^{U_y}\times\Mm_{\emptyset,1}^{D_x}
\,\amalg\,
\tcoprod_{\substack{y\in L_1 \\ x\in L_2}}
\Mm_{\emptyset,1}^{U_y}\times\Mm_{\emptyset,1}^{D_x}
\end{multline*}
In this way we get a diagram:
	\[
	\xymatrix{
		\BAR\bigl(|B_{L_1}|,\Mm_\emptyset,|B_{L_2}|\bigr)_2\cong |d^*F_1|\ar[r]^-{\simeq}\ar[d] &
		|B_{L,2}^\bg| \ar[d]\\
		\BAR(\Mm_{L_1},\Mm_\emptyset,\Mm_{L_2})_2\ar[r] &
		\Mm_{L,2}
	}
	\]
\out{
		To check that this diagram commutes we only need to check that  
	the following diagram commutes (in the $k=2$ components):             
\[
\xymatrix@C=1em{
\Bigl(\Mm_\emptyset\times\!\prod\limits_{x\in L_1}\!\Mm_x\Bigr)
\times\Mm_\emptyset\times
\Bigl(\Mm_\emptyset\times\!\prod\limits_{x\in L_2}\!\Mm_x\Bigr)
\ar[d]_-{\cong}
\ar[rr]^-{\boxplus_{\emptyset,\mathbf{L_1}}\times\ident\times\boxplus_{\emptyset,\mathbf{L_2}}} & &
\Mm_{L_1}\times\Mm_\emptyset\times\Mm_{L_2}
\ar[dd]_{\boxplus_{L_1,\emptyset,L_2}}
\\
\prod\limits_{x\in L_1}\!\Mm_x\times
\Bigl(\Mm_\emptyset\times\Mm_\emptyset\times\Mm_\emptyset\Bigr)
\times\!\prod\limits_{x\in L_2}\!\Mm_x
\ar[d]_{\boxplus_{\emptyset,\emptyset,\emptyset}}
\\
\prod\limits_{x\in L_1}\!\Mm_x\times
\Mm_\emptyset\times\!\prod\limits_{x\in L_2}\!\Mm_x
\ar[r]^-{\cong} &
\Mm_\emptyset\times\!\prod\limits_{x\in L}\!\Mm_x
\ar[r]^-{\boxplus_{\emptyset,\mathbf{L}}} & \Mm_L            
}\]                                                     
}
which commutes by the commutativity and associativity of the gluing maps
(Proposition~\ref{prop:mapsboxplus}).
        This finishes the proof for $k=2$.
        The proof for $k=1$ uses analogous but much simpler arguments.
\end{proof}

\subsection{Proof of Theorem~\ref{theor0}}\label{sec:8.5proof1.1}

We can now prove Theorem~\ref{theor0}, which we restate here. Let $\Vv$ be a finite dimensional complex Hermitian vector space and
denote the dependence of the bar construction on $\Vv$ by $B_L^\Vv$.

{
\renewcommand{\theequation}{\ref{theor0}}
\begin{theorem}
Let $I=\{x_1,\dots,x_q\}\subset\Cc^2$. Then:
\begin{enumerate}
\item The map 
$h_\boxplus\colon|B^\Vv_I|\to \Mm^\Vv_I$ from equations~\eqref{eq:hindegree1} and~\eqref{eq:maphboxplus} is a homotopy equivalence in the $k=1,2$ components.
\item If $I=J\cup K$, with $J\cap K=\emptyset$, then
the map 
$\BAR(\Mm^\Vv_J,\Mm^\Vv_\emptyset,\Mm^\Vv_K)\to \Mm^\Vv_{I}$ from equations~\eqref{eq:mapbarMMM->Mdeg1} and~\eqref{eq:mapbarMMM->M} is a homotopy equivalence in the $k=1,2$ components.
\end{enumerate}
\end{theorem}
\addtocounter{equation}{-1}
}
\begin{proof}
From Proposition~\ref{prop:barbarbardiagramcommutesinfiniterank} with $L_1=J$ and $L_2=K$
it follows that part (2) of the theorem is a consequence of part (1), which we now prove.
Let $\cat$ be as in Definition~\ref{def:cat} and let $\cat_{I,k}$ be as in Definition~\ref{def:catI,k}.
Let $\Mm^\Vv_k\colon \cat\to\Top$ be the functor defined on objects by $\Mm^\Vv_k(I)=\Mm^\Vv_{I,k}$ and defined on morphisms by pullback.
Also, let $B^\Vv_k\colon\cat\times\Delta^{\op}\to\Top$ be the functor defined on objects by 
$B^\Vv_k(n,J)=B^\Vv_{J,k}(n)$. Analogous but simpler arguments to those in the 
proof of Theorem~\ref{theoI} (on page~\pageref{proof:theorem1.3}) show that we have a commutative diagram
\begin{equation}\label{eq:nleqkisenoughinfiniter}
\entrymodifiers={+!!<0pt,\fontdimen22\textfont2>}
\xymatrix{
\hcolim\limits_{\cat_{I,k}}B^\Vv_k\ar[r]\ar[d]_-{\simeq} & \hcolim\limits_{\cat_{I,k}}\Mm^\Vv_k\ar[d]^-{\simeq} \\
|B^\Vv_{I,k}|\ar[r] & \Mm^\Vv_{I,k}
}\end{equation}
where the vertical maps are homotopy equivalences,
so the proof will be complete if we show that the top horizontal map is a homotopy equivalence.
Thus, for $k=1$ we only need to consider the trivial case when $\#I=1$, and
for $k=2$, we only need to consider the case where $I=\{x,y\}\subset\Cc^2$ which we now analyze.
Now consider Proposition~\ref{prop:barbarbardiagramcommutesinfiniterank} in the case where $L_1=\{x\}$, $L_2=\{y\}$ and $L=\{x,y\}$.
The maps $|B^\Vv_x|\to \Mm^\Vv_x$ in diagram~\eqref{eq:diagram-barbarbar} are trivially homotopy equivalences hence the left vertical map is a homotopy equivalence.
Thus we only have to show that the bottom horizontal map in diagram~\eqref{eq:diagram-barbarbar} is a homotopy equivalence. Consider the bar construction in Remark~\ref{remark:C_L=U_L}. We have maps
\[
\BAR(\Mm_x^\Vv,\Mm_\emptyset^\Vv,\Mm_y^\Vv)\xrightarrow{\simeq}\hatBAR(\Mm_x^\Vv,\Mm_\emptyset^\Vv,\Mm_y^\Vv)\to\Mm_L^\Vv\,.
\]
The face maps between the non-degenerate degree 2 summands in
$\hatBAR_\bullet(\Mm_x^\Vv,\Mm_\emptyset^\Vv,\Mm_y^\Vv)$ as well as the map $\hatBAR_0(\Mm_x^\Vv,\Mm_\emptyset^\Vv,\Mm_y^\Vv)\to\Mm_L^\Vv$ are
represented in the following cube diagram
(where we omit $\Vv$ from the notation):
\[
\xymatrix@C=5em{ &
  \Mm_{x,2}\ar[rr]^-{\pi_{x,L}^*} &&
  \Mm_{L,2} \\
  \Mm_{x,1}^{U_x}\times\Mm_{\emptyset,1}^{U_y}\ar[ru]^{\boxplus_{x,\emptyset}}\ar[rr]^(0.7){\ident\times\pi_{\emptyset,y}^*} &&
  \Mm_{x,1}^{U_x}\times\Mm_{y,1}^{U_y}\ar[ru]_{\boxplus_{x,y}} \\ &
  \Mm_{\emptyset,2}\ar'[r]_{\pi_{\emptyset,y}^*}[rr]\ar'[u]^{\pi_{\emptyset,x}^*}[uu] &&
  \Mm_{y,2}\ar[uu]_{\pi_{y,L}^*}\\
  \Mm_{\emptyset,1}^{U_x}\times\Mm_{\emptyset,1}^{U_y}\ar[uu]^{\pi_{\emptyset,x}^*\times\ident}\ar[ru]^{\boxplus_{\emptyset,\emptyset}}\ar[rr]^-{\ident\times\pi_{\emptyset,y}^*} &&
  \Mm_{\emptyset,1}^{U_x}\times\Mm_{y,1}^{U_y}\ar[uu]_(0.7){\pi_{\emptyset,x}^*\times\ident}\ar[ru]_-{\boxplus_{\emptyset,y}}  
}
\]
The diagram obtained from the cube by removing the vertex $\Mm_{L,2}^\Vv$
is homeomorphic to the nerve $N$ of the open cover of $\Mm_{L,2}^\Vv$ in Corollary~\ref{cor:opencoverofM_2} (see Example~\ref{ex:nerveofopencover} for the definition of nerve), and
the barycentric subdivision of 
$\hatBAR_\bullet(\Mm_x^\Vv,\Mm_\emptyset^\Vv,\Mm_y^\Vv)$
(see \cite[section 8.3.A, page 60]{ViFu04})
is homeomorphic to $B_\bullet(*,N,*)$. It follows from 
\cite[Proposition 4.1]{Seg68} that 
the map $\hatBAR(\Mm_x^\Vv,\Mm_\emptyset^\Vv,\Mm_y^\Vv)\to\Mm_L^\Vv$ is a homotopy equivalence, which concludes the proof of the theorem.
\end{proof}

Taking the limit when $\dim \Vv\to\infty$ we can now prove: 

\begin{theorem}\label{thm:hI2simeq}
Let $\Hh$ be a countably infinite dimensional complex Hermitian vector space. Then
the map $h_{I}\colon\|\Bb_I^\Hh\|\to\Mm_I^\Hh$ in Definition~\ref{def:Delta_I,Bb_I,h_I}
is a homotopy equivalence in degrees 1 and 2.
\end{theorem}
\begin{proof}
The maps in Proposition~\ref{prop:mapsboxplus} pass to the colimit to define maps
$\boxplus_{\mathbf I}\colon \prod_i\Mm_{I_i}^{\Hh,U_i}\to \Mm_I^{\Hh,U}$.
We need to see how these maps are related with the action of 
the linear isometries operad $\Ll^\Hh$.
Let $\imath_\alpha:\Hh\to\Hh^n$ be inclusion onto the $\alpha$-th factor. 
Let $\Oo$ be the suboperad of the endomorphism operad of $\Hh$ consisting of  complex linear maps $f\colon \Hh^n\to\Hh$ such that
$f\circ\imath_\alpha$ is an isometry for all $\alpha=1,\ldots,n$. 
Then $\Oo(n)\cong\Ll(1)^n$ so $\Oo(n)$ is contractible.
Also note that $\Oo$ contains the linear isometries operad $\Ll$.
We fix basepoints $*\in\Oo(n)$ given in matrix notation by
\[
*=\left[\ \Id\quad\cdots\quad\Id\ \right]:\Hh^n\to\Hh\,.
\]
For any $f\in\Oo(n)$ and each $\alpha=1,\ldots,n$, let
$f_\alpha=f\circ\imath_\alpha$. Given finite disjoint sets $J_1$, \ldots, $J_n$ we define a map $\boxplus_{\mathbf J,f}$
as the composition:
\begin{equation}\label{eq:fboxplus}
\boxplus_{\mathbf J,f}\colon\prod_{\alpha=1}^{n}\Mm_{J_\alpha,k_\alpha}^{\Hh,U_\alpha}
\xrightarrow{\prod f_\alpha}
\prod_{\alpha=1}^{n}\Mm_{J_\alpha,k_\alpha}^{\Hh,U_\alpha}
\xrightarrow{\boxplus_{\mathbf J}}\Mm_{J,k}^{\Hh,U}\qquad(k=\textstyle\sum k_\alpha\leq 2)\,.
\end{equation}
Note that, for $f=*$, we have $\boxplus_{\mathbf J,f}=\boxplus_{\mathbf J}$. The maps $\boxplus_{\mathbf J,f}$ are compatible with the operad data
in the following sense: Let $j_1,\dots,j_n$ be non-negative integers and let $j=\sum j_i$;
let $f\in\Oo(n)$ and $g_i\in\Oo(j_i)$;
consider finite disjoint sets $J_\alpha$, with $\alpha=1,\dots,j$ and
for each $i=1,\dots,n$, let $s_i=j_1+\dots+j_i$ and let:
\[
\mathbf J_i=(J_{s_{i-1}+1},\dots,J_{s_{i}})\,,\qquad 
K_i=\bigcup_{\alpha=s_{i-1}+1}^{s_i}J_\alpha\qquad
(i=1,\dots,n)\,;
\]
also, let:
\[
\mathbf J=(J_1,\dots,J_j)\,,\qquad \mathbf K=(K_1,\dots,K_n)\,,\qquad J=\bigcup\limits_{i=1}^nK_i=\bigcup\limits_{\alpha=1}^jJ_\alpha\,.
\]
Then we have:
\begin{equation}\label{eq:compatible}
\boxplus_{\mathbf J,f\circ\prod g_i}=\boxplus_{\mathbf K,f}\circ\prod_{i=1}^n\boxplus_{\mathbf J_i,g_i}\,.
\end{equation}
To prove this last statement, 
let $\imath_\alpha\colon\Hh\to \Hh^{j_i}$ (with $\alpha=1,\dots,j$) be inclusion into the $(\alpha-s_{i-1})$-th factor
and let  $f_i=f\circ\imath_i$ and $g_{i\alpha}=g_i\circ\imath_\alpha$.
Equation~\eqref{eq:compatible} then follows from the commutativity of the following diagram (the lower triangle is commutative by Proposition~\ref{prop:mapsboxplus}(2) and the upper right triangle is commutative by Proposition~\ref{prop:mapsboxplus}(5)): 
\[
\xymatrix@C=5em{
\ds\prod_{i=1}^n\left[\prod_{\alpha=s_{i-1}+1}^{s_i}\!\!\!\Mm_{J_\alpha}^{\Hh,U_\alpha}\right]
\ar[r]^-{\prod_i\prod_\alpha f_i\circ g_{i\alpha}}
\ar[d]_-{\prod_i\prod_\alpha g_{i\alpha}}&
\ds\prod_{i=1}^n\left[\prod_{\alpha=s_{i-1}+1}^{s_i}\!\!\!\Mm_{J_\alpha}^{\Hh,U_\alpha}\right]
\ar[r]^-{\boxplus_{\mathbf J}}\ar[rd]_{\prod_i\boxplus_{\mathbf J_i}}&\Mm_{J}^{\Hh,U}\\
\ds\prod_{i=1}^n\left[\prod_{\alpha=s_{i-1}+1}^{s_i}\!\!\!\Mm_{J_\alpha}^{\Hh,U_\alpha}\right]
\ar[r]_-{\prod_i\boxplus_{\mathbf J_i}}\ar[ru]_{\prod_i\prod_\alpha f_i}&
\ds\prod_{i=1}^n\Mm_{K_i}^{\Hh,U_i}\ar[r]_-{\prod_i f_i}&\ds\prod_{i=1}^n\Mm_{K_i}^{\Hh,U_i}\ar[u]_-{\boxplus_{\mathbf K}}}
\]
Let $q=\#I$, let $\Oo_I(n)=\Oo(1+nq)$ and $\Oo^I(n)=\prod_I\Oo(n)$
(compare with Definition~\ref{def:LlI}). Consider the category $\Delta(\Oo_I,\Oo^I)$ (see Definition~\ref{con:catDelta(P)}).
Then we can define a functor
$\Ff_I:\Delta(\Oo_I,\Oo^I)^\op\to \Top$ 
by letting $\Ff_I(n)=B^\Hh_I(n)$ (see equation~\eqref{eq:defofB_J,2})
and defining the morphisms using the maps $\boxplus_{\mathbf J,f}$ in equation~\eqref{eq:fboxplus}.
The inclusion of operads $\Ll^\Hh\subset\Oo$ induces an equivalence of categories 
$\imath\colon\Delta_I\to\Delta(\Oo_I,\Oo^I)$. 
For $f\in\Ll^\Hh(n)\subset\Oo(n)$ we have the commutative diagram (see Proposition~\ref{prop:mapsboxplus}):
\[
\xymatrix{
\ds\prod_{\alpha=1}^{n}\Mm_{J_\alpha,k_\alpha}^{\Hh,U_\alpha}
\ar[r]^-{\prod \imath_\alpha}\ar[rd]_-{\oplus\circ\pi^*}&
\ds\prod_{\alpha=1}^{n}\Mm_{J_\alpha,k_\alpha}^{\Hh^n\!,U_\alpha}
\ar[r]^-{\prod f}\ar[d]^-{\boxplus}&
\ds\prod_{\alpha=1}^{n}\Mm_{J_\alpha,k_\alpha}^{\Hh,U_\alpha}
\ar[d]^-{\boxplus}\\ &
\Mm_{J,k}^{\Hh^n\!,U}\ar[r]^-{f}&
\Mm_{J,k}^{\Hh,U}}
\]
so the inclusions $\Mm_J^{\Hh,U}\to\Mm_J^{\Hh}$ induce a weak equivalence from the restriction of $\Ff_I$ to $\Delta_I$ to $\Bb^\Hh_I$.
To finish the proof we observe that the inclusion of the base point in $\Oo(n)$ induces an equivalence of categories $\Delta\to\Delta(\Oo_I,\Oo^I)$ and
the restriction of $\Ff_I$ to $\Delta$ equals the simplicial space $B^\Hh_I$.
The functor $\Ff_I$ extends to a functor $\widetilde\Ff_I:\widetilde\Delta(\Oo_I,\Oo^I)\to \Top$ (see Definition~\ref{con:widetildedelta(P)})
and
$B^\Hh_I$ extends to a functor $\widetilde B^\Hh_I\colon\widetilde\Delta\to\Top$ by letting $\widetilde\Ff_I(-1)=B^\Hh_I(-1)=\Mm_I^{\Hh}$ and defining the functors on morphisms in the obvious way, and we have a commutative diagram
\[
\xymatrix{
\Bb^\Hh_I\ar[d]\ar[r]^-{\simeq}&\Ff_I\ar[d]&B^\Hh_I\ar[d]^-{\simeq}\ar[l]_-{\simeq}\\
\widetilde\Bb^\Hh_I\ar[r]^-{\simeq}&\widetilde\Ff_I&\widetilde B^\Hh_I\ar[l]_-{\simeq}}
\]
which completes the proof.
\end{proof}

\section{The limit when $k\to\infty$}\label{sec:7}

We now consider the limit when $k\to\infty$. We will work with the moduli space of instantons, which we now describe. 
For each finite set $J\subset\Cc^2$ let $\#_J\overline{\CP^2}$ be the smooth 4-manifold obtained by collapsing $L_\infty\subset\cp{J}$ to a point $y_\infty\in \#_J\overline{\CP^2}$ 
and let $\overline\pi\colon\cp{J}\to \#_J\overline{\CP^2}$ be the collapsing map. Let $g$ be a smooth metric on $\#_J\overline{\CP^2}$
such that $\overline\pi^*g$ is compatible with the complex structure on $\cp{J}$.
Let $\Vv$ be a finite dimensional complex Hermitian vector space,
let $E\to\#_J\overline{\CP^2}$ be an $SU(\Vv)$ vector bundle with second Chern class $c_2(E)=k$, let $\Aa(E)$ denote the space of smooth connections on $E$
and let $P_\infty$ denote the space of isomorphisms $\phi\colon E_{y_\infty}\to \Vv$ equivariant with respect to the $SU(\Vv)$ action.
The gauge group $\Gg$ of smooth automorphisms of $E$ acts freely on
$P_{\infty}\times\Aa(E)$ and we represent the quotient by $\Cs_{J,k}^\Vv$.
The \emph{moduli space of instantons} is the subspace $\MI_{J,k}^\Vv\subset\Cs_{J,k}^\Vv$ of
equivalence classes of connections whose curvature is self-dual with respect to the Hodge * operator. 
This space is isomorphic to $\Mm_{J,k}^\Vv$ (see \cite[Theorem 0.1]{Buc93}, \cite[Corollary 2.14]{Mat00}).
Let $\MI_{J}^\Vv=\coprod_k\MI_{J,k}^\Vv$ and 
$\Cs_{J}^\Vv=\coprod_k\Cs_{J,k}^\Vv$. The same arguments as the ones in section~\ref{sec:2}
show that $\MI_{J}^\Vv$ and $\Cs_{J}^\Vv$ are $\Ii_*$-functors and we have:

\begin{proposition}\label{prop:inclusionismapofLalgebras}
  The inclusion $\MI_{J}^\Vv\to\Cs_{J}^\Vv$ and the isomorphism $\psi_\Vv\colon\MI_{J}^\Vv\to\Mm_{J}^\Vv$ are maps of $\Ii_*$-functors.
\end{proposition}
\begin{proof}
  The result is clear for the inclusion
  and it follows easily for the map $\psi_\Vv$ from its description: given a self-dual connection $\xi$ on $E$, 
the $(0,1)$-part of the pullback $(\overline\pi^*\xi)^{(0,1)}$ is a holomorphic structure on the pullback bundle
$\overline\pi^*E\to\cp{J}$ and an isomorphism $\phi\in P_{\infty}$ induces a holomorphic trivialization
$\phi\times\ident\colon \overline\pi^*E|_{L_\infty}=E_{y_\infty}\times L_\infty\to \Vv\times L_{\infty}$; then
$\psi_\Vv\bigl([\phi,\xi]\bigr)=\bigl[(\overline\pi^*\xi)^{(0,1)},\phi\times\ident\bigr]\in\Mm_{J,k}^\Vv$ (recall Definition~\ref{def:ofMm}).
\end{proof}

Taking the direct limit when $k\to\infty$ 
(see the discussion on page~\pageref{theoIII}, before Theorem~\ref{theoIII})
the inclusion $\MI_{J}^\Vv\subset\Cs_J^\Vv$ induces a weak homotopy equivalence $\MI^\Vv_{J,\infty}\xrightarrow{\simeq}\Cs_{J,\infty}^\Vv$
\cite[Theorem 2*]{Tau89} and $\Cs_{J}^\Vv$ is homotopically equivalent to the space of based maps $\Map_*(\#_J\overline{\CP^2},BSU(\Vv))$ where $BSU(\Vv)$ is the classifying space of $SU(\Vv)$
(see \cite[Theorem 1.3]{CoMi94}).

\out{
The map $\lambda_\Vv\colon\Map(\#_J\CP^2,BSU(\Vv))\to \Cs_k^\Vv$ can be described as follows
\cite[page 101]{AtJo78}: Fix a connection $\xi$ on the universal bundle $ESU(\Vv)\to BSU(\Vv)$ and a point $\Id\in ESU(\Vv)$ over the basepoint of $BSU(\Vv)$.
Given a map $f\colon \#_J\CP^2\to BSU(\Vv)$ classifying the bundle $P$, we have
\begin{equation*}
\lambda_\Vv(f)=[f^{-1}(\Id),f^*\xi]\in P_{y_\infty}\times\Aa(P)/\Gg\,.
\end{equation*}
}%

\out{%
Let $\overline\pi\colon\cp{J}\to \#_J\CP^2$ be the collapsing map. Assume $\#_J\CP^2$ has a smooth metric $g$ such that $\overline\pi^*g$ is compatible with the complex structure on $\cp{I}$.
The moduli space of instantons is the subspace $\MI_{J,k}^\Vv\subset\Cs_{J,k}^\Vv$ of
equivalence classes of connections whose curvature is self-dual with respect to the Hodge * operator. There is an isomorphism $\psi_\Vv\colon\MI_{J,k}^\Vv\to\Mm_{J,k}^\Vv$
which we now describe: given a self-dual connection $\xi$ on $P$, let $\nabla_\xi$ be the associated connection on $P\times_{SU(\Vv)}\Vv$;
the $(0,1)$-part of the pullback $(\overline\pi^*\xi)^{(0,1)}$ is a holomorphic structure on the pullback bundle
$E=\overline\pi^*(P\times_{SU(\Vv)}\Vv)\to\cp{J}$. An element $p\in P_{y_\infty}$ induces an isomorphism $P_{y_\infty}\times_{SU(\Vv)}\Vv\to \Vv$ and hence a holomorphic trivialization
$\phi_p\colon E|_{L_\infty}\to \Vv\times L_{\infty}$. Thus we get a map $\psi_\Vv\colon\MI_{J,k}^\Vv\to\Mm_{J,k}^\Vv$ with $\psi_\Vv\bigl([p,\xi])=\bigl([\phi_p,(\overline\pi^*\xi)^{(0,1)}]\bigr)$.
Note that $\psi_\Vv$ is a map of $\Ii$-functors.
}%

\out{
In \cite{Tau89}, Taubes described, for $k'>k$, maps of pairs
$\bigl(\Cs_{J,k}^\Vv,\MI_{J,k}^{\,\Vv}\bigr)\to\bigl(\Cs_{J,k'}^\Vv\to\MI_{J,k'}^{\,\Vv}\bigr)$
such that the maps on $\mathscr C_{J,k}^\Vv$ are homotopy equivalences, and showed that in the direct limit
the inclusions $\imath_k$ induce a homotopy equivalence
$\MI_{J,\infty}^{\,\Vv}\xrightarrow{\simeq}\Cs_{J,\infty}^\Vv$.
}%

We now prove Theorem~\ref{theoIII}.
Since $\overline{\CP^2}$ has a CW-complex structure 
given by attaching a 4-dimensional cell $D^4$ to $S^2$ via the Hopf map
$h\colon S^3\to S^2$, we can view
maps from $\overline{\CP^2}$ as maps from $D^4$ whose restriction to the boundary factor through $h$.
In an analogous way, the space of maps from $\#_n\overline{\CP^2}$
to $BSU(\Vv)$ is homotopy equivalent to
the space of maps from $D^4$ to $BSU(\Vv)$ which factor through the map
$S^3\to\bigvee_nS^3\xrightarrow{\bigvee h}\bigvee_nS^2$.
For the remainder of this section $I$ denotes the unit interval: $I=[0,1]\subset\Rr$. Note that $S^3\cong I^3/\partial I^3$
and $\bigvee_nS^3$ is the quotient of $I^3$ by the subspace of triples
$\mathbf x=(x_1,x_2,x_3)$ with either $\mathbf x\in\partial I^3$
or $x_1=i/n$ with $i=0,\dots,n$.

\begin{definition}\label{def:loopmonoidandaction}
Given a based topological space $(X,*)$ let ${\mathcal M}_* X$ represent the space of 
compactly supported maps $f\colon \left(-\infty,0\right]\times I^3\to X$
such that $f(t,\mathbf x)=*$ whenever $\xx\in\partial(I^3)$.
\begin{enumerate}
\item We identify, up to homotopy, 
$\Omega^4X$ with the subspace of maps $f\in {\mathcal M}_*X$ such that
$f(0,\mathbf x)=*$ for any $\mathbf x\in I^3$.
\item Let $H\colon I^3\to S^2$ be the composition of the projection $I^3\to I^3/\partial I^3$ with the Hopf map.
We identify, up to homotopy, $\Map_*\bigl(\overline{\CP^2},X\bigr)$ with the subspace of maps
$f\in {\mathcal M}_*X$ whose restriction to $0\times I^3$ factors through $H$. Restriction to $0\times I^3$ induces
a map\label{maprho} $\rho\colon\Map_*\bigl(\overline{\CP^2},X\bigr)\to\Omega^2X$.
\item Let $\xx=(x_1,x_2,x_3)\in I^3$.
We identify, up to homotopy, $\Map_*\bigl(\#_n\overline{\CP^2},X\bigr)$ with the subspace of maps
$f\in {\mathcal M}_*X$ such that: 
\begin{enumerate}
\item $f(0,i/n,x_2,x_3)=*$ for $i=0,\dots,n$ and any
$x_2,x_3\in I$; 
\item for each $i=1,\dots,n$, the restriction of $f$ to $0\times\bigl[(i-1)/n,i/n\bigr]\times I^2\cong I^3$ factors through $H\colon I^3\to S^2$.
\end{enumerate}
We have a map \label{mapboldsymbolrho} $\boldsymbol\rho=(\rho_1,\dots,\rho_n)\colon\Map_*\bigl(\#_n\overline{\CP^2},X\bigr)\to(\Omega^2X)^n$
whose components $\rho_i$ are induced by restriction to $0\times \bigl[(i-1)/n,i/n\bigr]\times I^2$, for $i=1,\dots,n$.
\end{enumerate}
We now give $(\Omega^4X)^n$ the structure of an associative monoid and define actions of $(\Omega^4X)^n$ on
$(\Map_*(\overline{\CP^2},X))^n$ and $\Omega^4X$. 
\begin{enumerate}\setcounter{enumi}{3}
\item First we define a left action of $\Omega^4X$ on ${\mathcal M}_*X$.
Given a function $f\in {\mathcal M}_* X$ we let $s_f$ be the supremum of the set of 
$t\in\left(-\infty,0\right]$ such that $f(s,\xx)=*$ for any
$s\leq t$ and any $\xx\in I^3$. 
Then, given $g\in\Omega^4X$ we define
$g\cdot f\in {\mathcal M}_* X$ by:
\[
g\cdot f(t,\xx)=\begin{cases}f(t,\xx),&\text{if $t\geq s_f\,$;}\\
g(t-s_f,\xx),&\text{if $t<s_f\,$.}\end{cases}
\]
This action preserves the subspaces $\Omega^4X$, $\Map_*(\overline{\CP^2},X)$ and $\Map_*(\#_n\overline{\CP^2},X)$ and clearly  $(g_1\cdot g_2)\cdot f=g_1\cdot(g_2\cdot f)$.
Since $({\mathcal M}_*X)^n\cong {\mathcal M}_*(X^n)$, we get an associative product on $(\Omega^4X)^n$ and a left action
of $(\Omega^4X)^n$ on $({\mathcal M}_*X)^n$.
\item We now define a map $\omega\colon({\mathcal M}_*X)^n\to {\mathcal M}_*X$ by concatenation in the second variable. Given
$\ff=(f_1,\dots,f_n)\in({\mathcal M}_*X)^n$, for each $i=1,\dots,n$ and $x_1\in[(i-1)/n,i/n]$, we let
$\omega(\ff)(t,x_1,x_2,x_3)=f_i(t,nx_1-i+1,x_2,x_3)$.
Then, 
given $\ff\in({\mathcal M}_* X)^n$ and $\mathbf g\in(\Omega^4X)^n$, 
we have $\omega(\mathbf g\cdot\ff)=\omega(\mathbf g)\cdot\omega(\ff)$.
The map $\omega$ restricts
to define maps $(\Omega^4X)^n\to\Omega^4X$ and
$\bigl(\Map_*(\overline{\CP^2},X)\bigr)^n\to\Map_*\bigl(\#_n\overline{\CP^2},X\bigr)$. 
We define a right action of $(\Omega^4X)^n$ on $\Omega^4X$ by $f\cdot\mathbf g=f\cdot\omega(\mathbf g)$.
\end{enumerate}
We represent by
\begin{equation}\label{eq:k->inftybarwithloops}
  \BAR\bigl(\Omega^4X,(\Omega^4X)^n,(\Map_*(\overline{\CP^2},X))^n\bigr)
\end{equation}
the bar construction induced by the actions in (4) and (5)
(see section~\ref{sec:prelim-bar-monoid}).
\end{definition}

Theorem~\ref{theoIII} is a special case of the following theorem, when $X=BSU(\Vv)$:

{
\renewcommand{\theequation}{\ref{theoIII}*}
\begin{theorem}
The maps $\Omega^4X\times\bigl((\Omega^4X)^n\bigr)^k\times\bigl(\Map_*(\overline{\CP^2},X)\bigr)^n\to\Map_*\bigl(\#_n\overline{\CP^2},X\bigr)$ sending
$(g,\gG_1,\dots,\gG_k,\ff)$ to 
$g\cdot\omega(\gG_1\cdot\dotsm\cdot\gG_k\cdot\ff)$
induce a map 
\[
h\colon \BAR\bigl(\Omega^4X,(\Omega^4X)^n,(\Map_*(\overline{\CP^2},X))^n\bigr)\to\Map_*\bigl(\#_n\overline{\CP^2},X\bigr)
\] 
which is a homotopy equivalence.
\end{theorem}
\addtocounter{equation}{-1}
} 
\begin{proof}
We have a commutative diagram:
\[\xymatrix{
\BAR\bigl(\Omega^4X,(\Omega^4X)^n,(\Omega^4X)^n\bigr)\ar[r]^-{h}_-{\simeq}\ar[d]^i&
\Omega^4X\ar[d]^i\\
\BAR\bigl(\Omega^4X,(\Omega^4X)^n,(\Map_*(\overline{\CP^2},X))^n\bigr)\ar[r]^-{h}\ar[d]^-{\boldsymbol\rho\circ h}&
\Map_*(\#_n\overline{\CP^2},X)\ar[d]^-{\boldsymbol\rho}\\
(\Omega^2X)^n\ar@{-}[r]^-{=}&(\Omega^2X)^n
}\]
The maps $i$ are inclusions and the map $\boldsymbol\rho$ is the one defined above in~\ref{def:loopmonoidandaction}\eqref{mapboldsymbolrho}.
The right vertical maps are induced by the cofiber sequence $\bigvee_nS^2\to\#_n\overline{\CP^2}\to S^4$, and hence they form a fibration sequence.
Since the top row is a homotopy equivalence (see Proposition~\ref{prop:B(G,G,Y)=Y}),
to finish the proof we only have to show that the left vertical maps form a fibration sequence.
The map $\boldsymbol\rho\circ h$ is induced by the composition
\[
\Omega^4X\times(\Map_*(\overline{\CP^2},X))^n\xrightarrow{p_2}(\Map_*(\overline{\CP^2},X))^n\xrightarrow{\prod\rho}(\Omega^2X)^n
\]
where $p_2$ is projection onto the second factor and
$\rho\colon\Map_*(\overline{\CP^2},X)\to\Omega^2X$ is the map defined above in~\ref{def:loopmonoidandaction}\eqref{maprho}.
The homotopy fiber $F_{\!\boldsymbol\rho\circ h}$ of the map $\boldsymbol\rho\circ h$ is the fiber product: 
\[
\xymatrix{
F_{\!\boldsymbol\rho\circ h}\ar[r]\ar[d] &
\BAR\bigl(\Omega^4X,(\Omega^4X)^n,(\Map_*(\overline{\CP^2},X))^n\bigr)
\ar[d]^{(\boldsymbol\rho\circ h)\times\star} \\
\bigl((\Omega^2X)^n\bigr)^I\ar[r]^-{\text{ev}_0\times \text{ev}_1} &
(\Omega^2X)^n\times (\Omega^2X)^n}
\]
where, given a path $\gamma$ in $(\Omega^2X)^n$, 
$\text{ev}_0(\gamma)=\gamma(0)$ and $\text{ev}_1(\gamma)=\gamma(1)$,
so from \cite[Corollary~11.6]{May72} 
we see that $F_{\!\boldsymbol\rho\circ h}=\BAR\bigl(\Omega^4X,(\Omega^4X)^n,F_{\!\rho}^{n}\bigr)$,
where $F_{\!\rho}$ is the homotopy
fiber of the map $\rho$. Since the inclusion $\Omega^4X\to F_\rho$ is a homotopy equivalence, and the simplicial spaces 
involved are good, it follows from \cite[Proposition A.1]{Seg74} that the map 
\[
\BAR\bigl(\Omega^4X,(\Omega^4X)^n,(\Omega^4X)^n\bigr)\to\BAR\bigl(\Omega^4X,(\Omega^4X)^n,F_\rho^n\bigr)
\]
is a homotopy equivalence. This concludes the proof.
\end{proof}

We now study the rank-stable limit. Consider the $\Ii$-functors (see \cite[page 16]{May77}):
\[
X^\Vv=\frac{SU(\Vv\oplus \Vv)}{SU(\Vv)\times SU(\Vv)}\,,\qquad P^\Vv=\frac{SU(\Vv\oplus \Vv)}{\Id\times SU(\Vv)}\,,\qquad E^\Vv=P^\Vv\times_{SU(\Vv)}\Vv\,.
\]
For each $\Vv$ the natural map $E^\Vv\to X^\Vv$ is an $SU(\Vv)$ vector bundle and there is a canonical isomorphism of the fiber 
$E^\Vv_*$ over the basepoint with $\Vv$. 
For $\dim\Vv<\infty$
let $\xi^\Vv$ be the canonical universal connection on $E^\Vv$ \cite[section 2]{NaRa61}. Then pullback induces a map of $\Ii_*$-functors
$\psi_\Vv\colon\Map(\#_J\overline{\CP^2},X^\Vv)\to\Cs_{J}^\Vv$
which extends uniquely to a morphism of $\Ii$-functors
by passage to limits.

If $\Hh$ is a universe (Definition~\ref{ex:Ll^H}), then
we can identify the classifying space $BSU$ and the universal bundle $ESU$ with $X^\Hh$ and $P^\Hh$ respectively.
The maps $\psi_\Vv$ induce a map of $\Ll^\Hh$-algebras $\Map(\#_J\overline{\CP^2},X^\Hh)\to\Cs_{J,k}^\Hh$
which is a homotopy equivalence (see \cite[page 101, second paragraph]{AtJo78}). Let $n=\#J$ and consider 
the homotopy coherent bar construction $\Bb\bigl(\Omega^4BSU,(\Omega^4BSU)^n,(\Map_*(\overline{\CP^2},BSU))^n\bigr)$
constructed as in Definition~\ref{def:Delta_I,Bb_I,h_I}. We will now see how this bar construction is related to the one in equation~\eqref{eq:k->inftybarwithloops}.

\begin{proposition}\label{prop:loopcompatiblewithwhitney}
  There is a homotopy commutative diagram:
  \[
  \xymatrix{
    \BAR\bigl(\Omega^4BSU,(\Omega^4BSU)^n,(\Map_*(\overline{\CP^2},BSU))^n\bigr)\ar[r]^-h\ar[d] &
    \Map_*(\#_n\overline{\CP^2},BSU)\\
    \Bb\bigl(\Omega^4BSU,(\Omega^4BSU)^n,(\Map_*(\overline{\CP^2},BSU))^n\bigr)\ar[ru]_(0.6){h_J}
  }
  \]
  where all the maps are homotopy equivalences.
\end{proposition}
\begin{proof}
The idea of the proof is the same as in the proof of Theorem~\ref{thm:hI2simeq}.
Let $D_nX\subset(\Mc_*X)^n$ be the subspace of $n$-tuples $(f_1,\dots,f_n)$ such that, for all $x\in\left(-\infty,0\right]\times I^3$ we have $f_i(x)\neq *$ for at most one $i$.
  We then have a map $\mu\colon D_nX\to\Mc_*X$ with $\mu(f_1,\dots,f_n)(x)=f_i(x)$ if $f_i(x)\neq *$ for some $i$, and $\mu(f_1,\dots,f_n)(x)=*$ otherwise.
  An isometry $\alpha\colon \Vv\to {\Vv'}$ induces a map $X^\Vv\to X^{\Vv'}$ and hence a map $\Mc_*X^\Vv\to\Mc_*X^{\Vv'}$. Let $\imath_i\colon \Vv\to \Vv^n$
  be inclusion into the $i$-th factor. Then the following diagrams are commutative (compare with Proposition~\ref{prop:mapsboxplus}):
  \[
  \xymatrix{
    D_n X^\Vv\ar[r]^-{\prod\imath_i}\ar[rd]_-{\oplus} &
    D_n X^{\Vv^n}\ar[d]^\mu \\ &
    \Mc_*X^{\Vv^n}} \qquad
  \xymatrix{
    D_n X^\Vv\ar[r]^-{\prod\alpha}\ar[d]_\mu &
    D_n X^{\Vv'}\ar[d]^\mu \\
    \Mc_*X^\Vv\ar[r]^-\alpha &
    \Mc_*X^{\Vv'}}
  \]
  where $\prod\imath_i$ and $\prod\alpha$ denote the maps obtained by
  restricting the domain and the codomain of the product maps.
  The analogue of the maps in equation~\eqref{eq:fboxplus} are defined as follows:
  Let $\mathcal P$ be as in the proof of Theorem~\ref{thm:hI2simeq},
  let $\imath_i\colon\Hh\to\Hh^n$ be inclusion into the $i$-th factor and for each $f\in\mathcal P(n)$ define a map $\mu_f$ as the composition
  \[
  \mu_f\colon D_nX^\Hh\xrightarrow{\prod(f\circ\imath_i)}D_nX^{\Hh}\xrightarrow{\ \mu\ }\Mc_*X^\Hh\,.
  \]
  We now express the maps of Definition~\ref{def:loopmonoidandaction}(4)(5) in terms of $\mu$ as follows:
\begin{enumerate}
\item Let $j_i\colon\Mc_*X\to(\Mc_*X)^n$ be inclusion into the $i$-th factor and consider the composition $\omega_i=\omega\circ j_i$.
  For each $i$, the map $\omega_i$ is an embedding and the map $\omega$ is the composition
\[
\omega\colon(\Mc_*X)^n\xrightarrow{\prod \omega_i}D_nX\xrightarrow{\mu}\Mc_*X\,.
\]
\item For each $s\leq 0$ let $\tau_s\colon\Omega^4X\to\Omega^4X$ be defined by
\[
\tau_sg(t,\mathbf x)=\begin{cases}g(t-s,\mathbf x)& t<s; \\ * & t\geq s.\end{cases}
\]
For each $g\in\Mc_*X$ let $s_g$ be as in Definition~\ref{def:loopmonoidandaction}(4).
Given spaces $X_1,\dots,X_{m+1}$ and
$(g_1,\dots,g_m,g_{m+1})\in\bigl(\prod_{i=1}^m\Omega^4X_i\bigr)\times\Mc_*X_{m+1}$, for each $i$ let $r_i=s_{g_{i+1}}+\dots+s_{g_{m+1}}$ and let
$\tau\colon\bigl(\prod\Omega^4X_i\bigr)\times\Mc_*X_{m+1}\to\bigl(\prod\Omega^4X_i\bigr)\times\Mc_*X_{m+1}$ be defined by
\[
\tau(g_1,\dots,g_{m+1})=(\tau_{r_1}g_1,\dots,\tau_{r_m}g_m,g_{m+1})\,.
\]
If $X_1=\dots=X_{m+1}=X$
then $\mu\circ\tau(g_1,\dots,g_{m+1})=g_1\cdot\ldots\cdot g_m\cdot g_{m+1}$.
Note that the map $\tau$ is an embedding with image the subspace of $m+1$-tuples $(g_1,\dots,g_{m+1})\in\prod_i\Mc_*X_i$
such that $\supp g_i\subset\left(-\infty,s_{g_{i+1}}\right]$.
\end{enumerate}
Let $\Oo^J$ and $\Oo_J$ be as in the proof of Theorem~\ref{thm:hI2simeq} and
consider the functor $\Ff_J:\Delta(\Oo_J,\Oo^J)^\op\to \Top$
with $\mathcal F_J(m)$ equal to the image of the following embedding (where $n=\#J$):
\begin{multline*}
  \Omega^4BSU\times\bigl((\Omega^4BSU)^n\bigr)^m\times\bigl(\Map(\overline{\CP^2},BSU)\bigr)^n\\\xrightarrow{\Id\times(\prod\omega_i)^m\times\prod\omega_i}
  \Omega^4BSU\times\bigl((\Omega^4BSU)^n\bigr)^m\times\bigl(\Map(\overline{\CP^2},BSU)\bigr)^n\\
  =\Omega^4BSU\times\bigl(\Omega^4(BSU^n)\bigr)^m\times\Map(\overline{\CP^2},BSU^n)\\\xrightarrow{\tau}
  \Omega^4BSU\times\bigl(\Omega^4(BSU^n)\bigr)^m\times\Map(\overline{\CP^2},BSU^n)
\end{multline*}
and defined on morphisms by using the maps $\mu_f$.
The proof now proceeds exactly as in the proof of Theorem~\ref{thm:hI2simeq}. 
\end{proof}

\section{Homology}\label{sec:8}

In this section we prove Theorem~\ref{theoIV}. We'll always work with a fixed countably infinite dimensional vector space $\Hh$ which we omit from the notation. 
We begin by computing the homology of $\|\Bb_I\|$.
Fix a point $x\in\Cc^2$ and consider the homology rings $R=H_*(\Mm_\emptyset)$ and $M=H_*(\Mm_x)$.
It was shown in \cite[Theorem 1.2]{Kir94}, \cite[Theorem 3]{San95}, \cite[Theorem 1.1]{BrSa97} that $\Mm_{\emptyset}\simeq \coprod_{k\geq 0} BU(k)$ and
$\Mm_x\simeq \coprod_{k\geq 0} BU(k)\times BU(k)$ as $E_\infty$-spaces;
pullback $\pi^*\colon\Mm_\emptyset\to\Mm_x$ is induced by the diagonal inclusions
$\Delta_k\colon BU(k)\to BU(k)\times BU(k)$ (see \cite[equation (2)]{BrSa00}).
Thus we have $R=\bigoplus R_k$ and $M=\bigoplus M_k$ where $R_k\cong H_*\bigl(BU(k)\bigr)$
and $M_k\cong H_*\bigl(BU(k)\times BU(k)\bigr)\cong R_k\otimes R_k$.
If we write
\[
R=\Zz[r_i;i=0,1,2,\dots]
\]
where $r_i\in H_{2i}\bigl(BU(1)\bigr)$  is dual to $c_1^i$ (see \cite[Theorem 21.4.3]{MaPo12}),
then $R_k$ is the $\Zz$-submodule of homogeneous polynomials of degree $k$.
The inclusion $\coprod\bigl(BU(k)\times BU(k)\bigr)\to\bigl(\coprod BU(k)\bigr)\times\bigl(\coprod BU(k)\bigr)$
induces an inclusion
\[
M=\bigoplus\nolimits_k R_k\otimes R_k\subset R\otimes R=\Zz[x_i,y_j;i,j\geq 0]
\]
(where $x_i$ and $y_i$ have homological degree $2i$); namely, $M$ is the subring generated by $x_iy_j$, with $i,j\geq0$.
Consider the bigrading on the $\Zz$-modules $R$ and $M$, where the bidegree $(i,k)$ summand consist of the elements in $R_k$, $M_k$
with homological degree $2i$; in particular:
\begin{equation}\label{eq:bigradinginRkandMk}
  \deg r_i=(i,1)\quad\text{and}\quad \deg x_iy_j=(i+j,1)\,.
\end{equation}
The diagonal inclusions $\Delta_k\colon BU(k)\to BU(k)\times BU(k)$ induce a homomorphism
$\Delta_*\colon R\to M$ of bigraded rings with $\Delta_*(r_k)=\sum_{i+j=k}x_iy_j$, which makes $M$ into
an $R$-module.

\begin{proposition}\label{prop:MfreeoverR}
Let $T_k\subset\Zz[x_i;i\geq k]$ be the $\Zz$-submodule of homogeneous polynomials of degree $k$,
let $N_k=T_k\otimes R_k\subset M_k$ and let $N=\bigoplus_k N_k\subset M$.
Then $M$ is a free module over $R$ with basis $N$.
\end{proposition}
\begin{proof}
Let $\phi\colon R\otimes N\to M$ be the homomorphism of bigraded $\Zz$-modules induced by $\Delta_*$.
We want to show that $\phi$ is an isomorphism. 
We begin by showing that $\phi$ is surjective.
First we need to establish some notation. Let $A$, $B$ be eventually zero sequences
of non-negative integers:
\[
A=(a_0,a_1,\dots,a_n,\dots)\,,\qquad
B=(b_0,b_1\dots,b_n,\dots)\,.
\]
We order these sequences by lexicographic order. Let
\[
x^Ay^B=x_0^{a_0}x_1^{a_1}\dots x_n^{a_n}\dots y_0^{b_0}\dots y_n^{b_n}\ldots\in R\otimes R\,,
\]
and let $|A|=\sum_k a_k$ and $|B|=\sum_k b_k$. Then $x^Ay^B\in M$ if and only if $|A|=|B|$, and these monomials form a $\Zz$-basis for $M$.
We define a total order on the set of monomials $x^Ay^B$ as follows:
\begin{align*}
x^{A_1}y^{B_1}>x^{A_2}y^{B_2}\quad\text{if and only if}\quad&|A_1|>|A_2|,\text{ or }\\
&|A_1|=|A_2|\text{ and }A_1>A_2,\text{ or }\\
&A_1=A_2\text{ and }B_1<B_2
\end{align*}
(notice the reversed order on the $B$'s). 
We will show by induction on the ordering of the monomials 
that any monomial $x^Ay^B\in M$ is in the image of $\phi$. The statement is clearly true for $1\in M$ and
whenever $x^Ay^B\in N_n$ for some $n$, so fix a monomial
$x^{A_1}y^{B_1}\in M_n\setminus N_n$ and assume, by induction hypothesis,
that every monomial $x^{A_2}y^{B_2}$ strictly smaller than $x^{A_1}y^{B_1}$ is in the image of $\phi$.
Since $x^{A_1}y^{B_1}\notin N_n$, there is a $k\leq n$ and a sequence $A=(a_0,a_1,\ldots)$ with $|A|=n-1$ and $a_i=0$ for $i<k-1$
such that $x^{A_1}=x_{k-1}x^{A}$.
Also, since $k\leq n$, we can write $y^{B_1}=y_{n_1}y_{n_2-1}\dots y_{n_k-k+1}y^B$ with
$n_1<n_2<\dots<n_k$ and $B=(b_0,b_1,\ldots)$ a sequence with $|B|=n-k$ and $b_i=0$ for $i<n_k-k+1$.
Set $x_q=y_q=0$ for $q<0$ and consider the following element in $\Zz[x_i,y_j]$:
\[
d=x^Ay^B\det\begin{pmatrix}
y_{n_1}   & \dots & y_{n_k}   \\
y_{n_1-1} & \dots & y_{n_k-1} \\
\vdots  & \ddots & \vdots \\
y_{n_1-k+2} & \dots & y_{n_k-k+2} \\
r_{n_1}   & \dots & r_{n_k}  
\end{pmatrix}
\]
(where, abusing notation, we write $r_i$ instead of $\Delta_*r_i$).
Using Laplace's formula in the last row we find that $d=\sum_{i=0}^kr_{n_i}p_i$ with $p_i\in M_{n-1}$.
In particular, all monomials in each polynomial $p_i$ are strictly smaller than $x^{A_1}y^{B_1}$ so,
by the induction hypothesis, we have $d\in\text{Im}\,\phi$. 
Now using the identities $r_{n_i}=\sum x_\ell y_{n_i-\ell}$ in the last row we find that:
\[
d = \sum_{\ell=0}^\infty x_\ell x^Ay^B\det\begin{pmatrix}
y_{n_1}   & \dots & y_{n_k}   \\
y_{n_1-1} & \dots & y_{n_k-1} \\
\vdots  & \ddots & \vdots \\
y_{n_1-k+2} & \dots & y_{n_k-k+2} \\
y_{n_1-\ell} & \dots & y_{n_k-\ell}
\end{pmatrix}\,.
\]
Notice that the terms in the sum vanish for $\ell=0,\dots,k-2$ and $x_\ell x^A<x_{k-1}x^A$ for $\ell\geq k$.
Also, for $\ell=k-1$, the leading term of the determinant is the product of the entries in the
main diagonal. Thus we can write
\[
d=x_{k-1}x^Ay^By_{n_1}y_{n_2-1}\dots y_{n_k-k+1}+p=x^{A_1}y^{B_1}+p\in\text{Im}\,\phi
\]
where $p$ is a finite sum of monomials strictly smaller than $x^{A_1}y^{B_1}$. Using the induction hypothesis 
we conclude that $x^{A_1}y^{B_1}\in\text{Im}\,\phi$, which finishes the proof of the surjectivity of $\phi$.
Since $R\otimes N$ and $M$ are free $\Zz$-modules, finitely generated in each bidegree, in order to show that
$\phi$ is an isomorphism we only need to check that the dimensions over $\Zz$ match, which we will prove in Lemma~\ref{lemma:H_RH_N=H_M} below.
\end{proof}

Before we continue we introduce some notation. Let $(t;q)_n=\prod_{i=0}^{n-1}(1-tq^i)$ denote the $q$-Pochhammer symbol
and let $(q)_n=(q;q)_n$. Also consider the $q$-multinomial
\[
\binom{k}{i_1,\dots,i_n}_{\!q}=\frac{(q)_k}{\prod_{j=1}^n(q)_{i_j}}\quad(k=i_1+\dots+i_n)
\]
and let $\binom{k}{i}_{\!q}=\binom{k}{i,k-i}_{\!q}$ be the $q$-binomial. One easily checks that
\begin{equation}\label{eq:qbinomialpascalidentities}
  \binom ki_{\!q}=q^i\binom{k-1}i_{\!q}+\binom{k-1}{i-1}_{\!q}\,.
\end{equation}

\begin{lemma}\label{lemma:HilbertSeriesInduction}
For any $n,j\in\Zz$ with $n\geq j\geq 0$ we have
\[
A_{n,j}=\sum_{k=0}^{n-j}\frac{(q)_n}{(q)_{k+j}}\binom{n-j}{k}_{\!\!q}
q^{k^2+kj}=1\,.
\]
\end{lemma}
\begin{proof}
The proof is by induction on $n-j$. The result is clear for $n=j$. Using the relations $\binom nk_{\!q}=\binom{n-1}k_{\!q}+q^{n-k}\binom{n-1}{k-1}_{\!q}$ and 
$(q)_n=(1-q^n)(q)_{n-1}$, we get:
\begin{align*}
A_{n,j}&=\sum_{k=0}^{n-j-1}\!\frac{(q)_n}{(q)_{k+j}}\binom{n-j-1}{k}_{\!\!q}q^{k^2+kj}
+\sum_{k=1}^{n-j}\frac{(q)_n}{(q)_{k+j}}q^{n-j-k}\binom{n-j-1}{k-1}_{\!\!q}q^{k^2+kj}\\
&=(1-q^n)A_{n-1,j}+\sum_{k=0}^{n-j-1}\!\frac{(q)_n}{(q)_{k+j+1}}\binom{n-j-1}{k}_{\!\!q}q^{n-j-k-1}q^{(k+1)^2+(k+1)j}\\
&=(1-q^n)A_{n-1,j}+q^nA_{n,j+1}
\end{align*}
hence $A_{n,j}=1$ by the induction hypothesis.
\end{proof}

Consider the bigrading in 
the $\Zz$-modules $R$, $M$ and $N$ (equation~\eqref{eq:bigradinginRkandMk}).
Write the Hilbert series as $P(q,t)=\sum d_{ik}q^it^k$ where $d_{ik}$ is the rank of the
bidegree $(i,k)$ summand.

\begin{lemma}\label{lemma:H_RH_N=H_M}
The Hilbert series of $R$, $M$ and $N$ are respectively:
\begin{align*}
P_R(q,t)&=\sum_{n=0}^\infty \frac{t^n}{(q)_n}\,; &
P_M(q,t)&=\sum_{n=0}^\infty \frac{t^n}{\bigl((q)_n\bigr)^2}\,; &
P_N(q,t)&=\sum_{n=0}^\infty \frac{q^{n^2}t^n}{\bigl((q)_n\bigr)^2}\,,
\end{align*}
and we have $P_M(q,t)=P_R(q,t)P_N(q,t)$.
\end{lemma}

\begin{remark}
Lemma~\ref{lemma:H_RH_N=H_M} shows that the Hilbert series of $R\otimes N$ equals $P_{R\otimes N}=P_RP_N=P_M$
which concludes the proof of Proposition~\ref{prop:MfreeoverR}.
\end{remark}

\begin{proof}
Since $R_k\cong H^*\bigl(BU(k)\bigr)=\Zz[c_1,\dots,c_k]$ with $\deg c_i=(i,1)$, and $M_k=R_k\otimes R_k$, 
the Hilbert series of $R_k$ and $M_k$ are respectively $1/(q)_n$ and $1/(q)_n^2$ from which the formulas
for $P_R$ and $P_M$ immediately follow. 
The Hilbert series of $\Zz[x_i;i\geq k]\cong\Zz[x_k]\otimes\Zz[x_{k+1}]\otimes\cdots$, where
$x_i$ has bidegree $(i,1)$, is
\[
\prod_{i=k}^\infty\frac1{1-tq^i}=\frac1{(tq^k;q)_\infty}=\sum_{n=0}^\infty\frac{(tq^k)^n}{(q)_n}
\]
(the second equality is a special case of the $q$-binomial theorem: see \cite[Corollary 2.3]{Ber10}).
The Hilbert series of $T_k\subset \Zz[x_i;i\geq k]$ is obtained by taking the coefficient of $t^k$ and hence
it is given by $q^{k^2}/(q)_k$,
so the Hilbert series of $N_k=T_k\otimes R_k$ is $q^{k^2}/(q)_k^2$,
from which we get the formula for $P_N(q,t)$.
Writing $P_M$, $P_R$ and $P_N$ as power series in $t$ and multiplying the series we see
that the identity $P_M=P_RP_N$ is equivalent to the identity
\[
\sum_{k=0}^n\frac{q^{k^2}}{\bigl((q)_k\bigr)^2(q)_{n-k}}=\frac1{\bigl((q)_n\bigr)^2}
\]
which follows immediately from Lemma~\ref{lemma:HilbertSeriesInduction} by taking $j=0$.
\end{proof}

\begin{proposition}\label{prop:H(B_I)}
Given a finite set $I\subset\Cc^2$ with cardinality $\#I=n$, we have
\[
H_*\bigl(\|\Bb_I\|\bigr)\cong\underbrace{M\otimes_R\dots\otimes_RM}_{n}\cong N^{\otimes n}\otimes R\,.
\]
\end{proposition}
\begin{proof}
Consider the equivalence of categories $F\colon \Delta_I\to \Delta$ and let
$F_{h*}\Bb_I\colon\Delta^{\mathrm{op}}\to \Top$ be Segal's homotopy pushdown construction 
(see \cite[section 5]{HoVo92}). Then we
have homotopy equivalences
$\|\Bb_I\|\simeq \hcolim_{\Delta^\op}F_{h*}\Bb_I$ 
\cite[Proposition 5.5]{HoVo92} and 
$\hcolim_{\Delta^\op}F_{h*}\Bb_I\simeq|\tau F_{h*}\Bb_I|$ (Proposition~\ref{prop:hcolim}(1)).
The homology version of Segal's spectral sequence \cite[Proposition 5.1]{Seg68},
converging to the homology
of $|\tau F_{h*}\Bb_I|$, has
as $E^1$ term the simplicial chain complex 
$E^1_{\bullet,q}=H_q(\tau F_{h*}\Bb_I)\cong H_q(F_{h*}\Bb_I)$.
If $1\colon\Delta_I\to\Delta_I$ denotes the identity functor, then we have weak equivalences
$F^*F_{h*}\Bb_I\leftarrow1^*1_{h*}\Bb_I\to\Bb_I$ 
\cite[Proposition 5.3]{HoVo92} so 
the functors $F_{h*}\Bb_I$ and $\Bb_I$ are naturally equivalent
as functors $\Delta^\op\to\hTop$.
It follows that the chain complex
$H_q(F_{h*}\Bb_I)$ is isomorphic to the 
bar complex 
$\BAR(R,R^{\otimes n},M^{\otimes n})$.
By Proposition~\ref{prop:MfreeoverR}, $M^{\otimes n}$ is a free module over $R^{\otimes n}$ 
with basis $N^{\otimes n}$ so
the spectral sequence collapses and we find that 
\[
H_*\bigl(\|\Bb_I\|\bigr)\cong M^{\otimes n}\otimes_{R^{\otimes n}}R\cong  N^{\otimes n}\otimes R\cong M\otimes_R\dots\otimes_RM
\]
which concludes the proof.
\end{proof}

\begin{proposition}\label{prop:simplyconnected}
	The degree $c_2=k$ components $\|\Bb_{I,k}\|$
	of the space $\|\Bb_I\|$ are simply connected.
\end{proposition}
\begin{proof}
	The proof is by induction on $\#I$. For $\#I=0,1$ it immediately follows
	from Proposition~\ref{thm:h0hxsimeq}. For the induction step we use
	Proposition~\ref{prop:IcupJ}. Write $I=J\cup L$ with $\#J,\#L<\#I$
	and $J\cap L=\emptyset$, and let
	\[
	\Bb_{J,L,k}(\bullet)=\Bb_\bullet(\|F_i^*\Bb_J\|,\Mm_\emptyset,\|F_i^*\Bb_L\|)_k\,,
	\quad
	\|\Bb_{J,L,k}\|=\Bb(\|F_i^*\Bb_J\|,\Mm_\emptyset,\|F_i^*\Bb_L\|)_k\,.
	\]
	Then $\|\Bb_{I,k}\|\simeq \|\Bb_{J,L,k}\|$.
	By induction hypothesis, the connected components of $\Bb_{J,L,k}(m)$:
	\[
	\|F_i^*\Bb_{J,\alpha_0}\|\times\Bigl(
	\tprod_{j=1}^m\Mm_{\emptyset,\alpha_j-\alpha_{j-1}}\Bigr)
	\times\|F_i^*\Bb_{L,k-\alpha_m}\|\quad(0\leq\alpha_0\leq\dots\leq\alpha_m\leq k)
	\]
	are simply connected. 
	Let $F_{h*}$ be 
	Segal's homotopy pushdown along the equivalence of categories 
	$F\colon \Delta(\Ll_+,\Ll)\to\Delta$. 
	The simplicial space obtained
	by replacing each connected component in $F_{h*}\Bb_{J,L,k}$
	by a singleton is $\BAR_\bullet(*,\cat_k,*)$, where $\cat_k$ is the
	category $\{0\to1\to\dots\to k\}$.
	A choice of basepoint in the connected component of $F_{h*}\Bb_{J,L,k}(k)$
	corresponding to the composition of arrows $0\to1\to\dots\to k\in \BAR_k(*,\cat_k,*)$
	determines via the face and degeneracy maps, for any $m$,
	a collection $K_0(m)\subset F_{h*}\Bb_{J,L,k}(m)$ of basepoints, one in
	each connected component, inducing a simplicial map 
	$\BAR_\bullet(*,\cat_k,*)\to F_{h*}\Bb_{J,L,k}$.
	Let $\tilde\pi_1\bigl(F_{h*}\Bb_{J,L,k}(m)\bigr)$ denote the fundamental
	groupoid with $K_0(m)$ as the set of objects and with morphisms
	the equivalence classes os paths between basepoints. The diagram of
	spaces $F_{h*}\Bb_{J,L,k}$ induces a diagram of groupoids 
	$\tilde\pi_1\bigl(F_{h*}\Bb_{J,L,k}\bigr)$ and since the 
	connected components
    of $F_{h*}\Bb_{J,L,k}$
	are simply connected, we have an isomorphism of diagrams
	\[
	\tilde\pi_1\bigl(F_{h*}\Bb_{J,L,k}\bigr)\cong
	\tilde\pi_1\bigl(\BAR_\bullet(*,\cat_k,*)\bigr)\,.
	\]
	By \cite[Theorem 1.1]{DFa04} we have natural equivalences of groupoids
	\[
	\tilde\pi_1\bigl(\hcolim_{\Delta^\op} F_{h*}\Bb_{J,L,k}\bigr)
	\cong\hcolim_{\Delta^\op}\tilde\pi_1\bigl(F_{h*}\Bb_{J,L,k}\bigr)
	\]
	and
	\[
	\tilde\pi_1\bigl(\hcolim_{\Delta^\op}\BAR_\bullet(*,\cat_k,*)\bigr)
	\cong \hcolim_{\Delta^\op}\tilde\pi_1\bigl(\BAR_\bullet(*,\cat_k,*)\bigr)\,.
	\]
	Since 
	$\hcolim_{\Delta^\op}\BAR_\bullet(*,\cat_k,*)\bigr)\simeq \BAR(*,\cat_k,*)$
	is contractible,
		\begin{align*}
		\tilde\pi_1\bigl(\|\Bb_{J,L,k}\|\bigr)&\cong
		\tilde\pi_1\bigl(\hcolim_{\Delta^\op} F_{h*}\Bb_{J,L,k}\bigr)
		\cong\hcolim_{\Delta^\op}\tilde\pi_1\bigl(F_{h*}\Bb_{J,L,k}
		\bigr)\\
		&\cong \hcolim_{\Delta^\op}\tilde\pi_1\bigl(\BAR_\bullet(*,\cat_k,*)\bigr)
		\cong \tilde\pi_1\bigl(\hcolim_{\Delta^\op}\BAR_\bullet(*,\cat_k,*)\bigr)\cong\{1\}
		\end{align*}
	which finishes the proof.
\end{proof}

As in Lemma~\ref{lemma:H_RH_N=H_M},
we write the Hilbert series of $H_*\bigl(\|\Bb_I\|\bigr)$ as $P_I(q,t)=\sum d_{ik}q^it^k$ where $d_{ik}$ is the rank of the
bidegree $(i,k)$ summand. Also, given power series $s_1, s_2\in\Zz[[q]]$ we write $s_1=s_2+\mathcal O(q^m)$
if $s_1=s_2$ in $\Zz[[q]]/\langle q^{m}\rangle$. 

\begin{lemma}\label{lemma:P_Ipnk}
Let $I\subset\Cc^2$  be a finite set with cardinality $n$. Then the Hilbert series of $H_*\bigl(\|\Bb_I\|\bigr)$ is
given by the $q$-series:
\begin{equation}\label{eq:P_I(q,t)}
P_I(q,t)=\sum_{k=0}^\infty \frac{p_{n,k}(q)}{(q)_k^2}t^k\,,\quad\text{where}\quad
p_{n,k}=\!\sum_{i_1+\dots+i_n=k}\!q^{i_2^2+\dots+i_n^2}\binom{k}{i_1,\dots,i_n}_{\!q}^2\,.
\end{equation}
The polynomials $p_{n,k}$ satisfy:
\begin{enumerate}
\item $p_{1,k}=1$, $p_{2,k}=\binom{2k}{k}_{q}$, $p_{n,1}=1+(n-1)q$ and
\[
p_{n,2}=1+(n-1)q+\tfrac12(n-1)(n+2)q^2+(n-1)^2q^3+\tfrac12n(n-1)q^4\,;
\]
\item For any $n$, $k$ we have $(q)_\infty^{n-1}p_{n,k}=1+\mathcal O(q^{k+1})$;
\item For $2\leq n\leq k+1$ we have 
\[
(q)_\infty^{n-1}p_{n,k}=1-(2^n-2)\sum_{i=k+1}^{2k-n+2}q^i-(2^n-3)q^{2k-n+3}+\mathcal O(q^{2k-n+4})\,.
\]
\end{enumerate}
\end{lemma}
\begin{proof}
By Proposition~\ref{prop:H(B_I)},
the homology of $\|\Bb_I\|$ is isomorphic to the tensor product
$M\otimes N^{\otimes(n-1)}$ so its Hilbert series is given by
\[
P_I(q,t)=\left(\sum_{k=0}^\infty\frac{t^k}{(q)_k^2}\right)\left(\sum_{k=0}^\infty\frac{q^{k^2}t^k}{(q)_k^2}\right)^{n-1}
=\sum_{k=0}^\infty \frac{p_{n,k}}{(q)_k^2}t^k
\]
with $p_{n,k}$ as in equation~\eqref{eq:P_I(q,t)}.
The cases $k=1,2$ and $n=1$ are easily computed, while the case $n=2$ is a special case of the $q$-Vandermonde identities \cite[Solution to exercise 1.100, page 188]{Sta12}.

We prove statement (2) by induction on $n$.
Statement (2) holds for $n=1,2$ since, by (1), we have $(q)_\infty p_{2,k}=(q^{k+1})_\infty(q^{k+1})_k$.
Assume (2) holds for $n-1$. 
The polynomials $p_{n,k}$ satisfy the recurrence relation:
\[
p_{n,k}=\sum_{i=0}^kq^{i^2}\binom{k}{i}_{\!q}^2p_{n-1,k-i}
\]
so we find that
\[
(q)_\infty^{n-2}p_{n,k}=\sum_{a=0}^kq^{a^2}\binom{k}{a}_{\!q}^2(q)_\infty^{n-2}p_{n-1,k-a}
=\sum_{a=0}^kq^{a^2}\binom{k}{a}_{\!q}^2\bigl(1+\mathcal O(q^{k-a+1})\bigr)\,.
\]
Since $q^{a^2}q^{k-a+1}=\mathcal O(q^{k+1})$, we get:
\[
(q)_\infty^{n-1}p_{n,k}=(q)_\infty\sum_{a=0}^kq^{a^2}\binom{k}{a}_{\!q}^2+\mathcal O(q^{k+1})
=(q)_\infty p_{2,k}+\mathcal O(q^{k+1})=1+\mathcal O(q^{k+1})
\]
which finishes the proof of (2).

We now prove (3) by induction on $n$.
For $n=2$ we have $(q)_\infty p_{2,k}=(q^{k+1})_\infty(q^{k+1})_k$ and the result follows.
Assume (3) holds for $n-1$. Then
\begin{align*}
(q)_\infty^{n-2}p_{n,k}&=\sum_{a=0}^kq^{a^2}\binom{k}{a}_{\!q}^2(q)_\infty^{n-2}p_{n-1,k-a}\\
&=\sum_{a=0}^kq^{a^2}\binom{k}{a}_{\!q}^2\left(1-(2^{n-1}-2)\sum_{i=k-a+1}^{2k-2a-n+3}q^i+\mathcal O(q^{2k-2a-n+4})\right)\,.
\end{align*}
For $a\neq 1$ we have $a^2+2k-2a-n+4=(a-1)^2+2k-n+3\geq 2k-n+4$, so:
\[
q^{a^2}\sum_{i=k-a+1}^{2k-2a-n+3}q^i
=q^{a^2-a}\sum_{i=k+1}^{2k-a-n+3}q^i
=q^{a^2-a}\sum_{i=k+1}^{2k-n+3}q^i+\mathcal O(q^{2k-n+4})
\]
while for $a=1$ we have, by the induction hypothesis,
\[
q(q)_\infty^{n-2}p_{n-1,k-1}=q-(2^{n-1}-2)\sum_{i=k+1}^{2k-n+3}q^i
+q^{2k-n+3}+\mathcal O(q^{2k-n+4})\,.
\]
Thus we have: 
\begin{multline*}
(q)_\infty^{n-2}p_{n,k}=
\sum_{a=0}^kq^{a^2}\binom{k}{a}_{\!q}^2
-(2^{n-1}-2)\sum_{i=k+1}^{2k-n+3}q^i
\sum_{a=0}^kq^{a^2-a}\binom{k}{a}_{\!q}^2\\
+q^{2k-n+3}+\mathcal O(q^{2k-n+4})\,.
\end{multline*}
Now, using equation~\eqref{eq:qbinomialpascalidentities} and the $q$-Vandermonde identities we get:
\[
\sum_{a=0}^kq^{a^2-a}\binom{k}{a}_{\!q}^2 
=\sum_{a=0}^{k-1}q^{a^2}\binom{k}{a}_{\!q}\binom{k-1}{a}_{\!q}+\sum_{a=1}^kq^{a^2-a}\binom{k}{a}_{\!q}\binom{k-1}{a-1}_{\!q}=2\binom{2k-1}{k}_{\!q}
\]
and substituting above we get
\[
(q)_\infty^{n-2}p_{n,k}=p_{2,k}-(2^{n}-4)\binom{2k-1}{k}_{\!q}\sum_{i=k+1}^{2k-n+3}q^i+q^{2k-n+3}+\mathcal O(q^{2k-n+4})\,.
\]
Multiplying by $(q)_\infty$ and observing that $(q)_\infty\binom{2k-1}{k}_{q}=(q^{k})_k(q^{k+1})_\infty=1+\mathcal O(q^{k})$
we get, since $n\geq 3$:
\[
(q)_\infty^{n-1}p_{n,k}=(q)_\infty p_{2,k}-(2^n-4)\sum_{i=k+1}^{2k-n+3}q^i+q^{2k-n+3}+\mathcal O(q^{2k-n+4})\,.
\]
The result now easily follows.
\end{proof}

Let $I\subset\Cc^2$ be a finite set with cardinality $\#I=n$
and let $\mathscr C_{I,k}=\mathscr C^{\mathbb H}_{I,k}$ denote the rank-stable limit 
of the space 
$\mathscr C^{\mathbb V}_{I,k}$ 
of equivalence classes of framed connections on an $SU(\mathbb V)$ vector bundle over $\#_n\overline{\CP^2}$ with $c_2=k$ (see section~\ref{sec:7}).
The space $\mathscr C_{I,k}$ is
homotopically equivalent to $BU^{n+1}$ (see \cite[Remark 1.1]{BrSa97}).	
We are ready to prove Theorem~\ref{theoIV}, which is a consequence of the following theorem:

{
\renewcommand{\theequation}{\ref{theoIV}*}
\begin{theorem}\label{theor:upperboundstocokernel}
	Let $\imath_k\colon\Mm_{I,k}\cong\MI_{I,k}\to\mathscr C_{I,k}$ 
	be the 
	natural inclusion map and let $h_{I,k}\colon\|\Bb_{I,k}\|\to\Mm_{I,k}$ be 
	the $c_2=k$ components of the map $h_I$ in Definition~\ref{def:Delta_I,Bb_I,h_I}:
	\[
	\|\Bb_{I,k}\|\xrightarrow{h_{I,k}}\Mm_{I,k}\cong\MI_{I,k}
	\xrightarrow{\imath_k}\mathscr C_{I,k}\,.
	\]
  \begin{enumerate}
    \item The map $\imath_k$ induces surjective homomorphisms
    in homology and homotopy up to degree $2k+1$.
  \item The map $h_{I,k}$ is injective in homology.
  \item The map $h_{I,k}$ is an isomorphism in homology in degree $i$ provided $i\leq 2k+1$ and the map
   $(\imath_k)_*\colon H_i(\MI_{I,k})\to H_i(\mathscr C_{I,k})$
   is an isomorphism.
   \item The map $h_{I,k}$ is an isomorphism in homotopy in degree $i$ provided $i\leq 2k$ and the map
   $(\imath_k)_*\colon \pi_i(\MI_{I,k})\to \pi_i(\mathscr C_{I,k})$
   is an isomorphism. 
  \item Let $b_{i}$ be the Betti numbers of $BU^{n+1}$,
let $c_{2i}$ be the coefficients of the power series:
\[
\sum_{i=0}^\infty c_{2i}q^i=\left(\sum_{i=0}^\infty b_{2i}q^i\right)
\Bigl(2^n(q^{k+1}+\dots+q^{2k-n+2})
+(2^n-1)q^{2k-n+3}\Bigr)\,,
\]
and set $c_i=0$ for $i$ odd.
Then the rank of the cokernel of the map $(\imath_k)_*\colon H_i(\MI_{I,k})\to H_i(\mathscr C_{I,k})$ is less than or equal to $c_{i}$ for
$i\leq 4k-2n+7$, provided $2\leq n\leq k+1$.
\end{enumerate}
\end{theorem}
\addtocounter{equation}{-1}
}

\begin{proof}
By Proposition~\ref{prop:inclusionismapofLalgebras} the inclusions $\Mm_I\to\mathscr C_I$ are maps
of $\Ll$-algebras which induce a map of the bar constructions
\[
\textstyle
\|\Bb_I\|=\Bb\bigl(\prod\Mm_x,\prod\Mm_\emptyset,\Mm_\emptyset\bigr)\to\Bb\bigl(\prod \Cs_x,\prod\Cs_\emptyset,\Cs_\emptyset\bigr)
\]
and we get a commutative diagram, where the right vertical map is a homotopy equivalence (Proposition~\ref{prop:loopcompatiblewithwhitney}):
\begin{equation}\label{eq:r=inftyk=infty}
\xymatrix{
  \|\Bb_I\|\ar[r]\ar[d]_{h_I}&
  \Bb\bigl(\prod \mathscr C_x,\prod\mathscr C_\emptyset,\mathscr C_\emptyset\bigr)\ar[d]\\
\Mm_I\ar[r] & \mathscr C_I
}\end{equation}
To prove (2) it is enough to show that 
the top horizontal map is injective in homology, which we now prove.
Let $\overline R=H_*(\mathscr C_\emptyset)\cong H_*(BU\times\Zz)$ and $\overline M=H_*(\mathscr C_x)\cong H_*(BU\times BU\times\Zz)$,
and let $r_0$ be a generator of $H_0(\{1\}\times BU)$. Then
$\overline R$ and $\overline M$ are respectively the localizations of $R$ and $M$ away from $r_0$:
$\overline R=R[r_0^{-1}]$ and $\overline M=M[r_0^{-1}]$. 
Also $\overline M$ is free over $\overline R$ and:
\[
\textstyle
H_*\bigl(\Bb(\prod \mathscr C_x,\prod\mathscr C_\emptyset,\mathscr C_\emptyset)\bigr)\cong \overline M\otimes_{\overline R}\dots\otimes_{\overline R}\overline M
\cong (M\otimes_R\dots\otimes_RM)[r_0^{-1}].
\]
Since $M\otimes_R\dots\otimes_RM$ is $r_0$-torsion free, the map
$M\otimes_R\dots\otimes_RM\to(M\otimes_R\dots\otimes_RM)[r_0^{-1}]$ is injective as claimed. 
It follows that the map  $h_I$ is injective in homology.

We will now compute the dimension of the cokernel of the 
map $H_*(\|\Bb_{I,k}\|)\to H_*(\mathscr C_{I,k})$ induced by 
the diagonal map in diagram~\eqref{eq:r=inftyk=infty}.
The ring $H_*(\mathscr C_{I,k})$ is graded by half the homology degree with 
Hilbert series
$\sum_i b_{2i}q^i=\sum_i1/(q)_\infty^{n+1}$ and the Hilbert series of $H_*(\|\Bb_{I,k}\|)$ was computed in  Lemma~\ref{lemma:P_Ipnk}.
Subtracting the coefficients of the Hilbert series 
we obtain:
\[
\frac{1}{(q)_\infty^{n+1}}-\frac{p_{n,k}}{(q)_k^2}
=\frac{1-(q^{k+1})_\infty^2(q)_\infty^{n-1}p_{n,k}}{(q)_\infty^{n+1}}
=\left(\sum_{i=0}^\infty b_{2i}q^i\right)
\Bigl(1-(q^{k+1})_\infty^2(q)_\infty^{n-1}p_{n,k}\Bigr)\,.
\]
Statement (5) follows from Lemma~\ref{lemma:P_Ipnk}(3),
observing that $(q^{k+1})_\infty^2=1-2(q^{k+1}+\dots+q^{2k+1})+\mathcal O(q^{2k+2})$.
Furthermore,
by Lemma~\ref{lemma:P_Ipnk}(2), the map 
$H_*(\|\Bb_I\|)\to H_*(\mathscr C_I)$ is an isomorphism
in degrees up to $2k+1$, proving (3). Since
$\mathscr C_{I,k}\simeq BU^{n+1}$  and $\|\Bb_{I,k}\|$  are simply connected
(Proposition~\ref{prop:simplyconnected}),
by the Hurewicz Theorem the map $\pi_*(\|\Bb_{I,k}\|)\to \pi_*(\mathscr C_{I,k})$ 
is an isomorphism in degrees up to $2k$ and is surjective in degree $2k+1$.
Statements (1) and (4) immediately follow. 
\end{proof}

\appendix

\section{Bar construction}\label{app:bar}

In this appendix we prove propositions~\ref{prop3.1}
and~\ref{prop:IcupJ}.
First we need to prove some lemmas. Let $\widetilde\Delta$ be as in Definition~\ref{def:widetildeDelta}.
The identity $M\times A^n\times A=M\times A^{n+1}$  leads us
to make the following definition:

\begin{definition}\label{def:widehatdelta}
We represent by $\overline\Delta$ and $\widehat\Delta$
the categories whose objects are the same as the objects of $\widetilde\Delta$
(with $\emptyset=[-1]$) 
and such that $\overline\Delta(m,n)\subset\Delta(m+1,n+1)$
is the set of order preserving maps $\overline\mu\colon\{0,\dots,m+1\}\to\{0,\dots,n+1\}$
with $\overline\mu(m+1)=n+1$, and $\widehat\Delta(m,n)$
is the set of order preserving maps $\widehat\mu\colon\{-1,0,\dots,m+1\}\to\{-1,0,\dots,n+1\}$
with $\widehat\mu(m+1)=n+1$ and $\widehat\mu(-1)=-1$.
\end{definition}

\begin{remark}
The categories $\Delta$ and $\widetilde\Delta$ are subcategories 
of $\overline\Delta$ since we can extend any function $[m]\to[n]$ 
canonically to a morphism
$[m+1]\to[n+1]$ by sending $m+1$ to $n+1$. In a similar way, $\overline\Delta$ is a subcategory of $\widehat\Delta$.
Also observe that the object $[-1]$ is a final object 
of both $\overline\Delta{}^\op$ and $\widehat\Delta{}^\op$ and, 
for $m\neq-1$,
restriction gives an isomorphism of sets
$\overline\Delta(m,n)\cong\Delta(m,n+1)$.
\end{remark}

\out{
Given a functor $F\colon\cat\to\mathfrak D$ and an object 
$D\in\mathfrak D$, the undercategory
$D\downarrow F$ is the category whose objects are the pairs $(f,C)$
with $C\in\cat$ and $f\colon D\to F(C)$ and whose morphisms
$p\colon (f,C)\to(f',C')$ are the morphisms $p\colon C\to C'$
in $\cat$ such that $F(p)\circ f=f'$. The overcategory
$F\downarrow D$
is the category whose objects are the pairs $(C,f)$
with $f\colon F(C)\to D$ and whose morphisms
$p\colon (f,C)\to(f',C')$ are the morphisms $p\colon C\to C'$
in $\cat$ such that $f'\circ F(p)=f$. We say
the functor $F$ is \emph{right cofinal} if
$B(*,D\downarrow F,*)\cong *$ for all $D\in\mathfrak D$.
}

\begin{lemma}\label{lemma:widehatdelta}
The inclusion functors $F:\Delta^{\mathrm{op}}\to\overline\Delta{}^{\mathrm{op}}$,
$\widetilde F:\widetilde\Delta^{\mathrm{op}}\to\overline\Delta{}^{\mathrm{op}}$
and $\widehat F:\overline\Delta{}^{\mathrm{op}}\to\widehat\Delta{}^{\mathrm{op}}$ 
are right cofinal.
\end{lemma}

For the definition of right cofinal see equation~\eqref{eq:cofinal}.

\begin{proof}
We first show that $\widetilde F$ is right cofinal. 
It is enough to show that, for any $[n]\in\overline\Delta$, the overcategory
$\widetilde F\downarrow [n]$ has a final object. This object is the pair
$([n+1],\imath)$, where $[n+1]\in\widetilde\Delta$ and
$\imath\colon F\bigl([n+1]\bigr)=[n+1]\to[n]$ corresponds to the identity map under the isomorphism
$\overline\Delta(n+1,n)\cong\Delta(n+1,n+1)=\widetilde\Delta(n+1,n+1)$.
The exact same proof shows that $F$ is right cofinal since
the pair $([n+1],\imath)$ is also an object in $F\downarrow [n]$.

The proof for $\widehat F$ is similar.
Let $f\in\widehat\Delta(n+1,n)$ be the map $f\colon\{-1,\dots,n+2\}\to\{-1,\dots,n+1\}$ with $f(-1)=-1$ and
$f(i)=i-1$ for $i=0,\dots,n+1$. We will show that $([n+1],f)$ 
is a final object of $\widehat F\downarrow [n]$:
given a map $\widehat\mu\colon\{-1,\dots,m+1\}\to\{-1,\dots,n+1\}$ there is a unique map
$\overline\mu\colon\{0,\dots,m+1\}\to\{0,\dots,n+2\}$ such that $\widehat\mu=f\circ\widehat F(\overline\mu)$, namely: 
$\overline\mu(i)=\widehat\mu(i)+1$. This completes the proof.
\end{proof}

Let $\Oo$ be an $E_\infty$ operad with composition data $\gamma$ and let $\Ms$ be an $E_\infty$ monoidal module over $\Oo$ with composition data $\Gamma_R$ (see Definition~\ref{def:monoidalmodule}).

\begin{definition}
We denote by $\overline\Delta(\Ms,\Oo)$ the topological category equivalent to $\overline\Delta$
whose morphisms are defined as follows: for each $\mu\in\Delta(m+1,n+1)$ let
\[
\overline\Delta_{\mu}(\Ms,\Oo)=\Ms(\mu_0)\times
\prod_{\alpha=1}^{m+1}\Oo(\mu_\alpha-\mu_{\alpha-1})
\]
and define \[\overline\Delta(\Ms,\Oo)(m,n)=\coprod\limits_{\mu\in\overline\Delta(m,n)}\overline\Delta_{\mu}(\Ms,\Oo)\,.\] 
Given $\mu\in\overline\Delta(m,n)$
and $\nu\in\overline\Delta(n,p)$, composition
$\overline\Delta_\mu(\Ms,\Oo)\times\overline\Delta_\nu(\Ms,\Oo)\to
\overline\Delta_{\nu\circ\mu}(\Ms,\Oo)$ is defined using the operad data:
  \begin{align*}
  \Ms(\mu_0)\times\Ms(\nu_0)\times\prod_{\beta=1}^{\mu_0}\Oo(\nu_\beta-\nu_{\beta-1})
  &\xrightarrow{\Gamma_R}\Ms(\nu_{\mu_0})\,;
  \\
  \Oo(\mu_\alpha-\mu_{\alpha-1})\times\prod_{\beta=\mu_{\alpha-1}+1}^{\mu_{\alpha}}\Oo(\nu_{\beta}-\nu_{\beta-1})
  &\xrightarrow{\gamma}
  \Oo(\nu_{\mu_\alpha}-\nu_{\mu_{\alpha-1}})\quad(\alpha=1,\dots,m+1).
  \end{align*}
\end{definition}

Notice that, for $\mu_m\leq n$, we have
$\Oo_+(n-\mu_m)=\Oo(\mu_{m+1}-\mu_m)$ hence 
$\Delta(\Ms,\Oo)$ and $\widetilde\Delta(\Ms,\Oo)$ (see Definitions~\ref{con:catDelta(P)} and~\ref{con:widetildedelta(P)})
are canonically subcategories of $\overline\Delta(\Ms,\Oo)$.

\begin{lemma}\label{lemma:overlineB(M,A)}
The functor $\widetilde\Bb_\bullet(M,A,A;M)$ from Definition~\ref{con:widetildedelta(P)}(2)
can be extended to a functor
$\overline\Bb_\bullet(M,A):\overline\Delta(\Ms,\Oo)^\op\to \Top$.
\end{lemma}
\begin{proof}
Given $\mu\in\Delta(m+1,n+1)$ with
$\mu_{m+1}=n+1$, the required maps
\[
\overline\Delta_{\mu}(\Ms,\Oo)\times M\times A^{n+1}
=\overline\Delta_{\mu}(\Ms,\Oo)\times M\times A^{\mu_0}\times\prod_{\alpha=1}^{m+1}A^{\mu_\alpha-\mu_{\alpha-1}}
\to M\times A^{m+1}
\]
are induced by the maps
\begin{align*}
&\Ms(\mu_0)\times M\times A^{\mu_0}\to M\,;\\
&\Oo(\mu_\alpha-\mu_{\alpha-1})\times A^{\mu_\alpha-\mu_{\alpha-1}}\to A\,.\qedhere
\end{align*}
\end{proof}

Let $\overline\Bb(M,A)$ be the homotopy colimit of the 
functor $\overline\Bb_\bullet(M,A)$ in Lemma~\ref{lemma:overlineB(M,A)}.
Proposition~\ref{prop3.1} is now a direct consequence of the following lemma:

\begin{lemma}\label{theo3.1}
We have a commutative diagram
\[
\xymatrix{
\Bb(M,A,A)\ar[d]\ar[r]&
\overline\Bb(M,A)\\
\widetilde\Bb(M,A,A;M)\ar[ru]&
M\ar[l]_-{[-1]}\ar[u]_-{[-1]}}
\]
where every map is a homotopy equivalence.
\end{lemma}
\begin{proof}
	The bottom horizontal map is a homotopy equivalence
	by Proposition~\ref{prop:B->Y} so
it is enough to show that
the inclusion functor $\Delta(\Ms,\Oo)^\op\to\overline\Delta(\Ms,\Oo)^\op$ 
and the functor 
$\star\to\overline\Delta(\Ms,\Oo)^\op$
which sends $\star$ to $[-1]$ are right cofinal. 
This follows from the pullback squares
\[
\xymatrix{
\Delta(\Ms,\Oo)\ar[r]\ar[d]&
\Delta\ar[d]\\
\overline\Delta(\Ms,\Oo)\ar[r]&
\overline\Delta}\qquad
\xymatrix{
\star\ar@{-}[r]^-{=}\ar[d]&
\star\ar[d]\\
\overline\Delta(\Ms,\Oo)\ar[r]&
\overline\Delta}
\]
since the horizontal maps are equivalences of categories and the right vertical maps are cofinal.
\end{proof}

We now turn to the proof of Proposition~\ref{prop:IcupJ}. Recall Definitions~\ref{def:Ii} and~\ref{def:LlI}.
Given disjoint finite sets $I,J\subset\Cc^2$, a universe $\Hh$ and integers $a,b,c$ with $a\geq1$ and $b,c\geq0$ let 
\[
\Ll_{IJ}^\Hh(a,b,c)=\Ii\bigl(\Hh^a\oplus(\Hh^I)^b\oplus(\Hh^J)^c,\Hh\bigr)\,.
\]
By Proposition~\ref{prop:MisI-functor}, $\Mm_\emptyset$ is an $\Ii$-functor so
we have maps
\begin{multline}\label{eq:LIJ(a,b,c)acting}
\Ll_{IJ}^\Hh(a,b,c)
\times(\Mm_\emptyset^\Hh)^a
\times\bigl(\tprod_I\Mm_\emptyset^\Hh\bigr)^b
\times\bigl(\tprod_J\Mm_\emptyset^\Hh)^c\\
\to\Ll_{IJ}^\Hh(a,b,c)
\times\Mm_\emptyset^{\Hh^a\oplus(\Hh^I)^b\oplus(\Hh^J)^c}
\to\Mm_\emptyset^\Hh\,.
\end{multline}
Note that $\Ll_{IJ}^\Hh(1,0,c)=\Ll_J^\Hh(c)$
and $\Ll_{IJ}^\Hh(1,b,0)=\Ll_I^\Hh(b)$.
We also have an isomorphism $\Ll_{IJ}^\Hh(1,a,a)\cong\Ll_{I\cup J}^\Hh(a)$ induced by the permutation
\[
\tau\colon\Hh\oplus(\Hh^I)^a\oplus(\Hh^J)^a\to\Hh\oplus(\Hh^{I\cup J})^a\,.
\]
The permutation $\tau$ also induces a map
\[\tau\colon\Mm^\Hh_\emptyset
\times(\tprod_I\Mm^\Hh_\emptyset)^a
\times(\tprod_J\Mm^\Hh_\emptyset)^a\to
\Mm^\Hh_\emptyset\times(\tprod_{I\cup J}\Mm^\Hh_\emptyset)^a\]
and we have a commutative diagram:
\begin{equation}\label{eq:LIJdiagnoal=LIcupJ}
  \xymatrix{
    \Ll^\Hh_{IJ}(1,a,a)
    \times\Mm^\Hh_\emptyset
    \times(\prod_I\Mm^\Hh_\emptyset)^a
    \times(\prod_J\Mm^\Hh_\emptyset)^a\ar[r]\ar[d]^{\tau\times\tau}&
    \Mm^\Hh_\emptyset\ar@{=}[d] \\
    \Ll^\Hh_{I\cup J}(a)\times \Mm^\Hh_\emptyset\times(\prod_{I\cup J}\Mm^\Hh_\emptyset)^a\ar[r] &
    \Mm^\Hh_\emptyset }
\end{equation}

{
\renewcommand{\theequation}{\ref{prop:IcupJ}}
\begin{proposition}\label{prop:IcupJinappendix}\addtocounter{equation}{-1}
Let $I,J\subset\Cc^2$ be finite disjoint sets.
Fix universes $\Hh_{IJ},\Hh_\emptyset$ and let $\Hh=\Hh_\emptyset\otimes\Hh_{IJ}$. Let
$i:\Ll^{\Hh_{IJ}}\to\Ll^\Hh$ be the canonical operad map. Then, in $\hTop$,
we have a commutative diagram:
\renewcommand{\theequation}{\mbox{\Alph{section}.\arabic{equation}}}
\begin{equation}\label{eq:barbarbardiagramwithwhitney}
\xymatrix{
\Bb^{\Hh_\emptyset}(\|F_i^*\Bb_I^\Hh\|,\Mm_\emptyset^\Hh,\|F_i^*\Bb_J^\Hh\|)\ar[d]\ar[r]^-{\simeq}&
\|\Bb_{I\cup J}^\Hh\|\ar[d]^-{h_{I\cup J}}\\
\Bb^\Hh(\Mm_I^\Hh,\Mm_\emptyset^\Hh,\Mm_J^\Hh)\ar[r]&
\Mm_{I\cup J}^\Hh}
\end{equation}
where the top horizontal map is a homotopy equivalence and
the left vertical map is induced by $h_I$ and $h_J$.
\end{proposition}
}

\begin{proof}
  As in the proof of Proposition~\ref{prop:barbarbardiagramcommutesinfiniterank}, the strategy is to compare the bar construction with a trisimplicial version and take the diagonal.
  In the proof of~\ref{prop:barbarbardiagramcommutesinfiniterank} we only needed two functors: $F_0$ and $F_1$. Here, however, we need 3 more functors to define the
  vertical maps in diagram~\eqref{eq:barbarbardiagramwithwhitney} (see section~\ref{sec:mapsfromthebarconstruction}).
  Recall the functors in equation~\eqref{eq:defofB(B,M,B)}.
  We will define diagrams of spaces (see section~\ref{sec:diagramsofspaces}) 
$\mathfrak F_0$, $\mathfrak F_1$, $\widehat{\mathfrak F}_1$, $\mathfrak F_2$, $\mathfrak F_3$
indexed by certain categories $C_0$, $C_1$, $\widehat C_1$, $C_2$, $C_3$ 
in such a way that we get a commutative diagram in $\hTop$ (compare with diagram~\ref{eq:diagram-barbarbar}):
\[
\xymatrix@C=1em{
\Bb^{\Hh_\emptyset}(\|F_i^*\Bb^\Hh_I\|,\Mm^\Hh_\emptyset,\|F_i^*\Bb^\Hh_J\|)\ar[dd]\ar[rr]^-{g_0}_-{\simeq}&&
\|\mathfrak F_0\|\ar[dd]\ar[r]^-{\simeq}\ar[rd]&
\|\widehat\FF_1\|&
\|\Bb_{I\cup J}^\Hh\|\ar[l]_-{d}^-{\simeq}\ar[dl]\ar[dd]
\\
&&&\|\mathfrak F_1\|\ar[u]_-{\simeq}\ar[d]
\\
\Bb^{\Hh_\emptyset}(\|F_i^*\widetilde\Bb^\Hh_I\|,\Mm^\Hh_\emptyset,\|F_i^*\widetilde\Bb^\Hh_J\|)\ar[rr]^-{g_2}_-{\simeq}&&
\|\mathfrak F_2\|\ar[r]&
\|\mathfrak F_3\|&
\|\widetilde\Bb^\Hh_{I\cup J}\|\ar[l]_-{d}^-{\simeq}
\\
\Bb^\Hh(\Mm^\Hh_I,\Mm^\Hh_\emptyset,\Mm^\Hh_J)\ar[u]^-{\simeq}\ar[rrr]\ar[rru]&&&
\widetilde\Bb^\Hh(\Mm^\Hh_I,\Mm^\Hh_\emptyset,\Mm^\Hh_J;\Mm^\Hh_{I\cup J})\ar[u]_-{\simeq}&
\Mm^\Hh_{I\cup J}\ar[u]_-{\simeq}\ar[l]_-{\simeq}}
\]
The result then immediately follows.

We begin by defining 
a category $C$ topologically equivalent to $\widetilde\Delta\times\widehat\Delta\times\widetilde\Delta$ (see Definition~\ref{def:widehatdelta}).
The categories $C_0$, $C_1$, $\widehat C_1$, $C_2$, $C_3$ 
will be subcategories of $C$. The objects of $C$
are the triples $(m^I,m,m^J)$ of integers with $m^I,m,m^J\geq-1$. To define the morphisms recall the notation of
 Definition~\ref{def:LlI}.
Given morphisms $\mu^I\in\widetilde\Delta(m^I,n^I)$, $\mu\in\widehat\Delta(m,n)$ and $\mu^J\in\widetilde\Delta(m^J,n^J)$, 
with $m\neq -1$, let
\begin{multline*}
C(\mu^I,\mu,\mu^J)
=\left(\prod_{\alpha=1}^{m^I+1}\Ll^{\Hh,I}(\mu^I_{\alpha}-\mu^I_{\alpha-1})\right)\\
\times\Ll_{IJ}^\Hh(1+\mu_0,\mu^I_{0},0)
\times\left(\prod_{\alpha=1}^m\Ll^\Hh(\mu_\alpha-\mu_{\alpha-1})\right)\\
\times\Ll_{IJ}^\Hh(1+n-\mu_m,0,\mu^J_{0})
\times\left(\prod_{\alpha=1}^{m^J+1}\Ll^{\Hh,J}(\mu^J_{\alpha}-\mu^J_{\alpha-1})\right),
\end{multline*}
(with the convention that $\mu^I_{m^I+1}=n^I+1$ and $\mu^J_{m^J+1}=n^J+1$).
For $m=-1$ let
\begin{multline*}
C(\mu_I,\imath,\mu_J)
=\left(\prod_{\alpha=1}^{m^I+1}\Ll^{\Hh,I}(\mu^I_{\alpha}-\mu^I_{\alpha-1})\right)\\
\times\Ll_{IJ}(2+n,\mu^I_{0},\mu^J_{0})
\times\left(\prod_{\alpha=1}^{m^J+1}\Ll^{\Hh,J}(\mu^J_{\alpha}-\mu^J_{\alpha-1})\right)
\end{multline*}
where $\imath$ is the unique morphism in $\widehat\Delta(-1,n)$. 
We define the morphisms in $C$ by
\[
C\bigl((m^I,m,m^J),(n^I,n,n^J)\bigr)=
\coprod_{\mu^I,\mu,\mu^J}
C(\mu^I,\mu,\mu^J)\,.
\]
We define the categories $\widehat C_1$ and $C_3$ by the pullback diagrams
\[
\xymatrix{
	\widehat C_1\ar[r]\ar[d]&\Delta\times\widehat\Delta\times\Delta\ar[d]\\
	C\ar[r]&\widetilde\Delta\times\widehat\Delta\times\widetilde\Delta}
\qquad\xymatrix{
	C_3\ar[r]\ar[d]&\widetilde\Delta\times\widetilde\Delta\times\widetilde\Delta\ar[d]\\
	C\ar[r]&\widetilde\Delta\times\widehat\Delta\times\widetilde\Delta}
\]
Also let $C_2\subset C_3$ be the full subcategory
whose objects are the triples of integers $(n^I,n,n^J)$ with
$n^I,n^J\geq -1$ and $n\geq 0$, let $C_1=C_3\cap\widehat C_1$ 
and let $C_0=C_2\cap C_1$. Notice that we have equivalences of 
categories $C_2\simeq\widetilde\Delta\times\Delta\times\widetilde\Delta$,
$C_1\simeq\Delta\times\widetilde\Delta\times\Delta$ and
$C_0\simeq\Delta\times\Delta\times\Delta$.

We now define the functor $\mathfrak F_2\colon C_2^\op\to\Top$.
On objects,
$\mathfrak F_2(n^I,n,n^J)=\widetilde\Bb^\Hh_I(n^I)\times(\Mm_\emptyset^\Hh)^n\times\widetilde\Bb^\Hh_J(n^J)$
and the functor $\mathfrak F_2$ is defined on morphisms as follows:
given morphisms $\mu^I\in\widetilde\Delta(m^I,n^I)$, 
$\mu\in\Delta(m,n)$, $\mu^J\in\widetilde\Delta(m^J,n^J)$,
and $f\in C(\mu^I,\mu,\mu^J)$, we can write, for $n^I,n^J\neq-1$,
\begin{multline*}
\mathfrak F_2(n^I,n,n^J)\\ \xrightarrow{\text{shuffle}}
\bigl(\tprod_{I}\!\Mm_\emptyset^\Hh\bigr)^{n^I-\mu^I_0}
\times\tprod_I\Mm_x^\Hh\times
\Bigl((\Mm_\emptyset^\Hh)^{1+\mu_0}\times
\bigl(\tprod_I\Mm_\emptyset^\Hh\bigr)^{\mu^I_0}\Bigr)
\times(\Mm_\emptyset^\Hh)^{n-\mu_0}
\times\widetilde\Bb^\Hh_J(n^J)\\
\xrightarrow{f}\mathfrak F_2(m^I,m,m^J)
\end{multline*}
where $f$ acts on the product as in
Definition~\ref{con:barconstruction} and
equation~\eqref{eq:LIJ(a,b,c)acting}.
If either $n^I$ or $n^J$ equal $-1$, $f$ acts using the structure
of $\Mm_I^\Hh$ and $\Mm_J^\Hh$ as modules over the $\Ll^\Hh$-algebra
$\Mm_\emptyset^\Hh$ (Remark~\ref{rmk:MImoduleoverMJ}).

We define the functor $\mathfrak F_0$ as the restriction of
$\mathfrak F_2$ to $C_0^\op$; then the inclusion $C_0\subset C_2$ induces a map $\|\FF_0\|\to\|\FF_2\|$.

We now define the homotopy equivalences $g_0$ and $g_2$. Let $\Delta_0^{\Hh_\emptyset}=\Delta(\Ll_+^{\Hh_\emptyset},\Ll^{\Hh_\emptyset})$
(see Definition~\ref{con:catDelta(P)}) and
consider the functor $F\colon\widetilde\Delta^{\Hh_{IJ}}_I\times\Delta_0^{\Hh_\emptyset}\times\widetilde\Delta^{\Hh_{IJ}}_J\to C_2$
(see Definition~\ref{def:Delta_I,Bb_I,h_I})
which is the identity on objects and is induced on morphisms 
as folows: 
given morphisms $\mu^I\in\widetilde\Delta(m^I,n^I)$, 
$\mu\in\Delta(m,n)$ and $\mu^J\in\widetilde\Delta(m^J,n^J)$,
the map
\[
\widetilde\Delta^{\Hh_{IJ}}_I\times\Delta_0^{\Hh_\emptyset}\times\widetilde\Delta^{\Hh_{IJ}}_J\to C_2
\]
is given by the maps
\begin{align*}
\prod_{\alpha=1}^{m^I+1}\Ll^{\Hh_{IJ},I}(\mu^I_\alpha-\mu^I_{\alpha-1})
&\to
\prod_{\alpha=1}^{m^I+1}\Ll^{\Hh,I}(\mu^I_\alpha-\mu^I_{\alpha-1})\\
\prod_{\alpha=1}^{m+1}\Ll^{\Hh_{\emptyset}}(\mu_\alpha-\mu_{\alpha-1})
&\to
\prod_{\alpha=1}^{m+1}\Ll^{\Hh}(\mu_\alpha-\mu_{\alpha-1})\\
\prod_{\alpha=1}^{m^J+1}\Ll^{\Hh_{IJ},J}(\mu^J_\alpha-\mu^J_{\alpha-1})
&\to
\prod_{\alpha=1}^{m^J+1}\Ll^{\Hh,J}(\mu^J_\alpha-\mu^J_{\alpha-1})
\end{align*}
induced by the the canonical maps $i:\Ll^{\Hh_{IJ}}\to\Ll^{\Hh}$ and $i_\emptyset\colon\Ll^{\Hh_\emptyset}\to\Ll^\Hh$, and
by maps
\begin{align*}
\Ll_I^{\Hh_{IJ}}(\mu^I_0)\times\Ll^{\Hh_\emptyset}(1+\mu_0)
&\to\Ll_{IJ}^\Hh(1+\mu_0,\mu^I_0,0)\\
\Ll^{\Hh_\emptyset}(1+n-\mu_m)\times\Ll_J^{\Hh_{IJ}}(\mu^J_0)
&\to\Ll_{IJ}^\Hh(1+n-\mu_m,0,\mu^J_0)
\end{align*}
which we now define.
Using matrix notation,
the image by the first map 
of a pair of isometries
\begin{align*}
\left[\ g_{IJ}\quad h_{IJ}\ \right]&:\Hh_{IJ}\oplus \bigl((\Hh_{IJ})^I\bigr)^{\mu^I_0}\to\Hh_{IJ}\,,&
\left[\ g_\emptyset\quad f_\emptyset\ \right]&:\Hh_\emptyset\oplus(\Hh_\emptyset)^{\mu_0}\to\Hh_\emptyset
\end{align*}
is the isometry
\[
\left[\ g_\emptyset\otimes g_{IJ}\quad f_\emptyset\otimes g_{IJ}\quad g_\emptyset\otimes h_{IJ}\ \right]:
\Hh\oplus\Hh^{\mu_0}\oplus (\Hh^I)^{\mu^I_0}\to\Hh\,,
\]
where we identified
$\Hh^{\mu_0}$ with $(\Hh_\emptyset)^{\mu_0}\otimes\Hh_{IJ}$ and
$(\Hh^I)^{\mu^I_0}$ with $\Hh_\emptyset\otimes\bigl((\Hh_{IJ})^I\bigr)^{\mu^I_0}$.
The second map is defined in an analogous way:
the image of a pair of isometries
\begin{align*}
\left[\ g_\emptyset\quad f_\emptyset\ \right]&:\Hh_\emptyset\oplus(\Hh_\emptyset)^{n-\mu_m}\to\Hh_\emptyset\,, &
\left[\ g_{IJ}\quad h_{IJ}\ \right]&:\Hh_{IJ}\oplus \bigl((\Hh_{IJ})^J\bigr)^{\mu^J_0}\to\Hh_{IJ}
\end{align*}
is the isometry
\[
\left[\ g_\emptyset\otimes g_{IJ}\quad f_\emptyset\otimes g_{IJ}\quad g_\emptyset\otimes h_{IJ}\ \right]:
\Hh\oplus\Hh^{n-\mu_m}\oplus (\Hh^J)^{\mu^J_0}\to\Hh\,.
\]
Now, a direct verification shows that we have a homeomorphism of $\Delta_0^{\Hh_\emptyset}$-diagrams
\[
\hcolim_{(\widetilde\Delta_I^{\Hh_{IJ}}\times\widetilde\Delta_J^{\Hh_{IJ}})^\op}F^*\FF_2^\Hh\cong
\Bb_\bullet^{\Hh_\emptyset}\bigl(\|F_i^*\widetilde\Bb_I^\Hh\|,\Mm_\emptyset^\Hh,\|F_i^*\widetilde\Bb_J^\Hh\|\bigr), 
\]
and since $F$ is an equivalence of categories, by
\cite[Proposition 3.1(6)]{HoVo92} and
the commutation of homotopy colimits 
(see \cite[section 6]{HoVo92}), we get
\begin{align*}
\hcolim_{C_2^\op}\FF_2
&\simeq\hcolim_{(\widetilde\Delta_I^{\Hh_{IJ}}\times\Delta_0^{\Hh_\emptyset}\times\Delta_J^{\Hh_{IJ}})^\op}F^*\FF_2\\
&\cong\hcolim_{(\Delta^{\Hh_\emptyset}_0)^\op}\left(\hcolim_{(\widetilde\Delta_I^{\Hh_{IJ}}\times\Delta_J^{\Hh_{IJ}})^\op}F^*\FF_2\right)\\
&\cong\hcolim_{(\Delta^{\Hh_\emptyset}_0)^\op}\Bb_\bullet^{\Hh_\emptyset}(\|F_i^*\widetilde\Bb_I^\Hh\|,\Mm_\emptyset^\Hh,\|F_i^*\widetilde\Bb_J^\Hh\|)
\end{align*}
Therefore we have a commutative diagram in $\Top$:
\begin{equation}\label{eq:trapezetrianglediagram}
\xymatrix@C=4em{
\Bb^{\Hh_\emptyset}(\|F_i^*\widetilde\Bb_I^\Hh\|,\Mm_\emptyset^\Hh,\|F_i^*\widetilde\Bb_J^\Hh\|)&
\displaystyle\|\hcolim_{(\widetilde\Delta_I^{\Hh_{IJ}}\times\widetilde\Delta_J^{\Hh_{IJ}})^\op}F^*\FF_2\|\ar[r]^-{F}_-\simeq\ar[l]^-{\cong}&
\|\FF_2\|\\
\Bb^{\Hh_\emptyset}(\Mm_I^\Hh,\Mm_\emptyset^\Hh,\Mm_J^\Hh)\ar[ru]\ar[u]^-{\simeq}\ar[r]^-{i_\emptyset}_-{\simeq}&
\Bb^{\Hh}(\Mm_I^\Hh,\Mm_\emptyset^\Hh,\Mm_J^\Hh)\ar[ru]_-{\simeq}
}
\end{equation}
where the right diagonal map is induced by the functor $\Delta_0^{\Hh}\to C_2$ which sends $n$ to $(-1,n,-1)$.
We define $g_2$ as the top row in diagram~\ref{eq:trapezetrianglediagram}.
The map $g_0$ is defined by taking the restriction of $g_2$ to 
$\Bb^{\Hh_\emptyset}(\|F_i^*\Bb_I^\Hh\|,\Mm_\emptyset^\Hh,\|F_i^*\Bb_J^\Hh\|)$ and observing that it factors through $\|\FF_0\|$.

We now construct functors 
$\widehat\FF_1\colon \widehat C_1^\op\to\Top$ and $\FF_3\colon C_3^\op\to\Top$.
On objects $\FF_3(n^I,n,n^J)=\FF_2(n^I,n,n^J)$ for $n\neq-1$,
\[
\FF_3(n^I,-1,n^J)=\begin{cases}
\widetilde\Bb^\Hh_I(n^I)\times\bigl(\prod_J\Mm_\emptyset^\Hh\bigr)^{n^J}\times\bigl(\prod_{x\in J}\Mm^\Hh_x\bigr),&\text{if $n^J\neq-1$};\\
\bigl(\prod_{x\in I}\Mm^\Hh_x\bigr)\times\bigl(\prod_I\Mm_\emptyset^\Hh\bigr)^{n^I}\times\widetilde\Bb^\Hh_J(n^J)&\text{if $n^I\neq-1$}
\end{cases}
\]
and $\FF_3(-1,-1,-1)=\Mm_{I\cup J}^\Hh$. On objects, the functor $\widehat\FF_1$ is defined by
$\widehat \FF_1(n^I,n,n^J)=\FF_3(n^I,n,n^J)$ for
$(n^I,n,n^J)\in \widehat C_1$.
The functors are defined on morphisms in the usual way. We also define $\FF_1$ as the restriction of $\FF_3$ to $C_1^\op$
(which coincides with the restriction of $\widehat\FF_1$ to $C_1^\op$). Since, by Lemma~\ref{lemma:widehatdelta}, the inclusions $\widetilde\Delta\to\widehat\Delta$ 
and $\overline\Delta\to\widehat\Delta$ are right cofinal, it follows
that the inclusions $C_0\to\widehat C_1$ and $C_1\to\widehat C_1$ are also right cofinal and hence the maps $\|\FF_0\|\to\|\widehat\FF_1\|$ 
and $\|\FF_1\|\to\|\widehat\FF_1\|$ are homotopy equivalences.

The map $\widetilde\Bb^\Hh(\Mm_I^\Hh,\Mm_\emptyset^\Hh,\Mm_J^\Hh;\Mm_{I\cup J}^\Hh)\to\|\FF_3\|$ is induced by the functor
$\widetilde\Delta\to\widetilde\Delta\times\widetilde\Delta\times\widetilde\Delta$ given on objects by $n\mapsto(-1,n,-1)$
(compare with the right diagonal map in diagram~\ref{eq:trapezetrianglediagram}).

We now define a diagonal functor $d\colon\widetilde\Delta_{I\cup J}\to C_3$, given
on objects by $n\mapsto(n,-1,n)$; to define
$d$ on morphisms just observe that, for any $\mu\in\widetilde\Delta(m,n)$,
the spaces of morphisms $\widetilde\Delta_{I\cup J}(\mu)$ and $C_3(\mu,\Id,\mu)$ are
isomorphic (see diagram~\eqref{eq:LIJdiagnoal=LIcupJ}). Now, direct inspection shows that
$\widetilde\Bb_{I\cup J}=d^*\FF_3$. Restricting $d$ we get a functor $d\colon\Delta_{I\cup J}\to C_1$ and we also have
$\Bb_{I\cup J}=d^*\FF_1$. We claim that
the map $\|d^*\FF_1\|\xrightarrow{d}\|\widehat\FF_1\|$ induced by $d$ and the inclusion $\FF_1\to\widehat\FF_1$
is a homotopy equivalence. This follows from the pullback rectangle in the diagram:
\[\xymatrix@C=5em{
\Delta_{I\cup J}\ar[rr]^-{d}\ar[d]_-{\simeq}&&
\widehat C_1\ar[r]^-{\widehat\FF_1}\ar[d]_-{\simeq}&\Top\\
\Delta\ar[r]^-{\text{diag}}&
\Delta\times\Delta\ar[r]^-{\ident\times[-1]\times\ident}&
\Delta\times\widehat\Delta\times\Delta
}\]
Since the bottom arrows are cofinal and the vertical arrows are equivalences,
it follows that $d$ is cofinal and hence $\|d^*\FF_1\|\xrightarrow{d}\|\widehat\FF_1\|$
is a homotopy equivalence. In an analogous way we can see that the map $\|d^*\FF_3\|\xrightarrow{d}\|\FF_3\|$
is a homotopy equivalence, which finishes the proof.
\end{proof}

\vspace*{1em}
\noindent {\it Aknowledgements.}
The author would like to thank the referee for his careful reading of the manuscript and his comments, which greatly improved the readability of the text,
and for pointing out that the proof of Theorem 1.1 was incomplete.
The author would also like to thank Gustavo Granja for innumerous very fruitful discussions.

\bibliography{bar}

\providecommand{\bysame}{\leavevmode\hbox to3em{\hrulefill}\thinspace}
\providecommand{\MR}{\relax\ifhmode\unskip\space\fi MR }
\providecommand{\MRhref}[2]{%
  \href{http://www.ams.org/mathscinet-getitem?mr=#1}{#2}
}
\providecommand{\href}[2]{#2}
\begin{thebibliography}{10}

\bibitem{Ang09}
Vigleik Angeltveit, \emph{The cyclic bar construction on {$A_\infty$}
  {$H$}-spaces}, Adv. Math. \textbf{222} (2009), no.~5, 1589--1610.

\bibitem{AtJo78}
M.~F. Atiyah and J.~Jones, \emph{Topological aspects of {Y}ang-{M}ills theory},
  Comm. Math. Phys. \textbf{61} (1978), no.~2, 97--118.

\bibitem{BeMo09}
Clemens Berger and Ieke Moerdijk, \emph{On the derived category of an algebra
  over an operad}, Georgian Math. J. \textbf{16} (2009), no.~1, 13--28.

\bibitem{Ber10}
Bruce~C. Berndt, \emph{What is a {$q$}-series?}, Ramanujan rediscovered,
  Ramanujan Math. Soc. Lect. Notes Ser., vol.~14, Ramanujan Math. Soc., Mysore,
  2010, pp.~31--51.

\bibitem{BoVo68}
J.~M. Boardman and R.~M. Vogt, \emph{Homotopy-everything {$H$}-spaces}, Bull.
  Amer. Math. Soc. \textbf{74} (1968), 1117--1122.

\bibitem{BHMM93}
Charles~P. Boyer, J.~C. Hurtubise, B.~M. Mann, and R.~James Milgram, \emph{The
  topology of instanton moduli spaces. {I}. {T}he {A}tiyah-{J}ones conjecture},
  Ann. of Math. (2) \textbf{137} (1993), no.~3, 561--609.

\bibitem{BM88}
Charles~P. Boyer and Benjamin~M. Mann, \emph{Homology operations on
  instantons}, J. Differential Geom. \textbf{28} (1988), no.~3, 423--465.
  \MR{965223}

\bibitem{BrSa97}
J.~Bryan and M.~Sanders, \emph{The rank stable topology of instantons of
  {$\overline{\bf C{\rm P}}{}^2$}}, Proc. Amer. Math. Soc. \textbf{125} (1997),
  no.~12, 3763--3768.

\bibitem{BrSa00}
\bysame, \emph{Instantons on ${{S}}\sp 4$ and $\overline{{\bf {{C}}}{\rm
  {{P}}}}{}\sp 2$, rank stabilization, and {B}ott periodicity}, Topology
  \textbf{39} (2000), no.~2, 331--352.

\bibitem{Buc93}
N.~Buchdahl, \emph{Instantons on $n\mathbb{CP}^2$}, J. Differential Geom.
  \textbf{37} (1993), no.~3, 669--687.

\bibitem{Buc00}
\bysame, \emph{Blowups and gauge fields}, Pacific J. Math. \textbf{196} (2000),
  no.~1, 69--111.

\bibitem{Buc04}
\bysame, \emph{Monads and bundles on rational surfaces}, Rocky Mountain J.
  Math. \textbf{34} (2004), no.~2, 513--540.

\bibitem{CoMi94}
Ralph~L. Cohen and R.~James Milgram, \emph{The homotopy type of gauge-theoretic
  moduli spaces}, Algebraic topology and its applications, Math. Sci. Res.
  Inst. Publ., vol.~27, Springer, New York, 1994, pp.~15--55. \MR{1268186}

\bibitem{Don84}
S.~K. Donaldson, \emph{Instantons and geometric invariant theory}, Comm. Math.
  Phys. \textbf{93} (1984), no.~4, 453--460.

\bibitem{Don86}
\bysame, \emph{Connections, cohomology and the intersection forms of
  {$4$}-manifolds}, J. Differential Geom. \textbf{24} (1986), no.~3, 275--341.

\bibitem{DoKr90}
S.~K. Donaldson and P.~Kronheimer, \emph{The geometry of four-manifolds}, The
  Clarendon Press Oxford University Press, New York, 1990, Oxford Science
  Publications.

\bibitem{DFa04}
Emmanuel Dror~Farjoun, \emph{Fundamental group of homotopy colimits}, Adv.
  Math. \textbf{182} (2004), no.~1, 1--27. \MR{2028495}

\bibitem{Fre09}
Benoit Fresse, \emph{Modules over operads and functors}, Lecture Notes in
  Mathematics, vol. 1967, Springer-Verlag, Berlin, 2009.

\bibitem{Fri98}
R.~Friedman, \emph{Algebraic surfaces and holomorphic vector bundles},
  Springer-Verlag, New York, 1998.

\bibitem{Gas08}
Elizabeth Gasparim, \emph{The {A}tiyah-{J}ones conjecture for rational
  surfaces}, Adv. Math. \textbf{218} (2008), no.~4, 1027--1050.

\bibitem{GiKa94}
Victor Ginzburg and Mikhail Kapranov, \emph{Koszul duality for operads}, Duke
  Math. J. \textbf{76} (1994), no.~1, 203--272.

\bibitem{Hen14}
Amar~Abdelmoubine Henni, \emph{Monads for framed torsion-free sheaves on
  multi-blow-ups of the projective plane}, Internat. J. Math. \textbf{25}
  (2014), no.~1, 1450008, 42.

\bibitem{HoVo92}
J.~Hollender and R.~Vogt, \emph{Modules of topological spaces, applications to
  homotopy limits and {$E\sb \infty$} structures}, Arch. Math. (Basel)
  \textbf{59} (1992), no.~2, 115--129.

\bibitem{HuMi95}
J.~C. Hurtubise and R.~James Milgram, \emph{The {A}tiyah-{J}ones conjecture for
  ruled surfaces}, J. Reine Angew. Math. \textbf{466} (1995), 111--143.

\bibitem{HuLe95}
D.~Huybrechts and M.~Lehn, \emph{Stable pairs on curves and surfaces}, J.
  Algebraic Geom. \textbf{4} (1995), no.~1, 67--104.

\bibitem{Kin89}
A.~King, \emph{Instantons and holomorphic bundles on the blown-up plane}, Ph.D.
  thesis, Worcester College, Oxford, 1989.

\bibitem{Kir84}
Frances~Clare Kirwan, \emph{Cohomology of quotients in symplectic and algebraic
  geometry}, Mathematical Notes, vol.~31, Princeton University Press,
  Princeton, NJ, 1984.

\bibitem{Kir94}
\bysame, \emph{Geometric invariant theory and the {A}tiyah-{J}ones conjecture},
  The Sophus Lie Memorial Conference (Oslo, 1992), Scand. Univ. Press, Oslo,
  1994, pp.~161--186.

\bibitem{Kob87}
Shoshichi Kobayashi, \emph{Differential geometry of complex vector bundles},
  Publications of the Mathematical Society of Japan, vol.~15, Princeton
  University Press, Princeton, NJ; Princeton University Press, Princeton, NJ,
  1987, Kan\^o Memorial Lectures, 5.

\bibitem{Lub93}
M.~L{\"u}bke, \emph{The analytic moduli space of framed vector bundles}, J.
  Reine Angew. Math. \textbf{441} (1993), 45--59.

\bibitem{MSS02}
Martin Markl, Steve Shnider, and Jim Stasheff, \emph{Operads in algebra,
  topology and physics}, Mathematical Surveys and Monographs, vol.~96, American
  Mathematical Society, Providence, RI, 2002.

\bibitem{Mat00}
Andreas Matuschke, \emph{On framed instanton bundles and their deformations},
  Math. Nachr. \textbf{211} (2000), 109--126.

\bibitem{May72}
J.~P. May, \emph{Geometry of iterated loop spaces}, Lecture notes in
  mathematics, Springer, 1972.

\bibitem{May77}
\bysame, \emph{{$E\sb{\infty }$} ring spaces and {$E\sb{\infty }$} ring
  spectra}, Springer-Verlag, Berlin, 1977, With contributions by Frank Quinn,
  Nigel Ray, and J\o rgen Tornehave, Lecture Notes in Mathematics, Vol. 577.

\bibitem{MaPo12}
J.~P. May and K.~Ponto, \emph{More concise algebraic topology}, Chicago
  Lectures in Mathematics, University of Chicago Press, Chicago, IL, 2012,
  Localization, completion, and model categories.

\bibitem{MMR94}
John~W. Morgan, Tomasz Mrowka, and Daniel Ruberman, \emph{The {$L^2$}-moduli
  space and a vanishing theorem for {D}onaldson polynomial invariants},
  Monographs in Geometry and Topology, II, International Press, Cambridge, MA,
  1994.

\bibitem{NaRa61}
M.~S. Narasimhan and S.~Ramanan, \emph{Existence of universal connections},
  Amer. J. Math. \textbf{83} (1961), 563--572. \MR{0133772}

\bibitem{ViFu04}
S.~P. Novikov and V.~A. Rokhlin (eds.), \emph{Topology. {II}}, Encyclopaedia of
  Mathematical Sciences, vol.~24, Springer-Verlag, Berlin, 2004, Homotopy and
  homology. Classical manifolds. \MR{2054455}

\bibitem{Qui73}
Daniel Quillen, \emph{Higher algebraic {$K$}-theory. {I}}, Algebraic
  {$K$}-theory, {I}: {H}igher {$K$}-theories ({P}roc. {C}onf., {B}attelle
  {M}emorial {I}nst., {S}eattle, {W}ash., 1972), Lecture Notes in Mathematics,
  vol. 341, Springer, Berlin, 1973, pp.~85--147.

\bibitem{San95}
M.~Sanders, \emph{Classifying spaces and {D}irac operators coupled to
  instantons}, Trans. Amer. Math. Soc. \textbf{347} (1995), no.~10, 4037--4072.

\bibitem{San02}
J.~Santos, \emph{Topology of moduli spaces of rank stable instantons and
  holomorphic bundles}, Ph.D. thesis, Stanford University, 2002.

\bibitem{San05}
\bysame, \emph{Framed holomorphic bundles on rational surfaces}, J. Reine
  Angew. Math. \textbf{589} (2005), 129--158.

\bibitem{Seg68}
Graeme Segal, \emph{Classifying spaces and spectral sequences}, Inst. Hautes
  \'Etudes Sci. Publ. Math. (1968), no.~34, 105--112.

\bibitem{Seg74}
\bysame, \emph{{Categories and cohomology theories.}}, {Topology} \textbf{13}
  (1974), 293--312 (English).

\bibitem{Sta12}
Richard~P. Stanley, \emph{Enumerative combinatorics. {V}olume 1}, second ed.,
  Cambridge Studies in Advanced Mathematics, vol.~49, Cambridge University
  Press, Cambridge, 2012.

\bibitem{Tau82}
Clifford~Henry Taubes, \emph{Self-dual {Y}ang-{M}ills connections on
  non-self-dual {$4$}-manifolds}, J. Differential Geom. \textbf{17} (1982),
  no.~1, 139--170.

\bibitem{Tau84}
\bysame, \emph{Self-dual connections on {$4$}-manifolds with indefinite
  intersection matrix}, J. Differential Geom. \textbf{19} (1984), no.~2,
  517--560.

\bibitem{Tau89}
\bysame, \emph{The stable topology of self-dual moduli spaces}, J. Differential
  Geom. \textbf{29} (1989), no.~1, 163--230.

\bibitem{Ytian94}
Youliang Tian, \emph{The based {${\rm SU}(n)$}-instanton moduli spaces}, Math.
  Ann. \textbf{298} (1994), no.~1, 117--139.

\bibitem{Vog71}
Rainer~M. Vogt, \emph{Convenient categories of topological spaces for homotopy
  theory}, Arch. Math. (Basel) \textbf{22} (1971), 545--555.

\end{thebibliography}

\end{document}